\documentclass[a4paper,oneside,10pt,draft]{amsart}

\usepackage{amssymb}
\usepackage[mathscr]{eucal}
\usepackage[matrix,arrow]{xy}
\usepackage{version}
\usepackage{color}
\usepackage{amscd}

\usepackage{xr}

\newenvironment{NB}{
\color{red}{\bf NB}. \footnotesize 
}{}
\excludeversion{NB}
\newenvironment{NB2}{
\color{blue} 
}{}

\def\cal{\mathcal}
\def\Bbb{\mathbb}
\def\frak{\mathfrak}

\newcommand{\GL}  {\operatorname{GL}}

\newcommand{\PGL} {\operatorname{PGL}}

\newcommand{\Aut}{\operatorname{Aut}}
\newcommand{\Coh}{\operatorname{Coh}}
\newcommand{\Cone}{\operatorname{Cone}}

\newcommand{\Ext}{\operatorname{Ext}}
\newcommand{\Hilb}{\operatorname{Hilb}}
\newcommand{\Hom}{\operatorname{Hom}}

\newcommand{\NS}{\operatorname{NS}}
\newcommand{\Per}{\operatorname{Per}}   
\newcommand{\Pic}{\operatorname{Pic}}
\newcommand{\Quot}{\operatorname{Quot}}
\newcommand{\Spec}{\operatorname{Spec}}
\newcommand{\Supp}{\operatorname{Supp}}
\newcommand{\WIT}{\operatorname{WIT}}   
\newcommand{\Stab}{\operatorname{Stab}}
\newcommand{\chr}{\operatorname{char}}

\newcommand{\alg}{\operatorname{alg}}
\newcommand{\ch}{\operatorname{ch}}
\newcommand{\codim}{\operatorname{codim}}
\newcommand{\coker}{\operatorname{coker}}
\newcommand{\denom}{\operatorname{d}}
\newcommand{\exc}{\operatorname{Exc}}
\newcommand{\id}{\operatorname{id}}
\newcommand{\im}{\operatorname{im}}
\newcommand{\rk}{\operatorname{rk}}
\newcommand{\td}{\operatorname{td}}

\newcommand{\Def}{\operatorname{Def}}

\theoremstyle{plain}
 \newtheorem{thm}{Theorem}[subsection]
 \newtheorem{lem}[thm]{Lemma}
 \newtheorem{prop}[thm]{Proposition}
 \newtheorem{cor}[thm]{Corollary}
 \newtheorem{fct}[thm]{Fact}
\theoremstyle{definition}
 \newtheorem{defn}[thm]{Definition}
\theoremstyle{remark}
 \newtheorem{rem}[thm]{Remark}
 \newtheorem{ex}[thm]{Example}

\numberwithin{equation}{section}
\allowdisplaybreaks

\begin{document}


\title{Fourier-Mukai transforms and the wall-crossing behavior 
for Bridgeland's stability conditions}
\author{Hiroki Minamide}
\address{Department of Mathematics, Faculty of Science,
Kobe University,
Kobe, 657, Japan}
\email{minamide@math.kobe-u.ac.jp}

\author{Shintarou Yanagida}
\address{Research Institute for Mathematical Sciences,
Kyoto University,
Kyoto, 606-8502, Japan}
\email{yanagida@kurims.kyoto-u.ac.jp}

\author{K\={o}ta Yoshioka}
\address{Department of Mathematics, Faculty of Science,
Kobe University,
Kobe, 657, Japan}
\email{yoshioka@math.kobe-u.ac.jp}

\thanks{The second author is supported by JSPS Fellowships 
for Young Scientists (No.\ 21-2241, 24-4759). 
The third author is supported by the Grant-in-aid for 
Scientific Research (No.\ 22340010), JSPS}




\maketitle
\tableofcontents


\section{Introduction}

Let $X$ be a $K3$ surface or an abelian surface over a field ${\frak k}$.
In \cite{Br:2}, Bridgeland introduced the notion
of stability condition for objects in the bounded derived category ${\bf D}(X)$
of coherent sheaves on $X$.
It consists of a $t$-structure of ${\bf D}(X)$ and a stability function
on the heart. Bridgeland showed that the set of stability conditions
$\Stab(X)$ has a structure of complex manifold.
Then he studied several properties of $\Stab(X)$.
In particular, the non-emptiness of this space was shown
by constructing interesting examples
of stability conditions.
For $\beta \in \NS(X)_{\Bbb Q}$ and an ample
${\Bbb Q}$-divisor $\omega$,
the example consists of an abelian category 
${\frak A}_{(\beta,\omega)}$ 
which is a tilting of $\Coh(X)$ by a torsion pair and
a stability function $Z_{(\beta,\omega)}$ on it. 
The structure of ${\frak A}_{(\beta,\omega)}$ was 
studied further by Huybrechts \cite{H:category}.
In particular, it was shown that ${\frak A}_{(\beta,\omega)}$ 
is useful to study the usual Gieseker semi-stable sheaves.

In this paper, we shall slightly generalize Bridgeland's construction.
In particular, we shall relax the requirement on $\omega$ in \cite{Br:3}.
Then we study some basic properties of the categories.
We introduce a special parameter space of 
the category ${\frak A}_{(\beta,\omega)}$ and
study its chamber structure.  
Since $\sigma_{(\beta,\omega)}=({\frak A}_{(\beta,\omega)},Z_{(\beta,\omega)})$
is a stability condition, our parameter space
is a subspace of $\Stab(X)$.

In order to study the moduli space of Gieseker semi-stable sheaves,
it is important to study the behavior of Gieseker semi-stable sheaves under
Fourier-Mukai transforms.
One of the authors studied this problem 
in \cite{Stability} and \cite[sect. 2.7]{PerverseII}.
In this paper, we shall translate our previous results
into the theory of Bridgeland's stability conditions.
In particular, we shall discuss the relation between 
Bridgeland's stability and the twisted stability
under the so-called large volume limit, i.e., $(\omega^2) \gg 0$.
This relation was already discussed by Bridgeland \cite{Br:3},
Toda \cite{Toda}, Bayer \cite{Bayer} and Ohkawa \cite{O}, 
so our result is regarded as a supplement of theirs.
Independently similar supplements and generalizations are obtained by 
Kawatani \cite{Kawatani}, \cite{Kawatani2}, 
Lo and Qin \cite{Lo-Qin}.
In particular, Lo and Qin investigated the relation 
for arbitrary surfaces.

By the definition of Bridgeland's stability conditions, 
any Fourier-Mukai transform preserves a suitable stability
condition, but $(\omega^2)$ is replaced by $1/(\omega^2)$.
Thus this stability condition is far from Gieseker's stability, 
and we need to cross walls 
for Bridgeland's stability conditions. 
So the wall crossing behavior of stability conditions comes into the story.
Since the stability condition 
$\sigma_{(\beta,\omega)}=({\frak A}_{(\beta,\omega)},Z_{(\beta,\omega)})$ 
consists of the core of $t$-structure and the stability function,
there are two kinds of walls.
One changes the abelian category ${\frak A}_{(\beta,\omega)}$, 
and we call such one  
\emph{the wall for categories}.
The other one changes the stability in the fixed abelian category,
and we call such one \emph{the wall for stabilities}.

The wall crossing behavior for Bridgeland's stability condition
was studied by Arcara-Bertram \cite{AB} 
and Toda \cite{Toda} for $K3$ surfaces, and by
Lo-Qin for arbitrary surfaces.
In particular, Toda defined a counting invariant for moduli 
stack of semi-stable objects and
studied the wall crossing behavior of the invariants.
In this paper we will study the number of 
semi-stable objects defined over finite fields 
and its behavior under wall-crossings.
This is enough for our applications of the present study.

As an application, we shall explain previous results 
on the birational maps
induced by Fourier-Mukai transforms on abelian surfaces.
The previous works \cite{Y:birational} and \cite{YY} 
introduced operations to improve the unstability.
We shall show that these operations correspond to crossing
walls of Bridgeland's stability conditions.  
In particular, we shall recover 
the birational transforms of moduli spaces of
stable sheaves in \cite{Y:birational}.
\\

Let us explain the outline of this paper.

In \S\,\ref{sect:stability},
we introduce Bridgeland's stability conditions in our context
and study the wall/chamber structure for categories. 
After some preparation of notations in \S\,\ref{subsect:stability:prelim},
we recall basic notions for Bridgeland's stability 
in \S\,\ref{subsect:definitions}.
We will introduce the torsion pair using twisted (Gieseker) stability,
and define our stability conditions 
$\sigma_{(\beta,\omega)}=({\frak A}_{(\sigma,\beta)},Z_{(\sigma,\beta)})$ 
for $\beta \in \NS(X)_{\Bbb Q}$ and $\omega \in {\Bbb Q}_{>0} H$,
where $H$ is a fixed ample divisor on $X$.
In \S\,\ref{subsect:stability:ex}, 
we prove that $\sigma_{(\beta,\omega)}$ 
does satisfy the requirements of 
Bridgeland's stability condition 
(Proposition~\ref{prop:stability:beta-omega}).
The \S\,\ref{subsect:wall:category} is devoted to 
the study of wall/chamber structure for categories.
We shall introduce some wall/chamber structure on 
a suitable subspace ${\frak H}_{\Bbb R}$ of the space of stability conditions.
The wall corresponds to a $(-2)$-vector of Mukai lattice, 
and in each chamber the core ${\frak U}$ of the stability condition 
does not change.
In \S\,\ref{subsect:parameter-FM}, 
we mention the relation of Fourier-Mukai transform 
and Bridgeland's stability 
in terms of the language of Mukai vectors.
The final \S\,\ref{subsect:eta} gives techniques 
for the perturbation of stability conditions.

In \S\,\ref{sect:Gieseker}, 
we study the relation between Gieseker's stability and Bridgeland's stability.
The numerical conditions ($\star 1$)--($\star 3$) given 
in \S\,\ref{subsect:Gieseker:numcond} turn out to guarantee the equivalence 
between the two notions of stabilities, 
as shown in Proposition~\ref{prop:star-1}, Lemma~\ref{lem:star-3} 
and Proposition~\ref{prop:star-2}.
As a result, we can show Corollary~\ref{cor:Toda},
which states the equivalence of Gieseker's and Bridgeland's stabilities 
in the so-called \emph{large volume limit} $(\omega^2) \gg 0$.

The next \S\,\ref{sect:wall-crossing} is devoted to the analysis 
of wall-crossing behavior.
In \S\,\ref{subsect:wall:stability}, we introduce the wall 
for stabilities.
Its definition is an analogue of the usual walls for (Gieseker) stabilities.
In \S\,\ref{subsect:wall-crossing:category}, 
we study the wall-crossing behavior 
on walls for categories.
A special case of wall-crossing 
will be treated in \S\,\ref{subsect:wall-crossing:dmin}. 
\S\,\ref{subsect:wall-crossing:numbers} deals with 
the wall-crossing formula for numbers of semi-stable objects 
in the case when the defining field ${\frak k}$ is finite.

The final section shows some applications of our analysis 
for the case of abelian surfaces.
Theorem~\ref{thm:birational} 
gives the birational maps
induced by the Fourier-Mukai transform on an abelian surface.
This is a generalization of \cite[Thm. 1.1]{Y:birational}, 
where it was required that ${\frak k}={\Bbb C}$.
Our proof uses the relation of Gieseker's and Bridgeland's stabilities 
shown in \S\,\ref{sect:Gieseker},
the wall-crossing formula shown in \S\,\ref{sect:wall-crossing} 
and some estimates of wall-crossing terms prepared 
in \S\,\ref{subsect:wall=w_1} and \S\,\ref{subsect:abel:codim}.
Another application is shown in \S\,\ref{subsect:abel:Picard}, 
Corollary~\ref{cor:lagrange}.
Here we discuss the Picard groups of 
the fibers of Albanese maps of moduli spaces 
which are related by Fourier-Mukai transforms. 

The appendix \S\,\ref{sect:appendix} discusses 
the behavior of Bridgeland stabilities under field extension.
The result will be used in \S\,\ref{sect:stability}.
We also give some basic properties of the moduli spaces
of stable sheaves on abelian and $K3$ surfaces over any field.
In particular, we give a condition for the non-emptyness of the
moduli spaces, which ensure the existence of Fourier-Mukai transforms
in \S\,\ref{sect:abel}. 


\section{Bridgeland's stability conditions}
\label{sect:stability}


\subsection{Preliminaries}
\label{subsect:stability:prelim}
Let $f:{\cal X} \to S$ be an $S$-scheme over a scheme $S$.
For a point $s \in S$, $k(s)$ denotes the residue field
of $s$. Let ${\cal X}_s:={\cal X} \times_S \Spec(k(s))$ be the 
fiber of $s$. 
For a coherent sheaf $E$ on ${\cal X}$ which is flat over
$S$, we denone ${\cal E}_{| {\cal X}_s}$ by ${\cal E}_s$.
For a commutative ring $R$ and a morphism
$\Spec(R) \to S$, we also set
${\cal X}_R:={\cal X} \times_S \Spec(R)$ and
$E_R=E \otimes_S R$ denotes the pull-back of $E$ to
${\cal X}_R$.   

Let $X$ be an abelian surface or a $K3$ surface over a field ${\frak k}$.
Let us introduce the \emph{Mukai lattice} over $X$.
For the sake of convenience, we first assume that ${\frak k}={\Bbb C}$, 
and later we mention the modification in the case of arbitrary field.
Let $\varrho_X \in H^4(X,{\Bbb Z})$ be the fundamental class of $X$. 
We define a lattice structure $\langle \cdot,\cdot \rangle$
on $H^{ev}(X,{\Bbb Z}):=\bigoplus_{i=0}^2 H^{2i}(X,{\Bbb Z})$ by 
\begin{equation}\label{eq:mukai_pairing}
 \langle x,y \rangle:= (x_1,y_1)-(x_0 y_2+x_2 y_0),
\end{equation}
where $x=x_0+x_1+x_2 \varrho_X$ and $y=y_0+y_1+y_2 \varrho_X$ 
with $x_0,y_0 \in {\Bbb Z}=H^0(X,{\Bbb Z})$, 
$x_1,y_1 \in H^2(X,{\Bbb Z})$ and $x_2,y_2 \in {\Bbb Z}$.
For a coherent sheaf $E$ on $X$,
\begin{equation}\label{eq:mukai_vector}
\begin{split}
v(E):=&\ch(E) \sqrt{\td_X}\\
=&\rk E+c_1(E)+(\chi(E)-\varepsilon \rk E)\varrho_X \in H^{ev}(X,{\Bbb Z})
\end{split}
\end{equation}
is called the \emph{Mukai vector} of $E$, where $\varepsilon=0,1$ 
according as $X$ is an abelian surface or a $K3$ surface.
The Mukai vector of an object $E$ of ${\bf D}(X)$ is defined 
to be $\sum_{k}(-1)^k v(E^k)$, where $E$ is expressed 
as a bounded complex $(E^k)=(\cdots \to E^{-1} \to E^0 \to E^1 \to \cdots)$. 
Finally we define
\begin{align*}
 A^*_{\alg}(X) := {\Bbb Z} \oplus \NS(X) \oplus {\Bbb Z}\varrho_X.
\end{align*}
Then $ A^*_{\alg}(X) $ is a sublattice of $H^{ev}(X,{\Bbb Z})$, 
and $v(E)\in  A^*_{\alg}(X)$ for any element $E\in {\bf D}(X)$. 
We will call $(A^*_{\alg}(X), \langle \cdot,\cdot \rangle)$ 
the Mukai lattice for $X$.
In the case of ${\frak k}={\Bbb C}$, this lattice is 
sometimes denoted by $H^{*}(X,{\Bbb Z})_{\text{alg}}$ in literature. 

In the case where ${\frak k}$ is an arbitrary field,
we set $A^*_{\alg}(X)=\oplus_{i=0}^2 A^{i}_{\alg}(X)$ 
to be the quotient of the cycle group of $X$ 
by the algebraic equivalence.
Then we have $A^{0}_{\alg}(X)\cong {\Bbb Z}$,
$A^{1}_{\alg}(X)\cong \NS(X)$ and 
$A^{2}_{\alg}(X)\cong {\Bbb Z}$.
We denote the basis of $A^{2}_{\alg}(X)$ by $\varrho_X$,
and express an element $x\in A^{*}_{\alg}(X)$ 
by $x=x_0+x_1+x_2 \varrho_X$ 
with $x_0 \in {\Bbb Z}$, $x_1 \in \NS(X)$ and $x_2 \in {\Bbb Z}$.
The lattice structure of $A^*_{\alg}(X)$ is given 
by $\langle \cdot,\cdot\rangle$ with the same definition 
as \eqref{eq:mukai_pairing}.
The Mukai vector $v(E)\in A^*_{\alg}(X)$ is defined 
by \eqref{eq:mukai_vector} for a coherent sheaf $E$, 
and by $\sum_{k}(-1)^k v(E^k)$ for an object $E=(E^k)$ 
of ${\bf D}(X)$.
\\

Hereafter the base field ${\frak k}$ is arbitrary unless otherwise stated.
\\

For a ring extension $R \to R'$ and an $R$-module $M$,
we set $M_{R'}:=M \otimes_R R'$.
Let $E$ be an object of ${\bf D}(X)$.
$E^{\vee}:={\bf R}{\cal H}om_{{\cal O}_X} (E,{\cal O}_X)$ denotes 
the derived dual of $E$.
We denote the rank of $E$ by $\rk E$.
For a fixed nef and big divisor $H$ on $X$, 
$\deg(E)$ denotes the degree of $E$ with respect to $H$. 
For $G \in K(X)_{\Bbb Q}$ with $\rk G>0$,
we also define the twisted rank and the twisted degree by
\begin{align*}
\rk_G(E):=\rk (G^{\vee} \otimes E),\quad
\deg_G(E):=\deg(G^{\vee} \otimes E) 
\end{align*}
respectively.
Finally we define the twisted slope by 
$\mu_G(E):=\deg_G(E)/\rk_G(E)$, if $\rk E \ne 0$.
\\

We call an element $u \in A^*_{\alg}(X)$ \emph{(-2)-vector} 
if $\langle u^2\rangle:=\langle u,u\rangle =-2$.
For a $(-2)$-vector $u$,
\begin{equation}\label{eq:reflection}
\begin{matrix}
R_u: & A^*_{\alg}(X) & \to     & A^*_{\alg}(X)
\\
     & x             & \mapsto & x+\langle u,x \rangle u
\end{matrix}
\end{equation}
is the reflection by $u$.

Let $U$ be a complex of coherent sheaves on $X$ such that
$\Hom(U,U)={\frak k}$ and $\Hom(U,U[p])=0$ for $p \ne 0,2$.
Let $p_i:X \times X \to X$  ($i=1,2$) be the $i$-th projection.
We set 
\begin{align*}
{\bf E}:=\Cone(p_1^*(U) \otimes p_2^*(U^{\vee}) \to {\cal O}_\Delta).
\end{align*}
Then 
\begin{equation*}
\begin{matrix}
\Phi_U:& {\bf D}(X) & \to & {\bf D}(X)\\
& E & \mapsto & {\bf R}p_{1*}({\bf E} \otimes p_2^*(E))
\end{matrix}
\end{equation*}
is an equivalence and the quasi-inverse $\Phi_U^{-1}$ 
is given by
\begin{equation*}
\begin{matrix}
\Phi_U^{-1}:& {\bf D}(X) & \to & {\bf D}(X)\\
& F & \mapsto & {\bf R}p_{2*}({\bf E}^{\vee} \otimes p_1^*(F))[2].
\end{matrix}
\end{equation*}
Moreover $\Phi_U$ induces the $(-2)$-reflection $R_{v(U)}$ 
on the Mukai lattice.


\subsection{Several definitions}
\label{subsect:definitions}

In this subsection,
we shall explain several basic notions to define
Bridgeland's stability conditions.

\subsubsection{}
Let us recall the definition of stability conditions 
on triangulated categories, introduced in \cite[Definition 1.1]{Br:2}.

\begin{defn}\label{defn:bridgeland:stability}
A \emph{stability condition} on a triangulated category ${\frak T}$ 
consists of a group homomorphism $Z:K({\frak T})\to {\Bbb C}$ 
and full additive subcategories $P(\phi) \subset {\frak T}$ 
for all $\phi \in {\Bbb R}$,
satisfying the following conditions:
\begin{enumerate}
\item[(i)]
 For $E \in P(\phi) \setminus \{0\}$, 
 we have $Z(E)=m(E)\exp(\sqrt{-1}\pi\phi)$ 
 with some $m(E) \in {\Bbb R}_{>0}$.
\item[(ii)]
 $P(\phi+1)=P(\phi)[1]$ for all $\phi \in \Bbb R$.
\item[(iii)]
 If $\phi_{1}>\phi_{2}$ and $E_i \in P(\phi_i)$ ($i=1,2$),
 then $\Hom_{{\frak T}}(E_1,E_2)=0$.
\item[(iv)]
 For any $E \in {\frak T}\setminus\{0\}$. 
 we have a following collection of triangles
 \begin{align}\label{diag:HN}
 \xymatrix{
    0=E_0   \ar[rr]         &
  & E_1     \ar[dl] \ar[rr] &
  & E_2     \ar[r]  \ar[dl] & \cdots \ar[r]     
  & E_{n-1} \ar[rr]         &
  & E_n =E  \ar[dl]
 \\
                            & A_1 \ar[ul]^{[1]} 
  &                         & A_2 \ar[ul]^{[1]}
  &                         &                     
  &                         & A_n \ar[ul]^{[1]}
 }
 \end{align}
 such that $A_i \in P(\phi_i)$ with $\phi_1 > \phi_2 > \cdots >\phi_n$. 
\end{enumerate}
\end{defn}

In the above definition, \eqref{diag:HN} is unique up to isomorphism.
We will use the following notations: 
\begin{align}\label{eq:phi_maxmin}
\phi_{\max}(E) := \phi_1(E),\qquad 
\phi_{\min}(E) := \phi_n(E).
\end{align}

Given a stability condition $\sigma=(Z,P)$
as in Definition~\ref{defn:bridgeland:stability},
each subcategory $P(\phi)$ is abelian.
The non-zero objects of $P(\phi)$ are said to be 
\emph{semi-stable of phase $\phi$ with respect to $\sigma$},
and the simple objects of $P(\phi)$ are said to be \emph{stable}.
\\

Let us also recall an equivalent definition of stability condition,
given in \cite[Proposition 5.3]{Br:2}.
Before stating that, we need to prepare

\begin{defn}
Let ${\frak A}$ be an abelian category.
\begin{enumerate}
\item[(i)]
A \emph{stability function} on  ${\frak A}$ is a group homomorphism 
$Z:K({\frak A}) \to {\Bbb C}$ such that $Z(E) \in {\Bbb H}'$ 
for every $E \in {\frak A} \setminus \{0\}$, where 
${\Bbb H}':=
\{r e^{\sqrt{-1} \pi \phi} \mid 
  0<r,\ 0 <\phi \le 1\} \subset {\Bbb C}$.

\item[(ii)]
Given a stability function $Z$, 
the phase of non-zero object $E \in {\frak A}$ is defined to be 
$\phi(E) := (\arg Z(E))/\pi \in (0,1]$.

\item[(iii)]
A nonzero object $E \in {\frak A}$ is called 
\emph{semi-stable with respect to $Z$} 
if every nonzero subobject $A \subset E$ 
satisfies $\phi(A) \le \phi(E)$.

\item[(iv)]
A stability function $Z$ on ${\frak A}$ is said to have 
\emph{Harder-Narasimhan property} 
if every $E \in {\frak A}\setminus \{0\}$ admits a filtration 
\begin{equation*}
0 \subset E_1 \subset E_2 \subset \cdots \subset E_n=E
\end{equation*}
in ${\frak A}$ 
such that each $F_j:=E_j/E_{j-1}$ is a semi-stable object 
with respect to $Z$ 
and $\phi(F_1)>\phi(F_2)>\cdots>\phi(F_n)$.
\end{enumerate}
\end{defn}

Then \cite[Proposition 5.3]{Br:2} claims 
\begin{fct}
To give a stability condition on ${\frak T}$ is equivalent to 
giving a bounded $t$-structure on ${\frak T}$ and 
a stability function on its heart with Harder-Narasimhan property.
\end{fct}

We will denote a stability condition on ${\frak T}$ 
by a pair $\sigma=({\frak A},Z)$,
where ${\frak A}$ is the heart of the given bounded $t$-structure 
and $Z$ is the stability function on ${\frak A}$.
Hereafter we call a semi-stable object with respect to $\sigma$ 
by \emph{$\sigma$-semi-stable} object.

\subsubsection{}
In the remaining part of this subsection,
we shall explain the setting of categories and stability functions 
for defining certain Bridgeland's stability conditions. 

Let $X$ be an abelian surface or a $K3$ surface over ${\frak k}$, 
and $\pi:X \to Y$ a contraction of $X$ to 
a normal surface $Y$ over ${\frak k}$. 
We note that $\pi$ is an isomorphism 
if $X$ is an abelian surface.
If $\pi$ is not isomorphic, then
$Y$ has rational double points as singularities.
Let $H$ be the pull-back of an ample divisor on $Y$. 
We take $\beta \in \NS(X)_{\Bbb Q}$ such that 
$(\beta,D) \not \in {\Bbb Z}$ for every $(-2)$-curve $D$ 
with $(D,H)=0$.
Then the following proposition holds.

\begin{prop}[{\cite[Prop. 2.4.5]{PerverseI}}]\label{prop:perverse}
Assume that $\beta \in \NS(X)_{\Bbb Q}$ satisfies
$(\beta,D) \not \in {\Bbb Z}$ for all $(-2)$-curves $D$ 
on $X \otimes_{\frak k} \overline{\frak k}$ with $(D,H)=0$.
Then there is a category of perverse coherent sheaves ${\frak C}$
such that $\langle e^\beta,v(E) \rangle<0$ for all 0-dimensional objects
$E$ of ${\frak C}$. 
\end{prop}

\begin{NB}
For a suitable field extension ${\frak k} \subset {\frak k}'$,
all $(-2)$-curves $D$ with $(D,H)=0$ are defined over ${\frak k}'$.
Thus replacing $X$ by $X \otimes_{\frak k} {\frak k}'$, we may assume that
all $(-2)$-curves $D$ with $(D,H)=0$ are subscheme of $X$.
\end{NB}


If $\pi$ is isomorphic, then ${\frak C}$ is nothing but $\Coh(X)$.

Let $r_0$ be a positive integer such that $r_0 e^\beta$
is a primitive element of $A^*_{\alg}(X)$.   
Let $G$ be an element of $K(X)_{\Bbb Q}$ such that
$v(G)=r_0 e^\beta-a \varrho_X$, $a \in {\Bbb Q}$.

\begin{defn}[{\cite[sect. 1.4]{PerverseI}}]\label{defn:twisted-stability}

\begin{enumerate}
\item[(1)]
Let $E$ be an object of ${\frak C}$.
We set $E(n):=E(nH)$ for the fixed $H$.

\begin{enumerate}
\item
Assume that $\rk E>0$.
Then $E$ is \emph{$G$-twisted semi-stable} 
if
\begin{align*}
\chi(G,F(n)) \leq (\rk F)\frac{\chi(G,E(n))}{\rk E},\quad n \gg 0
\end{align*}
for any proper subobject $F$ of $E$.
If $E$ the inequality is strict for every $F$, 
$E$ is \emph{$G$-twisted stable}. 

\item
Assume that $\rk E=0$ and $(c_1(E),H)>0$.
Then $E$ is \emph{$G$-twisted semi-stable} 
if
$$
\chi(G,F) \leq (c_1(F),H)\frac{\chi(G,E)}{(c_1(E),H)}
$$
for any proper subobject $F$ of $E$.
If the inequality is strict for every $F$, 
$E$ is \emph{$G$-twisted stable}. 
\end{enumerate}

\item[(2)]
We define the $\beta$-twisted semi-stability as the 
${\cal O}(\beta)$-twisted semi-stability.

\item[(3)]
For $v \in A^*_{\alg}(X)$,
${\cal M}_H^G(v)^{ss}$
is the moduli stack of $G$-twisted semi-stable
objects $E$ of ${\frak C}$ with $v(E)=v$.
We also define $\mu$-semi-stability by using the slope $\mu_G$.
${\cal M}_H(v)^{\text{$\mu$-$ss$}}$ denotes the moduli stack of 
$\mu$-semi-stable objects $E$ with $v(E)=v$.
${\cal M}_H^G(v)^s$ is the open substack of 
${\cal M}_H^G(v)^{ss}$ consisting of $E$ such that
$E \otimes \overline{\frak k}$ 
is $G$-twisted stable.
We also define ${\cal M}_H^\beta(v)^{ss}$ and
${\cal M}_H^\beta(v)^{s}$ in a similar way.
\item[(4)]
For $v \in A^*_{\alg}(X)$,
$\overline{M}_H^G(v)$
is the moduli scheme of $G$-twisted semi-stable
objects $E$ of ${\frak C}$ with $v(E)=v$
and $M_H^G(v)$ the open subscheme of $G$-twisted stable objects.
We also define $\overline{M}_H^\beta(v)$ and
$M_H^\beta(v)$ in a similar way.
\end{enumerate}
\end{defn}

\begin{rem}\label{rem:relative-moduli}
In \cite{PerverseI}, we assumed ${\frak k}={\Bbb C}$.
As we shall see in the appendix (Corollary \ref{cor:relative-moduli}),
for any $S$ of finite type
over a universally Japanese ring,
we have a relative moduli scheme
of $\beta$-semi-stable objects as a GIT quotient of a quot-scheme 
in the category of perverse coherent sheaves.
In particular, we have a relative moduli stack as a quotient stack.
\begin{NB}
Old version:
Since \cite[Cor. 1.3.10]{PerverseI} holds for any $S$ of finite type
over universally japanese ring,
we have a relative moduli stack of $\beta$-semi-stable objects
as a quotient stack
$[Q^{ss}/\GL(N)]$, where
$Q^{ss}$ is an open subscheme
of $\Quot_{G^{\oplus N}/X/S}^{{\frak C},P}$, where
$G_s$ is a family of local projective generators of
${\frak C}_s$.  
\end{NB}
\end{rem}

\begin{defn}
Let $E \ne 0$ be an object of ${\frak C}$. 
\begin{enumerate}
\item[(1)]
There is a (unique) filtration
\begin{equation}\label{eq:HNF}
0 \subset F_1 \subset F_2 \subset \cdots \subset F_s=E
\end{equation}
such that each $E_j:=F_j/F_{j-1}$ is a torsion object or a torsion free 
$G$-twisted semi-stable object and
\begin{equation*}
(\rk E_{j+1})\chi(G,E_j(n))>(\rk E_j) \chi(G,E_{j+1}(n)),\quad n \gg 0.
\end{equation*}
We call it the \emph{Harder-Narasimhan filtration} of $E$.  

\item[(2)]
In the notation of (1), we set
\begin{align*}
\mu_{\max,G}(E):=
\begin{cases}
 \mu_{G}(E_1),&\rk E_1>0,\\
 \infty,&\rk E_1=0,
 \end{cases}
\qquad
\mu_{\min,G}(E):=
&\begin{cases}
 \mu_{G}(E_s),&\rk E_s>0,\\
 \infty,&\rk E_s=0.
 \end{cases}
\end{align*}
\end{enumerate}
\end{defn}

\begin{rem}\label{rem:def-field}
Let $\overline{\frak k}$ be the algebraic closure of ${\frak k}$.
For an object $E$ of ${\frak C}$,
$E$ is $G$-twisted semi-stable if and only if
$E \otimes_{\frak k}\overline{\frak k}$ is 
$G \otimes_{\frak k}\overline{\frak k}$-twisted semi-stable
as in the usual Gieseker semi-stability of sheaves.
Hence \eqref{eq:HNF} is invariant under the extension of the field.
\end{rem}

We define several torsion pairs of ${\frak C}$.
\begin{defn}
\begin{enumerate}
\item[(1)]
Let ${\frak T}^{\mu}$ be the full subcategory of ${\frak C}$
such that $E \in {\frak C}$ belongs to ${\frak T}^{\mu}$
if (i) $E$ is a torsion object or 
(ii) $\mu_{\min,G}(E) > 0$. 

\item[(2)]
Let ${\frak F}^{\mu}$ be the full subcategory of ${\frak C}$
such that $E \in {\frak C}$ belongs to ${\frak F}^{\mu}$
if $E=0$ or $E$ is a torsion free object 
with $\mu_{\max,G}(E) \leq 0$.
\end{enumerate}
\end{defn}

\begin{defn}
\begin{enumerate}
\item[(1)]
Let ${\frak T}_{G}$ be the full subcategory of ${\frak C}$
such that $E \in {\frak C}$ belongs to ${\frak T}_{G}$ 
if (i) $E$ is a torsion object or 
(ii) for the Harder-Narasimhan filtration \eqref{eq:HNF} of $E$,
$E_s$ satisfies
$\mu_{G}(E_s)>0$ or $\mu_{G}(E_s)=0$ and $\chi(G,E_s)>0$.

\item[(2)]
Let ${\frak F}_{G}$ 
be the full subcategory of ${\frak C}$
such that $E \in {\frak C}$ belongs to ${\frak F}_{G}$ 
if $E$ is a torsion free object and   
for the Harder-Narasimhan filtration \eqref{eq:HNF} of $E$,
$E_1$ satisfies $\mu_{G}(E_1)<0$ or $\mu_{G}(E_1)=0$ and $\chi(G,E_1) \leq 0$.
\end{enumerate}
\end{defn}

\begin{defn}\label{defn:category}
$({\frak T}^{\mu},{\frak F}^{\mu})$ 
and $({\frak T}_{G},{\frak F}_{G})$ 
are torsion pairs of ${\frak C}$.
We denote the tiltings of ${\frak C}$ by
${\frak A}^{\mu}$ and ${\frak A}_{G}$ respectively. 
\end{defn}

\begin{defn}
If $v(G)/\rk G=e^\beta$, then
we set $({\frak T},{\frak F}):= ({\frak T}_{G},{\frak F}_{G})$
and ${\frak A}:={\frak A}_{G}$.
\end{defn}

For $E \in {\bf D}(X)$, we can express $v(E)$ by
\begin{equation}\label{eq:Mukai-vector}
v(E)=r e^\beta+a \varrho_X+(dH+D)+(dH+D,\beta)\varrho_X,
\end{equation}
where $r,a,d \in {\Bbb Q}$ and 
$D \in \NS(X)_{\Bbb Q} \cap H^{\perp}$.
Then we have $r=\rk E$ and 
\begin{equation}\label{eq:mv:da}
d=\frac{\deg_G(E)}{r_0 (H^2)}=\frac{\deg(E(-\beta))}{(H^2)},
\quad
a=-\chi(E(-\beta)).
\end{equation}
We note that 
$(dH+D)+(dH+D,\beta)\varrho_X \in (e^\beta)^{\perp}$.

Hereafter we take $\omega \in {\Bbb R}_{>0} H$ with
$(\omega^2) \in {\Bbb Q}$, 
and for the pair $(\beta,\omega)$ we introduce some functions and categories 
in order to construct stability conditions. 

\begin{defn}
\begin{enumerate}
\item[(1)]
We define
$Z_{(\beta,\omega)}:{\bf D}(X) \to {\Bbb C}$ by
\begin{equation}\label{eq:Z}
\begin{split}
Z_{(\beta,\omega)}(E)
:=&\langle e^{\beta+\sqrt{-1}\omega},v(E) \rangle 
\\
 =&\Bigl\langle 
    e^\beta-\dfrac{(\omega^2)}{2}\varrho_X 
    + \sqrt{-1}(\omega+(\omega,\beta)\varrho_X),v(E) 
   \Bigr\rangle
\\
 =& -a+r\frac{(\omega^2)}{2}+\sqrt{-1} d(H,\omega). 
\end{split}
\end{equation}
Here we used \eqref{eq:mukai_vector}
\item[(2)]
If $Z_{(\beta,\omega)}({\frak A}_{G} \setminus \{ 0 \}) 
\subset {\Bbb H}'$, 
then $Z_{(\beta,\omega)}$ is a stability function on ${\frak A}_{G}$.
In this case, we have a function 
$\phi_{(\beta,\omega)}:{\frak A}_{G}[n] \setminus \{ 0\} \to (n,n+1]$ 
such that 
\begin{align*}
 Z_{(\beta,\omega)}(E)=|Z_{(\beta,\omega)}(E)|
                       e^{\sqrt{-1} \pi \phi_{(\beta,\omega)}(E)}
\end{align*}
for $E \in {\frak A}_{G}[n] \setminus \{ 0\} $.

\item[(3)]
For a non-zero Mukai vector $v \in H^*(X,{\Bbb Z})_{\alg}$, 
we define $Z_{(\beta,\omega)}(v) \in {\Bbb C}$ and 
$\phi_{(\beta,\omega)}(v) \in (0,2]$ by 
\begin{align*}
Z_{(\beta,\omega)}(v)
&:= \langle e^{\beta+\sqrt{-1} \omega},v \rangle 
  = |Z_{(\beta,\omega)}(v)|e^{\pi \sqrt{-1}\phi_{(\beta,\omega)}(v)}.
\end{align*}
\end{enumerate}
\end{defn}

Note that 
if $Z_{(\beta,\omega)}:{\frak A}_{G} \to {\Bbb C}$ 
is a stability function on ${\frak A}_{G}$,
then for 
$0 \ne 
 E \in {\frak A}_{(\beta,\omega)} \cup {\frak A}_{(\beta,\omega)}[1]$
we have
$$
\phi_{(\beta,\omega)}(v(E))
=\phi_{(\beta,\omega)}(E).
$$ 
$\phi_{(\beta,\omega)}$ is called the \emph{phase function} 
of $Z_{(\beta,\omega)}$.
If confusion does not occur, 
we denote $\phi(\cdot):=\phi_{(\beta,\omega)}(\cdot)$.

\begin{rem}\label{rem:twisted}
For the category of twisted sheaves, we take a locally free 
twisted sheaf $G$ with $\chi(G,G)=0$.
Then we replace the Mukai vector $v(E)$ by
\begin{equation*}
v_G(E):=\frac{\ch(G^{\vee} \otimes E)}
{\sqrt{\ch(G^{\vee} \otimes G)}}\sqrt{\td_X} \in A^*_{\alg}(X)\otimes{\Bbb Q}.
\end{equation*}
Then $v_G(G)=\rk G e^\beta$ with $\beta=0$ and we have an expression
\begin{equation*}
v_G(E)=r+a \varrho_X+(dH+D),
\end{equation*}
since $\beta=0$.
In this case, $Z_{(\beta,\omega)}$ is also well-defined. 
\end{rem}

\begin{defn}
For $(\beta, \omega)$,
$G_{(\beta,\omega)} \in K(X)_{\Bbb Q}$ is an element satisfying
\begin{equation}\label{eq:G_{omega}}
v(G_{(\beta,\omega)})=e^\beta-\frac{(\omega^2)}{2}\varrho_X.
\end{equation} 
\end{defn}

Then we have 
\begin{align*}
Z_{(\beta,\omega)}(E)=-\chi_{G_{(\beta,\omega)}}(E)+\sqrt{-1} d(H,\omega). 
\end{align*}
\begin{rem}\label{rem:discrete}
Since $G_{(\beta,\omega)} \in K(X)_{\Bbb Q}$ and 
$(H,\omega) \in {\Bbb Q}$,
there is an integer $N$ such that 
$N Z_{(\beta,\omega)}(E) \in {\Bbb Z}+{\Bbb Z}\sqrt{-1}$
for all $E \in K(X)$.
\end{rem}


\begin{defn}\label{defn:category2}
For $(\beta,\omega)$,
we define $({\frak T}_{(\beta,\omega)},{\frak F}_{(\beta,\omega)})$
and ${\frak A}_{(\beta,\omega)}$ to be the categories
$({\frak T}_{G_{(\beta,\omega)}},
{\frak F}_{G_{(\beta,\omega)}})$
and ${\frak A}_{G_{(\beta,\omega)}}$ with
\eqref{eq:G_{omega}}.
\end{defn}


\subsection{Examples of Bridgeland's stability conditions}
\label{subsect:stability:ex}

In this section, we shall generalize Bridgeland's explicit
construction of stability condition \cite[sect. 7]{Br:3}. 
It is also a geometric construction of the stability condition
in \cite[Lem. 4.8]{Ha:1}.
We keep the notation in the last subsection 
and fix the pair $(\beta,\omega)$.

\begin{prop}
\label{prop:stability:beta-omega}
Assume that there is no $G_{(\beta,\omega)}$-twisted stable object $E$ with
$\deg_{G_{(\beta,\omega)}}(E)=
\chi_{G_{(\beta,\omega)}}(E)=0$. Then
$\sigma_{(\beta,\omega)}:=({\frak A}_{(\beta,\omega)},Z_{(\beta,\omega)})$ 
is an example of Bridgeland's 
stability condition.
\end{prop}

\begin{proof}
By the definition of ${\frak A}_{(\beta,\omega)}$,
$Z_{(\beta,\omega)}$ is a stability function on ${\frak A}_{(\beta,\omega)}$.
We prove that 
the stability function $Z_{(\beta,\omega)}$ 
satisfies the Harder-Narasimhan properties.

Let $E$ be an object of ${\frak A}_{(\beta,\omega)}$.
By the same proof of \cite[Prop. 7.1]{Br:3},
there is no chain of monomorphisms in ${\frak A}_{(\beta,\omega)}$
\begin{equation*}
\cdots \subset E_{i+1} \subset E_i \subset \cdots \subset E_1 \subset E_0=E
\end{equation*}
with $\phi(E_{i+1})>\phi(E_i)$ for all $i$.

Assume that there is a chain of epimorphisms
\begin{equation*}
E=E_0 \twoheadrightarrow E_1 \twoheadrightarrow \cdots \twoheadrightarrow
E_i \twoheadrightarrow E_{i+1} \twoheadrightarrow \cdots
\end{equation*}
with $\phi(E_i)>\phi(E_{i+1})$ for all $i$.
By the same proof of \cite[Prop. 7.1]{Br:3},
we may assume that 
$\mathrm{Im}Z_{(\beta,\omega)}(E)=
\mathrm{Im}Z_{(\beta,\omega)}(E_i)$ and 
$H^0(E) \to H^0(E_i)$ is an isomorphism for all $i$.
We set $L_i:=\ker(E \to E_i)$.
Then there is a chain
\begin{equation*}
0=L_0 \subset L_1 \subset \cdots \subset L_i  \subset \cdots \subset E
\end{equation*}
and $\mathrm{Im} Z_{(\beta,\omega)}(L_i)=0$ for all $i$.
By the definition of ${\frak A}_{(\beta,\omega)}$,
$H^{-1}(L_i)$ is a $\mu$-semi-stable object with
$\deg_{G_{(\beta,\omega)}}(H^{-1}(L_i))=0$, and
$H^0(L_i)$ is an extension 
$$
0 \to T \to H^0(L_i) \to F \to 0
$$
of a $\mu$-semi-stable object $F$ with
$\deg_G(F)=0$ by a 0-dimensional object $T$.

We may assume that $H^{-1}(L_i) \to H^{-1}(L_{i+1})$
is an isomorphism for all $i$.
We set $B_i:=L_i/L_{i-1}$.
Then we have an exact sequence
\begin{equation*}
0 \to H^{-1}(B_i) \to H^0(L_{i-1}) \to H^0(L_i)
\to H^0(B_i) \to 0.
\end{equation*}
Since $\chi_{G_{(\beta,\omega)}}(H^{-1}(B_i)) \leq 0$ and 
$\chi_{G_{(\beta,\omega)}}(H^0(B_i))\geq 0$,
$\chi_{G_{(\beta,\omega)}}(H^0(L_{i-1})) \leq 
\chi_{G_{(\beta,\omega)}}(H^0(L_i))$.
 
We have
$$
0 \geq \chi_{G_{(\beta,\omega)}}(H^{-1}(E_i))=
\chi_{G_{(\beta,\omega)}}(H^{-1}(E))-
\chi_{G_{(\beta,\omega)}}(H^{-1}(L_i))
+\chi_{G_{(\beta,\omega)}}(H^0(L_i)).
$$
Hence $\chi_{G_{(\beta,\omega)}}(H^0(L_i))$ is bounded above.
Therefore $\chi_{G_{(\beta,\omega)}}(H^0(L_i))$ is constant for $i \gg 0$
by Remark \ref{rem:discrete}.
Then we have $\chi_{G_{(\beta,\omega)}}(H^0(B_i))=
\chi_{G_{(\beta,\omega)}}(H^{-1}(B_i))=0$, which implies that
$H^0(B_i)=0$.
\begin{NB}
Use $\deg_{G_{(\beta,\omega)}}(H^0(B_i))=0$ and
$\chi_{G_{(\beta,\omega)}}(H^0(B_i)) \geq 0$.
\end{NB}
Hence $H^0(L_{i-1}) \to H^0(L_i)$ is surjective for $i \gg 0$.
By the Noetherian properties of ${\frak C}$,
$H^0(L_{i-1}) \to H^0(L_i)$ is an isomorphism for $i \gg 0$.  
Therefore $L_{i-1} \to L_i$ is an isomorphism
for $i \gg 0$.
 \end{proof}

\begin{rem}\label{rem:category}
If $X$ is an abelian surface, then
${\frak A}_{(\beta,\omega)}={\frak A}^\mu$ for any
$\omega$. Thus Definition~\ref{defn:category2} is meaningful
only for a $K3$ surface.
\end{rem}

\begin{defn}\label{defn:area}
For $E, E' \in {\frak A}_{(\beta,\omega)}$,
we set
\begin{equation*}
\Sigma_{(\beta,\omega)}(E',E):=\det
\begin{pmatrix}
\mathrm{Re} Z_{(\beta,\omega)}(E') & \mathrm{Re}Z_{(\beta,\omega)}(E)\\
\mathrm{Im} Z_{(\beta,\omega)}(E') & \mathrm{Im} Z_{(\beta,\omega)}(E)
\end{pmatrix}.
\end{equation*}
\end{defn}

\begin{defn}
Let ${\frak k}$ be an arbitrary field.
Assume that $\sigma_{(\beta,\omega)}$ is a stability condition.
Then $E \in {\frak A}_{(\beta,\omega)}$ is 
$\sigma_{(\beta,\omega)}$-semi-stable
if
\begin{equation*}
\Sigma_{(\beta,\omega)}(E',E) \geq 0
\end{equation*}
for any subobject $E'$ of $E$.
\end{defn}

\begin{rem}\label{rem:def-stability}
Since 
$$
\Sigma_{(\beta,\omega)}(E',E)=
|Z_{(\beta,\omega)}(E')||Z_{(\beta,\omega)}(E)|\sin(\pi(\phi(E)-\phi(E')))
$$
and $0<\phi(E),\phi(E') \leq 1$,
$\Sigma_{(\beta,\omega)}(E',E) \geq 0$ if and only if
$\phi(E)-\phi(E') \geq 0$.
Thus the above definition is equivalent to Bridgeland's definition
of stability. 
\end{rem}

For $(\beta,\omega)$, let
$(\overline{\beta},\overline{\omega})$
be the corresponding element on 
$X \otimes_{\frak k}\overline{\frak k}$, where $\overline{\frak k}$
is the algebraic closure of ${\frak k}$.
Then we have a natural identification
${\frak A}_{(\overline{\beta},\overline{\omega})} =
({\frak A}_{(\beta,\omega)}) \otimes_{\frak k} \overline{\frak k}$
by Remark~\ref{rem:def-field}.
In \S\,\ref{sect:appendix},
we shall prove that the $\sigma_{(\beta,\omega)}$-semi-stability of $E$ 
is equivalent to
the $\sigma_{(\overline{\beta},\overline{\omega})}$-semi-stability 
of $E \otimes_{\frak k}\overline{\frak k}$.

%

\begin{defn}
Assume that $\sigma_{(\beta,\omega)}$ is a stability condition.
Then $E \in {\frak A}_{(\beta,\omega)}$ is 
$\sigma_{(\beta,\omega)}$-stable,
if
\begin{equation*}
\Sigma_{(\beta,\omega)}(E',E \otimes_{\frak k} \overline{\frak k})> 0
\end{equation*}
for any proper subobject $E' \ne 0$ of 
$E \otimes_{\frak k}\overline{\frak k}$.
\end{defn}

\begin{defn}\label{defn:moduli}
${\cal M}_{(\beta,\omega)}(v)$ denotes the moduli stack of 
$\sigma_{(\beta,\omega)}$-semi-stable objects $E$
such that $v(E)=v$. 
If there is a coarse moduli scheme of the $S$-equivalence classes of
$\sigma_{(\beta,\omega)}$-semi-stable objects,
then we denote it by $M_{(\beta,\omega)}(v)$.
\end{defn}

\begin{rem}\label{rem}
\begin{enumerate}
\item[(1)]
By Lieblich (\cite{Lieblich:1}, \cite[Appendix]{AB}),
 ${\cal M}_{(\beta,\omega)}(v)$
is an Artin stack.
\item[(2)]
Inaba \cite{Inaba} constructed the moduli space of simple complexes
as an algebraic space.
\item[(3)]
If the moduli scheme $M_{(\beta,\omega)}(v)$ exists and consists of
stable objects, then
the deformation theory implies that $M_{(\beta,\omega)}(v)$
is smooth of $\dim M_{(\beta,\omega)}(v)=\langle v^2 \rangle+2$.
\end{enumerate}
\end{rem}


\begin{lem}\label{lem:irreducible-stable}
Let $E$ be an object of 
${\frak A}_{(\beta,\omega)}$
such that $\deg(E(-\beta))=0$ and $E \otimes_{\frak k} \overline{\frak k}$
is an irreducible object of ${\frak A}_{(\overline{\beta},\overline{\omega})} =
({\frak A}_{(\beta,\omega)}) \otimes_{\frak k} \overline{\frak k}$,
where $\overline{\frak k}$ is the algebraic closure of ${\frak k}$.
\begin{enumerate}
\item[(1)]
If $\rk E \geq 0$, then
$H^{-1}(E)=0$ and 
$H^0(E)$ is 0-dimensional or 
$H^0(E)$ is a $\beta$-twisted stable object with
$\langle v(E)^2 \rangle=-2$.  
\item[(2)]
If $\rk E<0$, then
$H^0(E)=0$ and
$H^{-1}(E)$ is a $\beta$-twisted stable object of ${\frak C}$.
\end{enumerate}\end{lem}

\begin{proof}
Replacing ${\frak k}$ by $\overline{\frak k}$,
we may assume that ${\frak k}$ is algebraically closed.
We set $G:=G_{(\beta,\omega)}$.
For $E \in {\frak A}_{(\beta,\omega)}$,
we have an exact sequence
$$
0 \to H^{-1}(E)[1] \to E \to H^0(E) \to 0. 
$$

(1)
Assume that $\rk E \geq 0$. 
If $H^{-1}(E) \ne 0$, then the irreducibility 
of $E$ implies that $H^0(E)=0$ and 
$H^{-1}(E)$ is a torsion free object
of ${\frak C}$ with $\rk H^{-1}(E)>0$,
which is a contradiction.
Therefore $H^{-1}(E)=0$.
If $H^0(E)$ has a torsion subobject $T$, then
we have an exact sequence in ${\frak A}_{(\beta,\omega)}$:
$$
0 \to T \to E \to E/T \to 0.
$$ 
By the irreducibility of $E$,
$E$ is a torsion object.
If $\dim E=1$, then we have a non-trivial quotient 
$\varphi:E \to E_1$ in ${\frak C}$,
which gives a non-trivial quotient of $E$ in ${\frak A}_{(\beta,\omega)}$,
since $\dim \ker \varphi \leq 1$.
Therefore $\dim E=0$.
If $H^0(E)$ is torsion free, 
we take the Harder-Narasimhan filtration of $H^0(E)$: 
$$
0 \subset F_1 \subset F_2 \subset \cdots \subset F_s=H^0(E).
$$
Since $H^0(E) \in {\frak T}_{(\beta,\omega)}$
and $\deg(H^0(E)(-\beta))=0$, we see that
$\deg(F_i/F_{i-1}(-\beta))=0$ for all $i$.
Then we have 
$$
\frac{\chi_G(F_1)}{\rk F_1} >
\frac{\chi_G(F_2/F_1)}{\rk F_2/F_1} > \cdots >
\frac{\chi_G(F_s/F_{s-1})}{\rk F_s/F_{s-1}}> 0.
$$
Thus $F_i/F_{i-1} \in {\frak T}_{(\beta,\omega)}$
for all $i$.
By the irreducibility of $E$,
$s=1$. Thus $H^0(E)$ is $\beta$-twisted semi-stable.
By the irreducibility, we also see that
$H^0(E)$ is $\beta$-twisted stable.
Since $\chi(H^0(E)(-\beta))>\chi_G(H^0(E))>0$, we see that
$\langle v(E)^2 \rangle<0$. Hence $\langle v(E)^2 \rangle=-2$.
\begin{NB}
Use $v(G)=e^\beta-\frac{(\omega^2)}{2}\varrho_X$ to show
$\chi(H^0(E)(-\beta))>\chi_G(H^0(E))$.
\end{NB}

(2)
Since $\rk E<0$, 
we have $H^{-1}(E) \ne 0$.
By the irreducibility of $E$,
$H^0(E)=0$.
Let 
$$
0 \subset F_1 \subset F_2 \subset \cdots \subset F_s=H^{-1}(E)
$$
be the Harder-Narasimhan filtration of $H^{-1}(E)$.
Since $H^{-1}(E) \in {\frak F}_{(\beta,\omega)}$
and $\deg(H^{-1}(E)(-\beta))=0$, we see that
$\deg(F_i/F_{i-1}(-\beta))=0$ for all $i$.
Then we have 
$$
0 \geq \frac{\chi_G(F_1)}{\rk F_1} > 
\frac{\chi_G(F_2/F_1)}{\rk F_2/F_1} > \cdots >
\frac{\chi_G(F_s/F_{s-1})}{\rk F_s/F_{s-1}}.
$$
Thus $F_i/F_{i-1} \in {\frak F}_{(\beta,\omega)}$
for all $i$.
By the irreducibility of $E$,
$s=1$. Thus $H^{-1}(E)$ is $\beta$-twisted semi-stable.
By the irreducibility, we also see that
$H^{-1}(E)$ is $\beta$-twisted stable.
\end{proof}

\begin{lem}\label{lem:stable-irreducible}
Let $E$ be a $\beta$-twisted stable object of ${\frak C}$ 
with $\deg(E(-\beta))=\chi(E(-\beta))=0$.
Then $(E[1]) \otimes_{\frak k} \overline{\frak k}$ 
is an irreducible object of
${\frak A} \otimes_{\frak k} \overline{\frak k}$.
\end{lem}

\begin{proof}
We may assume that $\overline{\frak k}={\frak k}$. 
We set $G:=G_{(\beta,0)}$.
Assume that there is an exact sequence in 
${\frak A}_{(\beta,0)}={\frak A}$:
$$
0 \to F_1 \to E[1] \to F_2 \to 0.
$$
Then we have an exact sequence in ${\frak C}$:
$$
0 \to H^{-1}(F_1) \to E \overset{\psi}{\to} H^{-1}(F_2) \to
H^0(F_1) \to 0.
$$
Since $\deg_G(H^{-1}(F_1)),
\deg_G(H^{-1}(F_2)) \leq 0$ and
$\deg_G(H^0(F_1)) \geq 0$, we have
$$
0=\deg_G(E)=\deg_G(H^{-1}(F_1))+\deg_G(H^{-1}(F_2))
-\deg_G(H^0(F_1)) \leq 0.
$$
Hence $\deg_G(H^{-1}(F_1))=\deg_G(H^{-1}(F_2))=\deg_G(H^0(F_1))=0$.
We assume that $H^0(F_1) \ne 0$.
Since $\chi_G(H^0(F_1))>0$ and $\chi_G(H^{-1}(F_2)) \leq 0$,
we get $\chi_G(\im \psi) < 0$.
By the $\beta$-twisted stability of $E$
and $\chi_G(E)=0$, we have $\psi=0$. Then
$H^{-1}(F_2) \to H^0(F_1)$ is isomorphic, which is a contradiction.
Therefore $H^0(F_1)=0$. By the $\beta$-twisted stability of $E$,
we get $H^{-1}(F_1)=0$ or $H^{-1}(F_2)=0$,
which implies $E[1]$ is irreducible. 
\end{proof}


\subsection{The wall and chamber for categories}
\label{subsect:wall:category}

\subsubsection{}

For the stability condition $\sigma_{(\beta,\omega)}$,
the abelian category ${\frak A}_{(\beta,\omega)}$ depends on the choices of
$\beta$ and $\omega \in {\Bbb R}_{>0}H$.
In this subsection,
we shall study the dependence under 
fixing $b:=(\beta,H)/(H^2) \in {\Bbb Q}$. 
So we assume that $X$ is a $K3$ surface (cf. Remark~\ref{rem:category}).
We first note that $\eta:=\beta-bH \in H^{\perp}$.

\begin{defn}\label{defn:parameter-space}
We set 
\begin{equation*}
\begin{split}
{\frak H}:=& \{(\eta,\omega) \mid
 \eta \in \NS(X)_{\Bbb Q},\ (\eta,H)=0,\ 
 \omega \in {\Bbb R}_{>0}H \},
\\
{\frak H}_{\Bbb R}:=& \{(\eta,\omega) \mid
 \eta \in \NS(X)_{\Bbb R},\ (\eta,H)=0,\ 
 \omega \in {\Bbb R}_{>0}H \}.
\end{split}
\end{equation*}
\end{defn}

We have an embedding of 
${\frak H}_{\Bbb R}$ into $\NS(X)_{\Bbb C}$ via
$(\eta,\omega) \mapsto \eta+\sqrt{-1}\omega$.
Thus we have an identification:
\begin{align*}
{\frak H}_{\Bbb R}
\cong (\NS(X)_{\Bbb R} \cap H^{\perp})+\sqrt{-1}{\Bbb R}_{>0}H.
\end{align*}

\begin{rem}
We shall introduce an embedding of ${\frak H}_{\Bbb R}$ into a sphere.
For the vector space
\begin{equation}\label{eq:V_H}
V_H:=(\NS(X)_{\Bbb R} \cap H^{\perp})+\sqrt{-1}{\Bbb R}H,
\end{equation}
the intersection pairing is a negative definite real form.
We set
\begin{equation}\label{eq:I-space}
\begin{split}
{\frak I}:=&
\{{\Bbb R}x \in {\Bbb P}({\Bbb R}\oplus V_H \oplus{\Bbb R}\varrho_X)
  \mid
  \langle x^2 \rangle=0
\},
\\
{\frak I}_{bH}
:=& e^{bH}{\frak I}
 =\{{\Bbb R}x \in {\Bbb P}(A^*_{\alg}(X)_{\Bbb C})
    \mid
    {\Bbb R}xe^{-bH} \in {\frak I}\}.
\end{split}
\end{equation} 
We set $x=r+\xi+a\varrho_X \in {\frak I}_{bH}$.
If $r \ne 0$, then 
$x=r e^{bH+\eta+\sqrt{-1}\omega}$, 
$\eta+\sqrt{-1}\omega \in V_H$.
If $r=0$, then $x=a\varrho_X$.
Thus ${\frak I}$ is identified with a compactification 
$\overline{V}_H:=V_H  \cup \{\infty\}$ of $V_H$,
where ${\Bbb R} \varrho_X$ corresponds to $\infty$. 
We shall prove that
${\frak I}$ is diffeomorphic to $\rho$-dimensional sphere
$S^{\rho}$, where $\rho=\rk \NS(X)$.
For $x=r+\xi+a \varrho_X$ with $\xi \in V_H$,
$\langle x^2 \rangle=0$ if and only if $(\xi^2)=2ra$. 
We shall identify ${\Bbb R}^{\rho}$
with $V_H$ by sending $(y_1,\ldots,y_{\rho-1},y_\rho)$ 
to $\sum_{i=1}^{\rho-1} y_i \xi_i+\sqrt{-1}y_\rho h$,
where $h=H/\sqrt{(H^2)}$ and
$\xi_i \in H^{\perp}$ ($1 \leq i \leq \rho-1$)
satisfy $-(\xi_i,\xi_j)=\delta_{ij}$.
Let $S^{\rho}$ be a sphere in ${\Bbb R}^{\rho} \times {\Bbb R}$.
Then we have a diffeomorphism:
\begin{equation*}
\begin{matrix}
{\frak I} & \to & S^{\rho}\\
{\Bbb R}(r+\xi+a \varrho_X) & \mapsto &
\Bigl(\dfrac{-2\xi}{r-2a}, \dfrac{2a+r}{2a-r}\Bigr).
\end{matrix}
\end{equation*}
The correspondence $S^\rho \to \overline{V}_H$ is nothing but the 
stereographic projection from
$(\vec{0},1) \in S^\rho$, and we get a desired embedding
${\frak H}_{\Bbb R} \hookrightarrow S^\rho$. 
We set 
\begin{align*}
\overline{\frak H}_{\Bbb R}:={\frak H}_{\Bbb R} \cup \{\infty \}.
\end{align*}
This embedding will be used in \S\,\ref{subsect:parameter-FM} 
to describe the action of Fourier-Mukai transforms.
\begin{NB}
We set $\frac{\xi'}{r}=\eta+\sqrt{-1}\omega$.
$$
{\Bbb R}(r+\xi+a \varrho_X) \mapsto 
{\Bbb R}(r+\xi'+a' \varrho_X=re^{\eta+\sqrt{-1}\omega})
\mapsto 
(\frac{-2 (\eta+\sqrt{-1}\omega)}{1+(\omega^2)-(\eta^2)},
\frac{(\omega^2)-(\eta^2)-1}{(\omega^2)-(\eta^2)+1}).
$$
For $(X,Y)$ with $-(X^2)+Y^2=1$,
the inverse map is given by
$(X,Y) \mapsto e^{bH}[(Y-1)^2+(Y-1)X+(X^2)\varrho_X/2]$.
\end{NB}
\end{rem}

\begin{defn}\label{defn:wall:category}
We set
\begin{align*}
{\frak R}:=\{u \in A^*_{\alg}(X) \mid 
u \in (H+(H,bH)\varrho_X)^{\perp},\ 
\langle u^2 \rangle=-2 \}.
\end{align*}
For $u \in {\frak R}$, 
we define a \emph{wall $W_u$ for categories} 
of ${\frak H}_{\Bbb R}$ as
\begin{align*}
W_u := 
\{(\eta,\omega) \in {\frak H}_{\Bbb R} \mid
  \rk u \cdot (\omega^2)= -2 \langle  e^{bH+\eta},u \rangle \}.
\end{align*}
A connected component of
${\frak H}_{\Bbb R} \setminus \cup_{u \in {\frak R}} W_u$
is called a \emph{chamber for categories}.
\end{defn}

\begin{rem}\label{rem:frak-R}
\begin{enumerate}
\item[(1)]
Assume that
$u=r e^\beta+a \varrho_X+D+(D,\beta)\varrho_X$ belongs to
${\frak R}$.
If $Z_{(\beta,\omega)}(E)=0$, then
$a=r \frac{(\omega^2)}{2}$ implies that
$r>0$ if and only if $a>0$.
If $r=0$, then $u=D+(D,\beta)\varrho_X$ implies
that $(D,\beta) \in {\Bbb Z}$.
Hence $u \in {\frak R}$ with $\rk u=0$ are used to
describe the dependence on
the category ${\frak C}$.
\item[(2)]
For $u \in {\frak R}$ with $\rk u >0$,
$W_u$ is the half sphere defined by
\begin{align*}
-\left(\eta-\dfrac{c_1(u)}{\rk u}+b H\right)^2+(\omega^2)=\dfrac{2}{(\rk u)^2}.
\end{align*}
\end{enumerate}
Hence $u$ is determined by $W_u$.
\end{rem}

\begin{lem}\label{lem:wall:category-finite}
The set of walls is locally finite.
\end{lem}

\begin{proof}
Let $B$ be a compact subset of ${\frak H}_{\Bbb R}$.
We shall prove that
$$
\{u \in {\frak R} \mid W_u \cap B \ne \emptyset \}
$$
is a finite set.
An element $u \in {\frak R}$ can be expressed as 
$$
 u=r e^{bH}+a \varrho_X+D+(D,bH)\varrho_X
$$
with $D \in H^{\perp} \cap \NS(X)_{\Bbb Q}$ and 
$a=((D^2)+2)/(2r)$.
Then $W_u$ is the half sphere 
$$
-\left(\eta-\dfrac{D}{r}\right)^2+(\omega^2) =\dfrac{2}{r^2}.
$$
Hence $r^2<2/(\omega^2)$ and
\begin{equation*}
2>-(r\eta-D)^2 \geq \left(\sqrt{-r^2(\eta^2)}-\sqrt{-(D^2)} \right)^2.
\end{equation*}
Since $B$ is a compact subset of ${\frak H}_{\Bbb R}$,
the choice of $r$ and $(D^2)$ are finite.
We denote the denominator of $b$ by $b_0$.
Since $D=c_1(u)-r b H \in b_0^{-1}\cdot \NS(X)$,
the choice of $D$ is also finite. Hence the claim holds. 
\end{proof}

\subsubsection{}

For a fixed $\beta:=bH+\eta$, $\eta \in H^{\perp}$,
we have an embedding
\begin{equation}\label{eq:iota}
\begin{matrix}
\iota_\beta:& {\Bbb R}_{>0}H & \to& {\frak H}_{\Bbb R}\\
& \omega & \mapsto & (\eta,\omega).
\end{matrix}
\end{equation}
Then we also have the notion of walls and chambers
on ${\Bbb R}_{>0}H$.
In this case, the category ${\frak C}$ is fixed.

\begin{lem}\label{lem:R_beta}
We set
$$
{\frak R}_\beta:=
\{ u \in {\frak R} \mid 
  \rk u>0,\ -\langle e^\beta,u \rangle>0 \}.
$$
Then ${\frak R}_\beta$ is a finite set and $\rk u \leq r_0$.
\end{lem}

\begin{proof}
We set 
$$
u:=re^\beta+a \varrho_X+(D+(D,\beta)\varrho_X),\quad D \in H^\perp.
$$
By the assumption, $a=-\langle e^\beta,u \rangle >0$ and
$-2 =\langle u^2 \rangle=-2ra +(D^2) \leq -2ra$.
Hence 
$0<r(r_0 a) \leq r_0$.
Since $-r_0 a=\langle r_0 e^\beta,u \rangle \in {\Bbb Z}$
and $r \in {\Bbb Z}$,
$r$ and $r_0 a$ are positive integers with $r (r_0 a) \leq r_0$.
Thus the choices of $r$ and $a$ are finite.
Since $D=(c_1(u)-r\beta) 
\in (1/r_0)\NS(X) \cap H^{\perp}$ and
$0 \leq -(D^2) \leq -2ra+2$, the choice of $u$ is also finite.
\end{proof}

\begin{defn}\label{defn:wall:category2}
For $u \in {\frak R}_\beta$, 
we define a \emph{wall $W_{\beta,u}$ for categories} 
of ${\Bbb R}_{>0} H$ as
$$
\left\{\omega \in {\Bbb R}_{>0}H \,\Big|\, 
 \dfrac{(\omega^2)}{2}
 = -\dfrac{\langle  e^\beta,u \rangle}{\rk u} 
\right\}.
$$
A connected component of
${\Bbb R}_{>0}H \setminus \cup_{u \in {\frak R}_\beta} W_{\beta,u}$
is called a \emph{chamber for categories}.
\end{defn}

\begin{rem}
If $W_{\beta,u} \ne \emptyset$,
then $W_{\beta,u}=\iota_\beta^{-1}(W_u)$, and
$W_u$ intersects with $\iota_\beta({\Bbb R}_{>0}H)$ transversely.
\end{rem}

For $u \in {\frak R}_\beta$,
there is a $\beta$-twisted semi-stable object $E$ of ${\frak C}$
with $v(E)=u$.
Since $W_{\beta,u}$ depends only on $u/\rk u$,
we introduce the following definition. 
\begin{defn}
Let $\exc_\beta$ be the set of $\beta$-twisted stable objects $E$ 
of ${\frak C}$ with
$$
v(E)=r e^\beta+a \varrho_X+(D+(D,\beta)\varrho_X),\quad 
r,a>0,\quad 
D \in H^{\perp}.
$$
\end{defn}

\begin{lem}\label{lem:exc}
\begin{enumerate}
\item[(1)]
$\exc_\beta$ is a finite set and
$\{v(E) \mid E \in \exc_\beta \} \subset {\frak R}_\beta$.
\item[(2)]
For $E \in \exc_\beta$,
$\rk E \leq r_0$ and $\rk E<r_0$ unless $r_0=1$ and 
$v(E)=e^\beta+\varrho_X$.
\item[(3)]
Let $E_1,E_2,\ldots,E_n$ be the objects of $\exc_\beta$
and assume that $\{ (E_1)_{\overline{\frak k}},
(E_2)_{\overline{\frak k}},\dots,
(E_n)_{\overline{\frak k}}\}=\exc_{\overline{\beta}}$,
where $\overline{\beta}$ is the image of $\beta \in \NS(X)_{\Bbb Q}$
to $\NS(X_{\overline{\frak k}})_{\Bbb Q}$. 
Let $E_1,E_2,\ldots,E_s$ be the objects of $\exc_\beta$ with
$\chi_{G_{(\beta,\omega)}}(E_i)> 0$.
\begin{enumerate}
\item
For $E \in {\frak T}_{(\beta,\omega)}$, there is an exact sequence
$$
0 \to F_1 \to E \to F_2 \to 0
$$ 
such that $F_1 \in {\frak T}^\mu$ and
$F_2$ is a successive extension of $E_i$, $1 \leq i \leq s$. 
\item
For $E \in {\frak F}_{(\beta,\omega)}$, there is an exact sequence
$$
0 \to F_1 \to E \to F_2 \to 0
$$ 
such that 
$F_1$ is a successive extension of $E_i$, $s+1 \leq i$ and
$F_2 \in {\frak F}$. 
\end{enumerate}
 \end{enumerate}
\end{lem}

\begin{proof}
(1), (2) 
In the notation of \eqref{eq:Mukai-vector},
if $E \in \exc_\beta$, then
$-2 \leq \langle v(E)^2 \rangle=-2ra +(D^2) \leq -2ra < 0$.
Hence 
$\langle v(E)^2 \rangle=-2$, which implies that
$v(E) \in {\frak R}_\beta$.
In particular, $E$ is an exceptional object:
$$
\Hom(E,E)={\frak k},\quad \Ext^1(E,E)=0.
$$
Therefore (1) and by Lemma~\ref{lem:R_beta} the first claim of (2) hold.
If $r=r_0$, then
we have $r_0 a=1$ and $D=0$.
Hence $v(E)=r_0 e^\beta-a \varrho_X$.
Since $r_0 e^\beta \in A^*_{\alg}(X)$,
$a \in {\Bbb Z}$. Then $r a=1$ implies that $r=r_0=1$
and $a=1$.

(3) 
We first assume that ${\frak k}$ is algebraically closed.
In the notation of \eqref{eq:Mukai-vector},
if $E$ is a $G_{(\beta,\omega)}$-stable object of ${\frak C}$ with
$\deg_{G_{(\beta,\omega)}}(E)=
\chi_{G_{(\beta,\omega)}}(E) \geq 0$, then $d=0$ and
$a \geq r (\omega^2)/2 \geq 0$.
If $r=0$, then $E$ is a 0-dimensional object. Thus
$E \in {\frak T}^\mu$.
If $r>0$, then $a>0$, which implies that $E \in \exc_\beta$.
Then the claim follows from the definition of ${\frak T}_{(\beta,\omega)}$
and $\exc_\beta$.

For the general case, we take the exact sequence 
$$
0 \to F_1 \to E_{\overline{\frak k}} \to F_2 \to 0
$$
in (a) and (b).
Since $\Hom(F_1,F_2)=0$,
we see that $F_1$ and $F_2$ are defined over
${\frak k}$
( see the proof of the Harder-Narasimhan filtration
over any field).
Since 
$\{ (E_1)_{\overline{\frak k}},
(E_2)_{\overline{\frak k}},\dots,
(E_n)_{\overline{\frak k}}\}=\exc_{\overline{\beta}}$,
the claims hold for general cases.
\end{proof}

\begin{cor}
We fix $\beta$ and take $\omega \in {\Bbb Q}_{>0}H$.
Then
${\frak A}_{(\beta,\omega)}$ depends only on the chamber 
in ${\Bbb R}_{>0}H$ where
$\omega$ belongs.
\end{cor}

\begin{NB}
Let $E_1,\ldots,E_t$ be the elements of $\exc_\beta$ 
such that $\omega \in W_{E_i}$.
We set
$$
R_+:=\{u=\sum_i a_i v(E_i) \mid a_i \geq 0, \langle u^2 \rangle=-2 \}. 
$$
Then $R$ is a finite set.
For a sufficiently small general element
$\eta \in \NS(X) \otimes {\Bbb Q}$,
we have 
$$
\langle u/\rk u-u'/\rk u', \eta+(\eta,\beta)\varrho_X \rangle \ne 0
$$
 for all 
$u,u' \in R_+$.
We may assume that $R_+=\{u_1,u_2,\ldots,u_n \}$,
$$
-\langle u_i/\rk u_i,\eta+(\eta,\beta)\varrho_X \rangle 
<  
-\langle u_j/\rk u_j,\eta+(\eta,\beta)\varrho_X \rangle 
$$
for $i<j$.
Let $U_i$ be the $(\beta+\eta)$-twisted stable object
of ${\frak C}$ with $v(U_i)=u_i$.
We set 
$$
v(E_i):=r_i e^\beta+a_i \varrho_X+(D_i+(D_i,\beta)\varrho_X).
$$
Then $\langle v(E_i)/r_i,v(E_j)/r_j,
\eta+(\eta,\beta)\varrho_X \rangle=
(D_i/r_i-D_j/r_j,\eta)$.
We define a sequence of categories
$$
{\frak T}_{(G,\omega_-)}={\frak T}_1 \supset {\frak T}_2 \supset
\cdots \supset {\frak T}_{n+1}={\frak T}_{(G,\omega_+)}
$$
by the descending induction on $i$.
For $i=n+1$, we set
${\frak T}_{n+1}:={\frak T}_{(G,\omega_+)}$.
Then $E \in {\frak T}_{(G,\omega_-)}$ belongs to
${\frak T}_i$, if
there is an exact sequence
$$
0 \to F_1 \to E \to F_2 \to 0
$$
such that $F_1 \in {\frak T}_{i+1}$ and
$F_2=U_i^{\oplus n_i}$. 
We set
$$
{\frak F}_i:=\{E \in {\frak C} \mid \Hom(T,E)=0, T \in {\frak T}_i \}. 
$$
Then we have a sequence
$$
{\frak F}_{(G,\omega_-)}={\frak F}_1 \subset {\frak F}_2 \subset
\cdots \subset {\frak F}_{n+1}={\frak F}_{(G,\omega_+)}
$$
and $({\frak T}_i,{\frak F}_i)$ are torsion pairs of
${\frak C}$.
Let ${\frak A}_i$ be the tilting of ${\frak C}$ by
$({\frak T}_i,{\frak F}_i)$.
\end{NB}

\begin{NB}
\begin{lem}
For $E \in \exc_\beta$, $\Hom(F,E) \ne 0$
for any $F \in M_H^\beta(r_0 e^\beta)$.
\end{lem}

\begin{proof}
$\chi(F,E)=-r_0 \langle e^\beta,v(E) \rangle>0$.
Since $E$ and $F$ are $\beta$-twisted stable,
$\Hom(E,F)=0$. Hence $\Hom(F,E) \ne 0$. 
\end{proof}
\end{NB}

We note that 
${\frak T} \supset {\frak T}_{(\beta,\omega)} \supset {\frak T}^{\mu}$.
\begin{cor}\label{cor:Bridgeland}
\begin{enumerate}
\item[(1)]
 If $(\omega^2)<2/r_0^2$, then
${\frak T}= {\frak T}_{(\beta,\omega)}$.
Thus $({\frak A},Z_{(\beta,\omega)})$ is an example of Bridgeland's 
stability condition.
\item[(2)]
If $(\omega^2)>2$, then
${\frak T}_{(\beta,\omega)}= {\frak T}^{\mu}$.
Thus $({\frak A}^{\mu},Z_{(\beta,\omega)})$ 
is an example of Bridgeland's 
stability condition.
\end{enumerate}
\end{cor}

\begin{proof}
By Lemma~\ref{lem:exc}~(3),
${\frak T}= {\frak T}_{(\beta,\omega)}$ 
if $(\omega^2)/2<-\langle e^\beta,v(E) \rangle/\rk E$ for all
$E \in \exc_\beta$, and
${\frak T}^\mu= {\frak T}_{(\beta,\omega)}$ 
if $(\omega^2)/2>-\langle e^\beta,v(E) \rangle/\rk E$ for all
$E \in \exc_\beta$.
For $E \in \exc_\beta$, we have
$$
\frac{1}{(\rk E)^2} \geq 
-\frac{\langle e^\beta,v(E) \rangle}{\rk E}
\geq \frac{1}{r_0 \rk E}. 
$$
By Lemma~\ref{lem:exc} (2) and the definition of
$\exc_\beta$, we have
$1 \leq \rk E \leq r_0$.
Hence we get
$$
1 \geq -\frac{\langle e^\beta,v(E) \rangle}{\rk E}
\geq \frac{1}{r_0^2},
$$
which implies the claims.
\end{proof}

\begin{NB}
\begin{lem}
Assume that $(\omega^2)/2>\langle e^\beta,v(E) \rangle/\rk E$ for all
$E \in \exc_\beta$.
Then there is a small number $t$
such that 
$Z_{(\beta+\eta,\omega)}$ defines a stability condition on
${\frak A}^\mu$ for 
$\eta \in \NS(X) \otimes {\Bbb Q} \cap H^{\perp}$
with $(\eta^2)<t$.
\end{lem} 

\begin{proof}
Since $\exc_\beta$ is a finite set, there is a number $t$ such that
$\langle e^{\beta+\eta},v(E) \rangle/\rk E<(\omega^2)/2$
for $E \in \exc_\beta$ and $\eta \in \NS(X) \otimes {\Bbb Q} \cap H^{\perp}$
with $(\eta^2)<t$.
Let $E$ be a $(\beta+\eta)$-twisted semi-stable object of
${\frak C}$ with $\deg_G(E)=0$ and $\langle v(E)^2 \rangle=-2$.
Since $\eta$ is sufficiently small,
$E$ is $\beta$ semi-stable. Thus $E$ is an extension of
objects of $\exc_\beta$.
Hence the claim holds.
\end{proof}

Assume that
$(\omega^2)/2<\langle e^\beta,v(E) \rangle/\rk E$ for all
$E \in \exc_\beta$.
Then there is a small number $t$
such that 
$Z_{(\beta+\eta,\omega)}(F) \in {\Bbb R}_{>0}$
for $\beta$-twisted stable objects $F$ of ${\frak C}$
with $\deg_G(F)=\langle e^\beta,v(F) \rangle=0$ and
$\eta \in \NS(X) \otimes {\Bbb Q} \cap H^{\perp}$
with $(\eta^2)<t$.

If $Z_{(\alpha,\omega)}(F) \in {\Bbb R}_{ \leq 0}$,
then $\langle v(F)^2 \rangle \leq 0$.
We note that
 the set of $\mu$-semi-stable
objects $F$ with $\deg_G(F)=0$ and $\langle v(F)^2 \rangle \leq 0$
is bounded.
\end{NB}

\begin{NB}
Assume that ${\frak C}=\Per(X/Y,{\bf b}_1,\ldots,{\bf b}_n)$.
For $E \in \Coh(X)$, 
we have an exact sequence
$$
0 \to E_1 \to E \to E_2 \to 0
$$
such that $\Hom(E_1,{\cal O}_{C_{ij}}(b_{ij}))=0$ for all
$i,j$, i.e.,
$E_1 \in \Coh(X) \cap {\frak C}$ and
$E_2 \in \Coh(X) \cap {\frak C}[-1]$.
Since $E_1 \in {\frak C}$, we have an exact sequence
$$
0 \to E_1' \to E_1 \to E_1'' \to 0
$$
such that $E_1' \in {\frak T}^\mu$ and $E_1'' \in {\frak F}^\mu$.
Since $E_1''$ is a torsion free object of ${\frak C}$,
$H^{-1}(E_1'')=0$. Since $H^{-1}(E)=0$, we also have 
$H^{-1}(E_1')=0$.
Hence $E_1' \to E$ is injective in $\Coh(X)$.
We set $F:=E/E_1'$.

For $E \in \Coh(X)$ with $\rk E>0$,
we define $H$-semi-stability as
$$
(c_1(E'),H) \leq (\rk E')(c_1(E),H)/\rk E
$$
for any subsheaf $E'$ of $E$.

\begin{defn}
For $E \in \Coh(X)$,
let $E_\pi$ be the subsheaf of $E$ such that
$\pi_*(E_\pi)$ is 0-dimensional, that is $(c_1(E_\pi),H)=0$.
We set ${\frak P}:=\{E \in \Coh(X) \mid E=E_\pi, E \in {\frak C} \}$.
\end{defn}

\begin{lem}
\begin{enumerate}
\item[(1)]
For $E \in \Coh(X)$,
$E$ is $H$-semi-stable if and only if
$E/E_\pi$ is $H$-semi-stable.
\item[(2)]
For $E \in \Coh(X) \cap {\frak C}$,
$E$ is a $H$-semi-stable object of $\Coh(X)$ if and only if
$E$ is a $\mu$-semi-stable object of ${\frak C}$.
\item[(3)]
${\frak C}[-1] \cap \Coh(X)=\{E \in \Coh(X) \mid E=E_\pi, 
\Hom(A,E)=0, A \in {\frak P}\}$.
\end{enumerate}
\end{lem}

\begin{proof}
(2)
For a subsheaf $F$ of $E$, we have an exact sequence
$$
0 \to F' \to F \to F'' \to 0
$$
in $\Coh(X)$ such that $F' \in \Coh(X) \cap {\frak C}$
and $F'' \in \Coh(X) \cap {\frak C}[-1]$.
Since $\Supp(F'')$ is contained in fibers,
$F''=F''_\pi$. Thus
$(\rk F',(c_1(F'),H))=(\rk F,(c_1(F),H))$.
Therefore the claim holds. 

(3)
For $E \in \Coh(X)$ with $E=E_\pi$,
we have an exact sequence
$$
0 \to E_1 \to E \to E_2 \to 0
$$
such that $E_1 \in {\frak C} \cap \Coh(X)$ and
$E_2 \in {\frak C}[-1] \cap \Coh(X)$.
Since $E_1 \in {\frak P}$,  
$\Hom(A,E)=0$ for all $A \in {\frak P}$ if and only if $E=E_2$.
Thus the claim holds.
\end{proof}

For a morphism $\varphi:E \to E'$ of $H$-semi-stable sheaves
with $(c_1(E),H)/\rk E>(c_1(E'),H)/\rk E'$,
$\im \varphi \subset E'_\pi$.
Moreover if $E \in {\frak C}$ and 
$\Hom(A_{Z_i},E')=0$, then $\Hom(A_{Z_i},\im \varphi)=0$ implies that
$\varphi=0$.

For $E \in \Coh(X)$ with $(\rk E,(c_1(E),H)) \ne (0,0)$,
i.e., $E \ne E_\pi$,
we set
\begin{equation*}
\mu_{G,H}(E):=
\begin{cases}
\frac{(c_1(G^{\vee} \otimes E),H)}{\rk G \rk E},\;& \rk E>0\\
\infty,\; & \rk E=0.
\end{cases}
\end{equation*}
We set
\begin{equation*}
\begin{split}
\mu_{\min,G,H}(E):=& \min\{\mu_{G,H}(F) \mid 
F \in \Coh(X), E \twoheadrightarrow F, F \ne F_\pi \},\\
\mu_{\max,G,H}(E):=& \max\{\mu_{G,H}(F) \mid 
F \in \Coh(X), F \subset E, F \ne F_\pi \}.
\end{split}
\end{equation*}
Then
$\mu_{\min,G,H}(E)=\mu_{\min,G,H}(E/E_\pi)$ and
$\mu_{\max,G,H}(E)=\mu_{\max,G,H}(E/E_\pi)$.
We set
\begin{equation*}
\begin{split}
{\cal T}:=& \{E \in \Coh(X) \mid \text{(i) $E=E_\pi$, or
(ii) $E \ne E_\pi$ and $\mu_{\min,G,H}(E)>0$} \}\\
{\cal F}:=& \{E \in \Coh(X) \mid \text{(i) $E=E_\pi$, or 
(ii) $E \ne E_\pi$ and $\mu_{\max,G,H}(E) \leq 0$ } \}.
\end{split}
\end{equation*}
Then 
${\cal F} \cap {\cal T}=\{E \in \Coh(X) \mid E=E_\pi \}$.

We set
\begin{equation*}
\begin{split}
\overline{\frak T}:=& \{E \in \Coh(X) \mid E \in {\cal T} \cap {\frak C} \}\\
\overline{\frak F}:=& \{E \in \Coh(X) \mid \Hom(A,E)=0, A \in {\frak P},
E \in {\cal F} \}.
\end{split}
\end{equation*}
Then $(\overline{\frak T},\overline{\frak F})$ is a torsion pair.
We denote the tilting by $\overline{\frak A}$.
\begin{lem}
$\overline{\frak A}={\frak A}^\mu$.
\end{lem}

\begin{proof}
We note that $\overline{\frak T}={\frak T}^\mu \cap \Coh(X)$.
For $E \in {\frak T}^\mu \cap \Coh(X)[1]$,
we have $h^{-1}(E)=H^{-1}(E)_\pi$ and $\Hom(A,H^{-1}(E))=0$ 
for all $A \in {\frak P}$.
Hence $E \in \overline{\frak F}[1]$.   
For $E \in {\frak F}^\mu[1]$,
$H^{-1}(E) \in {\cal F}$ and 
the torsion freeness of $H^{-1}(E)$ in ${\frak C}$ implies that
$\Hom(A,h^{-1}(E))=0$.
Hence $E \in \overline{\frak F}[1]$.

Conversely for $E \in \overline{\frak F}[1]$,
we take the decomposition
$$
0 \to E_1 \to H^{-1}(E) \to E_2 \to 0
$$
such that $E_1 \in {\frak C} \cap \Coh(X)$ and
$E_2 \in {\frak C}[-1] \cap \Coh(X)$.
Since $\Hom(A,E_1)=0$ for all $A \in {\frak P}$,
 $E_1$ is a torsion free object of ${\frak C}$.
Hence $E_1[1] \in {\frak F}^\mu[1]$.
Since $E_2[1]$ is a torsion object
of ${\frak C}$,
$E_2[1] \in {\frak T}^\mu$.
Therefore $E \in {\frak A}^\mu$.
\end{proof}

For $E \in {\frak A}^\mu$,
we have an exact triangle
$$
E_1[1] \to E \to E_2 \to E_1[2]
$$
such that $E_1 \in {\frak F}^\mu$ and
$E_2 \in {\frak T}^\mu$.
Then $E_1$ is a torsion free object of ${\frak C}$,
which implies that $E_1 \in \Coh(X) \cap {\frak C}$. 
For $E_2$, we have an exact triangle
$$
E_2'[1] \to E_2 \to E_2'' \to E_2'[2]
$$ 
such that $E_2' \in \Coh(X) \cap {\frak C}[-1]$
and $E_2'' \in \Coh(X) \cap {\frak C}$.
Then $E_2'' \in \overline{\frak T}$.
For $F:=\Cone(E \to E_2'')[-1]$,
we have an exact triangle
$$
E_1[1] \to F \to E_2'[1] \to E_1[2].
$$
Hence $F[-1] \in \Coh(X)$.
Since $E_1 \in {\frak F}^\mu$ and 
$E_2' \in \Coh(X) \cap {\frak C}[-1]$,
we see that $F \in \overline{\frak F}[1]$.

Let $\widetilde{\frak T}^\mu$ be the subcategory of $\Coh(X)$
such that $E \in \Coh(X)$ belongs to 
$\widetilde{\frak T}^\mu$ if (i) $E$ is a torsion sheaf or (ii) 
$\mu_{\min,G}(E)>0$. 
Let $\widetilde{\frak F}^\mu$ be the subcategory of $\Coh(X)$
such that $E \in \Coh(X)$ belongs to 
$\widetilde{\frak F}^\mu$ if $E$ is a torsion free sheaf with
$\mu_{\max,G}(E) \leq 0$. 
Then $(\widetilde{\frak T}^\mu,\widetilde{\frak F}^\mu)$ 
is a torsion pair of $\Coh(X)$. 
We denote the tilting by
$\widetilde{\frak A}^\mu$.

\begin{lem}
\begin{equation*}
\begin{split}
\widetilde{\frak T}^\mu \cap {\frak C}= & \overline{\frak T},\\
\widetilde{\frak T}^\mu \cap {\frak C}[-1]=& 
\{E \in \overline{\frak F} \mid E=E_\pi, \Hom(A,E)=0, A \in {\frak P} \},\\
\widetilde{\frak F}^\mu =& 
\{E \in \overline{\frak F} \mid \text{$E$ is torsion free }\}.
\end{split}
\end{equation*}
For $E \in \overline{\frak F}$,
we have a unique decomposition in $\Coh(X)$
$$
0 \to E_1 \to E \to E_2 \to 0
$$
such that $E_1 \in \widetilde{\frak T}^\mu \cap {\frak C}[-1]$ 
and $E_2 \in \widetilde{\frak F}^\mu$.
\end{lem}

\begin{proof}
Since
$\widetilde{\frak T}^\mu={\cal T}$,
we have 
$\widetilde{\frak T}^\mu \cap {\frak C}=
{\cal T} \cap {\frak C}=
\overline{\frak T}$.

We also see that 
$\widetilde{\frak T}^\mu \cap {\frak C}[-1]=
\Coh(X) \cap {\frak C}[-1]=
\{E \in \Coh(X) \mid E=E_\pi, \Hom(A,E)=0, A \in {\frak P} \}$
and 
$\widetilde{\frak F}^\mu = 
\{E \in \overline{\frak F} \mid \text{$E$ is torsion free }\}$.

For $E \in \overline{\frak F}$,
we have a decomposition in $\Coh(X)$
$$
0 \to E_1 \to E \to E_2 \to 0
$$
such that $E_1=(E_1)_\pi$ 
and $E_2$ is torsion free.
Since $\Hom(A,E_1)=0$ for all
$A \in {\frak P}$, $E_1 \in \widetilde{\frak T}^\mu \cap {\frak C}[-1]$.
Since $E_2 \in {\cal F}$, $E_2 \in \widetilde{\frak F}^\mu$. 
\end{proof}

\begin{cor}
Let $\langle \widetilde{\frak F}^\mu[1],
\widetilde{\frak T}^\mu \cap {\frak C} \rangle$ be the subcategory 
of $\widetilde{\frak A}^\mu$ generated by
$\widetilde{\frak F}^\mu[1]$ and $\widetilde{\frak T}^\mu \cap {\frak C}$.
Then $(\langle \widetilde{\frak F}^\mu[1],
\widetilde{\frak T}^\mu \cap {\frak C} \rangle, 
\widetilde{\frak T}^\mu \cap {\frak C}[-1])$ is a torsion pair of
$\widetilde{\frak A}^\mu$ and its tilting is
${\frak A}^\mu$.
\end{cor}

For $E \in \Coh(X)$, we have an exact sequence
$$
0 \to E_1 \to E \to E_2 \to 0
$$
such that $\Hom(E_1,{\cal O}_{C_{ij}}(b_{ij}))=0$ and
$\Hom(A_{Z_i},E_2)=0$ for all $i,j$.
In ${\frak C}$, we have a decomposition
$$
0 \to E_1' \to E_1 \to E_1'' \to 0
$$
such that $E_1'$ is a torsion object or $\mu_{\min,G}(E_1')>0$, and
$\mu_{\min,G}(E_1') \leq 0$.
Then $E_1' \in \overline{\frak T}$ and
$E_1'', E_2 \in \overline{\frak F}$. 
\end{NB}

\begin{ex}\label{ex:exceptional}
Let $X$ be a $K3$ surface with $\Pic(X)={\Bbb Z}H$
and $E_0$ be an exceptional vector bundle on $X$.
We set $\beta:=c_1(E_0)/\rk E_0$.
Then $r_0=(\rk E_0)^2$ and
$v(E_0)=\sqrt{r_0}e^\beta+\sqrt{r_0}^{-1} \varrho_X$.
Hence $\langle e^\beta-(\omega^2) \varrho_X/2,v(E_0) \rangle=0$
if and only if $(\omega^2)/2=1/r_0$.
Therefore 
\begin{equation*}
{\frak A}_{(\beta,\omega)}=
\begin{cases}
{\frak A}, & (\omega^2)< 2/r_0\\
{\frak A}^\mu, & (\omega^2)> 2/r_0.
\end{cases}
\end{equation*}
\end{ex}

\begin{NB}
For $u \in {\frak R}$ with $\rk u>0$, $u$ is represented by
a $\mu$-stable object of ${\frak C}$ if and only if
there is no decomposition $u=\sum_{i=1}^s u_i$ with
$u_i \in {\frak R}$, $\rk u_i>0$.

\begin{proof}
For $\gamma=c_1(u)/\rk u$,
let $E$ be a $\gamma$-twisted semi-stable object of ${\frak C}$
with $v(E)=u$.
For a decomposition
$u=\sum_{i=1}^s u_i$, $u_i \in {\frak R}$,
$-2=\langle u,\sum_{i=1}^s u_i \rangle$ implies that
there is $u_i$ with $\langle u,u_i \rangle<0$. 
For a $\mu$-semi-stable object $E_i$ with $v(E_i)=u_i$,
we have a morphism $E_i \to E$ or
$E \to E_i$.
Hence $E$ is not $\mu$-stable.

Conversely we assume that $E$ is not $\mu$-stable and construct
a decomposition of $u$.
Assume that $E$ is not $\gamma$-twisted stable.
Since $\langle u^2 \rangle<0$,
$u=\rk u e^\gamma+a\varrho_X$, $a>0$.
For a Jordan-H\"{o}lder filtration 
$0 \subset F_1 \subset F_2 \subset \cdots \subset F_s=E$
of $E$,
$v(F_i/F_{i-1})=
r_i e^\gamma+r_i a \varrho_X+
(D_i+(D_i,\gamma)\varrho_X)$.
Hence $\langle v(F_i/F_{i-1})^2 \rangle=-2$ and
$u=\sum_i v(F_i/F_{i-1})$.
Replacing $E$ by $F_i/F_{i-1}$, 
we may assume that $E$ is $\gamma$-twisted stable.
In particular $E$ is a rigid object.
Assume that $E$ is not $\eta$-twisted semi-stable for an 
$\eta \in H^\perp$.
For the Harder-Narasimhan filtration 
$0 \subset F_1 \subset F_2 \subset \cdots \subset F_s=E$ of $E$,
$F_i/F_{i-1}$ are rigid objects. 
If $F_i/F_{i-1}$ is not $\eta'$-twisted semi-stable, then we take
the Harder-Narasimhan filtration of $F_i/F_{i-1}$.
Finally we get a filtration 
$0 \subset F_1' \subset F_2' \subset \cdots \subset F_t'=E$
of $E$ such that
each $E_i:=F_i'/F_{i-1}'$ is $\eta$-twisted semi-stable
and rigid object for all $\eta$.   
Let $F$ be a subobject of $E_i$ with
$\deg(E_i(-\beta))=\deg(F(-\beta))$.
Since
$$
\frac{\chi(E_i(-\eta))}{\rk E_i}-\frac{\chi(F(-\eta))}{\rk F}
=-(c_1(E_i)/\rk E_i-c_1(F)/\rk F,\eta)+
(\chi(E_i)/\rk E_i-\chi(F)/\rk F),
$$
if
$c_1(E_i)/\rk E_i-c_1(F)/\rk F \ne 0$, then there is an $\eta$ 
such that $E_i$ is not $\eta$-twisted semi-stable.
Hence $c_1(E_i)/\rk E_i-c_1(F)/\rk F=0$ for any subobject $F$.
Then $E_i$ is $S$-equivalent to $\oplus_j E_{ij}$ such that
$E_{ij}$ are $c_1(E_i)/\rk E_i$-twisted stable.
Then $v(E_{ij}) \in {\Bbb Q}v(E_i)$ for all $j$.
Since $\langle v(E_i)^2 \rangle<0$,
we have $\langle v(E_{ij})^2 \rangle=-2$.
Thus we get a desired decomposition of $u$.   
 \end{proof} 
\end{NB}

The following example is inspired by Washino \cite{Washino}.

\begin{ex}\label{ex:exceptional2}
Let $\pi:X \to {\Bbb P}^1$ be an elliptic $K3$ surface with a
section $\sigma$. Let $f$ be a fiber of $\pi$.
We set $H:=\sigma+4f$ and $D:=\sigma-2f$.
Then $(H^2)=6$, $(H,D)=0$ and $(D^2)=-6$.
We set $\beta+\sqrt{-1}\omega=xD+\sqrt{-1}yH$, $x,y \in {\Bbb R}$.
For $\eta:=xD$, $x<1/2$,
$\chi({\cal O_X(\eta),\cal O}_X(D))<\chi({\cal O}_X(\eta),{\cal O}_X)$.
Then a non-trivial extension
\begin{align*}
0 \to {\cal O}_X(D) \to E_1 \to {\cal O}_X \to 0
\end{align*}
defines an $\eta$-twisted stable sheaf.
We set $u_1:=v({\cal O}_X)$, $u_2:=v({\cal O}_X(D))$ and
$u_3:=v(E_1)=u_1+u_2$.
Then the equations for $W_{u_i}$, $i=1,2,3$ are
\begin{align*}
W_{u_1}:\ \ & x^2+y^2=1/3,\\
W_{u_2}:\ \ & (x-1)^2+y^2=1/3,\\
W_{u_3}:\ \ & (x-1/2)^2+y^2=1/12.
\end{align*} 
They pass through the point $(\frac{1}{2},\frac{1}{2\sqrt{3}})$.
By the action of $R_{u_1}$,
$W_{u_2}$ and $W_{u_3}$ are exchanged.
It is easy to see that
\begin{equation*}
\begin{cases}
{\frak R}_{D/3}= \{ u_1,u_3 \},\\
{\frak R}_{D/2}= \{ u_1,u_2,u_3 \}.
\end{cases}
\end{equation*}
For $\beta=D/3$, we have three categories 
${\frak A},{\frak A}^\mu, {\frak A}_3$:
\begin{equation*}
{\frak A}_{(\beta,\omega)}=
\begin{cases}
{\frak A}, & (\omega^2)/2< 1/6\\
{\frak A}_3, & 1/6<(\omega^2)/2<2/3\\
{\frak A}^\mu, & (\omega^2)/2> 2/3.
\end{cases}
\end{equation*}
For $\beta=D/2$, 
we have
\begin{equation*}
{\frak A}_{(\beta,\omega)}=
\begin{cases}
{\frak A}, & (\omega^2)/2< 1/4\\
{\frak A}^\mu, & (\omega^2)/2> 1/4.
\end{cases}
\end{equation*}
In this example, $u_1,u_2$ generate a negative definite lattice 
of type $A_2$.
\end{ex}

\begin{NB}
\begin{ex}
We set $\beta:=D/3$.
Then
$\exc_\beta=\{E_1,E_2 \}$ such that
$E_1={\cal O}_X$ and $E_2$ is a $\beta$-twisted stable sheaf
defined by a non-trivial extension
\begin{align*}
0 \to {\cal O}_X(D) \to E_2 \to {\cal O}_X \to 0.
\end{align*}
We have
\begin{equation*}
\begin{split}
v(E_1) &=:e^\beta+\frac{2}{3}\varrho_X-(\beta+(\beta,\beta)\varrho_X),\\
v(E_2) &=:2e^\beta+\frac{1}{3}\varrho_X+(\beta+(\beta,\beta)\varrho_X).
\end{split}
\end{equation*}
We have $v(E_1)=1+\varrho_X$ and $v(E_2)=2e^{\frac{D}{2}}+
\frac{1}{2}\varrho_X$.

For $E \in M_H^\beta(4 e^{\frac{D}{2}})$, 
$\chi(E,L_\pm)=1$.
If $\Ext^1(L_+,E) \ne 0$, then
the non-trivial extension $F$ is a $\beta$-twisted stable
object of ${\frak C}$ with $\langle v(F)^2 \rangle=-4$.
Hence $\Ext^1(L_+,E)=0$.
Then $E':=\ker(E \to \Hom(E,L_+)^{\vee} \otimes L_+)$
is a simple object of ${\frak C}$ with 
$v(E')=3+(\sigma-2f)-\varrho_X$.
Since $\chi(E',E_1)=1$, we will also have
a simple object 
$E'':=\ker(E' \to \Hom(E',E_1)^{\vee} \otimes E_1)$.
We will see that $E''=I_x$, $x \in X$.
Thus $R_{L_+}R_{E_1}R_{L_-}({\cal O}_\Delta)$ will gives 
the universal family.

$\beta:=(\sigma-2f)/2$.
Then $\exc_\beta=\{{\cal O}_X,{\cal O}_X(D) \}$
and $4e^\beta=4+2D-3 \varrho_X$ is a primitive vector.
We set $L_-:={\cal O}_X$ and $L_+:={\cal O}_X(D)$.
Then
$\exc_\beta=\{L_\pm \}$ 
and
$v(L_\pm)=e^\beta+\frac{1}{4}\varrho_X \pm(\beta+(\beta,\beta)\varrho_X)$.

Assume that $\langle v,v(L_\pm)\rangle>0$.
Then for $w:=v+\langle v,v(L_+)+v(L_-) \rangle(v(L_+)+v(L_-))$,
we have 
\begin{equation*}
\begin{split}
\langle w,v(L_-) \rangle=& -\langle v,v(L_+) \rangle<0,\\
 \langle w,v(L_+) \rangle=& -\langle v,v(L_-) \rangle<0,\\ 
\langle w,v(L_+)+v(L_-) \rangle=& -\langle v,v(L_+)+v(L_-) \rangle <0.
\end{split}
\end{equation*}
\begin{NB2}
$R_{v(L_+)}R_{v(L_+)+v(L_-)}R_{v(L_-)}=R_{v(L_-)}R_{v(L_+)+v(L_-)}R_{v(L_+)}
=R_{v(L_+)+v(L_-)}$.
\end{NB2}

We now give a simple example of the wall-crossing behavior. 
Let $E$ be a stable object with $v(E)=w$.
Assume that $(c_1(E),H)=1$, that is, $c_1(E)=f+m(\sigma-2f)$, 
$m \in {\Bbb Z}$.
If $(\omega^2)>1/4$, then
$\Hom(E,L_\pm)=0$.
Let $E_1$ be a non-trivial extension fitting in
$$
0 \to L_+ \to E_1 \to L_- \to 0
$$
For a general member $E$, 
$E':=R_{v(L_-)}^2(E)$ fits in an exact sequence
$$
0 \to \Hom(L_-,E) \otimes L_- \to E \to R_{v(L_-)}^2(E)
\to \Ext^1(L_-,R_{v(L_-)}(E)) \otimes L_- \to 0.
$$  
We also see that
$E'':=R_{v(E_1)}^2(E')$ and $E''':=R_{v(L_+)}^2(E'')$
fit in exact sequences
 $$
0 \to \Hom(E_1,E') \otimes E_1 \to E' \to E''
\to \Ext^1(E_1,R_{v(E_1)}(E')) \otimes E_1 \to 0
$$ 
and
$$
0 \to \Hom(L_+,E'') \otimes L_+ \to E'' \to E'''
\to \Ext^1(L_+,R_{v(L_+)}(E'')) \otimes L_+ \to 0.
$$ 
Then $E'''$ is a stable object of ${\frak A}$, where
$(\omega^2)<1/4$.

\begin{NB2}
Let $v:=(r,mD,-a)$ be a $(-2)$-vector.
Then $v \in \oplus_{k \in {\Bbb Z}}
{\Bbb Z}_{\geq 0}v({\cal O}_X(kD))$.
\begin{proof}
By the action of $\otimes {\cal O}_X(D)$,
we may assume that $-r/2 \leq m \leq r/2$.
Since $-6m^2+2ra=-2$, $a/r=3(m/r)^2-1/r^2 \leq 3/4-1/r^2$.
Hence $0<a/r<1$ or $v=v({\cal O}_X)$.
In the first case, 
$v=(a,mD,-r)+(r-a)v({\cal O}_X)$.
Thus by the induction on the rank of $v$, we get the claim.  
\end{proof}
\end{NB2}

\end{ex}

\begin{ex}\label{ex:exceptional3}
Let $\pi:X \to {\Bbb P}^1$ be an elliptic $K3$ surface with a
section $\sigma$. Let $f$ be a fiber of $\pi$.
We set $H:=\sigma+3f$ and $\beta:=(\sigma-f)/2$.
We set $L_-:={\cal O}_X$ and $L_+:={\cal O}_X(\sigma-f)$.
Then $\exc_\beta=\{L_\pm \}$ 
and $v(L_\pm)=e^\beta+\frac{1}{2}\varrho_X \pm(\beta+(\beta,\beta)\varrho_X)$.
Hence 
\begin{equation*}
{\frak A}_{(\beta,\omega)}=
\begin{cases}
{\frak A}, & (\omega^2)/2< 1/2\\
{\frak A}^\mu, & (\omega^2)/2> 1/2.
\end{cases}
\end{equation*}
$M_H^\beta(2 e^\beta) \cong X$ and the universal family 
${\cal E}$ on $X \times M_H^\beta(2 e^\beta)$
fits in an exact sequence
$$
0 \to {\cal E} \to L_- \boxtimes L_-^{\vee} \oplus
L_+ \boxtimes L_+^{\vee} \to {\cal O}_\Delta \to 0.
$$  
\begin{NB2}
We set $D:=\sigma-f$. Then $(H^2)=-(D^2)=4$.
We set $\beta+\sqrt{-1}\omega:=xD+\sqrt{-1}yH$.
Then 
$W_{v(L_-)}:x^2+y^2=1/2$ and
$W_{v(L_+)}:(x-1)^2+y^2=1/2$.
By the action of $R_{v(L_-)}$, $W_{v(L_+)}$ is invariant.
\end{NB2}
\end{ex}

\begin{ex}
Let $\pi:X \to {\Bbb P}^1$ be an elliptic $K3$ surface with a
section $\sigma$. Let $f$ be a fiber of $\pi$.
We set $H:=\sigma+5f$ and $D:=\sigma-3f$.
Then $(H^2)=8$, $(H,D)=0$ and $(D^2)=-8$.
Let $v:=r+dD-a \varrho_X$, $r,a>0$ be a primitive isotropic vector.
Then $-4d^2+ra=0$. Hence $4=(r/d)(a/d)$. 
If $|d| \leq r/2$, 
then $|a/d| \leq 2$. Hence $r \geq 2|d| \geq a$.
If the equality holds, then $v=2 \pm D-2 \varrho_X$.
\end{ex}

\end{NB}

\begin{defn}
Let $W$ be a wall of ${\Bbb R}_{>0}H$.
${\frak S}_W$ denotes the category of $\beta$-twisted semi-stable objects $E$ 
of ${\frak C}$ with $\rk E>0$, $\deg(E(-\beta))=0$ and 
$\chi(E(-\beta))= (\omega^2) \rk E /2$ (constant).
\end{defn}

\begin{lem}\label{lem:S_W}
Assume that $\{ E_{\overline{\frak k}}\mid E \in \exc_\beta \}=
\exc_{\overline{\beta}}$.
Then ${\frak S}_W$ is generated by  
$E \in \exc_\beta$ with
$W_{\beta,v(E)}=W$.
\end{lem}

\begin{lem}\label{lem:irreducible}
Assume that $\omega$ belongs to a wall $W$.
We take $\omega_\pm \in {\Bbb Q}_{>0}H$ such that
$\omega_\pm $ are sufficiently close to $\omega$ 
and $(\omega_-^2)<(\omega^2) <(\omega_+^2)$.
\begin{enumerate}
\item[(1)]
Let $E_1$ be a subobject of $E \in {\frak S}_W$ 
in ${\frak A}_{(\beta,\omega_-)}$. 
Then $E_1 \in {\frak S}_W$. 
In particular, $E \in {\frak S}_W \cap \exc_\beta$ 
is an irreducible object of ${\frak A}_{(\beta,\omega_-)}$.

\item[(2)]
Let $E_1$ be a subobject of $E[1] \in {\frak S}_W[1]$ 
in ${\frak A}_{(\beta,\omega_+)}$.
Then $E_1 \in {\frak S}_W[1]$.
In particular, $E[1] \in ({\frak S}_W \cap \exc_\beta)[1]$ 
is an irrediducible object of ${\frak A}_{(\beta,\omega_+)}$. 
\end{enumerate}
\end{lem}

\begin{proof}
We set $G:=G_{(\beta,\omega)}, 
G_\pm:=G_{(\beta,\omega_\pm)} \in K(X)_{\Bbb Q}$.

(1)
Assume that there is an exact sequence 
\begin{align*}
0 \to E_1 \to E \to E_2 \to 0
\end{align*}
in ${\frak A}_{(\beta,\omega_-)}$.
Then we have an exact sequence
\begin{align*}
0 \to H^{-1}(E_2) \overset{\varphi}{\to} H^0(E_1) \to E \to H^0(E_2) \to 0
\end{align*}
in ${\frak C}$.
As in the proof of Lemma~\ref{lem:stable-irreducible},
we see that
\begin{align*}
\deg_G(H^{-1}(E_2))=\deg_G(H^0(E_1))=\deg_G(H^0(E_2))=0.
\end{align*}
Assume that $E_1 \ne 0$.
Then $\chi_{G_-}(H^0(E_1))>0$ and $\chi_{G_-}(H^{-1}(E_2)) \leq 0$.
Since $\omega_-$ is sufficiently close to $\omega$,
we have 
\begin{align*}
\chi_G(\im \varphi)
=     \chi_G(H^0(E_1))-\chi_{G}(H^{-1}(E_2)) 
\geq -\chi_{G}(H^{-1}(E_2)).
\end{align*}
If $H^{-1}(E_2) \ne 0$, then
$\rk H^{-1}(E_2) \ne 0$,
which implies that
$\chi_{G}(H^{-1}(E_2))< \chi_{G_-}(H^{-1}(E_2)) \leq 0$.
Therefore $\chi_G(\im \varphi)>0$.
By the semi-stability of $E$ and $\chi_G(E)=0$,
$\chi_G(\im \varphi) \leq 0$, which is a contradiction.
Therefore $H^{-1}(E_2)=0$ and
$E_1$ is a $\beta$-twisted semi-stable object
with $\chi_G(E_1)=0$. Thus $E_1 \in {\frak S}_W$.

(2)
Assume that there is an exact sequence in ${\frak A}_{(\beta,\omega_+)}$ 
$$
0 \to E_1 \to E[1] \to E_2 \to 0.
$$
Then we have an exact sequence
$$
0 \to H^{-1}(E_1) \to E \to 
H^{-1}(E_2) \overset{\varphi}{\to} H^0(E_1) \to 0.
$$

Assume that $E_2 \ne 0$.
If $H^0(E_1) \ne 0$, then
$\chi_G(H^0(E_1)) \geq \chi_{G_+}(H^0(E_1))>0$.
\begin{NB}
If $\rk H^0(E_1)>0$, then 
$\chi_G(H^0(E_1))> \chi_{G_+}(H^0(E))$.
\end{NB}
Since $\omega_+$ is sufficiently close to $\omega$,
we have $\chi_G(H^{-1}(E_2)) \leq 0$.
Hence $\chi_G(\ker \varphi) <0$, which is a contradiction.
Therefore $H^0(E_1)=0$.
Since $\chi_G(H^{-1}(E_2)) \leq 0$,
$H^{-1}(E_1)$ and $H^{-1}(E_2)$ are 
$\beta$-twisted semi-stable objects
with $\chi_G(H^{-1}(E_1))=\chi_G(H^{-1}(E_2))=0$.
Therefore $E_1 \in {\frak S}_W[1]$.
\end{proof}

\begin{lem}\label{lem:ADE}
For a fixed $G=G_{(\beta,\omega)}$,
$\{E \in K(X) \mid \deg_G(E)=\chi_G(E)=0 \}$
is a negative definite sublattice of $A^*_{\alg}(X)$.
In particular, the sublattice $\langle {\frak S}_W \rangle$
generated by ${\frak S}_W$ is a direct sum of lattices of type $ADE$.
\end{lem}

\begin{proof}
The signature of $A^*_{\alg}(X)$ is $(2,\rho(X))$.
Since $\langle v(G)^2 \rangle=r_0^2 (\omega^2)$,
$(H^2)>0$ and $(H+(H,\beta)\varrho_X) \perp v(G)$,
we get the claim.
\end{proof}

\begin{NB}
$\{ \eta+(\eta,\beta)\varrho_X \mid \eta \in H^{\perp} \}$
is a codimension 1 subspace of
$v(G)^{\perp} \cap (H+(H,\beta)\varrho_X)^{\perp}$.
\end{NB}


\subsection{Relation with the Fourier-Mukai transforms}
\label{subsect:parameter-FM}

Let $X'$ be an abelian surface or a $K3$ surface.
Let
\begin{equation*}
\begin{matrix}
\Phi_{X \to X'}^{{\bf E}^{\vee}}:& 
{\bf D}(X) & \to & {\bf D}(X')\\
& E & \mapsto & {\bf R}p_{X' *}(p_X^*(E) \otimes {\bf E}^{\vee})
\end{matrix}
\end{equation*}
be the Fourier-Mukai transform whose kernel is 
${\bf E} \in {\bf D}(X \times X')$.
Assume that $v({\bf E}_{|X \times \{ x' \}})=r_0 e^\beta$,
$r_0>0$. 
For simplicity, we set
$\Phi:=\Phi_{X \to X'}^{{\bf E}^{\vee}}$.
For $C \in \NS(X)_{\Bbb Q}$, we set
$\widehat{C}:=-c_1(\Phi(C+(C,\beta)\varrho_X)) 
\in \NS(X')_{\Bbb Q}$.
Since $\Phi(e^{\beta})=(1/r_0)\varrho_{X'}$ 
and $\Phi(\varrho_{X})=r_0 e^{\beta'}$, we have
\begin{equation}\label{eq:mv:FMimg}
v(\Phi(E)[1])=-r_0 a e^{\beta'}-\frac{r}{r_0}\varrho_{X'}+
(d\widehat{H}+\widehat{D})+(d\widehat{H}+\widehat{D},\beta')\varrho_{X'},
\end{equation}
where $v(E)$ is given by \eqref{eq:Mukai-vector}.
We set $b':=(\beta',\widehat{H})/(\widehat{H}^2)$.
Then we have a diffeomorphism
\begin{equation*}
{\frak I}_{bH} \to {\frak I}_{b' \widehat{H}}
\end{equation*}
induced by $\Phi$.
Let ${\frak H}'_{\Bbb R}$ be the space for $(X',\widehat{H})$
in Definition~\ref{defn:parameter-space}.
We shall study the relation of 
${\frak H}_{\Bbb R}$ and ${\frak H}'_{\Bbb R}$ explicitly.
For $\eta+\sqrt{-1}\omega$ with $(\eta,H)=0$,
we set 
\begin{equation}\label{eq:Phi(space)}
\widetilde{\eta}+\sqrt{-1}\widetilde{\omega}:=
-\frac{2}{r_0 ((\eta+\sqrt{-1}\omega)^2)}
(\widehat{\eta}+\sqrt{-1}\widehat{\omega}).
\end{equation}
Then we have
\begin{equation}\label{eq:Phi(exp)}
\begin{split}
\Phi(e^{\beta+\eta+\sqrt{-1}\omega})
&=
\Phi\Bigl(
 e^\beta
  \bigl(
   1+(\eta+\sqrt{-1}\omega)+
   \frac{((\eta+\sqrt{-1}\omega)^2)}{2}\varrho_X 
  \bigr)
 \Bigr)
\\
&= 
\Phi\Bigl(
 e^\beta+((\eta+\sqrt{-1}\omega)+(\eta+\sqrt{-1}\omega,\beta)\varrho_X)
+\frac{((\eta+\sqrt{-1}\omega)^2)}{2}\varrho_{X}
\Bigr)
\\
&= 
\frac{1}{r_0}\varrho_{X'}+
r_0 \frac{((\eta+\sqrt{-1}\omega)^2)}{2} e^{\beta'}
-((\widehat{\eta}+\sqrt{-1}\widehat{\omega})+
(\widehat{\eta}+\sqrt{-1}\widehat{\omega},\beta')\varrho_{X'})
\\
&= 
r_0 \frac{((\eta+\sqrt{-1}\omega)^2)}{2}
e^{\beta'+\widetilde{\eta}+\sqrt{-1}\widetilde{\omega}}.
\end{split}
\end{equation}
Hence
\begin{equation}\label{eq:stability-comm}
\begin{split}
Z_{(\beta'+\widetilde{\eta},\widetilde{\omega})}(\Phi(E)[1])
=& \langle e^{\beta'+\widetilde{\eta}+\sqrt{-1}\widetilde{\omega}},
v(\Phi(E)[1])\rangle\\
=& -\frac{2}{r_0((\eta+\sqrt{-1}\omega)^2)} 
\langle -\Phi(e^{\beta+\eta+\sqrt{-1}\omega}),v(\Phi(E)[1]) \rangle\\
=& -\frac{2}{r_0((\eta+\sqrt{-1}\omega)^2)} 
\langle e^{\beta+\eta+\sqrt{-1}\omega},v(E) \rangle\\
=&-\frac{2}{r_0((\eta+\sqrt{-1}\omega)^2)} 
Z_{(\beta+\eta,\omega)}(E).
\end{split}
\end{equation}
Therefore ${\frak H}_{{\Bbb R}} \to {\frak H}_{\Bbb R}'$
is the correspondence given by
\begin{equation}
(\beta-bH)+\eta+\sqrt{-1}\omega \mapsto
(\beta'-b'  \widehat{H})+\widetilde{\eta}+
\sqrt{-1}\widetilde{\omega}.
\end{equation}

\begin{rem}
Since $(\eta,H)=0$, we have
$-((\eta+\sqrt{-1}\omega)^2)=(\omega^2)-(\eta^2)>0$.
We also have
\begin{equation*}
((\widetilde{\omega}^2)-(\widetilde{\eta}^2))
((\omega^2)-(\eta^2))=\frac{4}{r_0^2}.
\end{equation*}
\end{rem}

\begin{lem}\label{lem:Phi(space)}
Assume that $\widehat{\omega}$ is nef and big.
\begin{enumerate}
\item[(1)]
The correspondence $(\eta,\omega) \mapsto
(\widetilde{\eta},\widetilde{\omega})$ induced by $\Phi[1]$
preserves the structure of chamber.
\item[(2)]
If $({\frak A}_{(\beta+\eta,\omega)},Z_{(\beta+\eta,\omega)})$
is a stability condition on $X$, then
$\Phi[1]$ induces a stability condition
\\
$(\Phi[1]({\frak A}_{(\beta+\eta,\omega)}),
Z_{(\beta'+\widetilde{\eta},\widetilde{\omega})})$
on $X'$.
\end{enumerate}
\end{lem}

\begin{proof}
(1)
Let ${\frak R}'$ be the set in Definition~\ref{defn:wall:category}
associated to $X'$.
Then $-\Phi({\frak R})={\frak R}'$.
By \eqref{eq:stability-comm},
$(\eta,\omega) \in W_u$ if and only if 
$(\widetilde{\eta},\widetilde{\omega}) 
\in W_{-\Phi(u)}$.
Hence the claim holds.  
(2) is obvious.
\end{proof}

\begin{NB}
\begin{rem}
We note that by \cite{HMS}, $\widehat{\omega}$
belongs to the positive cone.
Hence if $\widehat{\omega}$ is not nef and big, then 
there is an auto-equivalence $\Phi':{\bf D}(X') \to {\bf D}(X')$
such that $\Phi'(\widehat{\omega}+
(\widehat{\omega},\beta')\varrho_{X'})$
is nef and big.
Thus replacing $\Phi$, we may assume that $\widehat{\omega}$ is nef and big.
\end{rem}
\end{NB}

Let $U$ be a $\beta$-twisted stable object of ${\frak C}$
with $\rk U=r$, $\deg(U(-\beta))=0$ and $\langle v(U)^2 \rangle=-2$.
We set $\gamma:=c_1(U)/r$. 
Then $v(U)=r e^\gamma+(1/r)\varrho_X$.
We set
$${\bf E}:=\Cone(U \boxtimes U^{\vee} \to {\cal O}_\Delta)[-1].$$
Then $v({\bf E}_{|X \times \{ x \}})=r^2 e^\gamma$.
Since the Fourier-Mukai transform
$\Phi[1]:=\Phi_{X \to X}^{{\bf E}^{\vee}[1]}$ induces a $(-2)$-reflection 
$R_{v(U)}$,
we get $\widehat{\xi}=\xi$ for any $\xi \in H^2(X,{\Bbb Q})$.
We write $\gamma=bH+\upsilon$, $\upsilon \in H^{\perp}$.
Then we have a diffeomorphism
\begin{equation*}
\begin{matrix}
\Phi_{\frak H}:& \overline{\frak H}_{\Bbb R} & \to &  
\overline{\frak H}_{\Bbb R}\\
& \eta+\sqrt{-1}\omega +\upsilon & \mapsto & 
\dfrac{2(\eta+\sqrt{-1}\omega)}{r^2(-(\eta^2)+(\omega^2))}+\upsilon
\end{matrix}
\end{equation*}


\subsection{A small perturbation of $\sigma_{(\beta,\omega)}=({\frak A}_{(\beta,\omega)},Z_{(\beta,\omega)})$}
\label{subsect:eta}

We take $\eta \in \NS(X)_{\Bbb Q}$ with
$(H,\eta)=0$.
We shall study the perturbed stability condition
$\sigma_{(\beta+\eta,\omega)}$ of 
$\sigma_{(\beta,\omega)}$.
Since 
\begin{align*}
(c_1(E)-(\rk E)(\beta+\eta),H)=
  (c_1(E)-(\rk E) \beta,H)
\end{align*}
for $E \in {\bf D}(X)$,
$({\frak T}^\mu,{\frak F}^\mu)$ and ${\frak A}^\mu$
do not depend on the choice of $\eta$.
Since 
\begin{equation*}
\begin{split}
e^{\beta+\eta+\sqrt{-1}\omega}=&
e^{\beta+\sqrt{-1}\omega}+\eta+(\eta,\beta+\sqrt{-1}\omega)\varrho_X+
\frac{(\eta^2)}{2}\varrho_X\\
=& 
e^{\beta+\sqrt{-1}\omega}+\eta+(\eta,\beta)\varrho_X+
\frac{(\eta^2)}{2}\varrho_X,
\end{split}
\end{equation*}
we have
\begin{equation}\label{eq:eta-Z}
\begin{split}
Z_{(\beta+\eta,\omega)}(E)
&=Z_{(\beta,\omega)}(E)
  +\langle \eta+(\eta,\beta)\varrho_X,v(E) \rangle
  -\frac{(\eta^2)}{2}\rk E
\\
&=(-a+(\eta,D)-\frac{(\eta^2)}{2}r)
  +r\frac{(\omega^2)}{2}+\sqrt{-1} d(H,\omega).
\end{split}
\end{equation}

\begin{lem}\label{lem:eta:a<0}
Assume that $\eta$ satisfies 
\begin{NB}
Old version
\[-2(\eta^2)<1/r_0^2,\quad
  (\omega^2)>-3(\eta^2),\quad 
 \sqrt{-2(\eta^2)}-(\eta^2)/2<(\omega^2)/2.\]
\end{NB}
\begin{equation}\label{eq:ass-eta}
-(\eta^2)<\min\{r^2(\omega^2)^2/8,(\omega^2)\}.
\end{equation}
Let $E$ be a $(\beta+\eta)$-twisted stable
object of ${\frak C}$ with $\deg(E(-\beta))=0$.
If $\chi(E(-\beta)) \leq 0$, then
$Z_{(\beta+\eta,\omega)}(E) \in {\Bbb R}_{>0}$.
\end{lem}

\begin{proof}
We set 
$$
v(E)=r e^\beta+a \varrho_X+(D+(D,\beta)\varrho_X),\quad D \in H^{\perp}.
$$
Then $a=\chi(E(-\beta)) \leq 0$.
Since $-2 \leq \langle v(E)^2 \rangle=-2ra+(D^2)$,
we have $-(D^2) \leq 2(1-ra)$.
Since $H^{\perp}$ is negative definite,
the Schwarz inequality implies that
\begin{equation}\label{eq:Schwarz}
|(\eta,D)| \leq \sqrt{-(\eta^2)}\sqrt{- (D^2)} \leq \sqrt{-2(\eta^2)(1-ra)}.
\end{equation}
Then we see that
\begin{equation}
\begin{split}
& \left(-a+\frac{(\omega^2)-(\eta^2)}{2}r \right)^2-(-2(\eta^2)(1-ra))\\
\geq & (-a)((\omega^2)-(\eta^2))r-2(-(\eta^2))r(-a)+
\left(\frac{(\omega^2)-(\eta^2)}{2}r \right)^2-2(-(\eta^2))\\
= & r(-a)((\omega^2)+(\eta^2))+
\left(\frac{(\omega^2)-(\eta^2)}{2}r \right)^2-2(-(\eta^2))>0,
\end{split}
\end{equation}
where we used the inequalities $(\omega^2)+(\eta^2)>0$ and
$((\omega^2)r/2)^2>2(-(\eta^2))$ coming from \eqref{eq:ass-eta}.
Hence we have
$$
|(\eta,D)|< -a+\frac{(\omega^2)-(\eta^2)}{2}r.
$$
\begin{NB}
Old version(2012/11/1):
We first assume that $a<0$. Then $-a \geq 1/r_0$.
Assume that $-2(\eta^2)<1/r_0^2$ and $(\omega^2)>-3(\eta^2)$.
Since $-(\eta^2)<(\omega^2)/2+(\eta^2)/2$,
we can take $\lambda \in {\Bbb Q}$ such that
$-(\eta^2)<\lambda<(\omega^2)/2+(\eta^2)/2$.
Then 
\begin{equation*}
\begin{split}
(\lambda r-a)^2-(-2(\eta^2)(1-ra))=&
\lambda^2 r^2+a^2+2(\eta^2)-(2\lambda+2(\eta^2))ra\\
 >&
\lambda^2 r^2+(1/r_0^2+2(\eta^2))-(2\lambda+2(\eta^2))ra>0.
\end{split}
\end{equation*}
Hence $(\eta,D)^2 <(\lambda r-a)^2$, which implies that
$-(\lambda r-a)<(\eta,D)<(\lambda r-a)$.
\end{NB}
Thus we get
\begin{equation*}
\begin{split}
Z_{(\beta+\eta,\omega)}(E)=& (-a+(\eta,D)-\frac{(\eta^2)}{2}r)+
r\frac{(\omega^2)}{2}>0.
\end{split}
\end{equation*}
\begin{NB}
Old version (2012/11/1):
Assume that $a=0$. Since $d=0$, we have $r>0$.
We further assume that
$\sqrt{-2(\eta^2)}-(\eta^2)/2<(\omega^2)/2$.
In this case, we have
$-\sqrt{-2(\eta^2)} \leq (\eta,D) \leq \sqrt{-2(\eta^2)}$ by
\eqref{eq:Schwarz}.
Hence
$$
Z_{(\beta+\eta,\omega)}(E) \geq 
\left(
-\sqrt{-2(\eta^2)}+\frac{(\eta^2)}{2}+\frac{(\omega^2)}{2}
\right)r>0.
$$
\end{NB}
\begin{NB}
$-(D^2) \leq 2$ implies that the choice of $D$ is finite.
\end{NB} 
\end{proof}

\begin{lem}\label{lem:eta-F}
Under the assumption in Lemma~\ref{lem:eta:a<0},
${\frak F} \subset {\frak F}_{(\beta+\eta,\omega)}$.
\end{lem}

\begin{proof}
Let $E$ be an object of ${\frak F}$ with 
$$
v(E)=re^\beta+a \varrho_X+(D+(D,\beta)\varrho_X),\quad a \leq 0.
$$
We take a $(\beta+\eta)$-twisted stable subobject $E_1$
of $E$ such that 
$$
\dfrac{\chi(E_1(-\beta-\eta))}{\rk E_1}
=\max\left\{\dfrac{\chi(F(-\beta-\eta))}{\rk F} 
            \,\Big|\, F \subset E,\ \deg(F(-\beta))=0
     \right\}. 
$$
We set
$$
v(E_1)=r_1 e^\beta+a_1 \varrho_X+(D_1+(D_1,\beta)\varrho_X).
$$
Since $E_1 \in {\frak F}$ and $\deg(E_1(-\beta))=0$, 
$a_1 \leq 0$.
Then
$$
Z_{(\beta+\eta,\omega)}(E_1)=
\bigl(-a_1+(\eta,D_1)+\dfrac{(\eta^2)}{2}r_1 \bigr)+
r_1 \frac{(\omega^2)}{2}>0,
$$
which implies that $E \in {\frak F}_{(\beta+\eta,\omega)}$.
\end{proof}

\begin{NB}
We set
$$
v(E_i):=r_i e^\beta+a_i \varrho_X+D_i+(D_i \beta)\varrho_X.
$$

$$
(-a_i+\frac{(\omega^2)}{2}r_i
$$

Let $E$ be a $\beta+\eta$-twisted stable object of ${\frak C}$
with $\deg_G(E)=0$.
Assume that $Z_{(\beta+\eta,\omega)}(E)=0$.
We may assume that $a>0$.
Since $\langle v(E)^2 \rangle \geq -2$,
we see that $\langle v(E)^2 \rangle=-2$.
Then there is a $\beta$-twisted semi-stable object $E'$ with $v(E')=v(E)$
and $E'$ is generated by objects $E_{i_1},\ldots,E_{i_t}$
of $\exc_\beta$ such that
there is $\omega_0 \in {\Bbb Q}_{>0}H$ with
$a_{i_j}/r_{i_j}=(\omega_0^2)/2$.
In particular,
$v(E)=\sum_{j=1}^{t} n_j v(E_{i_j})$, $n_j \in {\Bbb Z}_{\geq 0}$. 
Since $0 \leq -(D_{i_j}^2)=2(1-r_{i_j} a_{i_j})<2$,
$|(D_{i_j},\eta)| \leq \sqrt{-(D_{i_j}^2)}\sqrt{-(\eta^2)}<\sqrt{-2(\eta^2)}$. 
Therefore 
$$
\left(\frac{(\eta^2)}{2}-\sqrt{-2(\eta^2)}\right)r
\leq  (D,\eta)+\frac{(\eta^2)}{2}r
\leq \left(\frac{(\eta^2)}{2}+\sqrt{-2(\eta^2)}\right)r
$$
Then
$$
\frac{(\omega^2)}{2}=
\frac{a}{r}-\frac{(D,\eta)}{r}-\frac{(\eta^2)}{2}
=\frac{(\omega_0^2)}{2}-\frac{(D,\eta)}{r}-\frac{(\eta^2)}{2}
$$
with $|-\frac{(D,\eta)}{r}-\frac{(\eta^2)}{2}|< N_\epsilon$,
where $N_\epsilon=\frac{3}{4}\sqrt{-2(\eta^2)}$.
\end{NB}

\begin{lem}\label{lem:eta-ss}
Let $E$ be a $\mu$-semi-stable object of ${\frak C}$ with
$v(E)=re^\beta+a \varrho_X+D+(D,\beta)\varrho_X$, $a>0$.
Assume that $-(\eta^2)<1/(2r_0^2 r^5)$. 
\begin{enumerate}
\item[(1)]
Let $E_1$ be a subobject of $E$.
Then
\begin{equation*}
\frac{\chi(E(-\beta-\eta))}{\rk E}\underset{(>)}{<}
\frac{\chi(E_1(-\beta-\eta))}{\rk E_1}
\end{equation*}
if  
\begin{equation*}
\frac{\chi(E(-\beta))}{\rk E}\underset{(>)}{<}
\frac{\chi(E_1(-\beta))}{\rk E_1}.
\end{equation*}
\item[(2)]
$E$ is $(\beta+\eta)$-twisted semi-stable if and only if 
$E$ is $\beta$-twisted semi-stable and 
$$
\frac{-\langle v(F),\eta+(\eta,\beta)\varrho_X \rangle}{\rk F}
 \leq 
\frac{-\langle v(E),\eta+(\eta,\beta)\varrho_X \rangle}{\rk E}
$$
for any subobject $F$ with $\chi_G(F(n))/\rk F=\chi_G(E(n))/\rk E$.
\end{enumerate}
\end{lem}

\begin{proof}
Let $E_1$ be a $\mu$-semi-stable object of $E$ and $E_2:=E/E_1$.
We set 
\begin{align*}
v(E_i)=r_i e^\beta+a_i \varrho_X+(D_i+(D_i,\beta)\varrho_X),\quad
D_i \in \NS(X)_{\Bbb Q}\cap H^{\perp} \quad (i=1,2).
\end{align*}
By \cite[Lem. 1.1]{tori}, we have
\begin{equation*}
\frac{\langle v(E)^2 \rangle}{r}=
\sum_{i=1}^2 \frac{\langle v(E_i)^2 \rangle}{r_i}-
\sum_{i=1}^2 r_i \left( \frac{D_i}{r_i}-\frac{D}{r}\right)^2.
\end{equation*}
Since $\langle v(E_i)^2 \rangle \geq -2 r_i^2$ and 
$\langle v(E)^2 \rangle=-2ra+(D^2)<0$,
we get
$$
2r \geq -\sum_{i=1}^2 r_i \left( \frac{D_i}{r_i}-\frac{D}{r}\right)^2.
$$
Since $D_2/r_2-D/r=-(r_1/r_2)(D_1/r_1-D/r)$, we get
$$
2r \geq -\frac{rr_1}{r_2}\left(\frac{D_1}{r_1}-\frac{D}{r}\right)^2.
$$
Hence we have
$$
\left|\left(\frac{D}{r}-\frac{D_1}{r_1},\eta \right) \right|
 \leq \sqrt{-\left(\frac{D_1}{r_1}-\frac{D}{r}\right)^2}
\sqrt{-(\eta^2)} \leq 
\sqrt{2r}\sqrt{-(\eta^2)}<\frac{1}{r_0 r^2}
$$
and
\begin{equation}\label{eq:eta-ss}
\begin{split}
\frac{\chi(E(-\beta-\eta))}{\rk E}-\frac{\chi(E_1(-\beta-\eta))}{\rk E_1}
=& 
\frac{-\langle v(E),e^\beta+\eta+(\eta,\beta)\varrho_X \rangle}{r}
-\frac{-\langle v(E_1),e^\beta+\eta+(\eta,\beta)\varrho_X \rangle}{r_1}
\\
=&\left(\frac{a}{r}-\frac{a_1}{r_1} \right)-
\left(\frac{D}{r}-\frac{D_1}{r_1},\eta \right).
\end{split}
\end{equation}
Assume that $a/r-a_1/r_1 \ne 0$.
Since $|a/r-a_1/r_1| \geq 1/r_0 r^2$,
(1) holds.
We note that
$$
-\left(\frac{D}{r}-\frac{D_1}{r_1},\eta \right)
=\frac{-\langle v(E),\eta+(\eta,\beta)\varrho_X \rangle}{r}
-\frac{-\langle v(E_1),\eta+(\eta,\beta)\varrho_X \rangle}{r_1}.
$$
Then (2) follows from (1) and \eqref{eq:eta-ss}.
\end{proof}

We note that
${\frak R}_\beta$ is a finite set (Lemma~\ref{lem:R_beta}).
\begin{defn}
We set $N:=\min_{v(E) \in {\frak R}_\beta}\{ 1/(2r_0^2 r^5)\}$.
Assume that $-(\eta^2)<N$.
Let ${\frak U}$ be the set of 
$(\beta+\eta)$-twisted stable objects $U$ of ${\frak C}$
with $\deg(U(-\beta))=0$ and $\chi(U(-\beta))>0$.
\end{defn}
$U \in {\frak U}$ are $(\beta+t\eta)$-twisted stable for $0<t \leq 1$ and
$$
\frac{\chi(U(-\beta-\eta))}{\rk U}
\leq \frac{\chi(U'(-\beta-\eta))}{\rk U'}
$$ 
if and only if
$$
\frac{\chi(U(-\beta))}{\rk U}
< \frac{\chi(U'(-\beta))}{\rk U'}
$$
or  
$$
\frac{\chi(U(-\beta))}{\rk U}
= \frac{\chi(U'(-\beta))}{\rk U'}, 
\quad
\frac{-\langle v(U),\eta+(\eta,\beta)\varrho_X \rangle}{\rk U}
\leq \frac{-\langle v(U'),\eta+(\eta,\beta)\varrho_X \rangle}{\rk U'}
$$ 
for $U, U' \in {\frak U}$.

\begin{cor}\label{cor:eta-ss}
Let $E$ be a $\mu$-semi-stable object of ${\frak C}$ with
$v(E)=re^\beta+a \varrho_X+D+(D,\beta)\varrho_X$.
For $E \in {\frak T}$, let
\begin{equation*}
0 \subset F_1 \subset F_2 \subset \cdots \subset F_s=E
\end{equation*}
be the Harder-Narasimhan filtration  
with respect to $(\beta+\eta)$-twisted semi-stability.
Then 
$F_i/F_{i-1}$ are generated by $U \in {\frak U}$ with 
\begin{align*}
\dfrac{\chi(U(-\beta-\eta))}{\rk U}=
\dfrac{\chi(F_i/F_{i-1}(-\beta-\eta))}{\rk F_i/F_{i-1}}.
\end{align*}
\end{cor}

\begin{proof}
Since $E \in {\frak T}$, we have $a>0$.
We take the Harder-Narasimhan filtration
\begin{equation*}
0 \subset F_1 \subset F_2 \subset \cdots \subset F_s=E
\end{equation*}
with respect to $(\beta+\eta)$-twisted semi-stability.
By Lemma~\ref{lem:eta-ss} (1), 
$\chi(F_1(-\beta))/\rk F_1 \geq a/r>0$.
Since $E \in {\frak T}$, we have $E/F_1 \in {\frak T}$,
which implies that $\chi(E/F_1(-\beta))>0$.
Then inductively we see that $\chi(F_i/F_{i-1}(-\beta))>0$
for all $i$. 
Hence the claim holds.
\end{proof}

\begin{prop}\label{prop:eta-F}
\begin{enumerate}
\item[(1)]
${\frak F}_{(\beta+\eta,\omega)}$ is generated by
${\frak F}$ and $U \in {\frak U}$ with 
$Z_{(\beta+\eta,\omega)}(U)>0$.
\item[(2)]
${\frak T}_{(\beta+\eta,\omega)}$ is generated by
${\frak T}^\mu$ and $U \in {\frak U}$ with
$Z_{(\beta+\eta,\omega)}(U)<0$.
\end{enumerate}
\end{prop}

\begin{proof}
(1)
For $E \in {\frak F}_{(\beta+\eta,\omega)}$,
we have an exact sequence 
$$
0 \to E_1 \to E \to E_2 \to 0
$$
in ${\frak C}$
such that $E_1 \in {\frak T} \cap {\frak F}_{(\beta+\eta,\omega)}$
and $E_2 \in {\frak F}$.
Then 
$E_1$ is a $\mu$-semi-stable object with $\deg_G(E_1)=0$.
Since $E_1 \in {\frak T}$,
Corollary~\ref{cor:eta-ss} implies that
$E_1$ is generated by $U \in {\frak U}$ with 
$Z_{(\beta+\eta,\omega)}(U)>0$.
\begin{NB}
$Z_{(\beta+\eta,\omega)}(U[1]) \in {\Bbb R}_{<0}$.
\end{NB}

(2)
For $E \in {\frak T}_{(\beta+\eta,\omega)}$,
we have an exact sequence 
$$
0 \to E_1 \to E \to E_2 \to 0
$$
in ${\frak C}$
such that $E_1 \in {\frak T}^\mu$ and
$E_2 \in {\frak F}^\mu \cap {\frak T}_{(\beta+\eta,\omega)}$.
If there is a quotient $E_2 \to F$ of $E_2$ with $F \in {\frak F}$,
then Lemma~\ref{lem:eta-F} implies that
$F \in {\frak F}_{(\beta+\eta,\omega)}$.
Hence we get $F=0$, which implies that $E_2 \in {\frak T}$.
By Corollary~\ref{cor:eta-ss}, $E_2$ is generated by
$U \in {\frak U}$ with 
$Z_{(\beta+\eta,\omega)}(U)<0$.
\end{proof}

By Lemma~\ref{lem:eta-ss}, we have $\exc_\beta \subset {\frak U}$
and every $U \in {\frak U}$ is $\beta$-twisted semi-stable.
Hence ${\frak U} \subset \cup_W {\frak S}_W$.

\begin{prop}\label{prop:eta-I}
Let $I:=\{t H \mid a \leq t \leq b \}$ 
be a closed interval of ${\Bbb R}_{>0}H$.
Then for a sufficiently small 
$\eta:=\eta_I \in H^{\perp} \cap \NS(X)_{\Bbb Q}$, 
we have the following claims:
\begin{enumerate}
\item[(1)]
If $I$ does not intersect with any wall, then
${\frak A}_{(\beta+\eta,\omega)}={\frak A}_{(\beta,\omega)}$
for all $\omega \in I$.
\item[(2)]
Assume that the interior of $I$ intersects with exactly one wall $W$.
We take $\omega_\pm \in I \cap {\Bbb Q}_{>0}H$ such that
$\omega_\pm $ are separated by $W$ and
$\omega_- <\omega_+$.
Assume that $\omega \in I \cap {\Bbb Q}_{>0}H$ satisfies
$Z_{(\beta+\eta,\omega)}(E) \ne 0$ for all $E$ with $v(E)\in {\frak R}$.
Then 
\begin{enumerate}
\item
${\frak F}_{(\beta+\eta,\omega)}$ is generated by
${\frak F}_{(\beta,\omega_-)}$ and $U \in {\frak U} \cap {\frak S}_W$
with $Z_{(\beta+\eta,\omega)}(U)>0$.
\item
${\frak T}_{(\beta+\eta,\omega)}$ is generated by
${\frak T}_{(\beta,\omega_+)}$ and 
$U \in {\frak U} \cap {\frak S}_W$
with $Z_{(\beta+\eta,\omega)}(U)<0$.
\end{enumerate}
\end{enumerate}
\end{prop}

\begin{proof}
We set 
$\exc_\beta^*
 :=\cap_{\omega \in I} 
   \{E \in \exc_\beta \mid Z_{(\beta,\omega)}(E) \ne 0 \}$.
Since $|Z_{(\beta,\omega)}(E)|>0$ on $I$ for all $E \in \exc_\beta^*$,
$$
N:=\min \{ |Z_{(\beta,\omega)}(E)| \mid E \in \exc_\beta^* \}
$$
is a positive real number.
Then we can take a sufficiently small $\eta$ with
$N>\langle v(E),\eta+(\eta,\beta)\varrho_X \rangle-(\eta^2) \rk E/2$
for all $E \in \exc_\beta$.
For $E \in \exc_\beta^*$,  
we have
\begin{equation}\label{eq:eta-ineq}
\begin{cases}
Z_{(\beta+\eta,\omega)}(E)>0, & \text{if $Z_{(\beta,\omega)}(E)>0$},\\
Z_{(\beta+\eta,\omega)}(E)<0, & \text{if $Z_{(\beta,\omega)}(E)<0$}
\end{cases}
\end{equation} 
by \eqref{eq:eta-Z}.

(1) By the assumption,
$\exc_\beta^*=\exc_\beta$.
Hence the claim follows from \eqref{eq:eta-ineq}.

(2)   
In this case, $\exc_\beta^*=\exc_\beta \setminus {\frak S}_W$.
By \eqref{eq:eta-ineq}, 
${\frak F}_{(\beta,\omega_-)} \subset {\frak F}_{(\beta+\eta,\omega)}$
and
${\frak T}_{(\beta,\omega_+)} \subset {\frak T}_{(\beta+\eta,\omega)}$.
\begin{NB}
For $E \in {\frak S}_W$, 
$Z_{(\beta,\omega_-)}(E)<0<Z_{(\beta,\omega_+)}(E)$.
Hence ${\frak F}_{(\beta,\omega_-)}$ does not contain
non-trivial object of ${\frak S}_W$. 
\end{NB}
\end{proof}

\begin{cor}\label{cor:eta-I}
In the notation of Proposition~\ref{prop:eta-I} (2),
we choose $\eta$ such that
$Z_{(\beta+\eta,aH)}(E)<0<Z_{(\beta+\eta,bH)}(E)$ for 
$E \in \exc_\beta \cap {\frak S}_W$.
Then ${\frak A}_{(\beta+\eta,aH)}={\frak A}_{(\beta,\omega_-)}$
and ${\frak A}_{(\beta+\eta,bH)}={\frak A}_{(\beta,\omega_+)}$. 
\end{cor}

\begin{proof}
\begin{NB}
For $E \in {\frak S}_W$,
we have $Z_{(\beta,aH)}(E)<0<Z_{(\beta,bH)}(E)$.
Hence for sufficiently small $\eta$ depending on $a,b$,
we have $Z_{(\beta+\eta,aH)}(E)<0$ and $Z_{(\beta+\eta,bH)}(E)>0$.
\end{NB}
We first note that
$\exc_\beta \cap {\frak S}_W
\subset {\frak U} \cap {\frak S}_W \subset {\frak S}_W$
and $\exc_\beta \cap {\frak S}_W$ generate the category ${\frak S}_W$. 
By Proposition~\ref{prop:eta-I} (2),
${\frak T}_{(\beta+\eta,aH)}={\frak T}_{(\beta,\omega_-)}$
if and only if $Z_{(\beta+\eta,aH)}(U)<0$ for 
$U \in {\frak U} \cap {\frak S}_W$.
Hence ${\frak T}_{(\beta+\eta,aH)}={\frak T}_{(\beta,\omega_-)}$,
which implies that
${\frak A}_{(\beta+\eta,aH)}={\frak A}_{(\beta,\omega_-)}$.
The proof of ${\frak A}_{(\beta+\eta,bH)}={\frak A}_{(\beta,\omega_+)}$
is similar.
\end{proof}

\begin{cor}\label{cor:cat=loc-const}
${\frak A}_{(bH+\eta,\omega)}$ depends only on the chamber where
$(\eta,\omega)$ belongs.
\end{cor}

\begin{proof}
It is sufficient to prove that ${\frak A}_{(bH+\eta,\omega)}$
is locally constant on 
${\frak H}_{\Bbb R} \setminus \cup_{u \in {\frak R}} W_u$.
By Proposition~\ref{prop:eta-I} (1), the claim holds.
\end{proof}

\begin{prop}\label{prop:equiv} 
Assume that ${\cal C}_1$ and ${\cal C}_2$ are chambers 
separated by exactly one wall $W_u$ with $\rk u>0$.
\begin{enumerate}
\item[(1)]
$R_u$ (cf. \eqref{eq:reflection})
induces a bijection ${\cal C}_1 \xrightarrow{\sim} {\cal C}_2$.
\item[(2)]
For $(\eta_i,\omega_i) \in {\cal C}_i$, $i=1,2$, 
we set $\beta_i:=bH+\eta_i$. Then
the stability conditions
$({\frak A}_{(\beta_1,\omega_1)},Z_{(\beta_1,\omega_1)})$
and
$({\frak A}_{(\beta_2',\omega_2')},Z_{(\beta_2',\omega_2')})$
are equivalent, where 
$(\beta_2',\omega_2') \in {\cal C}_2$ and 
$R_u((\beta_1,\omega_1))=(\beta_2',\omega_2')$.
\end{enumerate}
\end{prop}

\begin{proof}
(1) is a consequence of Lemma~\ref{lem:Phi(space)}.

(2) By Corollary~\ref{cor:cat=loc-const},
we may assume that $\eta_2=\eta_1$,
$(\omega_1^2)<(\omega_2^2)$. 
We set $\beta:=bH+\eta_1$.
Let $U$ be a $\beta$-twisted semi-stable object of ${\frak C}$
with $v(U)=u$.
Since $W_u$ is the unique wall separating 
$(\eta_1,\omega_1)$ and $(\eta_2,\omega_2)$,
$U$ is $\beta$-twisted stable. 
Then $U$ is an irreducible object of ${\frak A}_{(\beta,\omega_1)}$
and $U[1]$ is an irreducible object of ${\frak A}_{(\beta,\omega_2)}$.
We shall prove that
$\Phi_U^{-1}$ induces an equivalence
${\frak A}_{(\beta,\omega_1)} \to {\frak A}_{(\beta,\omega_2)}$.

We first note that ${\cal O}_{\Delta}^{\vee}={\cal O}_\Delta[-2]$.
Let $E$ be an object of ${\frak A}_{(\beta,\omega_1)}$.
Since $\Phi_U^{-1}(U)=U[1]$, we have 
$\Hom(\Phi_U^{-1}(E)[1],U)
=\Hom(E[1],U[-1])=0$.
We have an exact sequence in ${\frak C}$:
\begin{equation*}
\begin{array}{cccccc}
                  & 
&                 & 0 
& \longrightarrow & \Hom(U,E) \otimes U 
\\[0.7em]
  \longrightarrow & H^{-1}(\Phi_U^{-1}(E)) 
& \longrightarrow & H^{-1}(E)
& \longrightarrow & \Ext^1(U,E) \otimes U 
\\[0.4em]
  \longrightarrow & H^0(\Phi_U^{-1}(E))    
& \longrightarrow & H^0(E) 
& \xrightarrow{\ \psi\ }& \Ext^2(U,E) \otimes U 
\\[0.7em]
  \longrightarrow & H^1(\Phi_U^{-1}(E))
& \longrightarrow & H^1(E)=0 
&                 &
\end{array}
\end{equation*}
By Lemma~\ref{lem:irreducible} and 
$\Hom(H^1(\Phi_U^{-1}(E)),U)=\Hom(\Phi_U^{-1}(E)[1],U)=0$,
$\psi$ is a surjective morphism in ${\frak A}_{(\beta,\omega_1)}$,
which also implies the surjectivity of $\psi$ in ${\frak C}$.
Therefore
$H^1(\Phi_U^{-1}(E))=0$.
\begin{NB}
By the definition of ${\frak T}_{(\beta,\omega_1)}$,
$\im \psi \in {\frak T}_{(\beta,\omega_1)}$ implies that
$\im \psi$ is a direct sum of $U$.
Since $\Hom(\Phi_U^{-1}(E),U)=0$, we get
$H^1(\Phi_U^{-1}(E))=0$.
\end{NB}
Then we see that 
$H^{-1}(\Phi_U^{-1}(E)) 
\in {\frak F}_{(\beta,\omega_2)}$ and
$H^0(\Phi_U^{-1}(E)) \in
{\frak T}_{(\beta,\omega_1)}$.
Since
$\Hom(\Phi_U^{-1}(E),U)
=\Hom(E,U[-1])=0$, we also have
$\Hom(H^0(\Phi_U^{-1}(E)),U)=0$.
Hence $H^0(\Phi_U^{-1}(E)) \in
{\frak T}_{(\beta,\omega_2)}$ and 
$\Phi_U^{-1}(E)$ is an object of
${\frak A}_{(\beta,\omega_2)}$.
Thus $\Phi_U^{-1}({\frak A}_{(\beta,\omega_1)})
\subset {\frak A}_{(\beta,\omega_2)}$, which implies the claim.
\end{proof}

\begin{NB}
Assume that $E$ is an irreducible object of ${\frak A}_{(\beta,\omega_1)}$
and $E \ne U$. Then
$\Hom(U,E)=\Hom(E,U)=0$.
In particular $\Hom(H^0(E),U)=0$, which implies that
$H^0(E) \in {\frak T}_{(\beta,\omega_2)}$.
If $H^{-1}(E)=0$, then
we have an exact sequence in ${\frak A}_{(\beta,\omega_2)}$:
$$
0 \to
\Phi_{X_1 \to X_2}^{{\bf E}^{\vee}[2]}(E)
\to H^0(E) \to \Ext^1(U,E) \otimes U[1] \to 0 .
$$ 
If $H^0(E)=0$, then
$$
0 \to
\Phi_{X_1 \to X_2}^{{\bf E}^{\vee}[2]}(E)
\to H^{-1}(E)[1] \to \Ext^1(U,E) \otimes U[1] \to 0 .
$$ 
\begin{NB2}
For $I_x[1], {\cal O}_X \in {\frak A}$ with $\beta=0$,
we have an exact sequence in ${\frak A}^\mu$:
$$
0 \to {\frak k}_x \to I_x[1] \to 
\Ext^1({\cal O}_X,I_x[1]) \otimes {\cal O}_X[1]
\to 0.
$$
\end{NB2} 
\end{NB}

\begin{cor}\label{cor:equiv} 
Assume that neither $(\beta_1,\omega_1)$ nor 
$(\beta_2,\omega_2)$ belongs to any wall. 
Then ${\frak A}_{(\beta_1,\omega_1)}$ is equivalent
to ${\frak A}_{(\beta_2,\omega_2)}$.
\end{cor}


\subsection{Moduli of stable objects for isotropic Mukai vectors}

Let $v \in {\Bbb Q}_{>0}e^{bH+\eta} \cap A^*_{\alg}(X)$ 
be a primitive isotropic Mukai vector. 
Assume that $\langle v,u \rangle \ne 0$
for all $u \in {\frak R}$.
Then $M_H^{bH+\eta}(v)$ is a projective $K3$ surface
by \cite{Mukai}.
Since $H$ is not a general polarization,
$M_H^{bH+\eta}(v)$ may not contain slope stable objects.  
In this subsection, we shall relate
$M_H^{bH+\eta}(v)$ to the moduli of slope stable objects
as an application of the chamber structure.

Let 
\[I_t := \eta+\sqrt{-1}t \omega \quad (0 \leq t \leq t_0)\]
be a segment of $V_H$ (see \eqref{eq:V_H}) such that  
$I_t$, $0<t<t_0$ belongs to a chamber and
$I_{t_0}$ belongs to a wall.
Let $\{W_{u_i} \mid u_i \in {\frak R},\ 1 \leq i \leq n \}$
be the set of all walls containing $I_{t_0}$.  
Then
\begin{align*}
\langle e^{bH+I_t},u \rangle<0
\text{ for all }
0 \leq t< t_0,\ u \in {\frak R},
\quad
\text{ and }
\quad
\langle e^{bH+I_{t_0}},u_i \rangle=0 
\text{ for } 1 \leq i \leq n.
\end{align*}

We have an auto-equivalence $\Phi:{\bf D}(X) \to {\bf D}(X)$
inducing an equivalence
$\Phi:{\frak A}_{(bH+\eta,t_- \omega)} \to
{\frak A}_{(bH+\eta,t_+ \omega)}$, where $t_-<t_0<t_+$
and $t_+-t_- \ll 1$.
For a sufficiently small $\xi \in H^{\perp}$,
let $U_i$ be $(bH+\eta+\xi)$-twisted stable objects of $v(U_i)=u_i$
($1 \leq i \leq n$).
Then $\Phi$ is a composite of reflections
$\Phi_{U_i}^{-1}$, $1 \leq i \leq n$. 
Hence $\Phi$ induces an isomorphism $\Phi_{\frak H}:
\overline{\frak H}_{\Bbb R} \to \overline{\frak H}_{\Bbb R}$.

$\Phi_{\frak H}(I_t)$  ($0<t<t_0$) is contained in a chamber.
We set $\eta':=\Phi_{\frak H}(\eta)$. 
We set $I'_t:=\eta'+\sqrt{-1}t \omega$ and assume that 
$I'_t$, $0< t<t_1$ 
is contained in a chamber.
We take $t,t' \in {\Bbb R}$ such that 
$0<t<t_0$ and $0<t'<t_1$.
Then there is no wall between
$\Phi_{\frak H}(\eta+\sqrt{-1}t \omega)$ and 
$\eta'+\sqrt{-1}t' \omega$.
The proof is the following:
if the two points are separated by a wall $W_u$,
then 
$\langle e^{bH+\Phi_{\frak H}(\eta+\sqrt{-1}t \omega)},u \rangle$ and
$\langle e^{bH+\eta'+\sqrt{-1}t' \omega},u \rangle$ 
have different signatures.
Since the two points are connected by 
the curve 
\begin{equation*}
J_s:=
\begin{cases}
\Phi_{\frak H}(I_{t-s}),& 0 \leq s \leq t \\
I'_{s-t},& t \leq s \leq t+t'
\end{cases}
\end{equation*} 
and $\langle e^{bH+\eta'},u \rangle \ne 0$,
$\langle e^{bH+\Phi_{\frak H}(\eta+\sqrt{-1}s \omega)},u \rangle=0$,
$0<s<t$
or $\langle e^{bH+\eta'+\sqrt{-1}s \omega},u \rangle=0$,
$0<s<t'$, which is a contradiction.
Therefore the claim holds. Then we have
${\frak A}_{(bH+\eta,t_+ \omega)}=
{\frak A}_{(bH+\eta_t,\omega_t)}=
{\frak A}_{(bH+\eta',t' \omega)}$, where 
$\eta_t+\sqrt{-1}\omega_t=\Phi_{{\frak H}}(\eta+\sqrt{-1} t\omega)$,
$0<t<t_0$ and $0<t'<t_1$.

Then $\Phi$ induces an isomorphism
$M_H^{bH+\eta}(v) \to M_H^{bH+\eta'}(\Phi(v))$
by Lemma \ref{lem:stable-irreducible}.
Continuing this procedure, we get an isomorphism
$M_H^{bH+\eta}(v) \to M_H^{bH+\eta'}(\Phi(v))$
such that $M_H^{bH+\eta'}(\Phi(v))$ parametrizes irreducible objects
of ${\frak A}^\mu$.
In particular, $M_H^{bH+\eta'}(\Phi(v))$ consists of slope
stable objects or 0-dimensional objects.

\begin{ex}
Let $X$ be the elliptic $K3$ surface in Example~\ref{ex:exceptional2}.
Using the same notation,
we shall describe the universal family for $M_H^{D/2}(4 e^{D/2})$.
Let $I_t:=D/2+tH$ ($0 \leq t \leq 1$) be a segment in $V_H$. 
Then $I_t$ meets the walls at $t=(2\sqrt{3})^{-1}$. 
We take a small perturbation $\widetilde{I}_t=x(t)D+tH$ of $I_t$ so that 
$x(t)<1/2$ in a neighborhood of $t=(2\sqrt{3})^{-1}$.
Then 
$\Phi_{{\cal O}_X(D)}\Phi_{E_1}
\Phi_{{\cal O}_X}({\cal O}_\Delta)$ is the universal family 
on $X \times M_H^{D/2}(4 e^{\frac{D}{2}})$.
We also have the universal family
$\Phi_{E_1}\Phi_{{\cal O}_X}({\cal O}_\Delta)$  
on $X \times M_H^{D/3}(3 e^{\frac{D}{3}})$.
\end{ex}


\section{Relation of Gieseker's stability and Bridgeland's stability}
\label{sect:Gieseker}


\subsection{Some numerical conditions}
\label{subsect:Gieseker:numcond}

We shall discuss numerical conditions which relate
Gieseker's stability with Bridgeland's stability.

For a fixed Mukai vector $v\in A^*_{\alg}(X)$, 
let us recall the expression of $v$ 
defined in \eqref{eq:Mukai-vector}:
\begin{align}\label{eq:v}
v=r e^\beta+a \varrho_X+(d H+D+(d H+D,\beta)\varrho_X),
\quad
D \in H^\perp \cap \NS(X)_{\Bbb Q}.
\end{align}
In the same way, for $E_1 \in {\bf D}(X)$, we set
\begin{align}
\label{eq:v(E_1)}
v(E_1)=r_1 e^\beta+a_1 \varrho_X+(d_1 H+D_1+(d_1 H+D_1,\beta)\varrho_X),
\quad
D_1 \in H^\perp \cap \NS(X)_{\Bbb Q}.
\end{align}
We will use these notations freely in this section.
\\

We consider the following conditions for $v$:
\begin{itemize}
\item[$(\star 1)$]
(1) $r \geq 0$, $d>0$ and
(2) $dr_1-d_1 r>0$ implies 
$(dr_1-d_1 r)(\omega^2)/2-(d a_1-d_1 a) > 0$ for
all $E_1$ with $0<d_1< d$ and
$\langle v(E_1)^2 \rangle \geq -2$.
\begin{NB}
Thus $\phi(E_1) <\phi(E)$.
\end{NB}
\begin{NB}
$r_1$ may be bigger than $r$.
The assumption implies that
$dr_1>d_1r \geq 0$. Hence we have $r_1>0$.
\end{NB}

\item[$(\star 2)$]
(1) $r \geq 0$, $d<0$ and 
(2) $dr_1-d_1 r < 0$ implies 
    $(dr_1-d_1 r)(\omega^2)/2-(d a_1-d_1 a) < 0$ for
    all $E_1$ with $d < d_1 < 0$ and
    $\langle v(E_1)^2 \rangle \geq -2$.
\begin{NB}
Thus $\phi(E_1) > \phi(E)$.
\end{NB}

\item[$(\star 3)$]
(1) $r \geq 0$, $d<0$ and 
(2) $dr_1-d_1 r \leq 0$ implies 
    $(dr_1-d_1 r)(\omega^2)/2-(d a_1-d_1 a) \leq 0$ 
    and the inequality is strict if $dr_1-d_1 r < 0$ 
    for all $E_1$ with $d \leq d_1 \leq 0$ and
    $\langle v(E_1)^2 \rangle \geq -2$.
\begin{NB}
Thus $\phi(E_1) > \phi(E)$.
\end{NB}

\begin{NB}
\item[$(\star )$]
(1) $d<0$ and (2) 
$dr_1-d_1 r \leq 0$ implies 
$(dr_1-d_1 r)(\omega^2)/2-(d a_1-d_1 a) \leq 0$ for
all $E_1$ with $d \leq d_1 \leq 0$ and
$\langle v(E_1)^2 \rangle \geq -2$.
Thus $\phi(E_1) \geq \phi(E)$.
\end{NB}
\end{itemize}

\begin{rem}
\begin{enumerate}
\item[(1)]
Since
$(d r_1-d_1 r)(H^2)=(c_1(E \otimes E_1^{\vee}),H)$,
the condition $(\star 1)$ says that
$\mu(E)>\mu(E_1)$ implies $\phi(E_1)>\phi(E)$,
and the condition $(\star 2)$ says that
$\mu(E)<\mu(E_1)$ implies $\phi(E_1)<\phi(E)$.

\item[(2)]
$(\star 1)$ for $v$ is equivalent to $(\star 2)$ for $v^{\vee}$.

\item[(3)]
Assume that $(\star 3)$ holds for $v$.
Let $E$ be an object of ${\frak C}$ with $v(E)=v$.
Then
$E$ is $\mu$-semi-stable if and only if $E$ is 
$\beta$-twisted semi-stable.
Moreover $E$ is a local projective object of ${\frak C}$.

Assume that $E$ is $\mu$-semi-stable and take a Jordan-H\"{o}lder
filtration with respect to the $\mu$-stability
\begin{equation*}
0 \subset F_1 \subset F_2 \subset \cdots \subset F_s=E.
\end{equation*}
We set 
$$
v(F_i/F_{i-1})=
r_i e^\beta+a_i \varrho_X+(d_i H+D_i+(d_i H+D_i,\beta)\varrho_X).
$$
Then $d_i/r_i=d/r$.
By $(\star 3)$, we get $a_i/r_i \leq a/r$ for all $i$.
Hence $a_i/r_i= a/r$ for all $i$.
Thus $E$ is $\beta$-twisted semi-stable.
We set $E_i:=F_i/F_{i-1}$ and
assume that $\Ext^1(E_i,A) \ne 0$ for an irreducible object $A$ of 
${\frak C}$.
Since $\Ext^1(A,E_i) \cong \Ext^1(E_i,A)^{\vee}$,
we have a non-trivial extension
\begin{equation*}
0 \to E_i \to F \to A \to 0.
\end{equation*} 
Then 
$F$ is a $\mu$-stable object. 
Applying $(\star 3)$ to $v(F)=v(E_i)+b \varrho_X+
(D'+(D',\beta)\varrho_X)$, $b>0$,
we get $(a_i+b)/r_i \leq a/r$, which is a contradiction.
Therefore $\Ext^1(E_i,A)=0$ for all $i$.
Then $\Ext^1(E,A)=0$ and $E$ is a local projective object.
\end{enumerate}
\end{rem}

\begin{lem}\label{lem:N}
Let $N(v):=N$ be the number in \cite[Proposition 2.8]{Stability}.
Assume that $d>N(v)$.
Then $(\star 1)$ holds for $v$ and $(\star 3)$ holds for
$w=r_0 a e^\beta+(r/r_0)\varrho_X-
(dH+D+(dH+D,\beta)\varrho_X)$.
\end{lem}

\begin{proof}
If $dr_1-d_1 r>0$, then
\cite[Lemma 2.9 (2)]{Stability} implies that
$(d a_1-d_1 a)<0$. Hence
$(dr_1-d_1 r)(\omega^2)/2-(d a_1-d_1 a)> 0$.
Thus $(\star 1)$ holds.

Note that $(-d)a_1-(-d_1)a=d_1 a-d a_1$ and
$(-d)r_1-(-d_1) r=d_1 r-dr_1$.
Hence $(-d)a_1-(-d_1)a \leq 0$ implies 
 $((-d)a_1-(-d_1)a)(\omega^2)/2-((-d)r_1-(-d_1) r) \leq 0$.
Moreover the inequality is strict, if  $(-d)a_1-(-d_1)a< 0$.
Therefore $(\star 3)$ holds.
\end{proof}

\begin{defn}\label{defn:minimal}
We set 
\begin{equation*}
\begin{split}
d_{\min}
:=& \frac{1}{(H^2)}\min \{ \deg(E(-\beta))>0 \mid E \in K(X) \}
    \in \frac{1}{\denom[(\beta,H)] (H^2)}{\Bbb Z},\\
\end{split}
\end{equation*}
where $\denom[x]$ is the denominator of $x \in {\Bbb Q}$.
\begin{NB}
$\deg_G(E)=(r_0 c_1(E)-\rk E c_1(G),H)=
r_0(c_1(E)-\rk E \beta,H)=r_0 d(H^2)$.
\end{NB}
Then $d \in {\Bbb Z}d_{\min}$ for any $v(E)$ with \eqref{eq:Mukai-vector}.
\end{defn}

\begin{ex}
In Example~\ref{ex:exceptional},
$d_{\min}=\sqrt{r_0}(H^2)/r_0 (H^2)=1/\sqrt{r_0}$.
\end{ex}

\begin{lem}\label{lem:d_min}
If $d=d_{\min}$, then
$(\star 1)$ holds for $v=v(E)$, $E \in {\frak A}$ and $(\star 2)$ holds
for $v^\vee$.  
\end{lem}

\begin{proof}
For any $E_1 \in {\bf D}(X)$ with \eqref{eq:v(E_1)},
$d_1 \leq 0$ or $d_1 \geq d_{\min}$.
Hence the claim follows. 
\end{proof}

\begin{lem}\label{lem:large}
For a fixed Mukai vector $v$,
if $(\omega^2) \gg 0$, then
$(\star 1)$ holds.
\end{lem}

\begin{proof}
Assume that $(\star 1)$ does not hold.
Then there is an object $E_1$ with \eqref{eq:v(E_1)} such that
\begin{align*}
dr_1-d_1 r>0,\quad (dr_1-d_1 r) \dfrac{(\omega^2)}{2} 
\leq da_1-d_1 a,
\end{align*}
$0<d_1<d$ and  $\langle v(E_1)^2 \rangle \geq -2\varepsilon$.
Since $r \geq 0$ and $d,d_1>0$, $r_1>rd_1/d \geq 0$.
We set 
\begin{align*}
\delta:=\frac{1}{(H^2)}\min\{(D,H)>0 \mid D \in \Pic(X) \}.
\end{align*}
Then $d r_1-d_1 r \geq \delta$.
\begin{NB}
Since $-2\varepsilon  \leq 
\langle v(E_1)^2 \rangle=-2r_1 a_1+d_1^2(H^2)+(D_1^2)
\leq -2r_1 a_1+d_1^2(H^2)$,
we get
$$
a_1 \leq \frac{\varepsilon}{r_1}+\frac{(H^2)}{2}\frac{d_1}{r_1}d_1 
<\varepsilon+\frac{(H^2)}{2}\frac{d}{r}d_1 \leq
\varepsilon+\frac{(H^2)}{2}\frac{d}{r}d.
$$
\end{NB}
Assume that $r>0$.
Then we get
$$
a=\frac{d^2(H^2)-(\langle v^2 \rangle-(D^2))}{2r},\quad
a_1 \leq \frac{d_1^2(H^2)+2\varepsilon}{2r_1}.
$$
Hence
\begin{align*}
&(d r_1-d_1 r)\dfrac{(\omega^2)}{2}-(d a_1-d_1 a)
\\
&\geq 
 (d r_1-d_1 r)\dfrac{(\omega^2)}{2}
 -d   \dfrac{d_1^2(H^2)+2\varepsilon}{2r_1}
 +d_1 \dfrac{d^2(H^2)-(\langle v^2 \rangle-(D^2))}{2r}
\\
&= (d r_1-d_1 r)\dfrac{(\omega^2)}{2}
  +(d r_1-d_1 r)\dfrac{d d_1}{r r_1} \dfrac{(H^2)}{2}
  -\dfrac{1}{r r_1}
   \Bigl(d r \varepsilon+d_1 r_1 \dfrac{\langle v^2 \rangle-(D^2)}{2}\Bigr).
\end{align*}
Thus we have
$$
\delta \frac{(\omega^2)}{2} 
\leq \frac{d\varepsilon}{r_1}+
\frac{d_1}{2r}(\langle v^2 \rangle-(D^2)) \leq 
\max\left\{d\varepsilon, d\varepsilon+
\frac{d-d_{\min}}{2r}(\langle v^2 \rangle-(D^2))\right\}.
$$

If $r=0$, then
\begin{equation*}
\begin{split}
(dr_1-d_1 r)\dfrac{(\omega^2)}{2}-(d a_1-d_1 a)
&\geq 
 d r_1 \dfrac{(\omega^2)}{2}-d\frac{d_1^2(H^2)+2\varepsilon}{2r_1}
+d_1 a
\\
&\geq 
  d r_1 \dfrac{(\omega^2)}{2}
-(d-d_{\min})d^2 \frac{(H^2)}{2 }-\frac{d}{r_1}\varepsilon-(d-d_{\min})|a|.
\end{split}
\end{equation*}
Hence
\begin{equation*}
\delta \frac{(\omega^2)}{2} 
\leq d\varepsilon+\frac{(d-d_{\min})}{2}
(\langle v^2 \rangle-(D^2)) +d|a|.
\end{equation*}
These mean that $(\omega^2)$ is bounded above. 
\end{proof}

\begin{NB}
We take $\eta \in H^{\perp} \otimes {\Bbb Q}$
and write
$v(E)=r^{\beta+\eta}+a' \varrho_X+
(dH+D'+(dH+D',\beta+\eta)\varrho_X)$.
Then 
\begin{equation*}
\begin{split}
\langle v(E)^2 \rangle-({D'}^2)=& 
\langle v(E)^2 \rangle-({D}^2)+2r(D,\eta)-r^2(\eta^2)\\
 \leq & \langle v(E)^2 \rangle-({D}^2)+2r\sqrt{-(\eta^2)}\sqrt{-(D^2)}
-r^2(\eta^2)\\
=&
\langle v(E)^2 \rangle+
\left(\sqrt{-(D^2)}+r\sqrt{-(\eta^2)} \right)^2.
\end{split}
\end{equation*}
Thus $(\star 1)$ holds for $(\beta+\eta,\omega)$,
if
$$
\delta \frac{(\omega^2)}{2} 
> 
\max \left\{d\varepsilon,d\varepsilon+
\frac{d}{2r} \left(\langle v^2 \rangle+
\left(\sqrt{-(D^2)}+r\sqrt{-(\eta^2)} \right)^2 \right) \right\}.
$$
\end{NB}

\begin{NB}
\begin{lem}\label{lem:a>0}
Assume that ${\frak A}={\frak A}^\mu$.
Let $E$ be an object of ${\frak A}$
such that
$$
v(E)=re^\beta+a \varrho_X+(dH+D+(dH+D,\beta)\varrho_X).
$$
If $d=0$, then $a \geq 0$.
\end{lem}

\begin{proof}
We set
\begin{equation*}
\begin{split}
v(H^{-1}(E))=& r''e^\beta+a'' \varrho_X+
(d'' H+D''+(d'' H+D'',\beta)\varrho_X),\\
v(H^0(E))=& r'e^\beta+a' \varrho_X+(d' H+D'+(d' H+D',\beta)\varrho_X).
\end{split}
\end{equation*}
$0=d=d'-d''$ implies that
$d'=d''=0$.
Since $H^{-1}(E) \in {\frak F}$ and $H^0(E) \in {\frak T}$,
$a'' \leq 0$ and $a' \geq 0$.
Hence $a=a'-a'' \geq 0$.
\end{proof}
\end{NB}

\begin{lem}\label{lem:arg}
For $E,E_1\in{\bf D}(X)$ with $v(E)=v$ and $v(E_1)$ 
expressed as \eqref{eq:v} and \eqref{eq:v(E_1)},
we have
\begin{equation*}
\Sigma_{(\beta,\omega)}(E,E_1)=
(H,\omega)\left((rd_1-r_1 d)\dfrac{(\omega^2)}{2}-(ad_1-a_1 d)\right).
\end{equation*}
In particular,
$\phi(E_1) \geq \phi(E)$ if and only if
\begin{equation*}
(r d_1-r_1 d)\dfrac{(\omega^2)}{2}-(a d_1-a_1 d)
\geq 0.
\end{equation*}
\end{lem}

\begin{proof}
The claim follows from the equalities
\begin{equation}\label{eq:arg}
\begin{split}
\Sigma_{(\beta,\omega)}(E,E_1)
=&
\det
\begin{pmatrix}
-a+r\frac{(\omega^2)}{2} & -a_1+r_1\frac{(\omega^2)}{2}\\
d(H,\omega) & d_1(H,\omega)
\end{pmatrix}\\
=& (H,\omega)\Bigl((r d_1-r_1 d)\dfrac{(\omega^2)}{2}-(ad_1-a_1 d)\Bigr).
\end{split}
\end{equation}

\end{proof}


\subsection{Gieseker's stability and Bridgeland's stability}
\label{subsect:Gieseker:equiv}

\begin{prop}\label{prop:star-1}
Assume that $(\star 1)$ holds.
\begin{NB}
In order to exclude the case $d_1=0$, we use 
${\frak A}_{(\beta,\omega)}={\frak A}^\mu$, that is, 
$(\omega^2)$ is relatively large.
The case where $d_1=d$ can be treated for any ${\frak A}_{(\beta,\omega)}$. 
\end{NB}
For an object $E$ of ${\bf D}(X)$ with $v(E)=v$,
$E$ is a $\beta$-twisted semi-stable object of ${\frak C}$ 
if and only if 
$E$ is a semi-stable object of
${\frak A}^\mu$.
\end{prop}

\begin{proof}
(1) 
We shall use the expression \eqref{eq:v} for $v$.
Assume that $E \in {\frak C}$ is $\beta$-twisted semi-stable.
We first assume that $r>0$.
Let $\varphi:E_1 \to E$ be a stable subobject of $E$ and
take the expression \eqref{eq:v(E_1)}.
\begin{NB}
Since $E$ is torsion free, $r_1>0$.
\end{NB}
Then $E/E_1$ belongs to ${\frak A}^\mu$.
Since $H^{-1}(E/E_1)$ and $E$ are torsion free objects of
${\frak C}$, $r_1>0$.
Since ${\frak A}_{(\beta,\omega)}={\frak A}^\mu$, $r_1>0$ implies
$d_1>0$.
\begin{NB}
This is the place where we used the assumption
${\frak A}_{(\beta,\omega)}={\frak A}^\mu$.
\end{NB}
We set 
\begin{align}\label{eq:v'}
 v(H^{-1}(E/E_1))=
 r' e^\beta +a' \varrho_X + (d' H+D'+(d' H+D',\beta)\varrho_X),
\quad
D' \in H^{\perp} \cap \NS(X)_{\Bbb Q}.
\end{align}
Then $d' \leq 0$.
Let $E'$ be the image of $\varphi$ in ${\frak C}$.
Then $d_1/r_1 \leq (d_1-d')/(r_1-r') \leq d/r$, so that 
\begin{equation*}
d_1 \leq d_1-d' \leq \dfrac{d(r_1-r')}{r} \leq d.
\end{equation*}
If $d_1=d$, then
$\deg_G(E/E_1)=r_0(H^2)(d-d_1)=0$.
Hence $\phi(E/E_1)=1>\phi(E)$, which implies that
$\phi(E_1)<\phi(E)$.
\begin{NB}
We do not use the condition ${\frak A}_{(\beta,\omega)}={\frak A}^\mu$.
\end{NB}
So we assume $d_1<d$.
\begin{NB}
Old argument:
If ${\frak A}={\frak A}^\mu$ and $d_1=d$,
then Lemma \ref{lem:a>0} implies that
$a-a_1 \geq 0$.
Hence $(dr_1-d_1 r)(\omega^2)/2+(d_1 a-d a_1) \geq  
(dr_1-d_1 r)(\omega^2)/2$.
By Lemma \ref{lem:arg}, $\phi(E_1)<\phi(E)$
if $dr_1-d_1 r>0$.
So we assume $d_1<d$ if ${\frak A}={\frak A}^\mu$.
\end{NB}
\begin{NB} 
If $d_1=d$, then
$d'=0$ and $r_1-r'=r$.

(+)
If $d_1=0$, then $E_1=H^0(E_1) \in {\frak T}^\mu$ implies that
$E_1$ is a 0-dimensional object. 

By the semi-stability of $E$,
$E_1=0$. 
\end{NB}
If $d_1/r_1=d/r$, then $(r',d')=(0,0)$, which implies that $H^{-1}(E/E_1)=0$.
Assume that $d_1/r_1<d/r$. 
Then $(\star 1)$ implies that $(dr_1-d_1 r)(\omega^2)/2-(da_1-d_1 a)>0$.
If $d_1/r_1=d/r$, then $\beta$-twisted semi-stability of $E$ implies that
$a_1/r_1 \leq a/r$, which also means that
$(dr_1-d_1 r)(\omega^2)/2-(da_1-d_1 a) \geq 0$.
By Lemma~\ref{lem:arg}, $\phi(E_1) \leq \phi(E)$.
Thus $E$ is a semi-stable object of ${\frak A}^\mu$.
\begin{NB}
Old version:
Hence if $\phi(E_1) \geq \phi(E)$, then
$(da_1-d_1 a) \geq (dr_1-d_1 r)(\omega^2)/2$.
By \eqref{eq:mu}, $da_1-d_1 a \geq 0$.
If $d > N$, then \cite[Lemma 2.9 (1)]{Stability}
implies that $d/r \leq d_1/r_1$.
Since $E$ is $\beta$-twisted semi-stable,
$\varphi$ is injective in ${\frak C}$ and
$a_1/r_1 \leq a/r$.
Then we get $\phi(E_1)=\phi(E)$.
Therefore $E$ is semi-stable in the sense of Bridgeland.
\end{NB}

We next assume that $r=0$.
Let $\varphi:E_1 \to E$ be a stable subobject of $E$ and
take the expression \eqref{eq:v(E_1)}.
Then $E/E_1$ belongs to ${\frak A}^\mu$.
Again we take the expression \eqref{eq:v'} of $v(H^{-1}(E/E_1))$.
Then $d' \leq 0$.
Let $E'$ be the image of $\varphi$ in ${\frak C}$.
Then 
\begin{align*}
  d \geq d_1-d' \geq d_1 \geq 0.
\end{align*}
If $d=d_1$, then we see that $\phi(E/E_1)=1>\phi(E)$.
Hence $\phi(E_1)<\phi(E)$.
So let us assume $d>d_1$.
If $d_1=0$, then $E_1$ is a 0-dimensional object.
\begin{NB}
See (+).
\end{NB}
Since $E$ is purely 1-dimensional and $H^{-1}(E/E_1)$
is torsion free,
$E'=0$ and $H^{-1}(E/E_1)=0$. Thus $E_1=0$. 
Assume that $d_1>0$.
If $r_1>0$, then
$(\star 1)$ implies that $(dr_1-d_1 r)(\omega^2)/2+d_1 a-d a_1>0$.
Therefore $\phi(E_1)<\phi(E)$.
If $r_1=0$, then
$H^{-1}(E/E_1)=0$ and the $\beta$-twisted semi-stability of $E$ implies that
$a_1/d_1 \leq a/d$.
Therefore $\phi(E_1) \leq \phi(E)$. 

(2)
Conversely assume that $E \in {\frak A}^\mu$ is
semi-stable.
We first assume that $r>0$.
We set 
\begin{align}
\label{eq:v(H^0(E))}
v(H^0(E)) &= 
 r' e^\beta+a' \varrho_X+ (d' H+D'+(d' H+D',\beta)\varrho_X),
\quad
D' \in H^{\perp} \cap \NS(X)_{\Bbb Q}
\\
\label{eq:v(H^-1(E))}
v(H^{-1}(E)) &= 
 r'' e^\beta+a'' \varrho_X+ (d'' H+D''+(d'' H+D'',\beta)\varrho_X),
\quad
D'' \in H^{\perp} \cap \NS(X)_{\Bbb Q}.
\end{align}
Then $d' \geq 0$ and $d'' \leq 0$.
Let $H^0(E) \to E_2$ be a quotient in ${\frak C}$ such that
$E_2$ is a $\beta$-twisted stable object with
\begin{align*}
 v(E_2)=r_2 e^\beta+a_2 \varrho_X
 +(d_2 H+D_2+(d_2 H+D_2,\beta)\varrho_X),\quad
 d'/r' \geq d_2/r_2
\end{align*}
and $\ker(H^0(E) \to E_2) \in   {\frak T}^\mu \subset{\frak A}^\mu$.
\begin{NB}
Here we used ${\frak A}_{(\beta,\omega)}={\frak A}^\mu$.
\end{NB}
Since $H^0(E) \in {\frak T}^\mu$ 
and $\rk H^0(E) \geq r>0$, we get
$d_2>0$.
Then $E \to E_2$ is surjective in ${\frak A}^\mu$.
Let $E_1$ be the kernel of $E \to E_2$ in ${\frak A}^\mu$.
Since $E$ is semi-stable, 
$\phi(E) \leq \phi(E_2)$, which implies that
$(da_2-d_2 a) \geq (d r_2-d_2 r) (\omega^2)/2$. 
We have
$d/r=(d'-d'')/(r'-r'') \geq d'/r' \geq d_2/r_2$,
so that 
\begin{align*}
d=d'-d'' \geq d' \geq d_2 r'/r_2 \geq d_2.
\end{align*}
If $d_2=d$, then $\deg_G(E_1)=0$ and $\phi(E_1)=1>\phi(E)$.
Therefore $d_2<d$.
\begin{NB}
We do not use ${\frak A}_{(\beta,\omega)}={\frak A}^\mu$
\end{NB}
\begin{NB}
Old argument:
If ${\frak A}={\frak A}^\mu$ and $d_2=d$, then
applying Lemma \ref{lem:a>0} to
$E_1$, we get $a-a_2 \geq 0$, which implies that
$(d r_2-d_2 r) (\omega^2)/2+d_2 a-d a_2 \geq
(d r_2-d_2 r) (\omega^2)/2$.
\end{NB}
If $d/r>d_2/r_2$, then $(\star 1)$ implies that
$(da_2-d_2 a)< (d r_2-d_2 r) (\omega^2)/2$.
Therefore $d_2/r_2=d/r$ and $(r'',d'')=(0,0)$. 
In particular, $H^{-1}(E)=0$. 
Then $da_2-d_2 a \geq 0$ implies that $a_2/r_2 \geq a/r$.
Therefore $E$ is $\beta$-twisted semi-stable.

We next assume that $r=0$.
We take the expression of $v(H^0(E))$ and $(H^{-1}(E))$ 
as \eqref{eq:v(H^0(E))} and \eqref{eq:v(H^-1(E))}.
Then $d' \geq 0$ and $d'' \leq 0$.
Assume that $r'>0$.
Let $H^0(E) \to E_2$ be a quotient in ${\frak C}$ such that
$E_2$ is a $\beta$-twisted stable object with
\begin{align*}
 v(E_2)=r_2 e^\beta+a_2 \varrho_X
  +(d_2 H+D_2+(d_2 H+D_2,\beta)\varrho_X),\quad
 d'/r' \geq d_2/r_2
\end{align*}
and $\ker(H^0(E) \to E_2) \in {\frak T}^\mu \subset {\frak A}^\mu$.
Then $E \to E_2$ is surjective in ${\frak A}^\mu$.
Since $H^0(E) \in {\frak T}^\mu$, we see that $d_2>0$.
Let $E_1$ be the kernel of $E \to E_2$ in ${\frak A}^\mu$.
Since $E$ is semi-stable, $\phi(E) \leq \phi(E_2)$, 
which implies that
$(da_2-d_2 a) \geq (d r_2-d_2 r) (\omega^2)/2$. 
We have
$d=d'-d'' \geq d' \geq d_2 r'/r_2 \geq d_2$.
Since $d r_2-d_2 r=d r_2>0$,
$(\star 1)$ implies that
$(da_2-d_2 a)< (d r_2-d_2 r) (\omega^2)/2$, which is a contradiction.
Therefore $r'=r''=0$.
In particular, $H^{-1}(E)=0$. 
Let $E_2$ be a $\beta$-twisted stable quotient object of $E$ with 
\begin{align*}
v(E_2)=a_2 \varrho_X
 +(d_2 H+D_2+(d_2 H+D_2,\beta)\varrho_X),
\quad
a/d > a_2/d_2.
\end{align*}
Then $\phi(E) \leq \phi(E_2)$ implies 
$d a_2-d_2 a \geq 0$, which is a contradiction.
Therefore $E$ is $\beta$-twisted semi-stable.
\end{proof}

\begin{NB}
\begin{rem}
We shall replace ${\frak A}^\mu$ by
${\frak A}_{(\beta,\omega)}$. 
Then in the proof of (2), $H^0(E) \in {\frak T}_{(\beta,\omega)}$
implies that $d_2 \geq 0$. So $d_2$ may be 0. 
\end{rem}
\end{NB}

\begin{lem}\label{lem:star-3}
Assume that $(\star 3)$ holds.
\begin{NB}
If $(\star 2)$ holds, then
the case of $d=d_2$ will be treated   
by ${\frak A}_{(\beta,\omega)}={\frak A}^\mu$.
\end{NB}
For an object $E$ of ${\bf D}(X)$ with $v(E)=v$,
$E$ is a $\beta$-twisted semi-stable object of ${\frak C}$ if and only if 
$E[1]$ is a semi-stable object of ${\frak A}_{(\beta,\omega)}$.
\end{lem}

\begin{proof}
We shall use the expression \eqref{eq:v} for $v$.
Assume that $E \in {\frak C}$ is $\beta$-twisted semi-stable.
Then $E[1] \in {\frak A}_{(\beta,\omega)}$.
We consider an exact sequence
\begin{equation*}
0 \to F_1[1] \to E[1] \to F_2[1] \to 0
\end{equation*}
in ${\frak A}_{(\beta,\omega)}$ 
such that $F_2[1] \in {\frak A}_{(\beta,\omega)}$ 
is a stable quotient object of $E[1]$. 
Then $H^0(F_2[1])=0$.
We set 
\begin{equation*}
v(F_2)=r_2 e^\beta+a_2 \varrho_X+(d_2 H+D_2+(d_2 H+D_2,\beta)\varrho_X),
\quad
D_2 \in H^{\perp} \cap \NS(X)_{\Bbb Q}.
\end{equation*}
Let $F'$ be the image of $E \to H^{-1}(F_2[1])=F_2$ in ${\frak C}$.
We set 
\begin{equation*}
v(F')=r' e^\beta+a' \varrho_X +(d'H+D)+(d' H+D',\beta)\varrho_X,
\quad
D' \in H^{\perp} \cap \NS(X)_{\Bbb Q}.
\end{equation*}
Then $d/r \leq d'/r' \leq d_2/r_2$.
Thus $r d_2-r_2 d \geq 0$.
Moreover if $d'/r'=d_2/r_2$, then $H^0(F_1[1])=F_2/F'$ 
is a 0-dimensional object. 
\begin{NB}
If $d'/r'=d_2/r_2$, then
$\deg_G F_2/F'=(d/r) \rk(F_2/F') \leq 0$ and the equality holds only if
$\rk(F_2/F') =0$. Since $F_2/F'=H^0(F_1[1]) \in {\frak T}$,
we have $\deg_G F_2/F'=0$. Then $F_2/F'$ is 0-dimensional. 
\end{NB}
If $d_2=0$, then $Z_{(\beta,\omega)}(F_2[1]) \in {\Bbb R}_{<0}$.
Hence $\phi(E[1])<1=\phi(F_2[1])$.
\begin{NB}
We don't need ${\frak A}_{(\beta,\omega)}={\frak A}^\mu$.
\end{NB}
So we assume that $d_2<0$.
We note that $d_2 \geq d' \geq dr'/r \geq d$.
\begin{NB}
If $d_2=d$, then $r'=r$ and $d'=d_2=d$, which implies that
$\rk H^{-1}(F_1[1])=r-r'=0$.
Since $H^{-1}(F_1[1])$ is torsion free, 
$H^{-1}(F_1[1])=0$.
Moreover $\deg(H^0(F_1[1]))=0$.
In particular, if ${\frak A}_{(\beta,\omega)}={\frak A}^\mu$,
then $H^0(F_1[1])$ is a 0-dimensional object.
Thus we have an exact sequence in ${\frak C}$:
$$
0 \to E \to F_2 \to H^0(F_1[1]) \to 0.
$$
\end{NB}
If $r d_2-r_2 d>0$, then $(\star 3)$ implies that
$\phi(F_2[1]) > \phi(E[1])$.
If $r d_2-r_2 d=0$, then $(\star 3)$ implies 
that $da_2-d_2 a \geq 0$. 
Therefore $E[1]$ is semi-stable. 
Moreover we see that $E[1]$ is stable 
if $E$ is $\beta$-twisted stable.
\begin{NB}
If $d_2<0$, then
$(r d_2-r_2 d)\frac{(\omega^2)}{2}+(d a_2-d_2 a)  \geq 0$.
Thus $\phi(F_2[1]) \geq \phi(E[1])$.
If the equality holds, then
$r d_2-r_2 d=d a_2-d_2 a=0$.
Therefore $E[1]$ is semi-stable.
Moreover $E[1]$ is stable if $E$ is $\beta$-twisted stable.
\end{NB}

Conversely let $E[1]$ 
be a semi-stable object of ${\frak A}_{(\beta,\omega)}$.
We shall prove that $E$ is a $\beta$-twisted semi-stable object 
of ${\frak C}$.
We set 
\begin{equation*}
\begin{split}
v(H^{-1}(E[1]))=& r'' e^\beta+a'' \varrho_X+
(d''H+D''+(d''H+D'',\beta)\varrho_X),\\
v(H^{0}(E[1]))=& r' e^\beta+a' \varrho_X+(d'H+D'+(d'H+D',\beta)\varrho_X).\\
\end{split}
\end{equation*}
Then $r=r''-r'$, $d=d''-d'$ and $a=a''-a'$.
We also have $d''/r'' \geq d/r$ and if the equality holds, 
then $r'=d'=0$. 
Let $F_1$ be a $\beta$-twisted stable subobject of
$H^{-1}(E[1])$ such that
\begin{equation*}
v(F_1)= r_1 e^\beta+a_1 \varrho_X+(d_1 H+D_1+(d_1 H+D_1,\beta)\varrho_X),\quad
d_1/r_1 \geq d''/r''
\end{equation*}
and $H^{-1}(E[1])/F_1 \in {\frak F}_{(\beta,\omega)}$.
Then $d_1/r_1 \geq d/r$ and
$F_1[1] \to E[1]$ is injective in ${\frak A}_{(\beta,\omega)}$. 
If $d_1=0$, then
$Z_{(\beta,\omega)}(F_1[1]) \in {\Bbb R}_{<0}$, which implies that
$\phi(F_1[1])=1>\phi(E[1])$. Therefore $d_1<0$.
\begin{NB}
We don't need ${\frak A}_{(\beta,\omega)}={\frak A}^\mu$.
\end{NB}
We note that $d_1 \geq d'' r_1/r'' \geq d'' \geq d$.
Moreover if $d_1=d$, then
$r_1=r''$, $d_1=d''=d$ and $d'=0$.
\begin{NB}
Moreover
if ${\frak A}_{(\beta,\omega)}={\frak A}^\mu$, then 
we see that $H^0(E[1])$ is a 0-dimensional object.
In particular $r'=0$ and $r_1=r$.
\end{NB} 
If $d_1/r_1>d/r$, then
$(\star 3)$ implies that
$\phi(F_1[1]) > \phi(E[1])$, which is a contradiction.
Therefore $d_1/r_1=d/r$. 
Then we have $r'=d'=0$ and $a' \geq 0$. 
By $(\star 3)$ and $d_1/r_1=d/r$,
we get $a_1 d-a d_1 \geq 0$.
If $a_1/r_1 \geq a''/r=(a+a')/r$, then
we have $a_1 d-a d_1 \leq 0$.
Hence $a_1/r_1=(a+a')/r=a/r$, which implies that $a'=0$.
Therefore $H^0(E[1])=0$ and $E$ is $\beta$-twisted semi-stable.
\end{proof}

\begin{NB}
\begin{rem}
If $(\star 2')$ holds, then $d'=0$ only says that
$H^0(E(1)) 
\in {\frak T}_{(\beta,\omega)}$ is 
an object of positive rank.
For example, if ${\frak A}_{(\beta,\omega)}={\frak A}$
in Example \ref{ex:exceptional},
then $H^0(E(1))$ fits in an exact sequence
$$
0 \to A \to H^0(E(1)) \to E_0^{\oplus n} \to 0,
$$
where $A$ is a 0-dimensional sheaf.
Since $\Ext^1(E_0,A)=0$, $H^0(E(1)) \cong A \oplus E_0^{\oplus n}$.
By Lemma \ref{lem:proj-dim=1},
$E^{\vee}$ fits in an exact sequence
$$
0 \to (E_0^{\vee})^{\oplus n}[1] \to E^{\vee} \to
H^0(E^{\vee}) \to 0,
$$
where $H^0(E^{\vee})$ is torsion free or purely 1-dimensional. 
\end{rem}
\end{NB}

In order to treat the case ($\star 2$),
it is necessary to consider another category ${\frak C}^D$ 
given in Definition~\ref{defn:category-C:dual}.
We start with the next lemma.

\begin{lem}\label{lem:proj-dim=1}
Let $F$ be an object of
${\frak A}_{(\beta,\omega)}$.
\begin{enumerate}
\item[(1)]
If $F$ is a semi-stable object with 
$\phi(F)<1$, then
$\Hom(A,F)=0$ for all 0-dimensional object $A$ of ${\frak C}$.
\item[(2)] 
For $F \in {\frak A}_{(\beta,\omega)}$,
$\Hom(A,F)=0$ for all 0-dimensional object $A$ of ${\frak C}$
if and only if
$F$ is represented by a complex $\varphi:U_{-1} \to U_0$ 
such that $U_{-1}$ and $U_0$ are local projective objects
of ${\frak C}$.
\end{enumerate}
\end{lem}
  
\begin{proof}
(1)
Let $A$ be a 0-dimensional object of ${\frak C}$.
Then $A \in {\frak A}_{(\beta,\omega)}$ and $\phi(A)=1$.
Since $\phi(F)<1$, 
we get $\Hom(A,F)=0$.

(2)
Assume that $F \in {\frak A}_{(\beta,\omega)}$ 
satisfies 
$\Hom(A,F)=0$ for all 0-dimensional object $A$ of ${\frak C}$.
Let $G$ be a local projective generator of ${\frak C}$.
Since $\Ext^1(G(-n),H^{-1}(F)[1])=\Ext^2(G(-n), H^{-1}(F))=0$
for $n \gg 0$,
we have a morphism $G(-n)^{\oplus N} \to F$
which induces a surjective morphism
$G(-n)^{\oplus N} \to H^0(F)$.
We set $U_0:=G(-n)^{\oplus N}$ and set
$U_{-1}:=\mathrm{Cone}(U_0 \to F)[-1]$.
By the exact triangle
\begin{equation*}
U_{-1} \overset{\varphi}{\to} U_0 \to F \to U_{-1}[1],
\end{equation*}
$H^i(U_0)=0$, $i \ne 0$ and $H^i(F)=0$, $i \ne -1,0$ imply
that
$H^i(U_{-1})=0$ for $i \ne 0$.
\begin{NB}
Then we have an exact sequence
\begin{equation*}
\begin{CD}
H^{-1}(U_{-1}) @>>> H^{-1}(U_0)=0 @>>> H^{-1}(F)  @>>>\\
H^{0}(U_{-1}) @>>> H^{0}(U_0) @>>> H^{0}(F) @>>> \\
H^{1}(U_{-1}) @>>> H^{1}(U_0)=0 @>>> H^{1}(F) @>>> \ldots
\end{CD}
\end{equation*}
\end{NB}
Hence $U_{-1} \in {\frak C}$.
We shall prove that $U_{-1}$ is a local projective object of ${\frak C}$.
By \cite[Rem. 1.1.34]{PerverseI},
it is sufficient to show $\Hom(U_{-1},A)=0$ for all
0-dimensional object $A$ of ${\frak C}$.
By the exact triangle
\begin{equation}\label{U_dot}
U_{-1} \to U_0 \to F \to U_{-1}[1],
\end{equation}
we have an exact sequence
\begin{equation}
\Hom(A,F) \to \Ext^1(A,U_{-1}) \to \Ext^1(A,U_0).
\end{equation}
Since $U_0$ is a local projective object of 
${\frak C}$ and $A$ is a 0-dimensional object
of ${\frak C}$, we have $\Ext^1(A,U_0)=\Ext^1(U_0,A)^{\vee}=0$.
By the assumption, we also have $\Hom(A,F)=0$.
Therefore we get
$\Ext^1(U_{-1},A)=\Ext^1(A,U_{-1})^{\vee}=0$.

\begin{NB}
Old argument:
If $\Ext^1(U_{-1},A) \ne 0$ for an irreducible object $A$
of ${\frak C}$,
then we have a non-trivial extension
\begin{equation}\label{eq:U'}
0 \to U_{-1} \to U_{-1}' \to A \to 0.
\end{equation}
Since $\Ext^1(A,U_0) \cong \Ext^1(U_0,A)^{\vee}=0$,
we have a morphism $\varphi':U_{-1}' \to U_0$ such that 
$U_{-1} \to U_{-1}' \overset{\varphi'}{\to} U_0$
is $\varphi$.
Hence we have a morphism $A \to F$ which induces the morphism
$A \to U_{-1}[1]$ corresponding to \eqref{eq:U'}. 
It contradicts the assumption on $F$.
Therefore $\Ext^1(U_{-1},A)=0$ for all $A$, 
which implies that
 $U_{-1}$ is a local projective object of ${\frak C}$.  
\begin{NB2}
By the definition of ${\frak A}_{(\beta,\omega)}$,
$H^{-1}(F)$ is a torsion free object of ${\frak C}$.
Hence 
$\Hom(A,F[-1])=\Hom(A,H^{-1}(F))=0$. 
\end{NB2}
\begin{NB2}
By the exact sequence
$0 \to H^{-1}(F) \to U_{-1} \to U_0$,
$U_{-1}$ is a torsion free object of ${\frak C}$.
If $U_{-1}$ contains a torsion object $B$ in $\Coh(X)$,
then $A:=B[1] \in {\frak C}$ and
 we have an exact sequence in ${\frak C}$:
$ 0 \to U_{-1} \to U_{-1}/B \to B[1] \to 0$. 
\end{NB2}
\end{NB}

Conversely assume that $F$ is represented by a complex
$U_{-1} \to U_0$ such that $U_{-1},U_0$ are local projective
objects.
Let $A$ be a 0-dimensional object of ${\frak C}$.
By the exact triangle \eqref{U_dot}
and $\Hom(A,U_0)=\Ext^1(A,U_{-1})=0$,
we have $\Hom(A,F)=0$.
\end{proof}

\begin{NB}
\begin{rem}
Even if $A$ is irreducible in ${\frak C}$,
$\phi:A \to E$ may not be injective in ${\frak A}_{(\beta,\omega)}$.
If ${\frak A}_{(\beta,\omega)}={\frak A}^\mu$, then
$A$ is an irreducible object of ${\frak A}^\mu$, which shows
the injectivity of $\phi$.
However for a general ${\frak A}_{(\beta,\omega)}$,
we only have an exat sequence
$$
 H^{-1}(\im \phi) \to H^0(\ker \phi) \to A \to H^0(\im \phi) \to 0.
$$ 
If $H^0(\im \phi)=0$, then
$H^{-1}(\im \phi)[1] \subset E$.

\end{rem}
\end{NB}

\begin{rem}
Assume that $\rk E<0$ and $(c_1(E),H)>0$.
Then $\phi(E)<1$.
If $E$ is semi-stable, then 
$H^{-1}(E)$ is a local projective object.
Indeed for a 0-dimensional object $A$ of ${\frak C}$,
by using the exact sequence
$$
0 \to H^{-1}(E)[1] \to E \to H^0(E) \to 0,
$$
we have an exact sequence
$$
\Hom(A,H^0(E)[-1]) \to \Hom(A,H^{-1}(E)[1]) \to \Hom(A,E).
$$
Therefore $\Hom(A, H^{-1}(E)[1])=0$.
\end{rem}

\begin{NB}
Let $G$ be a local projective generator of ${\frak C}$.
Since $\Ext^1(G(-n),H^{-1}(E)[1])=\Ext^2(G(-n), H^{-1}(E))=0$
for $n \gg 0$,
we have a morphism $G(-n)^{\oplus N} \to E$
which induces a surjective morphism
$G(-n)^{\oplus N} \to H^0(E)$.
We set $V_0:=G(-n)^{\oplus N}$ and set
$V_{-1}:=\mathrm{Cone}(V_0 \to E)[-1]$.
By the exact triangle
\begin{equation*}
V_{-1} \to V_0 \to E \to V_{-1}[1],
\end{equation*}
we get $H^i(V_{-1})=0$ for $i \ne 0$.
Hence $V_{-1} \in {\frak C}$.
If $\Ext^1(V_{-1},A) \ne 0$ for an irreducible object $A$
of ${\frak C}$.
Then we have a non-trivial extension
$$
0 \to V_{-1} \to V_{-1}' \to A \to 0.
$$
Since $\Ext^1(A,V_0) \cong \Ext^1(V_0,A)^{\vee}=0$,
we have a complex $V_{-1}' \to V_0$ such that $V_{-1} \to V_{-1}' \to V_0$
is the original $V_{-1} \to V_0$.
Then we have a morphism $A \to E$ which induces an injection
$A \to H^0(E)$.
By the semi-stability of $E$, this is impossible.
Therefore $V_{-1}$ is a local projective object of ${\frak C}$.  
In particular,
$E^{\vee}[1]$ is represented by
the complex $V_0^{\vee} \to V_{-1}^{\vee}$.
By Remark \ref{rem:tilting:dual},
$V_0^{\vee}$ and $V_{-1}^{\vee}$ are local projective
objects of ${\frak C}^D$.
If $H^0(E)$ is 0-dimensional, then
$E^{\vee}[1]$ is an object of ${\frak C}^D$.
It is a $\mu$-semi-stable object.
Let $E^{\vee}[1] \to  F$ be a quotient of $E^{\vee}[1]$
such that $F$ is $-\beta$-stable and
\begin{equation*}
v(F)=r_1 e^{-\beta}+a_1 \varrho_X+
(-(d_1 H+D_1)+(d_1 H+D_1,\beta)\varrho_X).
\end{equation*}  
Then we have a morphism 
$F^{\vee}[1] \to E$, which implies that
$\phi(F^{\vee}[1]) \leq \phi(E)$.
Then we have $-d a_1+d_1 a \geq 0$, which implies that
$a_1/r_1 \geq a/r$.
Therefore $E^{\vee}[1]$ is $-\beta$-twisted semi-stable.
\end{NB}

\begin{defn}[{\cite[Lem. 1.1.14]{PerverseI}}]
\label{defn:category-C:dual}
Let $G$ be a local projective generator of ${\frak C}$.
Then we have an abelian category
\begin{equation*}
{\frak C}^D:=
\{E \in {\bf D}(X) \mid 
  H^i(E)=0\, (i \ne -1,0),\ 
  \pi_*(G \otimes H^{-1}(E))=0,\ 
  R^1 \pi_*(G \otimes H^0(E))=0
\}.
\end{equation*}
Since $G^{\vee}$ is a local projective generator of 
${\frak C}^D$, 
$G^{\vee}$-twisted semi-stability and 
     $(-\beta)$-twisted semi-stability are defined by Definition
\ref{defn:twisted-stability}. 
\end{defn}

\begin{NB}
\cite[Lem. \ref{lem:tilting:dual}].
\end{NB}

\begin{NB}
\begin{rem}
Let $F$ and $\varphi$ be the complexes in Lemma~\ref{lem:proj-dim=1}.
Since $\ker (\varphi^{\vee})$ and $H^0(F)^{\vee}$ are the same 
on $X \setminus \cup_i Z_i$, we have 
$\mu_{\max,G^{\vee}}(\ker (\varphi^{\vee})) \leq 0$.
Since $(\coker (\varphi^{\vee}))^{\vee \vee}$ and $H^{-1}(F)^{\vee}$
are the same on $X \setminus \cup_i Z_i$, we also have 
$\mu_{\min,G^{\vee}}(\coker (\varphi^{\vee})) \geq 0$.
\end{rem}
\end{NB}

\begin{lem}\label{lem:proj-dim=1:2}
Let $E$ be a torsion free object of ${\frak C}^D$.
\begin{enumerate}
\item[(1)]
There is a local projective resolution 
\begin{equation}
0 \to U_{-1} \to U_0 \to E \to 0
\end{equation}
of $E$ such that $U_{-1}$ and $U_0$ are
local projective objects of ${\frak C}^D$.
Moreover $H^0(E^{\vee}[1])$ is a 0-dimensional object
of ${\frak C}^D$.
\item[(2)]
If $E$ is a $\mu$-semi-stable object of ${\frak C}^D$
with $\deg_{G^{\vee}}(E)>0$,
then $E^{\vee}[1] \in {\frak A}_{(\beta,\omega)}$
for all $\omega$.
\end{enumerate}
\end{lem}

\begin{proof}
(1)
We note that $\Hom(A,E)=0$ for all 0-dimensional object $A$ of 
${\frak C}^D$.
Then the proof of the first claim is similar to that
of Lemma \ref{lem:proj-dim=1} (2).

For the second claim, we take a closed subset $D$ of $X$ such that
$D$ contains the exceptional set of $\pi:X \to Y$,
$\dim \pi(D)=0$ 
and $E_{|X \setminus D}$ is a locally free sheaf.
Then $H^0(E^{\vee}[1])_{|X \setminus D}=0$.
Hence $H^0(E^{\vee}[1])$ is a 0-dimensional object
of ${\frak C}=({\frak C}^D)^D$.

(2) Since $H^{-1}(E^{\vee}[1])$ is a $\mu$-semi-stable
object of ${\frak C}$ with $\deg_G(H^{-1}(E^{\vee}[1]))=
-\deg_{G^{\vee}}(E)<0$,
we get the claim. 
\end{proof}

\begin{NB}
The forst claim also holds for purely 1-dimensional objects.
\end{NB}

\begin{prop}\label{prop:star-2}
Assume that $(\star 2)$ holds for $v$
(i.e., $(\star 1)$ holds for $v^{\vee}$).
Let $F$ be an object of ${\bf D}(X)$ with
$v(F)=-v$. Then 
$F$ is a semi-stable object of ${\frak A}^\mu$ if and only if
$F^{\vee}[1]$ is a $(-\beta)$-twisted semi-stable object of ${\frak C}^D$
with $v(F^{\vee}[1])=v^{\vee}$.
\end{prop}

\begin{proof}
We divide the argument into two steps.

(1)
Let $E$ be a $(-\beta)$-twisted semi-stable object of
${\frak C}^D$ with $v(E)=v^{\vee}$.
There is a surjective morphism
$\psi:(G^{\vee}(-n))^{\oplus N} \to E$, where $n \gg 0$.
Since $E$ does not contain a 0-dimensional subobject of
${\frak C}^D$, we see that $\ker \psi$ is a local projective
object of ${\frak C}^D$.
\begin{NB}
For a 0-dimensional object $A$ of ${\frak C}^D$,
we have an exact sequence
$$
0=\Ext^1((G^{\vee}(-n))^{\oplus N},A) \to
\Ext^1(\ker \psi,A) \to
\Ext^2(E,A) \cong \Hom(A,E)^{\vee}=0.
$$
\end{NB}
Thus $E$ is represented by a complex
$V_{-1} \to V_0$ such that $V_i$ are local projective objects
of ${\frak C}^D$.
Then we have an exact triangle
\begin{align*}
E^{\vee} \to V_0^{\vee} \to V_{-1}^{\vee} \to E^{\vee}[1].
\end{align*}
Since $V_i^{\vee}$ are local projective objects of ${\frak C}$,
we see that $H^i(E^{\vee}[1])=0$ for $i \ne -1,0$.
Moreover we get $H^{-1}(E^{\vee}[1])$ is a $\mu$-semi-stable object
of ${\frak C}$ with $\deg_G(H^{-1}(E^{\vee}[1]))/\rk G=d(H^2)<0$
and $H^0(E^{\vee}[1])$ is a 0-dimensional object of ${\frak C}$.
\begin{NB}
Indeed we have
$H^{-1}(E^{\vee}[1])={\cal H}om(E,{\cal O}_X)$ and
$H^0(E^{\vee}[1])={\cal E}xt^1(E,{\cal O}_X)$ on $X \setminus \cup_i Z_i$.
\end{NB}
Thus $E^{\vee}[1] \in {\frak A}^\mu$. 
We take an exact sequence
\begin{equation*}
0 \to F_1 \to E^{\vee}[1] \to F_2 \to 0
\end{equation*}
in ${\frak A}^\mu$ such that $F_2$ is a stable object.
If $\phi(F_2)=1$, then
$\phi(E^{\vee}[1])<\phi(F_2)$.
We next assume that $\phi(F_2)<1$.
Lemma~\ref{lem:proj-dim=1} implies that $E^{\vee}[1]$,
and hence $F_1$ does not contain a 0-dimensional object.
We have an exact sequence
\begin{equation*}
\begin{array}{ccccccccc}
0 & \to & H^{-1}(F_1) 
  & \to & H^{-1}(E^{\vee}[1]) 
  & \to & H^{-1}(F_2)
\\[0.4em]
  & \to & H^0(F_1) 
  & \to & H^0(E^{\vee}[1])
  & \to & H^0(F_2) 
  & \to & 0
\end{array}
\end{equation*}
in ${\frak C}$.
Since 
$H^0(E^{\vee}[1])$ is 0-dimensional,
$H^0(F_2)$ is also 0-dimensional.
Since $F_2$ does not contain a 0-dimensional object,
by using Lemma~\ref{lem:proj-dim=1}, we see that
$F_2^{\vee}[1] \in {\frak C}^D$.
So we have an exact sequence
\begin{equation*}
0 \to H^{-1}(F_1^{\vee}[1]) \to
F_2^{\vee}[1] \overset{\psi}{\to} E \to H^0(F_1^{\vee}[1]) \to 0
\end{equation*}
in ${\frak C}^D$.

Now recall the expression \eqref{eq:v} for $v$.
We also set
\begin{equation*}
\begin{split}
v(F_2^{\vee}[1])
&= r_2 e^{-\beta}+a_2 \varrho_X+(-(d_2 H+D_2)+(d_2 H+D_2,\beta)\varrho_X),
\quad
D_2 \in H^{\perp} \cap \NS(X)_{\Bbb Q},
\\
v(\im \psi)
&= r' e^{-\beta}+a' \varrho_X+(-(d' H+D')+(d' H+D',\beta)\varrho_X),
\quad
D' \in H^{\perp} \cap \NS(X)_{\Bbb Q}.
\end{split}
\end{equation*}
Since $-d_2 \leq -d'$, we have $-d_2/r_2 \leq -d'/r'$.
Since $E$ is $(-\beta)$-twisted semi-stable,
$-d'/r' \leq -d/r$. 
Thus we have $-d_2/r_2 \leq -d/r$.

If the equality holds, then
$r_2=r'$ and $-d_2=-d'$.
If $d_2=0$, then $1=\phi(F_2)>\phi(E^{\vee}[1])$.
Assume that $d_2<0$.
We note that
$-d_2 \leq -d' \leq -d r'/r \leq -d$.
If $d=d_2$, then
$d=d'$ and $r=r'$.
Hence $\coker \psi$ is a 0-dimensional object.
We also have $\deg_{G^{\vee}}(\ker \psi)/\rk G=(d_2-d')(H^2)=0$.  
Then $\rk F_1=\rk H^{-1}(F_1^{\vee}) \geq 0$
and $\deg_G(F_1)=0$.
In particular, $\deg_G (H^0(F_1))=\deg_G(H^{-1}(F_1))=0$. 
Since $H^0(F_1) \in {\frak T}^\mu$ with $\deg_G(H^0(F_1))=0$, 
$H^0(F_1)$ is a 0-dimensional object and
$H^{-1}(F_1)=0$. 
This means that $F_1$ is a 0-dimensional object,
which is a contradiction.
Therefore $d_2>d$.  
If $rd_2-r_2 d >0$, then $(\star 2)$ implies that
$\phi(F_2)>\phi(E^{\vee}[1])$.
If $rd_2-r_2 d=0$, then $r_2=r'$ implies that $H^{-1}(F_1^{\vee}[1])$
is a torsion free object of rank 0. 
Therefore $H^{-1}(F_1^{\vee}[1])=0$.
By the $(-\beta)$-twisted semi-stability of $E$,
$a_2/r_2 \leq a/r$. Then we get $d a_2-d_2 a \geq 0$,
which implies that 
$\phi(F_2) \geq \phi(E^{\vee}[1])$. 
Therefore $E^{\vee}[1]$ is a semi-stable object of 
${\frak A}^\mu$
with $v(E^{\vee}[1])=-v$.

(2)
Conversely, let $F$ be a semi-stable object of ${\frak A}^\mu$
such that $v(F)=-v$.
By Lemma~\ref{lem:proj-dim=1}, 
$F$ is represented by a complex
$U_{-1} \to U_0$ such that
$U_{-1}$ and $U_0$ are local projective objects of ${\frak C}$.
Then $F^{\vee}[1]$ is represented by
the complex $U_0^{\vee} \to U_{-1}^{\vee}$.
By \cite[Lem. 1.1.14 (2)]{PerverseI},
\begin{NB}
By Remark~\ref{rem:tilting:dual},
\end{NB}
$U_0^{\vee}$ and $U_{-1}^{\vee}$ are local projective
objects of ${\frak C}^D$.
By the proof of Lemma~\ref{lem:star-3},
$H^0(F)$ is 0-dimensional and $H^{-1}(F)$ is a $\mu$-semi-stable
object of ${\frak C}$. Hence
$F^{\vee}[1] \in {\frak C}^D$ is a $\mu$-semi-stable object.
Let $F^{\vee}[1] \to  F_2$ be a quotient of $F^{\vee}[1]$
such that $F_2$ is $(-\beta)$-twisted stable and
\begin{equation*}
v(F_2)=r_1 e^{-\beta}+a_1 \varrho_X+
(-(d_1 H+D_1)+(d_1 H+D_1,\beta)\varrho_X),\;-d_1/r_1=-d/r.
\end{equation*}  
Since $F_1:=\ker(F^{\vee}[1] \to F_2)$ is a $\mu$-semi-stable object
of ${\frak C}^D$, Lemma \ref{lem:proj-dim=1:2}
implies that $F_1^{\vee}[1]$ and $F_2^{\vee}[1]$ are
object of ${\frak A}^\mu$ and we have an exact sequence in
${\frak A}^\mu$:
\begin{equation*}
0 \to F_2^{\vee}[1] \to F \to F_1^{\vee}[1] \to 0.
\end{equation*}
By the semi-stability of $F$,
$\phi(F_2^{\vee}[1]) \leq \phi(F)$.
Then we have $-d a_1+d_1 a \geq 0$, which implies that
$a_1/r_1 \geq a/r$.
Therefore $F^{\vee}[1]$ is $(-\beta)$-twisted semi-stable.
\end{proof}

By Lemma~\ref{lem:large},
Proposition~\ref{prop:star-1} and Proposition~\ref{prop:star-2},
we have the following result
of Toda \cite[Prop. 6.4, Lem. 6.5]{Toda}.
\begin{cor}\label{cor:Toda}
Let $E$ be an object of ${\bf D}(X)$ with 
$$
v(E)=r e^\beta+a \varrho_X+(dH+D+(dH+D,\beta)\varrho_X),\quad 
d>0,\  D \in H^{\perp}.
$$
Assume that $(\omega^2)  \gg \langle v(E)^2 \rangle-(D^2)$.
\begin{enumerate}
\item[(1)]
If $r \geq 0$, then
$E$ is a semi-stable object of
${\frak A}^\mu$
if and only if 
$E$ is a $\beta$-twisted semi-stable object of ${\frak C}$.
\item[(2)]
If $r < 0$, then
$E$ is a semi-stable object of
${\frak A}^\mu$ if and only if 
$E^{\vee}[1]$ is a $(-\beta)$-twisted semi-stable object of ${\frak C}^D$ 
\end{enumerate}
\end{cor}

By this corollary, we also have the following. 
\begin{cor}\label{cor:Toda-moduli}
For
$$
v=r e^\beta+a \varrho_X+(dH+D+(dH+D,\beta)\varrho_X),\quad 
d>0,\ D \in H^{\perp},
$$
assume that $(\omega^2)  \gg \langle v^2 \rangle-(D^2)$.
Then there is a coarse moduli scheme $M_{(\beta,\omega)}(v)$,
which is given by 
\begin{equation*}
M_{(\beta,\omega)}(v)
\cong
\begin{cases}
M_H^\beta(v),& \rk v \geq 0,\\
M_H^{-\beta}(-v^{\vee}), & \rk v<0.
\end{cases}
\end{equation*}
\end{cor}

\begin{cor}\label{cor:isom}
Assume that $d=d_{\min}$.
Then there is a coarse moduli scheme
$M_{(\beta,\omega)}(v)$.
\end{cor}

\begin{proof}
If $(\omega^2)>2$, then ${\frak A}_{(\beta,\omega)}={\frak A}^\mu$.
In this case, by Lemma~\ref{lem:d_min} 
Proposition \ref{prop:star-1} and
Proposition \ref{prop:star-2},
we get 
the moduli scheme.
For a general case,
the claim follows from Proposition~\ref{prop:equiv}.
\end{proof}


\subsection{A relation of \cite{Stability} with Bridgeland's stability}
\label{subsect:FM}

We take a sufficiently small element $\alpha \in \NS(X)_{\Bbb Q}$
such that $-\langle e^{\beta+\alpha},v(A) \rangle>0$
for all 0-dimensional objects $A$ of ${\frak C}$.
Let $X':=\overline{M}_H^{\beta+\alpha}(v_0)$ be the moduli scheme of 
$(\beta+\alpha)$-twisted semi-stable
objects of ${\frak C}$ and there is a projective morphism
$X' \to Y'$, where $Y':=\overline{M}_H^{\beta}(v_0)$.
By choosing a general $\alpha$, we may assume that
$X'$ is a smooth projective surface.
Then there is a  universal family
${\cal E}$ on $X \times X'$ as a twisted object.
For simplicity, we assume that ${\cal E}$ is untwisted. 
Let
\begin{equation*}
\begin{matrix}
\Phi:& {\bf D}(X) & \to & {\bf D}(X')\\
& E & \mapsto & {\bf R}p_{X' *}(p_X^*(E) \otimes {\cal E}^{\vee})
\end{matrix}
\end{equation*}
be the Fourier-Mukai transform. 
Let ${\frak C}'$ be the category of perverse coherent sheaves
on $X'$ associated to $\beta'$, that is,
$-\langle e^{\beta'},v(A) \rangle>0$ for all
0-dimensional objects $A$ of ${\frak C}'$.
For $\omega$ with $(\omega^2)<2/r_0^2$, we shall consider 
the Fourier-Mukai transform of ${\frak A}={\frak A}_{(\beta,\omega)}$.
We take $\eta+\sqrt{-1}\omega$ such that
$\eta \in H^{\perp}$ and
$-(\eta^2)$ is sufficiently small.
Then ${\frak A}_{(\beta+\eta,\omega)}={\frak A}$.
Let ${{\frak A}'}^\mu$ be the category in Definition~\ref{defn:category}
associated to the pair $(\beta',\widetilde{\omega}')$
(cf. \eqref{eq:Phi(space)}).
By \eqref{eq:stability-comm} and a direct calculation, we have
\begin{equation*}
\begin{split}
Z_{(\beta'+\widetilde{\eta},\widetilde{\omega})}(\Phi(E)[1])
=-\frac{2}{r_0((\eta+\sqrt{-1}\omega)^2)} 
Z_{(\beta+\eta,\omega)}(E).
\end{split}
\end{equation*}

\begin{NB}
\begin{equation}
\begin{split}
Z_{(\beta',\widetilde{\omega})}(\Phi(E)[1])
=& \langle e^{\beta'+\sqrt{-1}\widetilde{\omega}},
v(\Phi(E)[1])\rangle\\
=& \frac{r}{r_0}-\frac{(\widetilde{\omega}^2)}{2}(r_0 a)
+\sqrt{-1}d(\widehat{H},\widetilde{\omega})\\
=& \frac{2}{(\omega^2)r_0}(-a+r\frac{(\omega^2)}{2}+\sqrt{-1}d(H,\omega))
=\frac{2}{(\omega^2)r_0}Z_{(\beta,\omega)}(E).
\end{split}
\end{equation}
\end{NB}

Thus we have the following diagram:
\begin{equation}\label{eq:Z-comm}
\xymatrix{
 {\frak A}   
 \ar@{->}[d]_{Z_{(\beta+\eta,\omega)}}
 \ar[rr]^{\Phi[1]} 
 & & 
 {{\frak A}'}^{\mu}
 \ar[d]^{Z_{(\beta'+\widetilde{\eta},\widetilde{\omega})}}
 \\
 {\Bbb C}
 \ar[rr]_{\times \frac{2}{r_0((\omega^2)-(\eta^2))}}
 & & 
 {\Bbb C}
}
\end{equation}

We note that $(\omega^2)<2/r_0^2$ if and only if 
$(\widetilde{\omega}^2)>2$.
\begin{NB}
In the example \ref{ex:exceptional},
$(\omega^2)<2/r_0$ if and only if $(\widetilde{\omega}^2)>2/r_0$.
\end{NB}
Since $\Phi[1]:{\frak A} \to {{\frak A}'}^\mu$ is an equivalence
(\cite[Thm. 2.5.9]{PerverseII})
\begin{NB}
Theorem \ref{II-thm:equiv-Phi}
\end{NB}
and $Z_{(\beta',\widetilde{\omega})}(\phi(E)[1])
=(2/((\omega^2)r_0)) Z_{(\beta,\omega)}(E)$,
Corollary~\ref{cor:Bridgeland} (1) also follows from (2).

\begin{NB}
$(\Phi[1]({\frak A}),Z_{(\beta',\widetilde{\omega})})$
is a good stability condition on
${\bf D}(X_1)$:
\begin{NB2}
We first note that 
for a stability condition
$\sigma=({\cal A},Z)$, we set 
$Z(E):=\langle \mho,v(E) \rangle$.
Then $\mathrm{Re}(\mho), \mathrm{Im}(\mho)$ span a
positive definite 2-plane if
$\langle \mho,\mho \rangle=0$ and $\langle \mho,\overline{\mho} \rangle>0$.
This condition is invariant under the isometry of
the Mukai lattice.
In particular, it is invariant under the Fourier-Mukai transform. 
The condition $\langle \mho,\mho \rangle=0$
means that $\mho=r e^{\beta+\sqrt{-1}\omega}$, $r \ne 0$.
Then $\langle \mho,\overline{\mho} \rangle>0$ is
nothing but $(\omega^2)>0$.

The condition that
$Z(E) \ne 0$ for all $(-2)$-vectors $v(E)$ of $A^*_{\alg}(X)$
is also invariant under the Fourier-Mukai transform.
\end{NB2}
Assume that $X'=Y'$.

We shall prove that ${\cal E}_{|X \times \{ x_1 \}}[1]$
is an irreducible object.
Assume that there is an exact sequence in ${\frak A}$.
\begin{align*}
0 \to E_1 \to {\cal E}_{|X \times \{ x_1 \}}[1] \to E_2 \to 0.
\end{align*}
Then we have an exact sequence
\begin{align*}
0 \to H^{-1}(E_1) \to {\cal E}_{|X \times \{ x_1 \}} \to H^{-1}(E_2)
\to H^0(E_1) \to 0.
\end{align*}
Then we see that 
\begin{align*}
\deg_G(H^{-1}(E_1))=\deg_G(H^{-1}(E_2))=
\deg_G(H^0(E_1))=0.
\end{align*}
Since ${\cal E}_{|X \times \{ x_1 \}}$ is 
$\beta$-twisted stable,
$\chi({\cal E}_{|X \times \{ x_1 \}},H^{-1}(E_1)) \leq 0$.
Hence 
$\chi({\cal E}_{|X \times \{ x_1 \}},H^{-1}(E_2)) \geq 0$.
Since $H^{-1}(E_2) \in {\frak F}$,
$\chi({\cal E}_{|X \times \{ x_1 \}},H^{-1}(E_2))=0$.
Then $H^0(E_1)=0$ and ${\cal E}_{|X \times \{ x_1 \}}$
is isomorphic to $H^{-1}(E_1)$ or $H^{-1}(E_2)$. 
Therefore ${\cal E}_{|X \times \{ x_1 \}}$ is irreducible.
Then $\Phi({\cal E}_{|X \times \{ x_1 \}}[1])[1]={\frak k}_{x_1}$
is an irreducible object of $\Phi[1]({\frak A})$.
In particular, ${\frak k}_{x_1}$ is an stable object
with $\phi({\frak k}_{x_1})=1$.
By \cite[Prop. 10.3]{Br:3}, $\Phi[1]({\frak A})={\frak A}^\mu$.
\end{NB}
The following is a refinement of \cite[Thm. 1.7]{Stability}
and \cite[Prop. 2.7.2 (1)]{PerverseII}.
\begin{NB}
Proposition \ref{II-prop:stability-asymptotic}.
\end{NB}

\begin{prop}\label{prop:FM}
We set
$w:=re^{\beta'}+a \varrho_{X'}+
(d\widehat{H}+\widehat{D}+(d\widehat{H}+\widehat{D},\beta')\varrho_{X'})$.
Assume that $d>N(v)$ (see Lemma~\ref{lem:N}).
Let $E \in {\bf D}(X)$ be a complex with 
$v(\Phi(E)[2])=w$.
Then the following conditions are equivalent.
\begin{enumerate}
\item[(i)]
$E$ is a $\beta$-twisted semi-stable object of ${\frak C}$.
\item[(ii)]
$E[1]$ is a semi-stable object of ${\frak A}$ for $(\beta,\omega)$.
\item[(iii)]
$\Phi(E)[2]$ is a semi-stable object of ${{\frak A}'}^\mu$
 for $(\beta',\widetilde{\omega})$.
\item[(iv)]
$\Phi(E)[2]$ is a $\beta'$-twisted semi-stable object of ${\frak C}'$.
\end{enumerate}
\end{prop}

\begin{proof}
By \eqref{eq:Z-comm}, (ii) is equivalent to (iii).
The equivalence between (i) and (ii) follows from 
Lemma~\ref{lem:star-Phi} (1).
The equivalence between (iii) and (iv) also follows from 
Lemma~\ref{lem:star-Phi} (2).
\end{proof}

\begin{rem}
If $(\omega^2) \ll 2/r_0^2$, then
(ii), (iii) and (iv) are equivalent.
Thus in order to compare the moduli spaces $M_H(\Phi^{-1}(w))$
and $M_{\widehat{H}}(w)$, it is sufficient to
study the wall-crossing behavior in ${\frak A}_{(\beta,\omega)}$.
\end{rem}

\begin{lem}\label{lem:star-Phi}
Assume that $d>N(v)$.
\begin{enumerate}
\item[(1)]
$(\star 3)$ holds for $\Phi^{-1}(w)$ and any $(\omega^2)>0$. 
\item[(2)]
$(\star 1)$ holds for $w$ and any 
$(\widetilde{\omega}^2)>0$. 
\end{enumerate}
\end{lem}

\begin{proof}
We note that
$\Phi^{-1}(w)=r_0 a e^{\beta'}+(r/r_0) \varrho_{X}
-(dH+D+(dH+D,\beta)\varrho_{X})$.
Hence the claims follow from Lemma~\ref{lem:N}.
\end{proof}

\begin{rem}
Assume that ${\frak A}'={{\frak A}'}^\mu$.
If the following two conditions hold, then
the conditions (ii), (iii), (iv) are equivalent to the  
$(-\beta)$-twisted semi-stability of
$E^{\vee}$ in ${\frak C}^D$.
\begin{enumerate}
\item[(1)]
$(\star 2)$ holds for $\Phi^{-1}(w)$ and any $(\omega^2)>0$ 
\item[(2)]
$(\star 1)$ holds for $w$ and any 
$(\widetilde{\omega}^2)>0$. 
\end{enumerate}
\end{rem}

By Lemma~\ref{lem:d_min}, we have the following claim.
\begin{lem}
Assume that ${\frak A}={\frak A}^\mu$.
If $d=d_{\min}$, then
the conditions (ii), (iii), (iv) are equivalent to the  
$(-\beta)$-twisted semi-stability of
$E^{\vee}$ in ${\frak C}^D$.
\end{lem}

\begin{NB}
\begin{lem}
Assume that $d=0$.
If $E \in {\frak A}_{(\beta,\omega)}$ is a stable object,
then 

\end{lem}

\begin{proof}
Since $\deg_G(H^{-1}(E)[1]),\deg_G(H^0(E)) \geq 0$
and $0=\deg_G(H^{-1}(E)[1])+\deg_G(H^0(E))$,
$\deg_G(H^{-1}(E)[1])=\deg_G(H^0(E))=0$.
Hence $H^{-1}(E)$ is a $\mu$-semi-stable object with
$\deg_G(H^{-1}(E))=0$
and $H^0(E)$ is an extension of a 
$\mu$-semi-stable object with
degree 0 by a 0-dimensional object.
Since all of these objects have the phase 1,
$E$ is semi-stable. Then $E$ is stable if and only if 
$E$ is stable. 
If ${\frak A}_{(\beta,\omega)}={\frak A}^\mu$,
then $E$ is stable if and only if $E$ is an irreducible 0-dimensional
object or $E=F[1]$ such that
$F$ is a $\mu$-stable local projective object. 
\end{proof}
\end{NB}


\section{The wall crossing behavior}
\label{sect:wall-crossing}


\subsection{The definition of wall and chamber for stabilities}
\label{subsect:wall:stability}

Let us  fix $H$ and $b:=(H,\beta)$ throughout this section.
Our main object here is to study the wall crossing behavior of 
Bridgeland's stability condition in ${\frak H}_{\Bbb R}$
or $i_\beta({\Bbb R}_{>0} H)$ 
on a $K3$ surface or an abelian surface
from the view of Fourier-Mukai transforms.

We set
\begin{equation*}
v_i:=r_i e^\beta+a_i \varrho_X+(d_i H+D_i+(d_i H+D_i,\beta)\varrho_X),\quad
i=1,2,\ D_i \in H^{\perp}.
\end{equation*}

\begin{lem}\label{lem:v_1,v_2}
Assume that $d_1,d_2>0$. Then
\begin{equation*}
\frac{\langle v_1,v_2 \rangle}{d_1 d_2}
=-\frac{1}{2}((D_1/d_1-D_2/d_2)^2)
+\frac{\langle v_1^2 \rangle}{2 d_1^2}+
\frac{\langle v_2^2 \rangle}{2 d_2^2}+
\left(\frac{r_1}{d_1}-\frac{r_2}{d_2} \right) 
\left(\frac{a_1}{d_1}-\frac{a_2}{d_2} \right).
\end{equation*}
\end{lem}

\begin{proof}
The claim follows by using the following equalities
\begin{equation*}
\begin{split}
\langle (v_1/d_1)^2 \rangle =&
(H^2)+((D_1/d_1)^2)-2(r_1/d_1)(a_1/d_1), \\
\langle (v_2/d_2)^2 \rangle =&
(H^2)+((D_2/d_2)^2)-2(r_2/d_2)(a_2/d_2), \\
\langle v_1/d_1,v_2/d_2 \rangle =&
(H^2)+(D_1/d_1,D_2/d_2)-(r_1/d_1)(a_2/d_2)-(r_2/d_2)(a_1/d_1).
\end{split}
\end{equation*}

\end{proof}

Lemma~\ref{lem:v_1,v_2} implies that
\begin{equation}\label{eq:v_1,v_2}
\langle v_1,v_2 \rangle
=-\frac{((d_2 D_1-d_1 D_2)^2)}{2 d_1 d_2}
+\frac{d_2}{2 d_1}\langle v_1^2 \rangle +
\frac{d_1}{2 d_2} \langle v_2^2 \rangle+
\frac{(d_2 r_1-d_1 r_2)(d_2 a_1-d_1 a_2)}{d_1 d_2}.
\end{equation}
Then we see that
\begin{equation}\label{eq:v_1+v_2}
\frac{\langle (v_1+v_2)^2 \rangle}{d_1+d_2}
= \frac{\langle v_1^2 \rangle}{d_1}
+\frac{\langle v_2^2 \rangle}{d_2}
-\frac{((d_2 D_1-d_1 D_2)^2)}{ d_1 d_2(d_1+d_2)}
+\frac{(d_2 r_1-d_1 r_2)(d_2 a_1-d_1 a_2)}{d_1 d_2(d_1+d_2)}. 
\end{equation}
We also have
\begin{equation*}
\frac{\langle v_1,v_2 \rangle-(D_1,D_2)}{d_1 d_2}
=
\frac{\langle v_1^2\rangle-(D_1^2)}{2 d_1^2}+
\frac{\langle v_2^2 \rangle-(D_2^2)}{2d_2^2} +
\left(\frac{r_1}{d_1}-\frac{r_2}{d_2}\right)
\left(\frac{a_1}{d_1}-\frac{a_2}{d_2}\right)
\end{equation*}
and
\begin{equation}\label{eq:v_1+v_2:D}
\frac{\langle (v_1+v_2)^2 \rangle-(D_1+D_2)^2}{d_1+d_2}
= \frac{\langle v_1^2 \rangle-(D_1^2)}{d_1}
+\frac{\langle v_2^2 \rangle-(D_2^2)}{d_2}
+\frac{(d_2 r_1-d_1 r_2)(d_2 a_1-d_1 a_2)}{d_1 d_2(d_1+d_2)}. 
\end{equation}

Assume that 
\begin{equation}\label{eq:v(eta)}
v=r e^\beta+a \varrho_X+(d H+D+(d H+D,\beta)\varrho_X),\quad d>0
\end{equation}
has a decomposition
\begin{equation}\label{eq:wall:decomp}
v=\sum_{i=1}^s v_i,\quad
\phi(v_i)=\phi(v), 
\end{equation}
where
\begin{equation}\label{eq:v_i(eta)}
v_i=r_i e^\beta+a_i \varrho_X+(d_i H+D_i+(d_i H+D_i,\beta)\varrho_X),
\quad 
d_i>0.
\end{equation}
We note that $d_i/d_{\min} \in {\Bbb Z}$ (Definition~\ref{defn:minimal}).
By using \eqref{eq:v_1+v_2} and \eqref{eq:v_1+v_2:D}, we have
\begin{equation}\label{eq:Bogomolov}
\begin{split}
\sum_{i=1}^s
\left(
 \frac{\langle v_i^2 \rangle}{d_i}+2\frac{d_i}{d_{\min}^2}\varepsilon
\right) 
\leq & \frac{\langle v^2 \rangle}{d}+2\frac{d}{d_{\min}^2}\varepsilon,
\\
\sum_{i=1}^s 
 \left(
  \frac{\langle v_i^2 \rangle-(D_i^2)}{d_i}
  +2\frac{d_i}{d_{\min}^2}\varepsilon
 \right) 
\leq & \frac{\langle v^2 \rangle-(D^2)}
{d}+2\frac{d}{d_{\min}^2}\varepsilon.
\end{split}
\end{equation}

\begin{lem}[Bogomolov inequality]
Let $E$ be a semi-stable object of ${\frak A}_{(\beta,\omega)}$
with the Mukai vector \eqref{eq:Mukai-vector}.
Assume that $d>0$.
Then
\begin{equation}\label{eq:Bogomolov-ineq}
\langle v(E)^2 \rangle \geq -2(d/d_{\min})^2 \varepsilon.
\end{equation}
\end{lem}

\begin{proof}
We may assume that ${\frak k}$ is algebraically closed.
Then for a stable object $E$, 
$\Hom(E,E)={\frak k}$.
Hence the claim holds.
In particular, if $d=d_{\min}$, the claim holds.
For a general case, we take a Jordan-H\"{o}lder filtration of
$E$. Then by \eqref{eq:Bogomolov}, we get the claim.
\end{proof}

%
%
%

\begin{defn}\label{defn:wall:stability}
Let ${\cal C}$ be a chamber for categories, that is,
${\frak A}_{(bH+\eta,\omega)}$ is constant for 
$(\eta,\omega) \in {\cal C} \cap {\frak H}$.
For a Mukai vector $v$, 
we take the expansion \eqref{eq:v(eta)} 
with $\beta=bH+\eta$. 
\begin{enumerate}
\item[(1)]
For $v_1$ in \eqref{eq:v_i(eta)} satisfying
\begin{enumerate}
\item
$0<d_1 <d$,
\item
$\langle v_1^2 \rangle<(d_1/d)\langle v^2 \rangle+
2d d_1 \varepsilon/ d_{\min}^2$,
\item
$\langle v_1^2 \rangle \geq -2d_1^2 \varepsilon/d_{\min}^2$,
\end{enumerate}
we define 
a \emph{wall for stabilities of type $v_1$}
as the set of 
$$
W_{v_1}:=
\left\{(\eta,\omega) \in {\frak H}_{\Bbb R} \mid
(\omega^2)(dr_1-d_1 r)=2(-d \langle e^{bH+\eta},v_1 \rangle
+d_1 \langle e^{bH+\eta},v \rangle) \right\}. 
$$
If it is necessary to emphasize the dependence on $v$,
then we call $W_{v_1}$ by a \emph{wall for $v$}. 
\item[(2)]
A \emph{chamber for stabilities}
is a connected component of 
${\cal C} \setminus \cup_{v_1} W_{v_1}$.  
\end{enumerate}
\end{defn}

\begin{rem}\label{rem:candidate}
\begin{enumerate}
\item[(1)]
The defining equation of the wall $W_{w_1}$ for stabilities corresponds to
the numerical condition for the existence
of a pair $(E,F)$ such that
$E$ is $\sigma_{(\beta,\omega)}$-semi-stable with $v(E)=v$,
$F$ is a subobject of $E$ with $v(F)=w_1$ and $\phi_{(\beta,\omega)}(F)=
\phi_{(\beta,\omega)}(E)$. 
In the case when $X$ is an abelian surface,
one can express the necessary and sufficient condition
for such a pair $(E,F)$ to exist.
We will discuss this topic in \cite[\S\, 4.1]{MYY2}.

\item[(2)]
For a fixed $\beta:=bH+\eta$, we have the injection
$\iota_\beta:{\Bbb R}_{>0}H \to {\frak H}_{\Bbb R}$
(see \eqref{eq:iota}).
Then we have the notion of walls and chambers for
${\Bbb R}_{>0}H$.
In this case, we can require that 
$v_1$ also satisfies
\begin{equation}\label{eq:eta-fixed}
\langle v_1^2 \rangle-(D_1)^2
<\dfrac{d_1}{d}\bigl(\langle v^2 \rangle-(D^2)\bigr)
 +2\frac{d d_1}{d_{\min}^2}\varepsilon.
\end{equation}
Note that the condition 
\eqref{eq:eta-fixed} depends on $\eta$, since
$D_1$ and $a_1$ depend on $\eta$.
We shall prove that the candidates of $(r_1,d_1,a_1,(D_1^2))$ 
are determined by $(H^2), b,\eta,v$. 
We note that $d_{\min} \geq 1/(b_0 (H^2))$, 
where $b_0$ is the denominator of $b$.
By (a), (c) and \eqref{eq:eta-fixed},
we have
\begin{align*}
-2 d^2 b_0^2 (H^2)^2 \leq
d_1^2(H^2)-2r_1 a_1 <d^2(H^2)-2ra+2d^2 b_0^2 (H^2)^2.
\end{align*}
Since $0<d_1<d$ and $a_1 \in (1/r_0)\Bbb Z$,
the choices of $r_1,a_1$ are finite.
Since 
\begin{align*}
-(D_1^2) \leq 2 d^2 b_0^2 (H^2)^2+d_1^2(H^2)-2r_1 a_1,
\end{align*}
$(D_1^2)$ is also bounded.
In particular, 
the candidates of $r_1,d_1,a_1$ and $(D_1^2)$
are determined by $(H^2),b,\eta,v$.   
\end{enumerate}
\end{rem}

\begin{lem}\label{lem:wall:stability-finite}
The set of walls is locally finite.
\end{lem}

\begin{proof}
The claim is a consequence of \cite[Prop. 9.3]{Br:3}.
For a convenience sake of the reader, we give a proof. 
Let $B$ be a compact subset of ${\frak H}_{\Bbb R}$.
We shall prove that 
$$
\{v_1 \mid 
  \text{$v_1$ satisfies (a), (b), (c) and } 
  W_{v_1} \cap B \ne \emptyset \}
$$
is a finite set.
We take $r_0 \in {\Bbb Z}$ with $r_0 e^{bH} \in A^*_{\alg}(X)$.
For $\beta=bH,bH+\eta$, we write
\begin{equation*}
\begin{split}
v=& r e^{bH}+a \varrho_X+
d H+D+(d H+D,bH)\varrho_X\\
= & r e^{bH+\eta}+a' \varrho_X+
d H+D'+(d H+D',bH+\eta)\varrho_X,\\
v_1=& r_1 e^{bH}+a_1 \varrho_X+
d_1 H+D_1+(d_1 H+D_1,bH)\varrho_X\\
=& r_1 e^{bH+\eta}+a_1' \varrho_X+
d_1 H+D_1'+(d_1 H+D_1',bH+\eta)\varrho_X.
\end{split}
\end{equation*}
Then 
\begin{equation}\label{eq:a-a'}
a'=a-(D,\eta)+r\frac{(\eta^2)}{2},\;
a_1'=a_1-(D_1,\eta)+r_1\frac{(\eta^2)}{2}.
\end{equation}
We note that
\begin{equation*}
a-\sqrt{-(D^2)}\sqrt{-(\eta^2)}+r\frac{(\eta^2)}{2}
\leq 
a'
\leq 
a+\sqrt{-(D^2)}\sqrt{-(\eta^2)}+r\frac{(\eta^2)}{2}.
\end{equation*}

\begin{NB}
We shall bound $d_1^2 (H^2)-2r_1 a_1$.
We first note that
$$
d^2(H^2)-2ra' \leq d^2(H^2)-2ra+2r\sqrt{-(\eta^2)}\sqrt{-(D^2)}
-r^2(\eta^2).
$$
We set
\begin{equation*}
\begin{split}
N:= & d^2(H^2)-2ra+2r\sqrt{-(\eta^2)}\sqrt{-(D^2)}
-r^2(\eta^2)\\
=& 
\langle v^2 \rangle+(\sqrt{-(D^2)}+\sqrt{-r^2(\eta^2)})^2.
\end{split}
\end{equation*}
We claim that
$$
\sqrt{-(D_1^2)} \leq
\sqrt{N+4d^2 b_0^2(H^2)^2}+
\sqrt{-r_1^2(\eta^2)}.
$$
By (c) and (d),
we have
\begin{equation*}
\begin{split}
N+2d^2 b_0^2(H^2)^2> & d_1^2(H^2)-2r_1 a_1+2r_1(D_1,\eta)-r_1^2(\eta^2)\\
\geq & -2d^2 b_0^2 (H^2)-(D_1^2)+2r_1(D_1,\eta)-r_1^2(\eta^2)\\
\geq & -2d^2 b_0^2 (H^2)-(D_1^2)-
2|r_1| \sqrt{-(D_1^2)}\sqrt{-(\eta^2)}-r_1^2(\eta^2)\\
= & -2d^2 b_0^2 (H^2)+(\sqrt{-(D_1^2)}-\sqrt{-r_1^2(\eta^2)})^2.
\end{split}
\end{equation*}
Hence the claim holds.
We have
\begin{equation*}
\begin{split}
d_1^2(H^2)-2r_1 a_1
< & N+2d^2 b_0^2(H^2)^2-2r_1(D_1,\eta)+r_1^2(\eta^2) \\
\leq & N+2d^2 b_0^2(H^2)^2+
2|r_1| \sqrt{-(D_1^2)}\sqrt{-(\eta^2)}+r_1^2 (\eta^2),
\end{split}
\end{equation*}
which implies that $d_1^2(H^2)-2r_1 a_1$ is uniformly bound
by $v$, $-(\eta^2)$ and $r_1$.
\end{NB}

Assume that $(\eta,\omega) \in W_{v_1}$, i.e.,
\begin{equation}\label{eq:W_{v_1}}
da_1'-d_1 a'=\frac{(\omega^2)}{2}(r_1 d-rd_1).
\end{equation} 
We shall give a bound of $r_1$ by $v$, $(\omega^2)$ and $(\eta^2)$.
We have $2r_1 a_1<d_1^2(H^2)+2d_1^2/d_{\min}^2$.
If $r_1<0$, then
$a_1> (d_1^2(H^2)+2d_1^2/d_{\min}^2)/2r_1 \geq
 -(d_1^2(H^2)/2+d_1^2/d_{\min}^2)$.
Hence 
\begin{equation*}
dr_1-d_1 r=\dfrac{2(da_1-d_1a)}{(\omega^2)}
>\dfrac{1}{(\omega^2)}
\biggl(-d\Bigl(d_1^2(H^2)+\dfrac{2d_1^2}{d_{\min}^2}\Bigr)-2d_1 a \biggr).
\end{equation*}
Therefore
\begin{equation*}
r_1>
 \dfrac{d_1}{d}r
-\dfrac{1}{(\omega^2)}\Bigl(d_1^2(H^2)+\dfrac{2d_1^2}{d_{\min}^2}\Bigr)
-2\dfrac{d_1}{d(\omega^2)} a.
\end{equation*}
If $r_1>0$, then $a_1<(d_1^2(H^2)+2d_1^2/d_{\min}^2)/2$.
Hence we see that
\begin{equation*}
r_1<\dfrac{d_1}{d}r
+\dfrac{1}{(\omega^2)}\Bigl(d_1^2(H^2)+\dfrac{2d_1^2}{d_{\min}^2}\Bigr)
-2\dfrac{d_1}{d(\omega^2)} a.
\end{equation*}
Since $r_1$ is an integer, the choice of $r_1$ is finite.

We have
\begin{equation*}
r_1 a_1'
=r_1 \left( \dfrac{(\omega^2)}{2}\Bigl(r_1-\dfrac{rd_1}{d}\Bigr)
           +\dfrac{a' d_1}{d} \right)
\geq r_1 \frac{d_1}{d}\left(a'-\frac{(\omega^2)}{2}r \right)
\end{equation*}
by \eqref{eq:W_{v_1}}.
Then by using \eqref{eq:a-a'}, we get 
\begin{equation*}
r_1 a_1 
\geq r_1 \frac{d_1}{d}\Bigl(a'-\frac{(\omega^2)}{2}r \Bigr)
    +r_1(D_1,\eta)-r_1^2\frac{(\eta^2)}{2}.
\end{equation*}
We have
\begin{equation*}
\begin{split}
d_1^2(H^2)+2\frac{dd_1}{d_{\min}^2} 
&\geq 2r_1 a_1-(D_1^2)
\\
&= 2r_1 a_1'+2r_1(D_1,\eta)-r_1^2(\eta^2)-(D_1^2)
\\
&\geq 2r_1 a_1'+\Bigl(\sqrt{-(D_1^2)}-\sqrt{-r_1^2(\eta^2)}\Bigr)^2
\\
&\geq 2r_1 \frac{d_1}{d}\Bigl(a'-\frac{(\omega^2)}{2}r \Bigr)
+\Bigl(\sqrt{-(D_1^2)}-\sqrt{-r_1^2(\eta^2)}\Bigr)^2.
\end{split} 
\end{equation*}
Hence $-(D_1^2)$ is bounded.
Since $D_1=(c_1(v_1)-r_1 bH)-(c_1(v_1)-r_1 bH,H)H/(H^2) \in
(1/b_0(H^2))\NS(X)$,
the choice of $D_1$ is finite.
Then $r_1 a_1$ is also bounded.
Since $r_0 a_1 \in {\Bbb Z}$,
the choice of $v_1$ is finite.
\end{proof}

By the same arguments in \cite{chamber}, we get the following claim.
\begin{lem}
If $(\eta,\omega), (\eta',\omega')$ 
belong to the same chamber, then
${\cal M}_{(bH+\eta,\omega)}(v)=
{\cal M}_{(bH+\eta',\omega')}(v)$.
\end{lem}

Assume that $d>0$.
Then $\langle e^{bH+\eta+\sqrt{-1}\omega},v \rangle \not \in {\Bbb R}$.
Since $\langle e^{bH+\eta+\sqrt{-1}\omega},v_1/d_1-v/d \rangle 
\in {\Bbb R}$,
$\phi(v)=\phi(v_1)$ if and only if
$\langle e^{bH+\eta+\sqrt{-1}\omega},v_1/d_1-v/d \rangle=0$.
Since
$$
\langle e^{bH+\eta+\sqrt{-1}\omega},v_1/d_1-v/d \rangle=
-\left(\dfrac{r_1}{d_1}-\dfrac{r}{d} \right) \dfrac{(\eta^2)-(\omega^2)}{2}
-\left(\dfrac{a_1}{d_1}-\dfrac{a}{d} \right)
+\left(\dfrac{D_1}{d_1}-\dfrac{D}{d},\eta \right),
$$
${\Bbb Q}(v_1/d_1-v/d)$ is determined by 
the wall $W_{v_1}$.

If $dr_1-d_1 r \ne 0$, then
$W_{v_1}$ is a half sphere in ${\frak H}_{\Bbb R}$:
$$
(\omega^2)-\left(\eta-\frac{dD_1-d_1 D}{dr_1-d_1 r} \right)^2
=2\frac{da_1-d_1 a}{dr_1-d_1 r}-
\left(\frac{dD_1-d_1 D}{dr_1-d_1 r}\right)^2.
$$
If $dr_1-d_1 r = 0$, then
$W_{v_1}$ is a hyperplane in ${\frak H}_{\Bbb R}$:
$$
(\eta,dD_1-d_1 D)=(da_1-d_1 a).
$$
In this case, $W_{v_1}$ is the wall for $\eta$-twisted semi-stability.

We can easily prove the
following lemma, which will be used later.
 
\begin{lem}\label{lem:xi_v}
$\langle v_1-d_1 v/d,e^{\beta+\sqrt{-1}\omega} \rangle=0$
if and only if
$\langle v_1,\xi_v \rangle=0$,
where 
\begin{equation*}
\begin{split}
\xi_v
&:= \mathrm{Re}
    \biggl(e^{\beta+\sqrt{-1}\omega}-
          \dfrac{\langle v/d,e^{\beta+\sqrt{-1}\omega} \rangle}{(H^2)}
          (H+(H,\beta)\varrho_X)\biggr)
\\
&= e^\beta-\frac{(\omega^2)}{2}\varrho_X-
\frac{\langle v/d,e^\beta-\frac{(\omega^2)}{2}\varrho_X \rangle}
{(H^2)}
(H+(H,\beta)\varrho_X).
\end{split}
\end{equation*}
\end{lem}

\begin{rem}
We give a few remarks on the paper \cite{AB},
where another stability function is defined:
\begin{equation*}
\begin{split}
Z_{(\beta,\omega)}'(E)
:=&\langle e^{\beta+\sqrt{-1}\omega},\ch(E) \rangle 
\\
 =&\langle e^{\beta+\sqrt{-1}\omega},v(E)(1-\varepsilon \varrho_X) \rangle
 = Z_{(\beta,\omega)}(E)+\varepsilon r
\\
 =&-a+r\frac{(\omega^2)+2\varepsilon}{2}+\sqrt{-1} d(H,\omega)
\end{split}
\end{equation*}
with $v(E)=r e^{\beta}+ a \varrho_X+(d H+D+(dH+D,\beta)\varrho_X)$, 
$D\in H^{\perp}$.
Therefore if $X$ is a $K3$ surface, 
then $Z_{(\beta,\omega)}'$ is a stability function for any $(\omega^2)>0$ 
and there is no wall for categories.

Assume that $X$ is a $K3$ surface with $\NS(X)={\Bbb Z}H$.
In this case the wall for stabilities $W'_{v_1}$ associated 
with $v_1=r_1 e^{\beta H}+ a_1 \varrho_X+(d_1 H+(d_1 H,\beta)\varrho_X)$ 
is given by 
$$
W'_{v_1}:=
 \left\{\omega \in {\Bbb R}_{>0} H \, \bigg| \,
        d_1\left(-a  +r  \dfrac{(\omega^2)+2}{2}\right)
       =d  \left(-a_1+r_1\dfrac{(\omega^2)+2}{2}\right) 
 \right\}. 
$$ 
Following \cite{AB}, 
let us consider the case $\beta=H/2$ and $v=H+((H^2)/2)\varrho_X$.
Note that $r=1$, $d=1$, $a=0$ and $d_{\min}=1/2$,
so that we choose $d_1=1/2$.
Then $\omega \in W_{v_1}$ if and only if $(\omega^2)=2a_1/r_1-2$.
For the case $r_1=1$, we have $v_1=e^H+(a_1-(H^2)/8)\varrho_X$. 
Putting $n_1:=(H^2)/8-a_1$, we have 
$(\omega^2)=(H^2)/4-2(n_1+1)$. 
Since $n_1\ge -1$, we recover the walls given in \cite{AB}.
\end{rem}

\vspace{1pc}

{\it The relative cases.}
Let ${\cal Y} \to S$ be a polarized family of normal $K3$ surfaces
or abelian surfaces over an integral scheme $S$. 
Assume that there is a smooth family of
polarized surfaces ${\cal X} \to S$ with a family of contractions
$\pi:{\cal X} \to {\cal Y}$ over $S$
such that ${\bf R}\pi_*({\cal O}_{\cal X})={\cal O}_{\cal Y}$.
Assume that there is a locally free sheaf ${\cal G}$ on ${\cal X}$
which defines a family of tiltings ${\frak C}_s$, $s \in S$
and ${\cal G}_s$ is a local projective generator
of ${\frak C}_s$.
Let ${\cal H}$ be a relative ${\Bbb Q}$-Cartier divisor on  
${\cal X}$ which is the pull-back of a relatively ample ${\Bbb Q}$-divisor
on ${\cal Y}$.
We assume that there is a section $\sigma$ of $f:{\cal X} \to S$.
Then $\sigma$ gives a family of fundamental classes
$\varrho_{{\cal X}_s}$, $s \in S$.
We denote it by $\varrho_{\cal X}$.
We take $\beta \in \NS({\cal X}/S)_{\Bbb Q}$.
Let $v \in {\Bbb Z} \oplus \NS({\cal X}/S) \oplus 
{\Bbb Z} \varrho_{\cal X}$ 
be a family of Mukai vectors. 
For $v_1 \in A^*_{\alg}({\cal X}_s)$, we write
$$
v_1=r_1 e^\beta+a_1 \varrho_{{\cal X}_s}+
d_1 {\cal H}_s+D_1+(d_1 {\cal H}_s+D_1,\beta)\varrho_{{\cal X}_s},
\quad
D_1 \in {\cal H}_s^{\perp}.
$$
If $v_1$ satisfies 
Definition~\ref{defn:wall:stability}~(1),
then Remark~\ref{rem:candidate}~(2)
implies that
the candidates of $r_1,d_1,a_1$ and $(D_1^2)$
are finite.
Assume that $r_0 \beta \in \NS({\cal X}_s)$.
Then $\xi:=r_0(d_1 {\cal H}_s+D_1) \in \NS({\cal X}_s)$
satisfies $(\xi,{\cal H}_s)=r_0 d_1({\cal H}_s^2)$ and
$(\xi^2)=r_0^2 (d_1^2(H^2)+(D_1^2))$.
Since 
$$
\chi({\cal O}_{{\cal X}_s}(\xi+n{\cal H}_s))=
n^2\frac{({\cal H}_s^2)}{2}+nr_0 d_1 ({\cal H}_s^2)+
\frac{r_0^2(d_1^2({\cal H}_s^2)+ (D_1^2))}{2}+
\chi({\cal O}_{{\cal X}_s}), 
$$
the set of Hilbert polynomials 
$\chi({\cal O}_{{\cal X}_s}(\xi+n{\cal H}_s))$ of ${\cal O}_{{\cal X}_s}(\xi)$
is finite.
Since the relative Picard scheme of a fixed Hilbert polynomial
is of finite type,
the equivalence class of $\xi$ is also finite, where
$\xi \in \NS({\cal X}_s)$ and $\xi' \in \NS({\cal X}_{s'})$
is equivalent if $\xi$ and $\xi'$ belong to 
the same connected component of the relative Picard scheme.

Thus we get the following lemma.
\begin{lem}\label{lem:extend}
There is a dominant morphism $S' \to S$ 
such that for any point $s \in S'$,
$v_1 \in A^*_{\alg}({\cal X}_s) $ in
Definition~\ref{defn:wall:stability}~(1)
extends to a family of Mukai vectors
$\widetilde{v}_1 \in {\Bbb Z} \oplus \NS({\cal X}'/S') 
\oplus {\Bbb Z} \varrho_{{\cal X}'}$, 
where ${\cal X}':={\cal X} \times_S S'$.  
\end{lem} 

\begin{proof}

For each $\xi$, we take $\Pic_{{\cal X}/S}^{\xi} \to S$.
For a suitable covering $S_{\xi} \to \Pic_{{\cal X}/S}^{\xi}$,
we have a family of line bundles
${\cal L}_{\xi}$ on ${\cal X} \times_S S_{\xi}$
with $({\cal L}_{\xi})_{k(s)} \in \Pic_{{\cal X}/S}^{\xi}$
for $s \in S_{\xi}$.
Since there are finitely many equivalence classes of $\xi$
and $\Pic_{{\cal X}/S}^{\xi} \to S$ is proper,
replacing $S$ by an open subset,
we may assume that
all $S_{\xi} \to S$ are surjective.
Let $\xi_1,\xi_2,...,\xi_n$ be the representative of $\xi$
and set $S':=S_{\xi_1} \times_S S_{\xi_2} \times \cdots \times_S S_{\xi_n}$.
Then we have a surjective
morphism 
$S' \to S$.
\end{proof}


\subsection{The wall crossing behavior under the change of categories}
\label{subsect:wall-crossing:category}

In this subsection, we assume that $X$ is a $K3$ surface and
we fix $\beta=bH+\eta$. 
Assume that $\omega \in {\Bbb R}_{>0} H$ belongs to
a wall $W:=W_{v(E)}$, where $E \in \exc_\beta$.  
Let $\omega_\pm \in {\Bbb Q}_{>0}H$ be ample ${\Bbb Q}$-divisors
which are sufficiently close to $\omega$ and
$(\omega_-^2)<(\omega^2)<(\omega_+^2)$.
We shall study the wall crossing behavior of the
moduli stacks ${\cal M}_{(\beta,\omega_\pm)}(v)$.
Let $\phi_\pm:=\phi_{(\beta,\omega_\pm)}$ 
be the phase function for $Z_{(\beta,\omega_\pm)}$,
and 
$\sigma_{(\beta,\omega_\pm)}
:=({\frak A}_{(\beta,\omega_{\pm})},Z_{(\beta,\omega_\pm)})$.
We also set
\begin{align*}
v=re^\beta+a \varrho_X+(dH+D+(dH+D,\beta)\varrho_X),\quad 
 D \in H^\perp.
\end{align*}

First of all, we shall slightly generalize the definition of
the stability.
\begin{defn}
$E \in {\frak A}_{(\beta,\omega_-)}$ is 
$\sigma_{(\beta,\omega)}$-semi-stable,
if
\begin{equation*}
\Sigma_{(\beta,\omega)}(E',E)
\geq 0
\end{equation*}
for any subobject $E'$ of $E$ (cf. Definition~\ref{defn:area}).
${\cal M}_{(\beta,\omega)}(v)$ denotes the 
moduli stack of $\sigma_{(\beta,\omega)}$-semi-stable
object $E$ such that $v(E)=v$.
\end{defn}

\begin{NB}
\begin{rem}
For $E \in {\frak S}_W$, $\chi_G(E)=0$, where
$v(G)=r_0(e^\beta-\frac{(\omega^2)}{2}\varrho_X)$.
Hence ${\frak A}_{(\beta,\omega)}={\frak A}_{(\beta,\omega_+)}$.
\end{rem}
\end{NB}

We study the categories of complexes $E \in {\bf D}(X)$
such that $E \in {\frak A}_{(\beta,\omega_-)}$ 
and is $\sigma_{(\beta,\omega)}$-semi-stable.

\begin{NB}
We assume that $d>0$.
For any semi-stable subobject $E'$ of $E$ with $\phi_-(E')=1$,
if $Z_{(\beta,\omega)}(E') \ne 0$, then $\Sigma_{(\beta,\omega)}(E',E) <0$.
Hence $Z_{(\beta,\omega)}(E')= 0$, 
that is, $E' \in {\frak S}_W$. 
\end{NB}

\begin{lem}\label{lem:wall:limit}
\begin{enumerate}
\item[(1)]
For an object $E$ of ${\frak A}_{(\beta,\omega_+)}$,
if $\Hom(E_0[1],E)=0$ for all $E_0 \in {\frak S}_W$, 
then $E \in {\frak A}_{(\beta,\omega_-)}$.
In particular, for a semi-stable object $E$ of 
${\frak A}_{(\beta,\omega_+)}$
with $\phi_+(E)<1$, $E \in {\frak A}_{(\beta,\omega_-)}$.
\item[(2)]
For an object $E$ of ${\frak A}_{(\beta,\omega_-)}$,
if $\Hom(E, E_0)=0$ for all $E_0 \in {\frak S}_W$, 
then $E \in {\frak A}_{(\beta,\omega_+)}$.
\item[(3)]
Assume that $d>0$. Then
${\cal M}_{(\beta,\omega_\pm)}(v) \subset {\cal M}_{(\beta,\omega)}(v)$. 
\end{enumerate}
\end{lem}

\begin{proof}
(1)
Assume that
$\Hom(E_0[1],E)=\Hom(E_0,H^{-1}(E))=0$ for all $E_0 \in {\frak S}_W$. 
Then $H^{-1}(E) \in {\frak F}_{(\beta,\omega_-)}$, which implies that
$E \in {\frak A}_{(\beta,\omega_-)}$.
Since $\phi_+(E)<1$ and $\phi_+(E_0[1])=1$ for $E_0 \in {\frak S}_W$, 
$\Hom(E_0[1],E)=0$. Hence $E \in {\frak A}_{(\beta,\omega_-)}$.

(2)
Assume that 
$\Hom(E,E_0)=0$ for all $E_0 \in {\frak S}_W$. 
By Lemma~\ref{lem:exc}~(3),
we have $H^0(E) \in {\frak T}_{(\beta,\omega_+)}$, which 
implies that
$E \in {\frak A}_{(\beta,\omega_+)}$.

(3)
For an object $E$ of ${\cal M}_{(\beta,\omega_+)}(v)$,
$d>0$ implies that $\phi_+(E)<1$. Then
(1) implies $E \in {\frak A}_{(\beta,\omega_-)}$.
Let $F$ be a subobject of $E$ in ${\frak A}_{(\beta,\omega_-)}$.
Then there is an exact sequence in ${\frak A}_{(\beta,\omega_-)}$
$$
0 \to F_1 \to F \to F_2 \to 0
$$  
such that $H^{-1}(F_1)=H^{-1}(F) 
\in {\frak F}_{(\beta,\omega_-)}$,
$H^0(F_1) \in {\frak T}_{(\beta,\omega_+)}$ and
$F_2 \in {\frak S}_W$.
Then $F_1$ is a subobject of $E$ in 
${\frak A}_{(\beta,\omega_-)}$ such that
the exact sequence in ${\frak A}_{(\beta,\omega_-)}$
$$
0 \to F_1 \to E \to E/F_1 \to 0
$$
is an exact sequence in 
${\frak A}_{(\beta,\omega_+)}$.
Indeed $E \in {\frak A}_{(\beta,\omega_+)}$ and 
$\Hom(E/F_1,E_0) \subset \Hom(E,E_0)=0$ for all 
$E_0 \in {\frak S}_W$ implies 
$E/F_1 \in {\frak A}_{(\beta,\omega_+)}$ by (2).
Since $Z_{(\beta,\omega)}(F)=Z_{(\beta,\omega)}(F_1)$
and $\omega_+$ is sufficiently close to $\omega$,
we have $\Sigma_{(\beta,\omega)}(F,E)=\Sigma_{(\beta,\omega)}(F_1,E) \geq 0$.
Therefore $E \in {\cal M}_{(\beta,\omega)}(v)$.

Since the $\sigma_{(\beta,\omega)}$-semi-stability 
is defined in the category ${\frak A}_{(\beta,\omega_-)}$,
${\cal M}_{(\beta,\omega_-)}(v) \subset {\cal M}_{(\beta,\omega)}(v)$
is obvious.
\end{proof}

\begin{lem}\label{lem:deg=0}
\begin{enumerate}
\item[(1)]
If $E \in {\frak A}_{(\beta,\omega_-)}$ satisfies $Z_{(\beta,\omega)}(E)=0$,
then $E \in {\frak S}_W$.
\item[(2)]
If $E \in {\frak A}_{(\beta,\omega_+)}$ satisfies $Z_{(\beta,\omega)}(E)=0$,
then $E \in {\frak S}_W[1]$.
\end{enumerate}
\end{lem}

\begin{proof}
We set $G:=G_{(\beta,\omega)}, G_\pm:=G_{(\beta,\omega_\pm)} 
\in K(X)_{\Bbb Q}$. 
For $E \in {\frak A}_{(\beta,\omega_\pm)}$,
$Z_{(\beta,\omega)}(E)=0$ implies that
$H^{-1}(E)$ and $H^0(E)$ are $\beta$-twisted semi-stable objects
of ${\frak C}$ with
$\deg_G(H^{-1}(E))=\chi_G(H^{-1}(E))=0$ and
$\deg_G(H^0(E))=\chi_G(H^0(E))=0$.
In particular, $H^0(E)$ and $H^{-1}(E)$ are torsion free. 

(1)
If $H^{-1}(E) \ne 0$, then
$\chi_{G_-}(H^{-1}(E))>\chi_G(H^{-1}(E))$.
Hence $E \in {\frak A}_{(\beta,\omega_-)}$ implies that
$H^{-1}(E)=0$. Therefore $E=H^0(E) \in {\frak S}_W$. 

(2)
If $H^0(E) \ne 0$, then $\rk H^0(E)>0$ implies that
$\chi_{G_+}(H^0(E))<\chi_G(H^0(E))$.
Hence $E \in {\frak A}_{(\beta,\omega_+)}$ implies that
$H^0(E)= 0$. Therefore $E=H^{-1}(E)[1] \in {\frak S}_W[1]$.
\end{proof}

\begin{cor}\label{cor:deg=0} 
\begin{enumerate}
\item[(1)]
If $E \in {\frak A}_{(\beta,\omega_-)}$ satisfies
$\phi_-(E)=1$ and 
$Z_{(\beta,\omega)}(E) \in {\Bbb R}_{\geq 0}e^{\pi \sqrt{-1} \phi}$,
$0<\phi<1$,
then $E \in {\frak S}_W$.
\item[(2)]
If $E \in {\frak A}_{(\beta,\omega_+)}$ satisfies
 $\phi_+(E)=1$ and 
$Z_{(\beta,\omega)}(E) \in {\Bbb R}_{\geq 0}e^{\pi \sqrt{-1} \phi}$,
$0<\phi<1$,
then $E \in {\frak S}_W[1]$.
\end{enumerate}
\end{cor}

\begin{proof}
We note that $Z_{(\beta,\omega)}(E) \in {\Bbb R}_{\leq 0}$.
Then $Z_{(\beta,\omega)}(E) \in {\Bbb R}_{\geq 0}e^{\pi \sqrt{-1} \phi}$
implies that $Z_{(\beta,\omega)}(E)=0$.
Hence the claims follow from Lemma~\ref{lem:deg=0}. 
\end{proof}

\begin{prop}\label{prop:wall:classification}
Assume that $d>0$.
\begin{enumerate}
\item[(1)]
$E \in {\cal M}_{(\beta,\omega)}(v)$ if and only if 
there is a filtration
$$
0=F_0 \subset F_1 \subset F_2 \subset \cdots \subset F_s=E 
$$
in ${\frak A}_{(\beta,\omega_-)}$
such that
\begin{enumerate}
\item
$Z_{(\beta,\omega)}(F_i/F_{i-1}) \in 
{\Bbb R}_{\geq 0}Z_{(\beta,\omega)}(E)$,

\item
$F_i/F_{i-1}$ are semi-stable with respect to
$Z_{(\beta,\omega_-)}$,

\item
\begin{align*}
1 \geq \phi_-(F_1/F_0) >\phi_-(F_2/F_1)>\cdots>\phi_-(F_s/F_{s-1})>0.
\end{align*}
\end{enumerate}

\item[(2)]
$E \in {\cal M}_{(\beta,\omega)}(v)$ if and only if 
there is a filtration
\begin{align*}
0=F_0 \subset F_1 \subset F_2 \subset \cdots \subset F_s=E 
\end{align*}
in ${\frak A}_{(\beta,\omega_-)}$
such that
\begin{enumerate}
\item
$Z_{(\beta,\omega)}(F_i/F_{i-1}) \in 
{\Bbb R}_{\geq 0}Z_{(\beta,\omega)}(E)$,

\item
$F_i/F_{i-1}$ are semi-stable with respect to
$Z_{(\beta,\omega_+)}$ or $F_i/F_{i-1} \in {\frak S}_W$,

\item
\begin{align*}
1> \phi_+(F_1/F_0) >\phi_+(F_2/F_1)>\cdots>\phi_+(F_s/F_{s-1}) \geq 0.
\end{align*}
\end{enumerate}
\end{enumerate}
\end{prop}

\begin{proof}
(1)
If $E \in {\cal M}_{(\beta,\omega)}(v)$ 
is not a $\sigma_{(\beta,\omega_-)}$-semi-stable object of
${\frak A}_{(\beta,\omega_-)}$, 
then we have the Harder-Narasimhan filtration
\begin{align*}
0=F_0 \subset F_1 \subset F_2 \subset \cdots \subset F_s=E
\end{align*}
such that 
$1 \geq \phi_-(F_1/F_0)>\phi_-(F_2/F_1)>\cdots > \phi_-(F_s/F_{s-1})>0$.
Since $\omega_-$ is sufficiently close to $\omega$,
we have $Z_{(\beta,\omega)}(F_i/F_{i-1}) \in 
{\Bbb R}_{ \geq 0}Z_{(\beta,\omega)}(E)$ for all $i$.
Conversely if $E$ has a filtration with (a), (b), (c),
then $F_i/F_{i-1}$ with $\phi_-(F_i/F_{i-1})<1$
are semi-stable with respect to $Z_{(\beta,\omega)}$
by Lemma~\ref{lem:wall:limit}~(3).  
If $\phi_-(F_i/F_{i-1})=1$, then Corollary~\ref{cor:deg=0}~(1)
implies that $F_i/F_{i-1} \in {\frak S}_W$.
Hence $E \in {\cal M}_{(\beta,\omega)}(v)$.
 
(2)
If $E \in {\cal M}_{(\beta,\omega)}(v)$ 
is not semi-stable with respect to $\sigma_{(\beta,\omega_+)}$,
then we have an exact sequence 
\begin{equation*}
0 \to E' \to E \to E_0 \to 0
\end{equation*}  
in ${\frak A}_{(\beta,\omega_-)}$,
where $E_0 \in {\frak S}_W$ and
$\Hom(E',F)=0$ for all $F \in {\frak S}_W$.
Thus $E' \in {\frak A}_{(\beta,\omega_+)}$
by Lemma~\ref{lem:wall:limit}~(2).
We take the Harder-Narasimhan filtration of $E'$
in ${\frak A}_{(\beta,\omega_+)}$:
\begin{equation}\label{eq:E'}
0=F_0 \subset F_1 \subset F_2 \subset \cdots \subset F_t=E'.
\end{equation}
Since $\omega_+$ is sufficiently close to $\omega$,
$Z_{(\beta,\omega)}(F_i/F_{i-1}) \in 
{\Bbb R}_{\geq 0}Z_{(\beta,\omega)}(E)$.
If $\phi_+(F_1)=1$, then Corollary~\ref{cor:deg=0}~(2) implies that
$F_1 \in {\frak S}_W[1]$.
By $E \in {\frak A}_{(\beta,\omega_-)}$,
we have $\phi_+(F_1)<1$.
Thus 
\begin{equation*}
1>\phi_+(F_1/F_0)>\phi_+(F_2/F_1)>\cdots >
\phi_+(F_t/F_{t-1})>\phi_+(E_0)=0.
\end{equation*}
By Lemma~\ref{lem:wall:limit} (1), (3),
$F_i/F_{i-1} \in {\frak A}_{(\beta,\omega_-)}$ and 
they are semi-stable with respect to $Z_{(\beta,\omega)}$.
In particular, 
\eqref{eq:E'} is a filtration in ${\frak A}_{(\beta,\omega_-)}$. 
We set $F_s:=E$, where 
$s:=t+1$ for $E_0 \ne 0$ and $s:=t$ for $E_0=0$.
Then we get a desired filtration of $E$. 
Conversely for a filtration with (a), (b), (c),
$F_i/F_{i-1}$ are semi-stable with respect to
$Z_{(\beta,\omega)}$ by Lemma~\ref{lem:wall:limit}~(3)
or $F_i/F_{i-1} \in {\frak S}_W$.
Hence $E \in {\cal M}_{(\beta,\omega)}(v)$.
\end{proof}

\begin{NB}
Assume that $\omega$ belongs to exactly one wall 
$W_{E_0}$.
We study the categories of complexes $E \in {\bf D}(X)$
such that $E \in {\frak A}_{(\beta,\omega_-)}$ and 
$E$ is semi-stable with respect to $Z_{(\beta,\omega)}$.
We assume that $d>0$.
For any semi-stable subobject $E'$ of $E$ with $\phi_-(E')=1$,
if $Z_{(\beta,\omega)}(E') \ne 0$, then $\Sigma_{(\beta,\omega)}(E',E) <0$.
Hence $Z_{(\beta,\omega)}(E')= 0$, 
that is,
$E'=E_0^{\oplus n}$. 

\begin{lem}
For a semi-stable object $E$ of ${\frak A}_{(\beta,\omega_+)}$,
if $\phi_+(E)<1$, then $E \in {\frak A}_{(\beta,\omega_-)}$.
\end{lem}

\begin{proof}
Since $\phi_+(E)<1$ and $\phi_+(E_0[1])=1$, 
$\Hom(E_0[1],E)=0$. Hence $E \in {\frak A}_{(\beta,\omega_-)}$.
\begin{NB2}
$\Hom(E_0,H^{-1}(F_i/F_{i-1}))=0$.
\end{NB2}
\end{proof}

If $E$ is not a semi-stable object of
${\frak A}_{(\beta,\omega_-)}$ with respect to
$Z_{(\beta,\omega)}$, then we have an exact sequence 
\begin{equation*}
0 \to E_0^{\oplus n} \to E \to E' \to 0
\end{equation*}  
such that for the Harder-Narasimhan filtration
$$
0 \subset F_1 \subset F_2 \subset \cdots \subset F_s=E'
$$
of $E'$ in ${\frak A}_{(\beta,\omega_-)}$, 
we have $\Hom(E_0,F_i/F_{i-1})=0$.
Thus $1=\phi_-(E_0^{\oplus n})>\phi_-(F_1/F_0)>\phi_-(F_2/F_1)>\cdots >
\phi_-(F_s/F_{s-1})>0$.
Moreover $Z_{(\beta,\omega)}(F_i/F_{i-1}) \in 
{\Bbb R}_{>0}Z_{(\beta,\omega)}(E)$ for all $i$.

If $E$ is not semi-stable with respect to $\sigma_{(\beta,\omega_+)}$,
then we have an exact sequence 
\begin{equation*}
0 \to E' \to E \to E_0^{\oplus n} \to 0
\end{equation*}  
in ${\frak A}_{(\beta,\omega_-)}$,
where $\Hom(E',E_0)=0$ and $n \geq 0$.
Thus $E' \in {\frak A}_{(\beta,\omega_+)}$.
We take the Harder-Narasimhan filtration of $E'$
in ${\frak A}_{(\beta,\omega_+)}$:
$$
0 \subset F_1 \subset F_2 \subset \cdots \subset F_s=E'.
$$
Since $E \in {\frak A}_{(\beta,\omega_-)}$,
$\Hom(E_0[1],E)=0$, which implies that
$\Hom(E_0[1],F_1/F_0)=0$.
Since $F_i/F_{i-1} \in {\frak A}_{(\beta,\omega_+)}$,
we have $\Hom(F_i/F_{i-1},E_0)=0$.
Thus $1>\phi_+(F_1/F_0)>\phi_+(F_2/F_1)>\cdots >
\phi_+(F_s/F_{s-1})>\phi_+(E_0^{\oplus n})=0$. 

Conversely for an object $E'$ of ${\frak A}_{(\beta,\omega_+)}$
with the filtration,
$\phi_+(F_i/F_{i-1})<1$ imply that
$F_i/F_{i-1} \in {\frak A}_{(\beta,\omega_-)}$.
Therefore $E' \in {\frak A}_{(\beta,\omega_-)}$.
\end{NB}

By the following lemma, the choice of the Mukai vectors
$v(F_i/F_{i-1})$ in Proposition~\ref{prop:wall:classification}
is finite.

\begin{lem}\label{lem:wall:category:finite}
Let $E$ be 
a $\sigma_{(\beta,\omega)}$-semi-stable object.
Assume that $E$ is $S$-equivalent to
$\oplus_{i=0}^s E_i$ such that
\begin{enumerate}
\item[(i)] 
$Z_{(\beta,\omega)}(E_0)=0$ 

\item[(ii)] 
$E_i$ are semi-stable with respect to
$\sigma_{(\beta,\omega_+)}$ for all $i>0$
or $E_i$ are semi-stable with respect to
$\sigma_{(\beta,\omega_-)}$ for all $i>0$.
\end{enumerate}
Then the choice of $v(E_0)$ is finite.
\end{lem}

\begin{proof}
We set $v(E_0):=r_0 e^\beta+a_0 \varrho_X+D_0+(D_0,\beta)\varrho_X$.
By $Z_{(\beta,\omega)}(E_0)=0$, we have
$a_0=r_0 (\omega^2)/2$.
Then we have
\begin{equation*}
\begin{split}
-2 \left(\dfrac{d}{d_{\min}}\right)^2 
&\leq \bigl\langle v(\oplus_{i=1}^s E_i)^2 \bigr\rangle -(D-D_0)^2
\\
&= -2(r-r_0)(a-a_0)+d^2(H^2)
\\
&= -(\omega^2)
    \left(r_0-\frac{1}{2}\Bigl(r+\frac{2}{(\omega^2)}a\Bigr)\right)^2
   +\dfrac{(\omega^2)}{4}
    \left(r+\frac{2}{(\omega^2)}a \right)^2+
    \langle v(E)^2 \rangle-(D^2).
\end{split}
\end{equation*}
Hence the choice of $r_0$ is finite.
Since $E_0$ is a successive extension of objects in $\exc_\beta$,
the choice of $E_0$ is also finite. 
\end{proof}


\subsection{The wall crossing behavior for $d=d_{\min}$}
\label{subsect:wall-crossing:dmin}

As an example of the wall crossing behavior, we shall study
${\cal M}_{(\beta,\omega)}(v)$ for $d=d_{\min}$.
Throughout section \ref{subsect:wall-crossing:dmin},
we assume that ${\frak k}$ is an algebraically closed field.
In this case, there is no wall for stabilities and 
the filtration in Proposition~\ref{prop:wall:classification}
is of length $s=2$.
Assume that $\omega$ belongs to a wall $W$ for categories.
In order to study the contribution of ${\frak S}_W$,
we set
\begin{align*}
R_+:=\{v(E) \mid  E \in {\frak S}_W, \langle v(E)^2 \rangle=-2 \}. 
\end{align*}
\begin{NB}
If the moduli spaces are non-empty for any algebraically closed 
field, then we need to add the following sentence:

Replacing ${\frak k}$ by a finite extension,
we assume that each $u_i \in R_+$ are defined over ${\frak k}$.
\end{NB}
By Lemma~\ref{lem:ADE}, $R_+$ is a finite set.
For a sufficiently small general element $\eta \in \NS(X)_{\Bbb Q}$,
we have 
\begin{equation*}
\langle u/\rk u-u'/\rk u', \eta+(\eta,\beta)\varrho_X \rangle \ne 0
\end{equation*}
for all $u,u' \in R_+$ with $u \ne u'$.
We may assume that $R_+=\{u_1,u_2,\ldots,u_n \}$ with 
\begin{equation}\label{eq:eta}
 -\langle u_i/\rk u_i,\eta+(\eta,\beta)\varrho_X \rangle 
<-\langle u_j/\rk u_j,\eta+(\eta,\beta)\varrho_X \rangle 
\end{equation}
for $i<j$.

\begin{lem}\label{lem:U_i}
\begin{enumerate}
\item[(1)]
There is a unique $(\beta+\eta)$-twisted semi-stable object
$U_i$ of ${\frak C}$ with $v(U_i)=u_i$.
\item[(2)]
\begin{equation}\label{eq:order}
\Hom(U_j,U_i)=0,\quad j>i.
\end{equation}
\end{enumerate}
\end{lem}

\begin{proof}
(1)
By Proposition \ref{prop:mod-p},
\begin{NB}
\cite[Cor. \ref{II-cor:reduction}]{PerverseII} 
\end{NB}
there is a $(\beta+\eta)$-twisted semi-stable object $U_i$
of ${\frak C}$ with $v(U_i)=u_i$.
If $U_i$ contains a $(\beta+\eta)$-twisted stable subobject $U_i'$ such that
$U_i' \in {\frak S}_W$ and 
\begin{align*}
\langle v(U_i'),\eta+(\eta,\beta)\varrho_X \rangle/\rk U_i'=
\langle v(U_i),\eta+(\eta,\beta)\varrho_X \rangle/\rk U_i,
\end{align*}
then $v(U_i') \in R_+$. By our choice of $\eta$,
$U_i'=U_i$. Thus $U_i$ is $(\beta+\eta)$-twisted stable.
Since $\langle v(U_i)^2 \rangle=-2$,
$U_i$ is the unique $(\beta+\eta)$-twisted stable object
of ${\frak C}$ with $v(U_i)=u_i$.
(2) follows from \eqref{eq:eta}.
\end{proof}

\begin{rem}
If the index of $[\alpha] \in 
H^2_{\text{\'{e}t}}(X,{\cal O}_X^{\times})$ and $\chr({\frak k})$
is relatively prime, then the same claim holds for the twisted case. 
\end{rem}

\begin{NB}
We set 
$$
v(E_i):=r_i e^\beta+a_i \varrho_X+(D_i+(D_i,\beta)\varrho_X).
$$
Then 
$\langle v(E_i)/r_i-v(E_j)/r_j, \eta+(\eta,\beta)\varrho_X \rangle
=(D_i/r_i-D_j/r_j,\eta)$.
\end{NB}

\begin{lem}\label{lem:eval}
Let $E$ be an object of ${\frak A}_{(\beta,\omega_-)}$.
Assume that $E$ is $\sigma_{(\beta,\omega)}$-semi-stable such that
$\Hom(E,U_i)=0$ for $i<k$.

\begin{enumerate}
\item[(1)]
$\varphi:E \to \Hom(E,U_k)^{\vee} \otimes U_k$
is surjective 
\begin{NB} in ${\frak A}_{(\beta,\omega_-)}$
\end{NB}
and 
$\ker \varphi$ is a $\sigma_{(\beta,\omega)}$-semi-stable object
such that
$\Hom(\ker \varphi,U_i)=0$ for $i \leq k$.
Moreover for the universal extension
\begin{equation}\label{eq:univ-ext}
0 \to U_k \otimes \Ext^1(\ker \varphi,U_k)^{\vee}
\to E' \to \ker \varphi \to 0,
\end{equation}
$E'$ is a $\sigma_{(\beta,\omega)}$-semi-stable object such that
$\Hom(E',U_i)=0$ for $i \leq k$.

\item[(2)]
If $\Hom(U_i,E)=0$ for $i \geq k$, then
\begin{align*}
\Hom(U_i,\ker \varphi)=0 \ \text{ for }\ i \geq k,\quad 
\Hom(U_i,E')=0\ \text{  for }\ i>k.
\end{align*}
\end{enumerate}

\end{lem}

\begin{proof}
(1)
By the proof of Lemma~\ref{lem:irreducible},
$\im \varphi \in {\frak S}_W$ and $\coker \varphi \in {\frak S}_W$.
Let $F_1$ be a subobject of $\im \varphi$ in ${\frak C}$
such that $F_1$ is $(\beta+\eta)$-twisted stable with
\begin{align*}
-\frac{\langle v(F_1),\eta+(\eta,\beta)\varrho_X \rangle}{\rk F_1} 
\geq
-\frac{\langle v(\im \varphi),\eta+(\eta,\beta)\varrho_X \rangle}
      {\rk \im \varphi}.
\end{align*}
Then we have 
\begin{align*}
-\frac{\langle v(F_1),\eta+(\eta,\beta)\varrho_X \rangle}{\rk F_1} 
\leq
-\frac{\langle u_k,\eta+(\eta,\beta)\varrho_X \rangle}{\rk u_k}.
\end{align*} 
On the other hand, for a quotient object $F_2$ of 
$\im \varphi$ in ${\frak C}$
such that $F_2 \in {\frak S}_W$ and $F_2$ is 
$(\beta+\eta)$-twisted stable with
\begin{align*}
-\frac{\langle v(F_2),\eta+(\eta,\beta)\varrho_X \rangle}{\rk F_2} 
\leq
-\frac{\langle v(\im \varphi),\eta+(\eta,\beta)\varrho_X \rangle}
      {\rk \im \varphi},
\end{align*}
our assumption implies that 
\begin{align*}
-\frac{\langle v(F_2),\eta+(\eta,\beta)\varrho_X \rangle}{\rk F_2}
\geq 
-\frac{\langle u_k,\eta+(\eta,\beta)\varrho_X \rangle}{\rk u_k}.
\end{align*}
Then each inequalities becomes equalities.
Hence $\im \varphi$ is a $(\beta+\eta)$-twisted semi-stable
object with 
$$
-\frac{\langle v(\im \varphi),\eta+(\eta,\beta)\varrho_X \rangle}
{\rk \im \varphi}
=
-\frac{\langle u_k,\eta+(\eta,\beta)\varrho_X \rangle}{\rk u_k}.
$$
By \eqref{eq:eta}, we see that $\im \varphi$ is a successive extension of
$U_k$. Since $\Ext^1(U_k,U_k)=0$,
we get $\im \varphi=U_k^{\oplus m}$.
Then we see that $m=\dim \Hom(E,U_k)$ and $\coker \varphi=0$.
Since $\Ext^1(U_k,U_k)=0$, we get $\Hom(\ker \varphi,U_k)=0$.
If there is a non-trivial homomorphism
$\psi:\ker \varphi \to U_i$, $i<k$,
then by Lemma~\ref{lem:irreducible},
$F:=\im \psi$ belongs to ${\frak S}_W$ and 
\begin{align*}
-\frac{\langle v(F), \eta+(\eta,\beta)\varrho_X \rangle}{\rk F}
<
-\frac{\langle u_k,\eta+(\eta,\beta)\varrho_X \rangle}{\rk u_k}.
\end{align*}
Then we have a quotient object
$U$ of $E$ fitting in an exact sequence
\begin{align*}
0 \to F \to U \to U_k^{\oplus m} \to 0. 
\end{align*}
Hence $U \in {\frak S}_W$ and 
\begin{align*}
-\langle v(U), \eta+(\eta,\beta)\varrho_X \rangle/\rk U
<
-\langle u_k,\eta+(\eta,\beta)\varrho_X \rangle/\rk u_k.
\end{align*}
This means that there is a quotient $E \to U_j$, $j<k$.
Therefore $\Hom(\ker \varphi,U_i)=0$ for all $i \leq k$.
By the exact sequence \eqref{eq:univ-ext},
we get an exact sequence
\begin{equation*}
\begin{array}{ccccccc}
0 & \longrightarrow & \Hom(\ker \varphi,U_i) 
  & \longrightarrow & \Hom(E',U_i) 
  & \longrightarrow & \Hom(U_k,U_i) \otimes \Ext^1(\ker \varphi,U_k)
\\
  & \xrightarrow{\ \delta_\varphi\ } &\Ext^1(\ker \varphi,U_i).
  &                 &
  &                 &
\end{array}
\end{equation*}
Since $\delta_\varphi$ is isomorphic for $i=k$ 
and $\Hom(\ker \varphi,U_i)=0$ for $i \leq k$, 
we get $\Hom(E',U_i)=0$ for $i \leq k$. 
Since $Z_{(\beta,\omega)}(U_k)=0$,
the semi-stability is a consequence of its definition. 
(2) easily follows from (1).
\end{proof}

\begin{lem}\label{lem:eval2}
Let $E$ be an object of ${\frak A}_{(\beta,\omega_-)}$.
Assume that $E$ is $\sigma_{(\beta,\omega)}$-semi-stable such that
$\Hom(U_i,E)=0$ for $i>k$.

\begin{enumerate}
\item[(1)]
$\varphi:\Hom(U_k,E) \otimes U_k \to E$ is injective and 
$\coker \varphi$ is a $\sigma_{(\beta,\omega)}$-semi-stable object
such that
$\Hom(U_i, \coker \varphi)=0$ for $i \geq k$.
Moreover for the universal extension
\begin{equation}
0 \to \coker \varphi \to E' \to U_k \otimes \Ext^1(U_k, \coker \varphi)
\to 0,
\end{equation}
$E'$ is a $\sigma_{(\beta,\omega)}$-semi-stable object such that
$\Hom(U_i,E')=0$ for $i \geq k$.

\item[(2)]
If $\Hom(E,U_i)=0$ for $i \leq k$, then
\begin{align*}
\Hom(\coker \varphi,U_i)=0 \ \text{ for }\ i \leq k,\quad 
\Hom(E',U_i)=0\ \text{  for }\ i<k.
\end{align*}
\end{enumerate}
\end{lem}

\begin{proof}
By the proof of Lemma \ref{lem:irreducible},
we have
$\ker \varphi \in {\frak S}_W$ and $\im \varphi \in {\frak S}_W$.
Let $F_1 \in {\frak S}_W$ be a subobject of $\im \varphi$
such that $F_1$ is a $(\beta+\eta)$-twisted stable object of ${\frak C}$.
By our assumption,
$$
-\frac{\langle v(F_1),\eta+(\eta,\beta)\varrho_X \rangle}{\rk F_1}
\leq -\frac{\langle u_k,\eta+(\eta,\beta)\varrho_X \rangle}{\rk u_k}.
$$
Then we get
$$
-\frac{\langle v(F_1),\eta+(\eta,\beta)\varrho_X \rangle}{\rk F_1} \leq 
-\frac{\langle u_k,\eta+(\eta,\beta)\varrho_X \rangle}{\rk u_k} \leq
-\frac{\langle v(\im \varphi),\eta+(\eta,\beta)\varrho_X \rangle}
{\rk \im \varphi}.
$$
Hence $\im \varphi$ is $(\beta+\eta)$-twisted semi-stable
with 
$$ 
-\frac{\langle u_k,\eta+(\eta,\beta)\varrho_X \rangle}{\rk u_k}=
-\frac{\langle v(\im \varphi),\eta+(\eta,\beta)\varrho_X \rangle}
{\rk \im \varphi}.
$$
As in the proof of Lemma \ref{lem:eval}, we see that 
$\im \varphi \cong U_k^{\oplus m}$.
Then we get $\ker \varphi=0$ and $m=\dim \Ext^1(U_k,E)$.
By similar arguments as in Lemma \ref{lem:eval},
we also see that the remaining statements hold.
\end{proof}

\begin{defn}\label{defn:BN}
\begin{equation*}
{\cal M}_{(\beta,\omega,k)}(v)
:=\{E \in {\cal M}_{(\beta,\omega)}(v) \mid
    \Hom(U_i,E)=0,\ i \geq k,\quad 
    \Hom(E,U_j)=0,\ j<k \}. 
\end{equation*}
We define the Brill-Noether locus by
\begin{equation*}
\begin{split}
{\cal M}_{(\beta,\omega,k)}(v)_m
&:= \{E \in {\cal M}_{(\beta,\omega,k)}(v) \mid
      \dim \Hom(E,U_k)=m \},
\\
{\cal M}_{(\beta,\omega,k)}(v)^m
&:= \{E \in {\cal M}_{(\beta,\omega,k)}(v) \mid
      \dim \Hom(U_{k-1},E)=m \}. 
\end{split}
\end{equation*}
\end{defn}

We note that
${\cal M}_{(\beta,\omega,k)}(v)_0={\cal M}_{(\beta,\omega,k+1)}(v)^0$.
As in \cite{Reflection},
we have the following description of the Brill-Noether locus.
\begin{prop}\label{prop:Brill-Noether}
\begin{enumerate}
\item[(1)]
${\cal M}_{(\beta,\omega,k)}(v)_m$ is a 
$Gr(2m+\langle v,u_k \rangle,m)$-bundle
over ${\cal M}_{(\beta,\omega,k)}(v-mu_k)_0$.
\item[(2)]
${\cal M}_{(\beta,\omega,k+1)}(v)^m$ is a 
$Gr(2m+\langle v,u_k \rangle,m)$-bundle
over ${\cal M}_{(\beta,\omega,k+1)}(v-mu_k)^0$.
\end{enumerate}
\end{prop}

\begin{proof}
(1)
For $F \in {\cal M}_{(\beta,\omega,k)}(v-m u_k)_0$,
$\Hom(U_k,F)=\Ext^2(U_k,F)=0$.
Hence $\dim \Ext^1(U_k,F)=\langle u_k,v-mu_k \rangle=
2m+\langle v,u_k \rangle$.
For a subspace $V$ of $\Ext^1(U_k,F)$ with $\dim V=m$,
take the associated extension
\begin{equation}\label{eq:grass1}
0 \to F \to E \to U_k \otimes V \to 0, 
\end{equation}
then $E \in {\cal M}_{(\beta,\omega,k)}(v)_m$.
Conversely for $E \in {\cal M}_{(\beta,\omega,k)}(v)_m$,
we have $F \in {\cal M}_{(\beta,\omega,k)}(v-m u_k)_0$
fitting in \eqref{eq:grass1} by Lemma \ref{lem:eval}.

(2) For $F \in {\cal M}_{(\beta,\omega,k+1)}(v-m u_k)^0$,
$\Hom(F,U_k)=\Ext^2(F,U_k)=0$.
For a subspace $V$ of $\Ext^1(F,U_k)$ with $\dim V=m$, 
we take the associated extension
\begin{equation}\label{eq:grass2}
0 \to U_k \otimes V^{\vee} \to E\to F \to 0.
\end{equation}
Then $E \in {\cal M}_{(\beta,\omega,k+1)}(v)^m$.
Conversely for $E \in {\cal M}_{(\beta,\omega,k)}(v)^m$,
we have $F \in {\cal M}_{(\beta,\omega,k)}(v-m u_k)^0$
fitting in \eqref{eq:grass2} by Lemma \ref{lem:eval2}.
\end{proof}

\begin{prop}\label{prop:simple}
For $E \in {\cal M}_{(\beta,\omega,k)}(v)$,
if $E$ is $S$-equivalent to $E_0 \oplus E_1$ 
such that $E_0$ is $\sigma_{(\beta,\omega)}$-stable
and $E_1 \in {\frak S}_W$, 
then $\Hom(E,E) \cong {\frak k}$.  
In particular, if $d=d_{\min}$, 
then $\Hom(E,E) \cong {\frak k}$ 
for all $E \in {\cal M}_{(\beta,\omega,k)}(v)$.
\end{prop}

\begin{proof}
Let $\varphi:E \to E$  be a homomorphism in
${\frak A}_{(\beta,\omega_-)}$ such that
$\ker \varphi \ne 0$ and $\coker \varphi \ne 0$.
Then $\ker \varphi, \im \varphi, \coker \varphi$ are
$\sigma_{(\beta,\omega)}$-semi-stable objects.  
By our assumption, 
$\ker \varphi, \coker \varphi \in {\frak S}_W$ or
$\im \varphi \in {\frak S}_W$.
In the first case,
$E \in {\cal M}_{(\beta,\omega,k)}(v)$ implies that
\begin{equation*}
\begin{split}
-\frac{\langle v(\ker \varphi),\eta+(\eta,\beta)\varrho_X \rangle}
{\rk \ker \varphi} & \leq
-\frac{\langle u_{k-1},\eta+(\eta,\beta)\varrho_X \rangle}
{\rk u_{k-1}},\\
-\frac{\langle v(\coker \varphi),\eta+(\eta,\beta)\varrho_X \rangle}
{\rk \coker \varphi} & \geq
-\frac{\langle u_k,\eta+(\eta,\beta)\varrho_X \rangle}
{\rk u_k}. 
\end{split}
\end{equation*}
Since $v(\ker \varphi)=v(\coker \varphi)$,
we have $-\frac{\langle u_{k},\eta+(\eta,\beta)\varrho_X \rangle}
{\rk u_{k}} \leq -\frac{\langle u_{k-1},\eta+(\eta,\beta)\varrho_X \rangle}
{\rk u_{k-1}}$, which is a contradiction.
In the second case, we also see that
\begin{equation*}
-\frac{\langle u_k,\eta+(\eta,\beta)\varrho_X \rangle}
{\rk u_k} \leq 
-\frac{\langle v(\im \varphi),\eta+(\eta,\beta)\varrho_X \rangle}
{\rk \im \varphi}  \leq
-\frac{\langle u_{k-1},\eta+(\eta,\beta)\varrho_X \rangle}
{\rk u_{k-1}},
\end{equation*}
which is a contradiction.
Therefore $\varphi$ is an isomorphism or $\varphi=0$.
Then we see that 
$\Hom(E,E)={\frak k}$ by a standard argument.
\end{proof}

\begin{NB}
\begin{prop}\label{prop:isom}
Assume that $d=d_{\min}$. 
Then ${\cal M}_{(\beta,\omega,k)}(v) \cong 
{\cal M}_{(\beta,\omega,k+1)}(v+\langle v,u_k \rangle u_k)$.
In particular, there is a Mukai vector 
$w=v+\sum_i a_i u_i$, $a_i \in {\Bbb Z}$,
 such that
 ${\cal M}_{(\beta,\omega_-)}(v) \cong 
{\cal M}_{(\beta,\omega_+)}(w)$.
\end{prop}

\begin{proof}
We set $X_i:=X$, $i=1,2$ and
${\bf E}_k:=\ker(U_k^{\vee} \boxtimes U_k \to {\cal O}_\Delta)
\in {\bf D}(X_1 \times X_2)$.
Then for $E \in {\cal M}_{(\beta,\omega,k)}(v)$,
$\Phi_{X_2 \to X_1}^{{\bf E}_k^{\vee}[1]}(E)$ fits in an exact sequence
in ${\frak A}_{(\beta,\omega_-)}$:
\begin{equation*}
0 \to \Ext^1(U_k,E) \otimes U_k \to
\Phi_{X_2 \to X_1}^{{\bf E}_k^{\vee}[1]}(E)
\to E \overset{\psi}{\to} \Ext^2(U_k,E) \otimes U_k \to 0.
\end{equation*}
Under the identification 
$\Ext^2(U_k,E) \cong \Hom(E,U_k)^{\vee}$,
$\psi$ corresponds to the evaluation map
$\varphi$ in Lemma \ref{lem:eval}.
Thus $\Phi_{X_2 \to X_1}^{{\bf E}_k^{\vee}[1]}(E)
\in {\cal M}_{(\beta,\omega,k+1)}(v+\langle v,u_k \rangle u_k)$.
Conversely for $F \in
{\cal M}_{(\beta,\omega,k+1)}(v+\langle v,u_k \rangle u_k)$,
it is easy to see that
$\Phi_{X_1 \to X_2}^{{\bf E}_k[1]}(F) \in {\cal M}_{(\beta,\omega,k)}(v)$.
Thus we get an isomorphism
$\Phi_{X_2 \to X_1}^{{\bf E}_k^{\vee}[1]}:
{\cal M}_{(\beta,\omega,k)}(v) \cong 
{\cal M}_{(\beta,\omega,k+1)}(v+\langle v,u_k \rangle u_k)$.
Applying $\Phi_{X_2 \to X_1}^{{\bf E}_k^{\vee}[1]}$, $k=1,2,\ldots,n$ 
successively, we get an isomorphism
${\cal M}_{(\beta,\omega,1)}(v) \cong 
{\cal M}_{(\beta,\omega,n+1)}(w)$, where $w$ is 
the corresponding Mukai vector.
By the assumption
$d=d_{\min}$, we have ${\cal M}_{(\beta,\omega,1)}(v)=
{\cal M}_{(\beta,\omega_-)}(v)$ and
${\cal M}_{(\beta,\omega,n+1)}(w)=
{\cal M}_{(\beta,\omega_+)}(w)$.
Hence ${\cal M}_{(\beta,\omega_-)}(v) \cong
{\cal M}_{(\beta,\omega_+)}(w)$. 

\begin{NB2}
$\Phi_{X_2 \to X_1}^{{\bf E}_k^{\vee}[1]}$ induces an equivalence
${\frak A}_{(\beta,\omega,k)} \to {\frak A}_{(\beta,\omega,k+1)}$:
Let $E$ be an object of ${\frak A}_{(\beta,\omega,k)}$.
${\cal O}_{\Delta}^{\vee}={\cal O}_\Delta[-2]$.
Since $\Phi_{X_2 \to X_1}^{{\bf E}_k^{\vee}[1]}(U_k)=U_k[1]$,
$\Hom(\Phi_{X_2 \to X_1}^{{\bf E}_k^{\vee}[1]}(E)[1],U_k)
=\Hom(E[1],U_k[-1])=0$.
We have an exact sequence in ${\frak C}$:
By the definition of ${\frak T}_{(\beta,\omega,k)}$,
$\im \psi \in {\frak T}_{(\beta,\omega,k)}$ implies that
$\im \psi$ is a direct sum of $U_k$.
Since $\Hom(\Phi_{X_2 \to X_1}^{{\bf E}_k^{\vee}[1]}(E)[1],U_k)=0$, we get
$H^1(\Phi_{X_2 \to X_1}^{{\bf E}_k^{\vee}[1]}(E)[1])=0$.
Then we see that 
$H^{-1}(\Phi_{X_2 \to X_1}^{{\bf E}_k^{\vee}[1]}(E)) 
\in {\frak F}_{(\beta,\omega,k+1)}$ and
$H^0(\Phi_{X_2 \to X_1}^{{\bf E}_k^{\vee}[1]}(E)) \in
{\frak T}_{(\beta,\omega,k)}$.
Since
$\Hom(\Phi_{X_2 \to X_1}^{{\bf E}_k^{\vee}[1]}(E),U_k)
=\Hom(E,U_k[-1])=0$, we also have
$\Hom(H^0(\Phi_{X_2 \to X_1}^{{\bf E}_k^{\vee}[1]}(E)),U_k)=0$.
Therefore $H^0(\Phi_{X_2 \to X_1}^{{\bf E}_k^{\vee}[1]}(E)) \in
{\frak T}_{(\beta,\omega,k+1)}$.
\end{NB2}
\end{proof}
\end{NB}

\begin{NB}
Let $U$ be an object of ${\frak U}$.
Assume that $Z_{(\beta+\eta,t_0 H)}(U)=0$.
Then $\langle e^{\beta+\eta+\sqrt{-1}t H},v(U) \rangle=
r(t^2-t_0^2)\frac{(H^2)}{2}$.
The claim follows from the following equalities.
\begin{equation}
\begin{split}
e^{\beta+\eta+\sqrt{-1}t H}= & e^{\beta+\eta+\sqrt{-1}t_0 H}
(1+\sqrt{-1}(t-t_0)H-(t-t_0)^2 \frac{(H^2)}{2}\varrho_X)\\
=& e^{\beta+\eta+\sqrt{-1}t_0 H}-(t^2-t_0^2) \frac{(H^2)}{2}\varrho_X
+\sqrt{-1}(t-t_0)(H+(H,\beta)\varrho_X). 
\end{split}
\end{equation}

Hence 
$$
e^{\beta+\eta+\sqrt{-1}t H}+
\langle e^{\beta+\eta+\sqrt{-1}t H},v(U) \rangle v(U)
=e^{\beta+\eta+\sqrt{-1}t H}+r(t^2-t_0^2)\frac{(H^2)}{2}v(U).
$$

We set
\begin{equation}
\begin{split}
\eta':=& \eta+\frac{r(t^2-t_0^2)\frac{(H^2)}{2}(c_1(U)-r(\beta+\eta))}
{1+r^2(t^2-t_0^2){\frac{(H^2)}{2}}},\\
t':=& \frac{1}{1+r^2(t^2-t_0^2){\frac{(H^2)}{2}}}.
\end{split}
\end{equation}
Then

\begin{equation}
Z_{(\beta+\eta',t' H)}(R_U(E))=
\frac{1}{1+r^2(t^2-t_0^2){\frac{(H^2)}{2}}}
Z_{(\beta+\eta,\omega)}(E).
\end{equation}

We set $\xi:=\beta-\frac{c_1(U)}{r}$ and 
$\xi':=\beta'-\frac{c_1(U)}{r}$.
Then $\xi,\xi' \in H^{\perp}$.

\begin{equation}
\begin{split}
\xi'=& \frac{2\xi}{r^2((\omega^2)-(\xi^2))},\\
\omega'=& \frac{2\omega}{r^2((\omega^2)-(\xi^2))}.
\end{split}
\end{equation}
Indeed $v(U)=r e^{\beta-\xi}+\frac{1}{r}\varrho_X$.
\begin{equation}
\langle e^{\beta+\sqrt{-1} \omega}, v(U) \rangle
=r \langle e^{\beta+\sqrt{-1} \omega},e^{\beta-\xi}  \rangle
-\frac{1}{r}
=r \langle e^{\xi+\sqrt{-1} \omega},1  \rangle
-\frac{1}{r}=
-r \frac{((\xi+\sqrt{-1}\omega)^2)}{2}-\frac{1}{r}.
\end{equation}
Hence
\begin{equation}
\begin{split}
R_U(e^{\beta+\sqrt{-1} \omega})=&
e^{\beta+\sqrt{-1} \omega}+
\left(-r \frac{((\xi+\sqrt{-1}\omega)^2)}{2}-\frac{1}{r} \right)v(U)\\
=& e^{\beta+\sqrt{-1} \omega}+
\left(-r^2 \frac{((\xi+\sqrt{-1}\omega)^2)}{2}-1 \right)e^{\beta-\xi}
+\left( -\frac{((\xi+\sqrt{-1}\omega)^2)}{2}-\frac{1}{r^2} \right)\varrho_X\\
=& e^{\beta-\xi}\left(e^{\xi+\sqrt{-1} \omega}+
\left(-r^2 \frac{((\xi+\sqrt{-1}\omega)^2)}{2}-1 \right)
+\left( -\frac{((\xi+\sqrt{-1}\omega)^2)}{2}-\frac{1}{r^2} \right)\varrho_X
\right)\\
=& e^{\beta-\xi} r^2 \frac{-((\xi+\sqrt{-1}\omega)^2)}{2}
e^{\xi'+\sqrt{-1}\omega'}.
\end{split}
\end{equation}

We write $\beta=bH+\eta$, $b \in {\Bbb Q}$, $\eta \in H^{\perp}$
and $c_1(U)=bH+\upsilon$.
Then the fixed point set is a sphere 
$$
(\omega^2)-(\eta-\upsilon)^2)=\frac{2}{r^2}
$$
and is the wall for categories defined by $U$.
\end{NB}

\begin{prop}\label{prop:category:birat}
Assume that $d=d_{\min}$.
Then ${\cal M}_{(\beta,\omega,k)}(v)$, $1 \leq k \leq n$, 
are all birationally equivalent.
\end{prop}

\begin{proof}
We set $m:=\max\{-\langle v,u_k \rangle, 0 \}$.
Then
${\cal M}_{(\beta,\omega,k)}(v)_m$ is an open dense substack 
of ${\cal M}_{(\beta,\omega,k)}(v)$ and 
${\cal M}_{(\beta,\omega,k+1)}(v)^m$ is an open dense
substack of ${\cal M}_{(\beta,\omega,k+1)}(v)$.
For the proof of the claim, 
it is sufficient to show
${\cal M}_{(\beta,\omega,k+1)}(v)^m \cong
{\cal M}_{(\beta,\omega,k)}(v)_m$.
If $m=0$, then 
${\cal M}_{(\beta,\omega,k)}(v)_0 \cong 
{\cal M}_{(\beta,\omega,k+1)}(v)^0$.
Assume that $m>0$.
Let $\Phi_{U_k}$ be the reflection functor.
Then for $E \in {\cal M}_{(\beta,\omega,k+1)}(v)^m$, we have
$$
\Phi_{U_k}(E)=\coker(U_k \otimes \Hom(U_k,E) \to E)
$$
and 
$E':=\Phi_{U_k} \circ \Phi_{U_k}(E)$ fits in an exact sequence
$$
0 \to E \to E' \to \Ext^1(U_k,\Phi_{U_k}(E)) \otimes U_k \to 0.
$$
Then we see that $E' \in {\cal M}_{(\beta,\omega,k)}(v)_m$.
Hence we have an isomorphism
${\cal M}_{(\beta,\omega,k+1)}(v)^m \to
{\cal M}_{(\beta,\omega,k)}(v)_m$. 
\end{proof}

\begin{cor}\label{cor:category:birat}
Assume that $d=d_{\min}$. 

\begin{enumerate}
\item[(1)] 
The birational type of
$M_{(\beta,\omega)}(v)$ does not depend on the choice of
general $\omega$.
\item[(2)]
Let $\Phi$ be the Fourier-Mukai transform 
in \S\,\ref{subsect:FM}.
If $\rk v \geq 0$, then
$M_H(v)$ is birationally equivalent to $M_{H'}(w)$,
where $w=-\Phi(v)$ for $-\rk \Phi(v) \geq 0$ and
$w=\Phi(v)^{\vee}$ for $-\rk \Phi(v) < 0$.
\end{enumerate}
\end{cor}

\begin{proof}
(1)
We note that ${\cal M}_{(\beta,\omega,1)}(v)=
{\cal M}_{(\beta,\omega_-)}(v)$ and
${\cal M}_{(\beta,\omega,n+1)}(v)=
{\cal M}_{(\beta,\omega_+)}(v)$.
By Proposition~\ref{prop:category:birat},
${\cal M}_{(\beta,\omega_-)}(v)$ is birationally equivalent to
${\cal M}_{(\beta,\omega_+)}(v)$.
Since the stability is independent of each chamber, we get
our claim.

(2)
The claim follows from (1), the equivalence
$\Phi[1]:{\frak A} \to {{\frak A}'}^\mu$ and 
the proof of Corollary \ref{cor:isom}.
\end{proof}


\subsection{The wall crossing formula for the numbers 
of semi-stable objects over ${\Bbb F}_q$}
\label{subsect:wall-crossing:numbers}

\subsubsection{}
The wall crossing formula for
counting invariants
was obtained by Toda \cite{Toda}.
Here we only consider its specialization.
 
Let ${\Bbb F}_q$ be the finite field with $q$ elements.
\begin{defn}
We denote the set of semi-stable objects over ${\Bbb F}_q$
by ${\cal M}_{(\beta,\omega)}(v)({\Bbb F}_q)$. 
\end{defn}
\begin{NB}
We don't need this.
\begin{rem}\label{rem:difference}
We do not know whether
${\cal M}_{(\beta,\omega)}(v)({\Bbb F}_q)$ is the same as 
the set of ${\Bbb F}_q$-valued point of the stack
${\cal M}_{(\beta,\omega)}(v)$. 
\end{rem}
\end{NB}

We shall study the weighted number of semi-stable objects
over ${\Bbb F}_q$:
\begin{align*}
\sum_{E \in {\cal M}_{(\beta,\omega)}(v)({\Bbb F}_q)}
\frac{1}{\#\Aut(E)}.
\end{align*}

We start with the moduli of 
$\beta$-twisted semi-stable objects of ${\frak C}$.
By Remark \ref{rem:relative-moduli},
we have a quotient stack description 
${\cal M}_H^\beta(v)^{ss}=[Q^{ss}/\GL(N)]$,
where $Q^{ss}$ is an open subscheme of a suitable quot-scheme.
Since every $\GL(N)$-orbit $O$ over ${\Bbb F}_q$ contains a 
${\Bbb F}_q$-rational point
(cf. \cite[Thm. 2]{Lang}), 
we have $\# O({\Bbb F}_q)=\# \GL(N)({\Bbb F}_q)$
and 
\begin{equation}\label{eq:wt-number}
\sum_{E \in {\cal M}_H^\beta(v)^{ss}({\Bbb F}_{q^n})}\frac{1}{\# \Aut(E)}
=\frac{\# Q^{ss}({\Bbb F}_{q^n})}{\# \GL(N)({\Bbb F}_{q^n})}.
\end{equation}
\begin{NB}
\cite[Thm. 2]{Lang} says that every homogeneous space over ${\Bbb F}_q$
has a ${\Bbb F}_q$-rational point.
\end{NB}

\begin{rem}
Let $Q^s$ be the open subset of $Q^{ss}$ parametrizing stable sheaves.
Since $Q^s \to M_H^\beta(v)$ is a principal $\PGL(N)$-bundle,
\cite[Thm. 2]{Lang} also says that
\begin{align*}
\# M_H^\beta(v)({\Bbb F}_q)=
(q-1) \frac{\#Q^s({\Bbb F}_q)}{\#\GL(N)({\Bbb F}_q)}.
\end{align*}
\end{rem}

Let us compute the wall crossing formula for 
$\sum_{E \in {\cal M}_{(\beta,\omega)}(v)({\Bbb F}_q)}
\frac{1}{\#\Aut(E)}$.
We first treat the case where $\omega$ belongs to a wall $W$ for stabilities
in Definition~\ref{defn:wall:stability}.
By using Desale and Ramanan~\cite{D-R:1}, we see that 
\begin{multline}\label{eq:7}
\sum_{E \in {\cal M}_{(\beta,\omega)}(v)({\Bbb F}_q)}
\frac{1}{\# \Aut(E)}=
\sum_{E \in {\cal M}_{(\beta,\omega_\pm)}(v)({\Bbb F}_q)}
\frac{1}{\# \Aut(E)}\\
+\sum_{(v_1,\ldots,v_s) }
q^{\sum_{i>j}\langle v_i,v_i \rangle}
\prod_{i=1}^s 
\left(\sum_{E \in {\cal M}_{(\beta,\omega_\pm)}(v_i)({\Bbb F}_q)}
\frac{1}{\#\Aut(E)} \right),
\end{multline}
where $v_1,\ldots,v_s$ satisfies
$v=\sum_{i=1}^s v_i$,
$\phi(v_i)=\phi(v)$ and
$\phi_\pm(v_1)>\phi_{\pm}(v_2)>\cdots > \phi_\pm(v_s)$.

We next treat the case where $\omega$ belongs to a wall 
$W$ for categories in Definition~\ref{defn:wall:category}.
For $v$ with $Z_{(\beta,\omega)}(v)=0$,
we set 
\begin{align*}
{\cal M}_{(\beta,\omega_-)}(v):= 
& \{E \in {\frak S}_W \mid v(E)=v \},\quad 
\phi_-(v)=1,
\\
{\cal M}_{(\beta,\omega_+)}(v):= 
& \{E \in {\frak S}_W \mid v(E)=v \},\quad
\phi_+(v)=0.
\end{align*}
\begin{NB}
By the definition of ${\frak S}_W$,
${\frak S}_W$ and ${\cal M}_{(\beta,\omega_\pm)}(v)$
are well-defined.
\end{NB}
Then we get
\begin{multline}\label{eq:8}
\sum_{E \in {\cal M}_{(\beta,\omega)}(v)({\Bbb F}_q)}\frac{1}{\# \Aut(E)}=
\sum_{E \in {\cal M}_{(\beta,\omega_\pm)}(v)({\Bbb F}_q)}\frac{1}{\# \Aut(E)}
\\
+\sum_{(v_1,\ldots,v_s) }
q^{\sum_{i>j}\langle v_i,v_i \rangle}
\prod_{i=1}^s 
\left(\sum_{E \in {\cal M}_{(\beta,\omega_\pm)}(v_i)({\Bbb F}_q)}
\frac{1}{\#\Aut(E)} \right),
\end{multline}
where $v_1,\ldots,v_s$ satisfies
\begin{enumerate}
\item
$v=\sum_i v_i$,
\item
$Z_{(\beta,\omega)}(v_i) \in
{\Bbb R}_{\geq 0}Z_{(\beta,\omega)}(v)$,
\item
$1 \geq \phi_-(v_1)>\phi_-(v_2)>\cdots > \phi_-(v_s)>0$
and
$1>\phi_+(v_1)>\phi_+(v_2)>\cdots > \phi_+(v_s) \geq 0$.
\end{enumerate}
For $v_1,\ldots,v_s$ with $v=\sum_i v_i$,
$Z_{(\beta,\omega)}(v_i) \in
{\Bbb R}_{\geq 0}Z_{(\beta,\omega)}(v)$,
the condition 
$$
1 \geq \phi_-(v_1)>\phi_-(v_2)>\cdots > \phi_-(v_s)>0
$$
is equivalent to
$$
1>\phi_+(v_s)>\phi_+(v_{s-1})>\cdots > \phi_+(v_1) \geq 0.
$$

By the induction on $d$, we get the following claim, which is a special
case of \cite{Toda}.
\begin{prop}\label{prop:wt-number}
$$
\sum_{E \in {\cal M}_{(\beta,\omega_-)}(v)({\Bbb F}_q)}\frac{1}{\# \Aut(E)}
=\sum_{E \in {\cal M}_{(\beta,\omega_+)}(v)({\Bbb F}_q)}\frac{1}{\# \Aut(E)}.
$$
\end{prop}

\begin{proof}
We prove the claim by induction on $d$.
By the wall crossing formula \eqref{eq:7} \eqref{eq:8},
it is sufficient to prove the claim for $d=d_{\min}$.
In this case,
by Proposition~\ref{prop:Brill-Noether}, we get
$$
\sum_{E \in {\cal M}_{(\beta,\omega,k)}(v)_m({\Bbb F}_q)}
\frac{1}{\# \Aut(E)}
=\sum_{E \in {\cal M}_{(\beta,\omega,k+1)}(v)^m({\Bbb F}_q)}
\frac{1}{\# \Aut(E)}.
$$
Then using the Brill-Noether locus given in Definition \ref{defn:BN},
we have
\begin{equation*}
\begin{split}
\sum_{E \in {\cal M}_{(\beta,\omega,k)}(v)({\Bbb F}_q)}
\frac{1}{\# \Aut(E)}
=& \sum_m \sum_{E \in {\cal M}_{(\beta,\omega,k)}(v)_m({\Bbb F}_q)}
\frac{1}{\# \Aut(E)}\\
=&\sum_m \sum_{E \in {\cal M}_{(\beta,\omega,k+1)}(v)^m({\Bbb F}_q)}
\frac{1}{\# \Aut(E)}\\
=&\sum_{E \in {\cal M}_{(\beta,\omega,k+1)}(v)({\Bbb F}_q)}
\frac{1}{\# \Aut(E)}.
\end{split}
\end{equation*}
Since ${\cal M}_{(\beta,\omega,1)}(v)={\cal M}_{(\beta,\omega_-)}(v)$
and ${\cal M}_{(\beta,\omega,n+1)}(v)={\cal M}_{(\beta,\omega_+)}(v)$,
we get the claim.
\end{proof}

\begin{prop}\label{prop:wt-number:stable}
If $\omega$ is general, then
\begin{equation*}
\sum_{E \in {\cal M}_{(\beta,\omega)(v)({\Bbb F}_q)}}\frac{1}{\# \Aut(E)}
=
\sum_{E \in {\cal M}_H^\beta( \pm v)^{ss}({\Bbb F}_q)}\frac{1}{\# \Aut(E)}.
\end{equation*}
\end{prop}

\begin{proof}

By Proposition~\ref{prop:wt-number} and
Corollary~\ref{cor:Toda}, we have
\begin{equation*}
\sum_{E \in {\cal M}_{(\beta,\omega)(v)({\Bbb F}_q)}}\frac{1}{\# \Aut(E)}
=
\begin{cases}
\sum_{E \in {\cal M}_H^\beta(v)^{ss}({\Bbb F}_q)}\frac{1}{\# \Aut(E)},& 
\rk v\geq 0\\
\sum_{E \in {\cal M}_H^{-\beta}(-v^{\vee})^{ss}({\Bbb F}_q)}
\frac{1}{\# \Aut(E)},& \rk v<0.
\end{cases}
\end{equation*}
Then the claim follows from Proposition~\ref{prop:wt-number:dual}
below.
\end{proof}

\begin{cor}
Let $\Phi:{\bf D}(X) \to {\bf D}(X')$ be a Fourier-Mukai
transform associated to a moduli of stable sheaves.
Then
$$
\sum_{E \in {\cal M}_H^\beta(v)^{ss}({\Bbb F}_q)}\frac{1}{\# \Aut(E)}
=\sum_{E \in 
{\cal M}_{H'}^{\beta'}(\pm \Phi(v))^{ss}({\Bbb F}_q)}\frac{1}{\# \Aut(E)}.
$$ 
\end{cor}

\begin{proof}
The claim follows from Proposition~\ref{prop:FM} and
Proposition~\ref{prop:wt-number:stable}.
\end{proof}

\subsubsection{}

We shall prove the following result.
\begin{prop}\label{prop:wt-number:dual}
Assume that $\rk v>0$.
\begin{equation*}
\sum_{E \in {\cal M}_H^\beta(v)^{ss}({\Bbb F}_q)}
\frac{1}{\# \Aut(E)}=
\sum_{E \in {\cal M}_H^{-\beta}(v^{\vee})^{ss}({\Bbb F}_q)}
\frac{1}{\# \Aut(E)}.
\end{equation*}
\end{prop}

By the Harder-Narasimhan filtration, we can write down   
\begin{align*}
\sum_{E \in {\cal M}_H(v)^{\text{$\mu$-ss}}({\Bbb F}_q)}
\frac{1}{\# \Aut(E)}-
\sum_{E \in {\cal M}_H^{\beta}(v)^{ss}({\Bbb F}_q)}
\frac{1}{\# \Aut(E)}
\end{align*}
by using 
$\sum_{E \in {\cal M}_H^{\beta}(v')^{ss}({\Bbb F}_q)} (\# \Aut(E))^{-1}$ 
with $\rk v'< \rk v$.
By the induction on $\rk v$, the proof of Proposition 
\ref{prop:wt-number:dual} is reduced to show the following 
claim.
\begin{lem}\label{lem:wt-number:dual}
Assume that $\rk v>0$.
\begin{equation*}
\sum_{E \in {\cal M}_H(v)^{\text{$\mu$-ss}}({\Bbb F}_q)}
\frac{1}{\# \Aut(E)}=
\sum_{E \in {\cal M}_H(v^{\vee})^{\text{$\mu$-ss}}({\Bbb F}_q)}
\frac{1}{\# \Aut(E)}.
\end{equation*}
\end{lem}
Indeed if $\rk v=1$, then Lemma~\ref{lem:wt-number:dual}
implies Proposition~\ref{prop:wt-number:dual}.

\begin{defn}
Let ${\cal M}_{(\beta,\infty)}(v)$ be the stack consisting of $E
\in {\frak A}^\mu$ with $v(E)=v$
such that
$H^{-1}(E)$ is a $\mu$-semi-stable object and
$H^0(E)$ is a 0-dimensional object.
\end{defn}

\begin{lem}\label{lem:M_infty}
Let $E$ be an object of ${\frak A}^\mu$ with $v(E)=v$.
Then 
$E \in {\cal M}_{(\beta,\infty)}(v)$
if and only if 
we have a filtration 
\begin{equation}\label{eq:filtration3}
0 \subset F_1 \subset F_2 \subset F_3=E
\end{equation}
such that 
$F_1$ and $F_3/F_2$ are 0-dimensional, $F_2/F_1[-1]$ is a local
projective object of ${\frak C}$,
$\Hom(A,F_3/F_1)=0$ for any $0$-dimensional object $A \in {\frak C}$. 
\begin{NB}
Old version. It may be wrong:
$F_1$ and $F_3/F_2$ are 0-dimensional, $F_2/F_1[-1]$ is a local
projective object of ${\frak C}$,
$F_3/F_1$ is torsion free, $\Hom(F_2,A)=0$ 
for any 0-dimensional object $A$ of ${\frak C}$.
\end{NB}
\end{lem}

\begin{proof}
For an object $E$ of ${\frak A}^\mu$ with
$\dim H^0(E)=0$,
by Lemma~\ref{lem:reflexive-hull}, we have an exact sequence 
$$
0 \to H^{-1}(E) \to F \to T \to 0
$$
in ${\frak C}$ such that $F$ is a local projective object
of ${\frak C}$ and $T$ is a 0-dimensional object of ${\frak C}$.

Then we have an injective morphism $T \to H^{-1}(E)[1] \to E$
in ${\frak A}^\mu$ and we get a quotient $E/T$ in ${\frak A}^\mu$.
$E/T$ is a complex such that
$H^{-1}(E/T)=F$ and $H^0(E/T)=H^0(E)$.
Let $A$ be a 0-dimensional subobject of $E/T$,
then $A \to H^0(E/T)$ is injective.
\begin{NB}
We have an exact sequence
$$
0 \to H^{-1}(E/T) \to H^{-1}((E/T)/A) \to A \to H^0(E/T).
$$
Since $H^{-1}(E/T)$ is local projective and 
$H^{-1}((E/T)/A)$ is torsion free,
$H^{-1}(E/T) \to H^{-1}((E/T)/A)$ is an isomorphism, which implies
$A \to H^0(E/T)$ is injective.
\end{NB}
Hence there is a maximal 0-dimensional subobject
$B$ of $E/T$.
Then there is a 0-dimensional subobject $F_1$ of $E$
such that $T \subset F_1$ and $F_1/T=B$.
It is easy to see that
$H^{-1}(E/T) \to H^{-1}(E/F_1)$ is isomorphic.
We take $F_1 \subset F_2 \subset F_3=E$ 
as $F_2/F_1=H^{-1}(E/F_1)[1]$. 
Then the filtration satisfies the required properties.

Conversely if there is 
a filtration \eqref{eq:filtration3},
we have an exact sequence
\begin{equation*}
      0 \to H^{-1}(E) \to H^{-1}(F_2/F_1) 
\to F_1 \to H^0(E)    \to H^0(F_3/F_2)   \to 0
\end{equation*}
in ${\frak C}$.
Hence the claim holds.
\end{proof}

\begin{NB}
Assume that $E \in {\frak A}^\mu$ has a subobject $E_1$ in ${\frak A}^\mu$
such that $E_1$ is a 0-dimensional object
of ${\frak C}$ and $E/E_1$ is an object of ${\frak A}^\mu$
such that $H^{-1}(E/E_1)$ is a $\mu$-semi-stable
object of ${\frak F}^\mu$,
$H^0(E/E_1)$ is a 0-dimensional object of ${\frak C}$
and $\Hom(A,E/E_1)=0$ for all 0-dimensional
object of ${\frak C}$.
Then $H^0(E)$ is a 0-dimensional object of ${\frak C}$
and $H^{-1}(E)$ is a $\mu$-semi-stable object of ${\frak C}$.
\end{NB}

\begin{lem}\label{lem:reflexive-hull}
For a torsion free object $E$ of ${\frak C}$,
we have a unique extension
$$
0 \to E \to F \to T \to 0
$$
in ${\frak C}$ such that $F$ is a local projective object
of ${\frak C}$ and $T$ is a 0-dimensional object of ${\frak C}$.
\end{lem}

\begin{proof}
For a torsion free object $E_1:=E$ of ${\frak C}$,
if $\Hom(A_1, E_1[1]) \ne 0$ for a 0-dimensional irreducible
object $A_1$, then
we take a non-trivial extension
$$
0 \to E_1 \to E_2 \to A_1 \to 0.
$$ 
Then $E_2$ is a torsion free object of ${\frak C}$.
If $\Hom(A_2,E_2[1]) \ne 0$ for 0-dimensional object $A_2$
of ${\frak C}$,
then
a non-trivial extension
$$
0 \to E_2 \to E_3 \to A_2 \to 0
$$ 
gives a torsion free object $E_3$ of ${\frak C}$.
Continuing this procedure,
we get a sequence of torsion free objects
$E_1 \subset E_2 \subset \cdots \subset E_n \subset \cdots$
with $v(E_i)=v(E_1)+\sum_{j=1}^{i-1}v(A_j)$.
By using the Bogomolov's inequality for $\mu$-semi-stable objects
and Lemma \ref{lem:bogomolov-estimate} below,
we see that $\langle v(E_i)^2 \rangle \geq -N$,
where $N$ depends on $\rk E$ and the Harder-Narasimhan
filtration of $E_1$ with respect to the $\mu$-semi-stability.
We set 
\begin{equation*}
\begin{split}
v(E_1)= & re^\beta+a \varrho_X+(dH+D+(dH+D,\beta)\varrho_X),\quad r>0,
\\
v(A_i)= & b_i \varrho_X+(D_i+(D_i,\beta)\varrho_X),\quad b_i>0.
\end{split}
\end{equation*}
Then
$$
\langle v(E_i)^2 \rangle-\langle v(E_1)^2 \rangle
=-2r\sum_{j=1}^{i-1}b_j+((D+\sum_{j=1}^{i-1} D_j)^2)-(D^2)
\leq -2r\sum_{j=1}^{i-1}b_j-(D^2).
$$
Therefore there is an $n$
such that
$\Hom(A,E_n[1])=0$ for all 0-dimensional object $A$.
Thus we get a desired local projective object.
The uniqueness follows from $\Hom(T,F[1])=0$.  
\end{proof}

\begin{lem}\label{lem:bogomolov-estimate}
For $\mu$-semi-stable objects $E_i$ ($i=1,2$) of ${\frak C}$ with
$$
v(E_i)=r_i e^\beta+a_i \varrho_X +(d_i H+D_i+(d_i H+D_i,\beta)\varrho_X),\;
r_i>0,\; (D_i,H)=0,
$$
we have 
$$
\langle v(E_1),v(E_2) \rangle
\geq -\frac{1}{2 r_1 r_2 }(r_2 d_1-r_1 d_2)^2(H^2)-2r_1 r_2.
$$   
\end{lem}

\begin{proof}
The claim follows from the Bogomolov inequality and 
$$
\frac{\langle v(E_1),v(E_2) \rangle}{r_1 r_2}
=-\frac{1}{2}
\left(\frac{d_1}{r_1}-\frac{d_2}{r_2} \right)^2(H^2)-
\frac{1}{2}\left(\frac{D_1}{r_1}-\frac{D_2}{r_2} \right)^2+
\frac{1}{2}\left( \frac{\langle v(E_1)^2 \rangle}{r_1^2}
+ \frac{\langle v(E_2)^2 \rangle}{r_2^2} \right).
$$
\end{proof}

By the proof of Proposition~\ref{prop:star-2}, we get the following result.
\begin{lem}\label{lem:characterize-dual}
$E$ is a $\mu$-semi-stable object of ${\frak C}^D$ with $\deg_{G^{\vee}}(E)>0$ 
if and only if
$E^{\vee}[1]$ is an object of ${\frak F}^\mu$
such that $H^{-1}(E^{\vee}[1])$ is a $\mu$-semi-stable
object of ${\frak C}$, $H^0(E^{\vee}[1])$ is a 0-dimensional 
object and $\Hom(A,E^{\vee}[1])=0$ for all 0-dimensional
objects $A$ of ${\frak C}$.
\end{lem}

By Lemma~\ref{lem:M_infty} and
Lemma~\ref{lem:characterize-dual}, 
we have the following expressions:
\begin{equation*}
\sum_{E \in {\cal M}_{(\beta,\infty)}(v)({\Bbb F}_q)}\frac{1}{\# \Aut(E)}=
\sum_{v_1+v_2=v }
q^{\langle v_2,v_1 \rangle}
\left(\sum_{E \in {\cal M}_H(-v_1)^{\text{$\mu$-ss}}({\Bbb F}_q)}
\frac{1}{\# \Aut(E)} \right)
\left( 
\sum_{E \in {\cal M}_H(v_2)^{ss}({\Bbb F}_q)}\frac{1}{\# \Aut(E)}
\right)
\end{equation*}
with 
$\rk v_2=(c_1(v_2), H)=0$.
We also have  
\begin{equation*}
\sum_{E \in {\cal M}_{(\beta,\infty)}(v)({\Bbb F}_q)}\frac{1}{\# \Aut(E)}=
\sum_{v_1+v_2=v }
q^{\langle v_2,v_1 \rangle}
\left(\sum_{E \in {\cal M}_{H}(-v_1^{\vee})^{\text{$\mu$-ss}}({\Bbb F}_q)}
\frac{1}{\# \Aut(E)} \right)
\left( 
\sum_{E \in {\cal M}_H(v_2)^{ss}({\Bbb F}_q)}\frac{1}{\# \Aut(E)}
\right)
\end{equation*}
with $\rk v_2=(c_1(v_2), H)=0$.
Note that if we set 
\begin{align*}
v=-re^\beta+a \varrho_X+(dH+D+(dH+D,\beta)\varrho_X),
\quad
v_2=a_2 \varrho_X+(D_2+(D_2,\beta)\varrho_X),
\end{align*}
then 
\begin{align*}
-2r^2 \leq \langle v_1^2 \rangle-(\langle v^2 \rangle-(D^2))
=-2ra_2+((D-D_2)^2).
\end{align*}
Hence the choice of $v_2$ is finite, 
so the above equalities are well-defined. 
By the induction on $\rk v$, we get Lemma~\ref{lem:wt-number:dual}.


\section{The wall crossing behavior on an abelian surface}
\label{sect:abel}


\subsection{The wall defined by an isotropic Mukai vector}
\label{subsect:wall=w_1}

In this subsection, we assume that $X$ is an abelian surface
over a field ${\frak k}$.
Let $\overline{\frak k}$ be the algebraic closure of ${\frak k}$.
We note that ${\frak A}={\frak A}^\mu$ 
and all ${\frak A}_{(\beta,\omega)}$
are the same. 
Since there is no wall for categories, we shall simply use
the word \emph{a wall} in the meaning of a wall for stabilities.
We fix $\beta=bH+\eta$ and study the wall crossing behavior
with respect to $\omega \in {\Bbb Q}_{>0}H$.
By Definition~\ref{defn:wall:stability}, we have the following proposition.

\begin{prop}\label{prop:homog}
Assume that $d>0$.
If $v$ is a primitive and isotropic Mukai vector, then
the stability does not depend on the choice of $\omega$.
In particular,
\begin{equation*}
\begin{split}
{\cal M}_{(\beta,\omega)}(v)=& 
\{ E \mid \text{$E$ is a semi-homogeneous sheaf of $v(E)=v$} \},\quad 
\rk v \geq 0,
\\
{\cal M}_{(\beta,\omega)}(v)=&
\{ F[1] \mid \text{$F$ is a semi-homogeneous sheaf of $v(F)=-v$} \},\quad
\rk v < 0.
\end{split}
\end{equation*}
\end{prop}
By Proposition \ref{prop:mod-p}, 
we have ${\cal M}_{(\beta,\omega)}(v) \ne \emptyset$
if $v$ is defined over ${\frak k}$. 

\begin{NB}
Let $v$ be a Mukai vector such that $v=w_0+l w_1$,
$\langle w_0^2 \rangle=\langle w_1^2 \rangle=0$,
$\langle w_0,w_1 \rangle=1$,
$\deg_G(w_0),\deg_G(w_1)>0$.
Assume that $\phi(w_0)=\phi(w_1)$.
Let $E$ be a stable object of ${\frak A}$ with 
$\langle v(E),w_1 \rangle=1$.

For $E_1 \in {\cal M}_{(\beta,\omega)}(w_1)$,
If $\Hom(E_1,E) \ne 0$ or $\Hom(E,E_1) \ne 0$,
then $E \cong E_1$, which implies that $v(E)=w_1$.
Therefore $\Hom(E_1,E)=\Ext^2(E_1,E)=0$
for all $E_1 \in {\cal M}_{(\beta,\omega)}(w_1)$.
Hence $\Ext^1(E_1,E) \cong {\frak k}$ for
all $E_1 \in {\cal M}_{(\beta,\omega)}(w_1)$.
\end{NB}

Let $w_1$ be a primitive and isotropic Mukai vector.
We shall study $\sigma_{(\beta,\omega)}$-stable objects 
for $(\beta,\omega)$ lying on the wall $W_{w_1}$
of type $w_1$ for $v$.
We set $X_1:=M_{(\beta,\omega)}(w_1)$. 
Let $\Phi_{X \to X_1}^{{\bf E}^{\vee}}:{\bf D}(X) \to {\bf D}^\alpha(X_1)$
be the Fourier-Mukai transform defined by the universal family
${\bf E}$ on $X \times X_1$ as a $(1_X \times \alpha^{-1})$-twisted sheaf, 
where $\alpha$ is a representative of 
$[\alpha] \in H^2_{\text{\'{e}t}}(X_1,{\cal O}_{X_1}^{\times})$
(for the construction of ${\bf E}$, 
see \S\, \ref{subsect:relative-FM} below).

\begin{NB}
\begin{lem}
Let $F$ be an irreducible $\alpha$-twisted sheaf on a closed point
$W$ of $X_1$. Then $\Phi_{X_1 \to X}^{{\bf E}}(F)$ is a stable
sheaf on $X$. 
\end{lem}

\begin{proof}
Let $y$ be a point of $X_1 \otimes_{\frak k} \overline{\frak k}$.
Since $H_{\text{\'et}}^2(y,\overline{\frak k}_y^{\times})$
is trivial, 
there is an $\alpha$-twisted sheaf ${\cal O}_y^\alpha$ on $y$
such that $\dim_{\overline{\frak k}} ({\cal O}_y^\alpha)=1$.
Then $F \otimes_{\frak k} \overline{\frak k}$
is a successive extension of ${\cal O}_{y_i}^\alpha$, 
$y_1,\ldots,y_n \in X_1 \otimes_{\frak k} \overline{\frak k}$,
where $n:=\dim_{\frak k} W$.
Hence $E:=\Phi_{X_1 \to X}^{{\bf E}}(F)$ is a sheaf on $X$
such that $E \otimes_{\frak k} \overline{\frak k}$
is a successive extension of stable
sheaves $\Phi_{X_1 \to X}^{\bf E}({\cal O}_{y_i}^\alpha)$.
If $E$ is not stable, then there is an exact sequence
$$
0 \to E_1 \to E \to E_2 \to 0 
$$
such that $E_1 \otimes_{\frak k} \overline{\frak k}$
and $E_2 \otimes_{\frak k} \overline{\frak k}$
are succissive extensions of stable
sheaves $\Phi_{X_1 \to X}^{\bf E}({\cal O}_{y_i}^\alpha)$. 
Then we see that $F_i:=
\Phi_{X \to X_1}^{{\bf E}^{\vee}[2]}(E_i)$ ($i=1,2$)
are 0-dimensional $\alpha$-twisted sheaves
and we have an exact sequence
of $\alpha$-twisted sheaves
$$
0 \to F_1 \to F \to F_2 \to 0
$$
on $W$.
It contradicts the irreducibility of $F$.
Therefore $E$ is stable.
\end{proof}
\end{NB}

\begin{NB}
Let $G$ be a locally free $\alpha$-twisted sheaf on $X_1$. 
Let $W$ be a closed point of $X_1$.
Then there are finitely many irreducible $\alpha$-twisted
sheaves on $W$, which are quotient of $G_{|W}$.
Indeed $\Hom(G_{|W}, F)=H^0(W,G_{|W}^{\vee} \otimes F) \ne 0$ 
implies that there is a non-trivial morphism
$G_{|W} \to F$, which is surjective.
For irreducible objects $F_1,F_2$,
non-trivial morphism $F_1 \to F_2$ is an isomorphism.
Since $G_{|W}$ is a finite dimensional vector space
over ${\cal O}_W$,
$G_{|W}$ is a successive extension of irreducible objects.
Therefore the claim holds.
\end{NB}

\begin{lem}\label{lem:stable-on-wall}
Assume that ${\frak k}$ is algebraically closed.
Let $E$ be a stable object of ${\frak A}$
and assume that $\phi(E)=\phi(w_1)$.
\begin{enumerate}
\item
[(1)]
$\langle v(E),w_1 \rangle \geq 0$.
\item
[(2)]
If $\langle v(E),w_1 \rangle=0$, then
there is a point $x_1$ of $X_1$ such that 
$E=\Phi_{X_1 \to X}^{{\bf E}}(F_{x_1}) \in 
{\cal M}_{(\beta,\omega)}(w_1)$, where
$F_{x_1}$ corresponds to ${\frak k}_{x_1}$ via
$\Coh^\alpha(\Spec({\frak k})) \cong \Coh(\Spec({\frak k}))$.
\item
[(3)]
If $\langle v(E),w_1 \rangle>0$, then
$\Hom(E_1,E)=\Ext^2(E_1,E)=0$
for any stable object $E_1$ with $v(E_1) \in {\Bbb Z}w_1$.
In particular, if $\langle v(E),w_1 \rangle=1$,
then $\langle v(E)^2 \rangle=0$.
\end{enumerate}
\end{lem}

\begin{proof}
Let $E_1$ be a stable object with $v(E_1) \in {\Bbb Z}w_1$.
We first note that
$\Hom(E_1,E) \ne 0$ or $\Hom(E,E_1) \ne 0$ 
implies that
$E \cong E_1$. 

Since $\dim X_1>0$,
we have a point $x_1$ of $X_1$ such that
$\Phi_{X_1 \to X}^{{\bf E}}(F_{x_1}) \ne E$.
Then
\begin{align*}
\langle v(E), w_1 \rangle
=\langle v(E),v(\Phi_{X_1 \to X}^{{\bf E}}(F_{x_1})) \rangle
=-\chi(E,\Phi_{X_1 \to X}^{{\bf E}}(F_{x_1})) \geq 0.
\end{align*} 
Thus (1) holds.
We also have the first part of (3).

Assume that $E \ne \Phi_{X_1 \to X}^{{\bf E}}(F_{x_1})$
for any point $x_1 \in X_1$.
Then $\Phi_{X \to X_1}^{{\bf E}^{\vee}}(E)[1]$ is a locally free 
$\alpha$-twisted sheaf 
of rank $\langle v(E),w_1 \rangle$ on $X_1$. 
If $\langle v(E),w_1 \rangle=0$, then we have 
$\Phi_{X \to X_1}^{{\bf E}^{\vee}}(E)[1]=0$,
which implies that $E=0$. Therefore
$E \cong \Phi_{X_1 \to X}^{{\bf E}}(F_{x_1})$
for a point $x_1 \in X_1$.
Thus (2) holds.

If $\langle v(E),w_1 \rangle=1$, then
$\Phi_{X \to X_1}^{{\bf E}^{\vee}}(E)[1]$ is a line bundle on $X_1$.
Hence $\langle v(E)^2 \rangle=
\langle v(\Phi_{X \to X_1}^{{\bf E}^{\vee}}(E)[1])^2 \rangle=0$.
Thus the second part of (3) follows.
\end{proof}

\begin{NB}
\begin{lem}\label{lem:stable-on-wall}
Assume that ${\frak k}$ is algebraically closed.
Let $E$ be a stable object of ${\frak A}$
and assume that $\phi(E)=\phi(w_1)$.
\begin{enumerate}
\item
[(1)]
$\langle v(E),w_1 \rangle \geq 0$.
\item
[(2)]
If $\langle v(E),w_1 \rangle=0$, then
there are a closed point $W$ of $X_1$ and an irreducible
$\alpha$-twisted sheaf $F$ on $W$ such that 
$E=\Phi_{X_1 \to X}^{{\bf E}}(F)$.
In particular, if ${\frak k}$ is an algebraically closed field, then
$E \in {\cal M}_{(\beta,\omega)}(w_1)$.
\item
[(3)]
If $\langle v(E),w_1 \rangle>0$, then
$\Hom(E_1,E)=\Ext^2(E_1,E)=0$
for any stable object $E_1$ with $v(E_1) \in {\Bbb Z}w_1$.
In particular, if $\langle v(E),w_1 \rangle=1$,
then $\langle v(E)^2 \rangle=0$.
\end{enumerate}
\end{lem}

\begin{proof}
Let $E_1$ be a stable object with $v(E_1) \in {\Bbb Z}w_1$.
We first note that
$\Hom(E_1,E) \ne 0$ or $\Hom(E,E_1) \ne 0$ 
implies that
$E \cong E_1$. 

Since $\dim X_1>0$,
we have a closed point $W$ of $X_1$ and an irreducible 
$\alpha$-twisted sheaf $F$ on $W$ such that  
$\Phi_{X_1 \to X}^{{\bf E}}(F) \ne E$.
Then
\begin{align*}
\langle v(E),n w_1 \rangle
=\langle v(E),v(\Phi_{X_1 \to X}^{{\bf E}}(F)) \rangle
=-\chi(E,\Phi_{X_1 \to X}^{{\bf E}}(F)) \geq 0
\end{align*}
with $n=\dim_{\frak k} F$. 
Thus (1) holds.
We also have the first part of (3).

Assume that $E \ne \Phi_{X_1 \to X}^{{\bf E}}(F)$
for any irreducible 
$\alpha$-twisted sheaf $F$ on a closed point $W$ of $X_1$.
Then $\Phi_{X \to X_1}^{{\bf E}^{\vee}}(E)[1]$ is a locally free 
$\alpha$-twisted sheaf 
of rank $\langle v(E),w_1 \rangle$ on $X_1$. 
\begin{NB2}
Let $G$ be a locally free $\alpha$-twisted sheaf on $X_1$. 
$\Phi_{X \to X_1}^{{\bf E}^{\vee}}(E)\otimes G^{\vee}$
is untwisted, and
$H^i(\Phi_{X \to X_1}^{{\bf E}^{\vee}}(E)\otimes G^{\vee} \otimes W)
=\Hom(G_{|W},{\bf E}_{|X \times W}[i])=0$ for $i=0,2$.
\end{NB2}
If $\langle v(E),w_1 \rangle=0$, then we have 
$\Phi_{X \to X_1}^{{\bf E}^{\vee}}(E)[1]=0$,
which implies that $E=0$. Therefore
$E \cong \Phi_{X_1 \to X}^{{\bf E}}(F)$
for an irreducible 
$\alpha$-twisted sheaf $F$ on a closed point $W$ of $X_1$.
Thus (2) holds.

If $\langle v(E),w_1 \rangle=1$, then
$\Phi_{X \to X_1}^{{\bf E}^{\vee}}(E)[1]$ is a line bundle on $X_1$.
Hence $\langle v(E)^2 \rangle=
\langle v(\Phi_{X \to X_1}^{{\bf E}^{\vee}}(E)[1])^2 \rangle=0$.
Thus the second part of (3) follows.
\end{proof}
\end{NB}


By the above lemma, 
we are interested in a (semi-)stable object $E$ 
with $\langle v(E),w_1 \rangle=1$.
We have the following descriptions of such objects.

\begin{lem}\label{lem:classification}
Let $E$ be a semi-stable object with $\langle v(E),w_1 \rangle=1$.
Then $E$ is $S$-equivalent to $E_0 \oplus \oplus_{i=1}^s E_i$,
where $E_0$ is a stable object with $\langle v(E_0)^2 \rangle=0$
and $\langle v(E_0),w_1 \rangle=1$, and
$E_i$ ($i>0$) are stable objects with $v(E_i) \in {\Bbb Z}w_1$. 
\end{lem}

\begin{prop}\label{prop:moduli-on-wall}
Assume that $\omega_{\pm} \in {\Bbb Q}_{>0}H$ 
are sufficiently close to $\omega$ and 
$(\omega_-^2)<(\omega^2)<(\omega_+^2)$.
For $v$, assume that
$\langle v,w_1 \rangle=1$ and $\deg_{G}(w_1)>0$.
\begin{enumerate}
\item[(1)]
There are fine moduli schemes 
$M_{(\beta,\omega_{\pm})}(v)$, which are isomorphic to
$\Pic^0(X_1) \times \Hilb_{X_1}^{\langle v^2 \rangle/2}$.
\item[(2)] 
The universal families on $X \times M_{(\beta,\omega_{\pm})}(v)$
are the simple complexes in \cite[Thm. 4.9]{YY}.
\end{enumerate}
\end{prop} 

\begin{proof}
Let $\phi_{\pm}$ be the phase function of $Z_{(\beta,\omega_{\pm})}$.
Since $\langle v,w_1 \rangle=1$, $X_1$ is a fine moduli space, i.e.,
${\bf E}$ is a coherent sheaf.
We first assume that $\phi_{\pm}(w_1)<\phi_{\pm}(v)$.
We shall show that $E \in {\cal M}_{(\beta,\omega_\pm)}(v)$ 
if and only if
$\Phi_{X \to X_1}^{{\bf E}^{\vee}}(E)[1]$
is a torsion free sheaf of rank 1.
Since $(\Phi_{X \to X_1}^{{\bf E}^{\vee}}(E)) 
\otimes _{\frak k} \overline{\frak k}=
\Phi_{X_{\overline{\frak k}} \to (X_1)_{\overline{\frak k}}}
^{({\bf E} \otimes_{\frak k} {\overline{\frak k}})^{\vee}}
(E \otimes_{\frak k} {\overline{\frak k}})$,
we may assume that
${\frak k}$ is algebraically closed. 
In this case, for $E \in {\cal M}_{(\beta,\omega_{\pm})}(v)$ and a 
point $x_1$ of $X_1$,
$\Ext^2({\bf E}_{|X \times \{x_1\}},E)
=\Hom(E,{\bf E}_{|X \times \{ x_1 \}})^{\vee}=0$.
If $\psi:E \to {\bf E}_{|X \times \{ x_1 \}}$ is a non-trivial
morphism, then 
$\psi$ is surjective and $\ker \psi$ is semi-stable.
Since $\Hom({\bf E}_{|X \times \{ x_1 \}},{\bf E}_{|X \times \{ x_1' \}}) 
\ne 0$
if and only if $x_1=x_1'$,
we see that $\Hom(E,{\bf E}_{|X \times \{ x_1 \}})=0$
except for finitely many points of $X_1$.
Hence $\Phi_{X \to X_1}^{{\bf E}^{\vee}}(E)[1]$ is a torsion free sheaf
of rank 1. Replacing the universal family ${\bf E}$,
we may assume that $c_1(\Phi_{X \to X_1}^{{\bf E}^{\vee}}(E)[1])=0$.
Thus $\Phi_{X \to X_1}^{{\bf E}^{\vee}}(E)[1]=I_Z \otimes L$,
$L \in \Pic^0(X_1)$ and $I_Z \in \Hilb_{X_1}^{\langle v^2 \rangle/2}$.

Conversely for a torsion free sheaf 
$I_Z \otimes L$, we shall prove the stability of
$\Phi_{X_1 \to X}^{{\bf E}}(I_Z \otimes L)[1]$.
We note that 
$\Phi_{X_1 \to X}^{{\bf E}}(L)[1] \in {\cal M}_{(\beta,\omega)}(w_0)$
where $w_0$ is a Mukai vector such that $\langle w_0^2 \rangle=0$ and
$\langle w_0,w_1 \rangle=1$,
and $\Phi_{X_1 \to X}^{{\bf E}}({\cal O}_Z)
={\bf E}_{|X \times Z}$ is a semi-stable object
with $v(\Phi_{X_1 \to X}^{{\bf E}}({\cal O}_Z))=nw_1$,
$n=\dim_{\frak k} {\cal O}_Z$.
Hence we have an exact sequence in ${\frak A}$:
\begin{align*}
0 \to \Phi_{X_1 \to X}^{{\bf E}}({\cal O}_Z)
\to \Phi_{X_1 \to X}^{{\bf E}}(I_Z \otimes L)[1] \to
\Phi_{X_1 \to X}^{{\bf E}}(L)[1] \to 0.
\end{align*}
Let $E_1$ be a stable quotient object of 
$\Phi_{X_1 \to X}^{{\bf E}}(I_Z \otimes L)[1]$
such that $\phi_{\pm}(v)>\phi_{\pm}(E_1)$.
Then $E_1$ is $\sigma_{(\beta,\omega)}$-semi-stable
and $\phi(E_1)=\phi(v)$.
By Lemma~\ref{lem:classification},
$v(E_1)=n_0 w_0+n_1 w_1$
with $n_0=0$ or $1$.
Since $\phi_\pm(w_1)<\phi_\pm(v)$, we see that
$n_0=0$. By Lemma~\ref{lem:stable-on-wall} (2),
$E_1=\Phi_{X_1 \to X}^{{\bf E}}({\frak k}_{x_1})$, $x_1 \in X_1$.
Since $\Hom(\Phi_{X_1 \to X}^{{\bf E}}(I_Z \otimes L)[1],
E_1)=
\Hom(I_Z \otimes L[1],F)=0$,
we get a contradiction.
Therefore $\Phi_{X_1 \to X}^{{\bf E}}(I_Z \otimes L)[1]$
is semi-stable.
By Lemma~\ref{lem:classification}, it is easy to see that
${\cal M}_{(\beta,\omega_\pm)}(v)$ consists of stable objects.
Hence
we get that
$\Phi_{X \to X_1}^{{\bf E}^{\vee}[1]}$ induces an isomorphism
$M_{(\beta,\omega_{\pm})}(v) \to 
\Pic^0(X_1) \times \Hilb_{X_1}^{\langle v^2 \rangle/2}$.
Thus (1) holds.

We next assume that $\phi_{\pm}(w_1)>\phi_{\pm}(v)$.
In this case, we see that 
$\Phi_{X \to X_1}^{{\bf E}}(E^{\vee})[1]$
is a torsion free sheaf of rank 1.
By this correspondence, we get (1).

The relation with \cite[Thm. 4.9]{YY} follows from
the proof of \cite[Thm. 4.9]{YY}.
\end{proof}

\begin{cor}\label{cor:moduli-on-wall}
Under the same assumption of Proposition~\ref{prop:moduli-on-wall},
we have an isomorphism 
$M_{(\beta,\omega_+)}(v) \to M_{(\beta,\omega_-)}(v)$
by a contravariant Fourier-Mukai functor.
\end{cor}

\begin{proof}
Assume that $\phi_\pm(w_1)<\phi_\pm(v)$. Then
the isomorphism 
$M_{(\beta,\omega_\pm)}(v) \to M_{(\beta,\omega_\mp)}(v)$
is given by
\begin{align*}
M_{(\beta,\omega_\pm)}(v) \ni E 
\mapsto 
\Phi_{X_1 \to X}^{{\bf E}[1]}(\Phi_{X \to X_1}^{{\bf E}^{\vee}[1]}(E)^{\vee}).
\end{align*}
\end{proof}

\begin{NB}
For an exact sequence
$0 \to {\cal O}_Z \to I_Z \otimes L[1] \to L[1] \to 0$
in ${\frak A}_{(0,\omega')}$,
we have an exact sequence
$0 \to L^{\vee}[1] \to (I_Z \otimes L)^{\vee}[1] 
\to {\cal O}_Z^{\vee}[2] \to 0$
in ${\frak A}_{(0,\omega')}$.
\end{NB}

\begin{NB}
Let $W$ be a closed point of $X$,
that is, ${\cal O}_W$ is a field.
For a semi-stable object $E$ with respect to $\omega$,
if $\psi:{\bf E}_{|X \times W} \to E$ is a non-trivial
morphism, then 
$\psi$ is injective and $\coker \psi$ is semi-stable.
Since $\Hom({\bf E}_{|X \times W},{\bf E}_{|X \times W'}) \ne 0$
if and only if $W=W'$,
we see that $\Hom({\bf E}_{|X \times W},E)=0$
except for finitely many closed points of $X$.
By the semi-stability of $E$ with respect to $\omega_\pm$,
we also have $\Hom(E,{\bf E}_{|X \times W})=0$
for any closed point of $X$.
Therefore $F:=\Phi_{X \to X_1}^{{\bf E}^{\vee}[1]}(E)$
is a torsion free sheaf of rank 1.
Then $E \otimes_{\frak k} \overline{\frak k}=
\Phi_{X_1 \to X}^{{\bf E}[1]}(F \otimes_{\frak k} \overline{\frak k} )$
is a semi-stable object with respect to $\omega_\pm$.  
\end{NB}

\begin{NB}
Let $E$ be a semi-stable object with respect to
$\omega_\pm$.
Assume that $\phi(w_1)=\phi(E)$ and
$\phi_\pm(w_1)<\phi_\pm(E)$.
Then $E \otimes_{\frak k} \overline{\frak k}$ is semi-stable.
We note that $\Ext^2({\bf E}_{|X \times W},E)=
\Hom(E,{\bf E}_{|X \times W})^{\vee}=0$ for all closed point $W$
of $X_1$
and $\Hom({\bf E}_{|X \times W},E)=0$ except for 
finitely many closed points $W$ of $X_1$.
Hence $\Phi_{X \to X_1}^{{\bf E}^{\vee}[1]}(E)$ is a torsion free
sheaf.
Assume that $E \otimes_{\frak k} \overline{\frak k}$ 
is not semi-stable with respect to
$\omega$
and take a decomposition
$$
0 \to E_1 \to E\otimes_{\frak k} \overline{\frak k} \to E_2 \to 0
$$
such that 
\begin{equation*}
\begin{split}
\min \{\phi(E_1/F)| F \subset E_1  \}> &\phi(E)=\phi(w_1)\\
\max \{\phi(F)| F \subset E_2 \} \leq & \phi(E)=\phi(w_1).
\end{split}
\end{equation*}
Obviously this decomposition is unique.
Then $\Hom({\bf E}_{|X \times W},E_2)=0$ except for finitely many
closed point $W$ of $X_1$ and
$\Ext^2({\bf E}_{|X \times W},E_1)=0$ for all closed point $W$ of
$X_1$.
Hence $\WIT_1$ holds for $E_1,E_2$ and
we have an exact sequence
$$
0 \to \Phi_{X \to X_1}^{{\bf E}^{\vee}[1]}(E_1) \to
\Phi_{X \to X_1}^{{\bf E}^{\vee}[1]}
(E \otimes_{\frak k} \overline{\frak k}) \to
\Phi_{X \to X_1}^{{\bf E}^{\vee}[1]}(E_2) \to 0.
$$
Then there are coherent sheaves $F_i$, $i=1,2$ over ${\frak k}$ 
such that $F_i \otimes_{\frak k} \overline{\frak k}=
\Phi_{X \to X_1}^{{\bf E}^{\vee}[1]}(E_i)$.
Hence $E_i=\Phi_{X_1 \to X}^{{\bf E}[1]}(F_i) 
\otimes_{\frak k} \overline{\frak k}$
and the decomposition is defined over ${\frak k}$.
Therefore $E_1=0$ and $E \otimes_{\frak k} \overline{\frak k}$
is semi-stable with respect to $\omega$.
Assume that there is an exact sequence
\begin{equation}\label{eq:HNF_k}
0 \to E_1 \to E \otimes_{\frak k} \overline{\frak k} \to E_2 \to 0
\end{equation}
such that $E_1$ is a semi-stable object with
$\phi_\pm(E_1)>\phi_\pm(E) (> \phi_\pm(w_1))$ and 
$\max\{\phi_\pm(F) \mid F \subset E_2 \}<\phi_\pm(E_1)$.
Since $\omega_\pm$ is sufficiently close to
$\omega$, $\phi(E_1)=\phi(E)$ and $E_1, E_2$ are semi-stable objects
with respect to $\omega$.
Then $\WIT_1$ holds for $E_1$ and $E_2$.
By the same arguments as above, we see that
the exact sequence \eqref{eq:HNF_k} is defined over $k$, which is a
contradiction. 
Therefore $E \otimes_{\frak k} \overline{\frak k}$ is semi-stable
with respect to $\omega_\pm$.

Assume that $\phi(w_1)=\phi(E)$ and $\phi_\pm(w_1)>\phi_\pm(E)$.
Then $\Hom(E,{\bf E}_{|X \times W})=0$ except for
finitely many closed point $W$ of $X_1$ and
$\Ext^2(E,{\bf E}_{|X \times W})=0$ for all closed point $W$
of $X_1$. Hence $\Phi_{X \to X_1}^{{\bf E}[1]}(E^{\vee})$
is a torsion free sheaf.
Assume that $E \otimes_{\frak k} \overline{\frak k}$ 
is not semi-stable with respect to
$\omega$
and take a decomposition
$$
0 \to E_1 \to E\otimes_{\frak k} \overline{\frak k} \to E_2 \to 0
$$
such that 
\begin{equation*}
\begin{split}
\min \{\phi(E_1/F) \mid F \subset E_1  \} \geq &\phi(E)=\phi(w_1)\\
\max \{\phi(F) \mid F \subset E_2 \} < & \phi(E)=\phi(w_1).
\end{split}
\end{equation*}
Then $\WIT_1$ holds for $E_1^{\vee}, E_2^{\vee}$
and we have an exact sequence
$$
0 \to \Phi_{X \to X_1}^{{\bf E}[1]}(E_2^{\vee}) \to
\Phi_{X \to X_1}^{{\bf E}^{\vee}[1]}(E^{\vee} 
\otimes_{\frak k} \overline{\frak k}) \to
\Phi_{X \to X_1}^{{\bf E}^{\vee}[1]}(E_1^{\vee}) \to 0.
$$
In the same way as in the above, we see that 
$E\otimes_{\frak k} \overline{\frak k}$ is semi-stable with respect to
$\omega$. Then we also see that
$E \otimes_{\frak k} \overline{\frak k}$ is semi-stable with respect to
$\omega_\pm$.
 
\end{NB}

\subsection{Twisted relative Fourier-Mukai transforms}
\label{subsect:relative-FM}

Let $S$ be a scheme of finite type over ${\Bbb Z}$ and 
$f:{\cal X} \to S$ a family of abelian surfaces such that
there is a family of Cartier divisors ${\cal H}$ which is relatively ample,
there is a section of $f$, and
there is a family of $\beta \in \NS({\cal X}/S)_{\Bbb Q}$
of algebraic classes.
Let $\varrho_{\cal X}$ 
be the cohomology class represented by a section of $f$. 
Then $\Pic^0_{{\cal X}/S} \to S$ is a smooth morphism
by \cite[Prop. 6.7]{MFK}.
 
For $w=(r,D,a) \in {\Bbb Z} \oplus \NS({\cal X}/S) \oplus 
{\Bbb Z}\varrho_{\cal X}$ and $\gamma \in \NS({\cal X}/S)_{\Bbb Q}$,
we have the relative moduli scheme
$M_{({\cal X},{\cal H})/S}^\gamma(w) \to S$ of $\gamma_s$-twisted semi-stable
sheaves $E$ on ${\cal X}_s$ with $v(E)=w_s$.
We shall recall the construction of a universal family as a twisted sheaf.
We note that $M:=M_{({\cal X},{\cal H})/S}^\gamma(w)$ 
is a GIT-quotient of an open subscheme
$Q$ of $\Quot_{V \otimes_S {\cal G}/{\cal X}/S}$ by
$G=\PGL(V)$, where $V={\cal O}_S^{\oplus N}$ and
${\cal G}$ is a locally free sheaf on ${\cal X}$.
By \cite[Prop. 6.4]{Ma}, $Q \to M$ is a principal $G$-bundle, 
that is,
$Q \times G \cong Q \times_M Q$.
Let 
$\xi:V \otimes_S {\cal O}_Q \otimes_S {\cal G} \to {\cal Q}$
be the universal quotient.
Since $Q \to M$ is a principal
$G$-bundle,
we have morphisms
$\iota_i:U_i \to Q$ $(i \in I)$ such that
$\varphi_i:U_i \to Q \to M$ gives 
an \'{e}tale covering. Thus
for $U:=\coprod_{i \in I} U_i$,
$U \to M$ is \'{e}tale and
surjective. 
We set $U_{ij}:=U_i \times_M U_j$ and
$U_{ijk}:=U_i \times_M U_j \times_M U_k$.
\begin{NB}
Then $\iota_i,\iota_j,\iota_k$ induces
morphisms $\iota_{ij}:U_{ij} \to Q$ and
$\iota_{ijk}:U_{ijk} \to Q$.
\end{NB}
Let $p_{ij}:U_{ij} \to U_i$ and 
$q_{ij}:U_{ij} \to U_j$ be the first and the second projections.
Let
$\varphi_{ij}:U_{ij} \to Q \to M$ be
the morphism defined by $\varphi_{ij}=\varphi_i \circ p_{ij}
=\varphi_j \circ q_{ij}$. 
We set ${\cal E}_i:=(\iota_i \times 1_{\cal X})^*({\cal Q})$.

Since the $G$-action on $Q$ induces 
an isomorphism 
$\mu:Q \times G \to Q \times_M Q$,
we have a morphism 
$g_{ij}:U_{ij} \to G$ such that for the morphism
$(\iota_i \circ p_{ij}) \times g_{ij}:U_{ij} \to Q \times_S G$,
the equality $\mu \circ ((\iota_i \circ p_{ij}) \times g_{ij})
=(\iota_i \circ p_{ij}) \times (\iota_j \circ q_{ij})$ holds.
\begin{NB}
\begin{equation}
\begin{CD}
U_{ij} @>{(\iota_i \circ p_{ij}) \times g_{ij}}>> Q \times G\\
@| @VV{\mu}V\\
U_{ij} @>>{(\iota_i \circ p_{ij}) \times (\iota_j \circ q_{ij})}> 
Q \times_M Q 
\end{CD}
\end{equation}
\end{NB}
Replacing $U$ by an \'{e}tale covering $U' \to U$,
we may assume that $g_{ij}:U_{ij} \to G$ is lifted to
a morphism $\phi_{ij}:U_{ij} \to \GL(V)$
(\cite[Lem. 2.19]{Milne}).
Then we have a commutative diagram
\begin{equation}
\begin{CD}
V \otimes {\cal O}_{U_{ij}} \otimes_S {\cal G} 
@>{(g_{ij} \cdot  (\iota_i \circ p_{ij}))^*(\xi)}>> 
(g_{ij} \cdot  (\iota_i \circ p_{ij}))^*({\cal Q})\\
@V{\phi_{ij}}VV @VVV\\
V \otimes {\cal O}_{U_{ij}} \otimes_S {\cal G}
@>>{(\iota_i \circ p_{ij})^*(\xi)}> 
(\iota_i \circ p_{ij})^*({\cal Q})\\
\end{CD},
\end{equation}
where $g_{ij} \cdot :Q(U_{ij}) \to Q(U_{ij})$ is the
multiplication by $g_{ij}$ on the set of $U_{ij}$-valued points
of $Q$.
Since $g_{ij} \cdot (\iota_i \circ p_{ij})
=\iota_j \circ q_{ij}$,
we have isomorphisms
$\phi_{ij}':(p_{ij} \times 1_{\cal X})^*({\cal E}_i) \to 
(q_{ij}\times 1_{\cal X})^*({\cal E}_j)$.
Since $\phi_{ki} \circ \phi_{jk} \circ \phi_{ij}
=\alpha_{ijk} \id_{V \otimes {\cal O}_{U_{ijk}}}$,
$\alpha_{ijk} \in H^0(U_{ijk},{\cal O}_{U_{ijk}}^{\times})$
and $(\phi_{ki}' \circ \phi_{jk}' \circ \phi_{ij}') \circ \xi
=\xi \circ (\phi_{ki} \circ \phi_{jk} \circ \phi_{ij})$,
we have 
$\phi_{ki}' \circ \phi_{jk}' \circ \phi_{ij}'=
\alpha_{ijk} \id_{({\cal E}_i)_{jk}}$,
where $({\cal E}_i)_{jk}$ is the pull-back of ${\cal E}_i$ to
$U_{ijk} \times_S {\cal X}$.
Thus ${\cal E}:=(\{{\cal E}_i\}, \{\phi_{ij}' \})$
defines an $(\alpha \times 1_{\cal X})$-twisted sheaf 
on $M \times_S {\cal X}$, where 
$\alpha=\{ \alpha_{ijk} \}$ is a \v{C}ech 2-cocycle of
${\cal O}_M^{\times}$ with respect to
the \'{e}tale covering $U \to M$.
Then ${\cal E}$ gives a desired universal family.

\begin{rem}
If there is a relative Cartier divisor $\xi$ such that
$\gcd(r,a,(\xi,D))=1$, then there is a universal family
${\cal E}$ on $X \times_S M_{({\cal X},{\cal H})/S}^\gamma(w)$.
\end{rem}

Let $(R,{\frak m})$ be an artinian local ring with the residue field
$\kappa$ and $I$ an ideal of $R$
with ${\frak m}I=0$.
We set $T=\Spec(R)$ and $T'=\Spec(R/I)$.
For a family of stable sheaves $E$ on $X \times_S T'$,
$\det E$ is unobstructed, hence the obstruction
for the lifting to $X \times_S T$ belongs to
$\Ext^2(E_\kappa,E_\kappa)_0=0$, where
$\Ext^2(E_\kappa,E_\kappa)_0$ is the kernel of the trace map
$\Ext^2(E_\kappa,E_\kappa) \to 
H^2({\cal X}_{\kappa},{\cal O}_{{\cal X}_{\kappa}})$.
Therefore $M_{({\cal X},{\cal H})/S}^\gamma(w) \to S$ is smooth.

If $\langle w^2 \rangle=0$, then
${\cal X}':=M_{({\cal X},{\cal H})/S}^\gamma(w)$ 
is a smooth family of projective surfaces
parametrizing semihomogeneous sheaves.  
Then $\Phi_{{\cal X}_s \to {\cal X}'_s}^{{\cal E}_s^{\vee}}:
{\bf D}({\cal X}_s)
\to {\bf D}^{\alpha^{-1}}({\cal X}'_s)$ is an equivalences
for any point $s \in S$.
Indeed
since the Kodaira-Spencer map for $M_{({\cal X},{\cal H})/S}^\gamma(w)$
is isomorphic, Bridgeland's criterion for equivalence
works over any field.  
\begin{NB}
For an integral functor $F$,
let $G$ and $H$ be left and right adjoint.
Then we have morphisms $1 \to H F$, $F H \to 1$ and
$1 \to FG$, $GF \to 1$.
In order to show the isomorphisms,
it is sufficient to check it over algebraically closed filds.   
\end{NB}

\begin{prop}\label{prop:FM-isom} 
Assume that 
$\Phi_{{\cal X}_s \to {\cal X}'_s}^{{\cal E}_s^{\vee}}$
induces an isomorphism
${\cal M}_{(\beta_s,\omega_s)}(v) \to 
{\cal M}_{{\cal H}_s'}^{G'_s}(v')^{ss}$
for all $s \in S$, where $v' \in {\Bbb Z} \oplus \NS({\cal X}'/S) 
\oplus {\Bbb Z}\varrho_{{\cal X}'}$ and
$G'$ is a locally free $\alpha^{-1}$-twisted sheaf on
${\cal X}'$.
Then 
the relative moduli stack ${\cal M}_{(\beta,\omega)}(v)$
is a quotient stack $[Q^{ss}/\GL(V)]$, where $Q^{ss}$ is
an open subscheme of a quot-scheme 
$\Quot_{V \otimes_S G'/{\cal X}'/S}$ and
we have a relative moduli scheme
$M_{(\beta,\omega)}(v) \to S$ as a projective scheme
over $S$.
\end{prop}

For a $S$-scheme $\phi:T \to S$ and
a family of complexes
$E$ on ${\cal X} \times_S T$ such that
$E_{|{\cal X}_{\phi(t)}}$ is 
$\sigma_{(\beta_{\phi(t)},\omega_{\phi(t)})}$-semi-stable
for all $t \in T$,
$\Phi_{{\cal X}_T \to {\cal X}'_T}^{{\cal E}_T^{\vee}}(E)$
is a family of $G'$-twisted semi-stable $\alpha^{-1}$-twisted sheaves.
Hence we have a morphism
$T \to M_{({\cal X}',{\cal H}')/S}^{G'}(v')$.

Under the same assumption,
the $\sigma_{(\beta_{\phi(t)},\omega_{\phi(t)})}$-semi-stability
is an open condition:
Indeed for a strictly perfect complex $E$ (which is the same as relatively 
perfect complex by the projectivity of ${\cal X}_T \to T$),
$\Phi_{{\cal X}_T \to {\cal X}_T'}^{{\cal E}_T^{\vee}}(E)$ 
is a strictly perfect complex
on ${\cal X}'_T \to T$.
$$
\{ t \in T \mid
\Phi_{{\cal X}_T \to {\cal X}_T'}^{{\cal E}_T^{\vee}}(E) \otimes_T k(t)
\in {\cal M}_{({\cal X}',{\cal H}')/S}^{G'}(v')\}
$$
is an open subset of $T$.

\begin{NB}
For a complex $E$ of locally free sheaves,
$H^i(E_t)=0$ means $H^i(E)=0$ in a neighborhood of $t$.
Thus $E_t \in \Coh(X_t)$ is an open condition. 
\end{NB}

\begin{rem}
The same claims hold for a family
of polarized $K3$ surfaces.
\end{rem}

\begin{NB}
For $U_{ij}:=U_i \times_{M_{({\cal X},{\cal H})/S}(v)} U_j$,
${\cal E}_{ij}$ denotes the pull-back of
${\cal E}_i$ and
${\cal E}_{ji}$ the pull-back of
${\cal E}_j$.

 Then $\Hom_{p_{U_{ij}}}({\cal E}_{ij},{\cal E}_{ji})
\cong L_{ij}$, $L_{ij} \in \Pic(U_{ij})$.
Replacing $U=\coprod_i U_i$ be a refinement 
$U' \to U$,
we may assume that 
$L_{ij} \cong {\cal O}_{U_{ij}}$
(cf. use \cite[Lem. 2.19]{Milne}).
Then we have isomorphims
$\phi_{ij}:{\cal E}_{ij} \to {\cal E}_{ji}$
such that $\phi_{ki}\phi_{jk}\phi_{ij}=
\alpha_{ijk}\id_{U_{ijk} \times_S {\cal X}}$,
where
$\alpha_{ijk} \in {\cal O}_{U_{ijk}}^{\times}$.
Thus $(\{ {\cal E}_i \},\{\phi_{ij} \})$
defines a twisted sheaf.

Let ${\cal O}_{{\Bbb P}(W)}(1)$ be the tautological line bundle
on ${\Bbb P}(W)$.
It is naturally $\GL(V)$-linearized.
Then $W \otimes {\cal O}_{{\Bbb P}(W)}(-1)$ is $\PGL(V)$-linearized.
\end{NB}


\subsection{Estimates of the wall-crossing terms}
\label{subsect:abel:codim}

We shall estimate the dimension of the wall-crossing terms 
by computing the weighted numbers.
For this purpose, we quote the following result of Lang and Weil \cite{L-W}.
\begin{prop}[Lang-Weil]\label{prop:L-W}
Let $Z$ be an algebraic set over ${\Bbb F}_q$ 
with $\dim Z=k$. Then
$$
\#Z({\Bbb F}_{q^n})=N q^{nk}+O(q^{n(k-1/2)}),
$$
where $N$ is the number of irreducible components of dimension $k$.
In particular,
$$
\frac{1}{q^{n\dim {\cal M}_H^\beta(v)^{ss}}}
\left(\sum_{E \in {\cal M}_H^\beta(v)^{ss}({\Bbb F}_{q^n})}\frac{1}{\# \Aut(E)}
\right)
$$
is bounded by a constant which is independent of $n$.
\end{prop}

\begin{lem}\label{lem:twisted-general}
For a Mukai vector $v$, we write $v=lv'$, where $v'$ is primitive and $l>0$.
Assume that $(H,\beta)$ is general with respect to $v$, 
that is, for any $E \in {\cal M}_H^\beta(v)^{ss}$, 
the $S$-equivalence class $\oplus_i E_i$ of $E$ 
satisfies $v(E_i) \in {\Bbb Q}v$. 
Then
\begin{align*}
\dim {\cal M}_H^\beta(v)^{ss}=
\begin{cases}
\langle v^2 \rangle+1,& \langle v^2 \rangle>0,\\
\langle v^2 \rangle+l,& \langle v^2 \rangle=0,\\
\langle v^2 \rangle+l^2,& \langle {v'}^2 \rangle=-2.
\end{cases}
\end{align*}
\end{lem}

\begin{proof}
The proof is similar to those for
\cite[Lem. 3.2, Lem. 3.3]{tori}.
\end{proof}

\begin{rem}
For an abelian surface, stronger results hold (\cite[Lem. 3.8]{tori}):
$$
\dim {\cal M}_H(v)^{\text{$\mu$-$ss$}}=\langle v^2 \rangle+1
$$
for $\langle v^2 \rangle>0$.
\end{rem}

\begin{lem}\label{lem:codim}
Let $v=\sum_{i=1}^s v_i$ be the decomposition in
\eqref{eq:wall:decomp} with \eqref{eq:v_i(eta)}.
Assume that $\langle v_i^2 \rangle \geq 0$ for all $i$
(e.g., $X$ is an abelian surface).
\begin{enumerate}
\item[(1)]
For $i \ne j$,
(i) $\langle v_i,v_j \rangle \geq 3$, or
(ii) $\langle v_i,v_j \rangle =1,2$ and 
$\langle v_i^2 \rangle=0$ or $\langle v_j^2 \rangle=0$.

\item[(2)]
Assume that $(H,\beta)$ is general with respect to $v$ and all $v_i$. 
Then
\begin{enumerate}
\item
$$\dim {\cal M}_H^\beta(v)^{ss}-\sum_{i>j}\langle v_i,v_j \rangle-
\sum_i \dim {\cal M}_H^\beta(v_i)^{ss} \geq 1
$$
unless $s=2$, 
$\{v_1,v_2 \}=\{lu_1, u_2 \}$, $\langle u_1^2 \rangle=0$,
$\langle u_1,u_2 \rangle=1$. 

\item
If 
\begin{align}\label{eq:codim:b}
\dim {\cal M}_H^\beta(v)^{ss}-\sum_{i>j}\langle v_i,v_j \rangle-
\sum_i \dim {\cal M}_H^\beta(v_i)^{ss}=1,
\end{align}
then one of the following conditions holds:
\begin{enumerate}
\item[(b1)]
$s=2$, $\{v_1,v_2 \}=\{u_1, u_2 \}$, $\langle u_1^2 \rangle=0$,
$\langle u_1,u_2 \rangle=2$, $u_1$ is primitive,

\item[(b2)]
 $s=2$, $\{v_1,v_2 \}=\{2u_1, 2u_2 \}$, $\langle u_1^2 \rangle=
\langle u_2^2 \rangle=0$,
$\langle u_1,u_2 \rangle=1$,

\item[(b3)]
$s=3$, $\langle v_i^2 \rangle=0$, $\langle v_i,v_j \rangle=1$
and $\langle v^2 \rangle=6$,

\item[(b4)]
$s=3$, $\{v_1,v_2,v_3 \}=\{u_1,u_2,u_1+u_2 \}$,
$\langle u_1^2 \rangle=\langle u_2^2 \rangle=0$,
$\langle u_1,u_2 \rangle=1$ and $v=2(u_1+u_2)$.
\end{enumerate}
\end{enumerate}
\end{enumerate}
\end{lem}

\begin{proof}
(1)
Since $(d_j r_i-d_i r_j)(d_j a_i-d_i a_j)>0$,
by \eqref{eq:v_1,v_2} and the Hodge index theorem, we have
\begin{equation*}
\langle v_i,v_j \rangle > 
\frac{1}{2}\left( 
 \frac{d_2}{d_1}\langle v_i^2 \rangle+\frac{d_1}{d_2}\langle v_j^2 \rangle
\right) 
\geq 
\sqrt{\langle v_i^2 \rangle \langle v_j^2 \rangle} \geq 0. 
\end{equation*}
If $\langle v_i^2 \rangle, \langle v_j^2 \rangle>0$, then
$\langle v_i,v_j \rangle \geq 3$.
Therefore (1) holds.

(2)
We write $v_i=l_i v_i'$, $v_i'$ is primitive.
\begin{equation}\label{eq:codim-estimate}
\begin{split}
\dim {\cal M}_H^\beta(v)^{ss}-\sum_{i>j}\langle v_i,v_j \rangle-
\sum_i \dim {\cal M}_H^\beta(v_i)^{ss}=&
\sum_{i<j}\langle v_i,v_j \rangle-
\sum_i (\dim {\cal M}_H^\beta(v_i)^{ss}-\langle v_i^2 \rangle)+1\\
\geq &\sum_{i<j}l_i l_j \langle v_i',v_j' \rangle-\sum_i l_i+1\\
\geq & \sum_{i<j}l_i l_j -\sum_i l_i+1.
\end{split}
\end{equation}
If the equality holds, then
(i) $\langle v_i',v_j' \rangle=1$ for all $i<j$ and (ii)
for all $i$, $l_i=1$ or $\langle v_i^2 \rangle=0$.  
If $s \geq 3$, then
\begin{equation*}
\sum_{i<j}l_i l_j-\sum_i l_i+1  
>\sum_{i<s} l_i l_s-\sum_i l_i+1 
=(\sum_{i<s}l_i)(l_s-1)-l_s+1 \geq 0.
\end{equation*}
Assume that $s=2$. Then
\begin{equation*}
\sum_{i<j}l_i l_j -\sum_i l_i+1 =(l_1-1)(l_2-1) \geq 0.
\end{equation*}
Hence if 
\begin{align*}
\dim {\cal M}_H^\beta(v)^{ss}-\sum_{i>j}\langle v_i,v_j \rangle-
\sum_i \dim {\cal M}_H^\beta(v_i)^{ss}=0,
\end{align*}
then $s=2$, $\langle v_1',v_2' \rangle=1$.
Moreover if $\langle v_i^2 \rangle>0$, then $l_i=1$.
Hence (a) holds. 

Assume that \eqref{eq:codim:b} holds.
Then by the estimate in \eqref{eq:codim-estimate}, we have
\begin{align}\label{eq:codim:b:1234}
\sum_{i<j} l_i l_j (\langle v_i' v_j' \rangle -1)=1
\quad \text{or} \quad 
\sum_{i<j} l_i l_j-\sum_i l_i +1=1.
\end{align}

In the first case,
$s=2$, $\langle v_1',v_2' \rangle=2$,
$l_1=l_2=1$ and $\langle v_1^2 \rangle \langle v_2^2 \rangle=0$.
This is the case (b1).

In the second case, $s=2$ or $3$.
If $s=2$, $l_1=l_2=2$,
$\langle v_1 \rangle=\langle v_2^2 \rangle=0$
and $\langle v_1',v_2' \rangle=1$, 
which is the case (b2).

Let us assume the second case in \eqref{eq:codim:b:1234} and $s=3$.
Then $\{v_1,v_2,v_3 \}=\{u_1,u_2,u_3 \}$ with 
$\langle u_1^2 \rangle=\langle u_2^2 \rangle=0$ and 
$\langle u_i,u_j \rangle=1$. 
Since $(\beta,\omega)$ belongs to the walls $W_{v_1},W_{v_2},W_{v_3}$,
we have from Lemma~\ref{lem:xi_v} that $u_1,u_2,u_3 \in \xi_v^{\perp}$. 
Since $(\xi_v^2)>0$,
$\langle (u_1+u_2)^2 \rangle>0$ and
$u_3-u_1-u_2 \in \xi_v^\perp \cap (u_1+u_2)^\perp$,
we have
$\langle (u_3-u_1-u_2)^2 \rangle \leq 0$ and the equality holds
if and only if $u_3=u_1+u_2$. 
On the other hand, we have 
\begin{align*}
0 \leq \langle u_3^2 \rangle
=\langle (u_1+u_2)^2 \rangle+\langle (u_3-u_1-u_2)^2 \rangle
=2+\langle (u_3-u_1-u_2)^2 \rangle.
\end{align*}
Hence $\langle (u_3-u_1-u_2)^2 \rangle=-2$ or $0$.
If $\langle (u_3-u_1-u_2)^2 \rangle=0$, then
$v=2(u_1+u_2)$ and $\langle v^2 \rangle=8$,
which is the case (b4).
If $\langle (u_3-u_1-u_2)^2 \rangle=-2$, then
$\langle v^2 \rangle=4\langle (u_1+u_2)^2 \rangle+
\langle (u_3-u_1-u_2)^2 \rangle=6$,
which is the case (b3).
\end{proof}

Assume that $X$ is an abelian surface over ${\frak k}$ 
and $(\beta,H)$ is general with respect to $v$.
We may assume that $X$ is defined over a finitely generated ring $R$ 
over ${\Bbb Z}$.
Thus there is a smooth surface $X_R$ over $R$ such that
${\frak k}$ is an extension field of the quotient field of $R$ and
$X \cong X_R \otimes_R {\frak k}$.
Lemma~\ref{lem:extend} implies that
by a suitable base change, we may assume that
all $v_1 \in A^*_{\alg}((X_R)_s)$, $s \in \Spec(R)$ 
in Definition~\ref{defn:wall:stability} (1) 
are defined over $R$.

\begin{prop}\label{prop:pic-isom}
For the family of surfaces $X_R \to R$,
let $\omega_1$ and $\omega_2$ be general members of
${\Bbb Q}_{>0}H$ which are not separated by any wall $W_{w_1}$
with $\langle v, w_1 \rangle=1$ and $\langle w_1^2 \rangle=0$.
For the relative moduli stacks 
${\cal M}_{(\beta,\omega_i)}(v)$ ($i=1,2$) over $R$,
we set
${\cal Z}_{ij}:={\cal M}_{(\beta,\omega_i)}(v) \setminus 
{\cal M}_{(\beta,\omega_j)}(v)$ 
for $\{ i,j \}=\{1,2\}$.
\begin{enumerate}
\item[(1)]
Assume ${\cal M}_{(\beta,\omega_i)}(v)$ ($i=1,2$) are isomorphic
to $[Q_i^{ss}/\GL(V_i)]$, where $Q^{ss}_i$ are
the open subschemes of quot-schemes $\Quot_{V_i \otimes_R G_i/X_i/R}$,
where $G_i$ are locally free twisted sheaves on $X_i$. 
Then we get
$\codim_{{\cal M}_{(\beta,\omega_i)}(v)}({\cal Z}_{ij}) \geq 1$
(and hence 
$\codim_{{\cal M}_{(\beta,\omega_i)}(v)_{\frak k}}
(({\cal Z}_{ij})_{\frak k}) \geq 1$).
Moreover if there is no wall $W_{w_1}$
with $\langle v, w_1 \rangle=2$ and $\langle w_1^2 \rangle=0$,
then $\codim_{{\cal M}_{(\beta,\omega_i)}(v)}({\cal Z}_{ij}) \geq 2$
(and hence 
$\codim_{{\cal M}_{(\beta,\omega_i)}(v)_{\frak k}}
(({\cal Z}_{ij})_{\frak k}) \geq 2$).
In particular,
${\cal M}_{(\beta,\omega_1)}(v) \cap {\cal M}_{(\beta,\omega_2)}(v)$ 
is an open dense substack of
each ${\cal M}_{(\beta,\omega_i)}(v)$ ($i=1,2$).
\item[(2)]
The assumption for ${\cal M}_{(\beta,\omega_i)}(v)$ hold, if 
$(\omega_i^2) \gg 0$ or $1 \gg (\omega_i^2)> 0$ or
$(\beta,\omega_i)$ is sufficiently close to a wall
$W_{w_1}$
with $\langle v, w_1 \rangle=1$ and $\langle w_1^2 \rangle=0$. 
\end{enumerate}
\end{prop}

\begin{proof}
(1)
\begin{NB}
We first assume that there is no wall $W_{w_1}$ 
with $\langle v, w_1 \rangle=1$ and $\langle w_1^2 \rangle=0$.
In this case, we may assume that
$(\omega_1^2) \gg 1$ and $(\omega_2^2) \ll 1$.
Then $M_{(\beta,\omega_1)}(v)$ parametrizes 
semi-stable sheaves on $X$ or their duals, 
and $M_{(\beta,\omega_2)}(v)$ parametrizes
semi-stable sheaves on $X'=M_H(r_0 e^\beta)$ 
via the Fourier-Mukai transform $\Phi_{X \to X'}^{{\bf E}^{\vee}[1]}$.
\end{NB}
We take a reduction over ${\Bbb F}_q$.
We note that there are closed subschemes $Z_{ij}$ of
$Q_i^{ss}$ such that ${\cal Z}_{ij}=[Z_{ij}/\GL(V_i)]$. 
Hence 
$$
\sum_{E \in {\cal Z}_{ij}({\Bbb F}_q)}\frac{1}{\# \Aut(E)}=
\frac{\# Z_{ij}({\Bbb F}_q)}{\# \GL(V_i)({\Bbb F}_q)}.
$$
\begin{NB}
Bridgeland stability is an open condition.
\end{NB}
We set $\beta=bH+\eta$, $\eta \in H^\perp$.
By Lemma~\ref{lem:wall:stability-finite},
we can find an open neighborhood $U$ of the segment
connecting $(\eta,\omega_1)$ and $(\eta,\omega_2)$
such that $U$ does not intersect $W_{w_1}$ with
$\langle v,w_1 \rangle=1$
and $\langle w_1^2 \rangle=0$.
We take a path $I_t:=(\eta_t,\omega_t)$ ($1 \leq t \leq 2$) in $U$
such that 
\begin{enumerate}
\item[(i)]   $\eta_1=\eta_2=\eta$, 
\item[(ii)]  $I_t$ is very close to the segment 
             $I_t':=(\eta_1,\omega_t), 1 \leq t \leq 2$,
\item[(iii)] $I_t \cap W_{v_1} \cap W_{v_2}=\emptyset$ 
             for any $W_{v_1} \ne W_{v_2}$ satisfying
             the conditions (a),(b),(c) in Definition~\ref{defn:wall:stability}
             and \eqref{eq:eta-fixed}.
\end{enumerate}
Assume that $(\eta_s,\omega_s)$ belongs to a wall $W$.
Then for $s_\pm$ with $s_-<s<s_+$, $s_+-s_- \ll 1$,
we have the formula \eqref{eq:7},
where $(\beta,\omega)$ and $(\beta,\omega_{\pm})$
are replaced by $(bH+\eta_s,\omega_s)$ 
and $(bH+\eta_{s_\pm},\omega_{s_\pm})$. 
Then $(H,bH+\eta_{s_\pm})$ are general with respect to $v_i$
in the sense of Lemma \ref{lem:twisted-general}. 
Indeed if $v_i=v_{i1}+v_{i2}$ with 
\begin{align*}
 v_{ij} = 
 r_{ij}e^{bH+\eta_s}+a_{ij}\varrho_X
 +d_{ij}H+D_{ij}+(d_{ij}H+D_{ij},bH+\eta_s)\varrho_X,\quad
 d_{ij}>0,\ 
  \dfrac{r_{ij}}{d_{ij}}=\dfrac{r_i}{d_i},\ 
  \dfrac{a_{ij}}{d_{ij}}=\dfrac{a_i}{d_i},
\end{align*}
then
$v_{ij}/d_{ij}-v/d$ and $v_i/d_i-v/d$ define the same wall.
By our assumption, we have 
$v_{ij}/d_{ij}-v/d \in {\Bbb Q}(v_i/d_i-v/d)$.
Hence we get 
$v_{i1}/d_{i1}-v_{i2}/d_{i2}=
(r_{i1}/d_{i1}-r_{i2}/d_{i2})w$, $w \in A^*_{\alg}(X)_{\Bbb Q}$.
Thus $r_{i1}/d_{i1}-r_{i2}/d_{i2}=0$ implies 
that $v_{i1}/d_{i1}-v_{i2}/d_{i2}=0$.
In particular, we can apply Lemma~\ref{lem:codim} 
to estimate the weighted numbers of ${\cal Z}_{ij}({\Bbb F}_{q^n})$. 
Thus by Lemma~\ref{lem:codim}, 
Proposition~\ref{prop:wt-number:stable},
Proposition~\ref{prop:L-W}
and \eqref{eq:7},
we see that
$$
\lim_{n \to \infty} \frac{1}{q^{n(\dim {\cal M}_H^\beta(v)^{ss}-1/2)}}
\left(\sum_{E \in {\cal Z}_{ij}({\Bbb F}_{q^n})}\frac{1}{\# \Aut(E)} \right)=0.
$$
By Proposition~\ref{prop:L-W},
 $\dim {\cal Z}_{ij} \leq \dim {\cal M}_H^\beta(v)^{ss}-1
=\langle v^2 \rangle$.
Moreover if there is no wall $W_{w_1}$
with $\langle v, w_1 \rangle=2$
and $\langle w_1^2 \rangle=0$,
then
$\dim {\cal Z}_{ij} \leq 
\dim {\cal M}_H^\beta(v)^{ss}-2=\langle v^2 \rangle-1$.

\begin{NB}
By Lemma \ref{lem:codim}, we see that
\begin{equation*}
\begin{split}
&\lim_{n \to \infty}\frac{1}{q^{n\dim M_H(v)^{ss}}}
 \left(
  \sum_{E \in M_{(\beta,\omega_1)}(v)^{ss}({\Bbb F}_{q^n})
  \cap M_{(\beta,\omega_2)}(v)^{ss}({\Bbb F}_{q^n}) }
  \frac{1}{\# \Aut(E)}
 \right)
\\
=&\lim_{n \to \infty}
  \frac{1}{q^{n\dim M_H(v)^{ss}}}
  \left(\sum_{E \in M_{(\beta,\omega_2)}(v)^{ss}({\Bbb F}_{q^n})}
   \frac{1}{\# \Aut(E)}\right)
=1.
\end{split}
\end{equation*}
Hence there is a ${\Bbb F}_{q^n}$-semi-stable
object with respect to $\omega_1$ and $\omega_2$.
Since $M_{(\beta,\omega_i)}(v)^{ss}$, $i=1,2$ are irreducible,
we get the claim.
\end{NB}

\begin{NB}
We next assume that there is a wall $W_{w_1}$ 
with $\langle v, w_1 \rangle=1$ and $\langle w_1^2 \rangle=0$.
If $\omega_i$ is close to a wall $W_{w_1}$, then
Proposition~\ref{prop:moduli-on-wall}
implies that
$M_{(\beta,\omega_i)}(v)$ is isomorphic to the moduli
scheme of torsion free sheaves of rank 1 on an abelian surface.
Then the same argument gives the claim. 
\end{NB}

(2)
The claim follows from Corollary \ref{cor:Toda-moduli} if $(\omega_i^2) \gg 0$.
If $1 \gg (\omega_i^2)$, then the claim follows 
from the first case by using the Fourier-Mukai transform 
$\Phi$ in subsection \ref{subsect:FM} and
Proposition \ref{prop:FM-isom}.
For the last case, the claim follows from 
Proposition \ref{prop:moduli-on-wall}.
\end{proof}

\begin{rem}\label{rem:assumption-OK}
In \cite{MYY2},
we shall show that the assumption of
Proposition \ref{prop:FM-isom} holds for any general
$(\beta,\omega)$.
Thus the assumption of ${\cal M}_{(\beta,\omega)}(v)$ holds
for any general $(\beta,\omega)$.
\end{rem}


\subsection{Application}

\begin{thm}[{\cite[Thm. 1.1]{Y:birational}}]
\label{thm:birational}
Assume that $X$ is an abelian surface over a field ${\frak k}$.
Let $\Phi_{X \to X'}^{{\bf E}^{\vee}[1]}:{\bf D}(X) \to {\bf D}(X')$
be a Fourier-Mukai transform.
Let $v \in A^*_{\alg}(X)$ be a primitive Mukai vector 
with $\langle v^2 \rangle>0$. 
Let $H$ be an ample divisor on $X$ 
which is general with respect to $v$.
We set $v':=\pm \Phi_{X \to X'}^{{\bf E}^{\vee}[1]}(v)$ and
assume that $H'$ is general with respect to $v'$.
Then there is an autoequivalence 
$\Phi_{X' \to X'}^{{\bf F}}: {\bf D}(X') \to {\bf D}(X')$ such that
for a general $E \in M_{H}(v)$,
$F:=\Phi_{X' \to X'}^{{\bf F}} \circ \Phi_{X \to X'}^{{\bf E}^{\vee}[1]}(E)$
is a stable sheaf with $v(F)=v'$
or $F^{\vee}$ is a stable sheaf with $v(F^{\vee})=v'$,
up to shift.
\end{thm}

\begin{proof}
If $\rk {\bf E}=0$, then
we can decompose 
$\Phi_{X \to X'}^{{\bf E}^{\vee}[1]}$ into a composition of 
two Fourier-Mukai transforms 
$\Phi_{X \to X'}^{{\bf E}^{\vee}[1]}=\Phi_{X'' \to X'}^{{\bf E}_1^{\vee}[1]}
\Phi_{X \to X''}^{{\bf E}_2^{\vee}[1]}$ such that
$\rk {\bf E}_i>0$ ($i=1,2$).
Thus we may assume that $\rk {\bf E}>0$. 
Then ${\bf E}_{|X \times \{x' \}}$ is a $\mu$-stable vector bundle
with respect to any polarization.
We also note that
the $\beta$-twisted semi-stability for $E$ with $v(E)=v$
does not depend on the choice of 
$\beta$, since 
$H$ is a general polarization.
We set 
\begin{align*}
\beta:=&
c_1({\bf E}_{|X \times \{x' \}})/\rk {\bf E}_{|X \times \{x' \}},
\\ 
\beta':=&
c_1({\bf E}_{|\{x \} \times X'}^{\vee})
/\rk {\bf E}_{|\{x \} \times X'}.
\end{align*}

(1)
We first assume that $d>0$.
By Corollary~\ref{cor:Toda-moduli},
$M_H^\beta(v) \cong M_{(\beta,\omega)}(v)$ for $(\omega^2) \gg 0$. 
Applying $\Phi_{X \to X'}^{{\bf E}^{\vee}[1]}$,
we have an isomorphism
$M_H^\beta(v) \cong M_{(\beta',\omega')}(v')$,
where $({\omega'}^2) \ll 1$.
We shall study the wall-crossing behavior
between $\omega'$ and $t \omega'$, $t \gg 0$.
By Proposition~\ref{prop:pic-isom},
it is sufficient to study the wall $W_{w_1}$ in 
Proposition~\ref{prop:moduli-on-wall}.
Then we have an isomorphism
$M_{(\beta',\omega'_-)}(v') \to M_{(\beta',\omega'_+)}(v')$
by a contravariant Fourier-Mukai transform.
Therefore the claim holds in this case.

(2)
We next assume that $d<0$.
We shall construct a rational map
$M_H^\beta(v) \dashrightarrow M_{(\beta,\omega)}(v)$ 
for $(\omega^2) \gg 0$.
Then the same proof of case (1) implies the claim.
We set $v:=l(r+\xi)+a \varrho_X$, $l \in {\Bbb Z}_{>0}$,
$\gcd(r,\xi)=1$.
If $r=1,2$, then
$\eta:=2 \xi/r \in \NS(X)$.
Then Corollary~\ref{cor:Toda} implies that
$E_1:=E^{\vee} \otimes L$ is a Bridgeland stable object
with $v(E_1)=v$, where $c_1(L)=\eta$.
Thus we have an isomorphism
$M_H^\beta(v) \to M_{(\beta,\omega)}(v)$ by this correspondence.
Assume that $r \geq 3$.
Let $M_H(v)^*$ be the open subset of $M_H^\beta(v)$
consisting of $\mu$-stable locally free sheaves.
Then
the proof of \cite[Lem. 3.5]{tori} implies
that
$\dim M_H^\beta(v) \setminus M_H(v)^* \leq \dim M_H^\beta(v)-1$.
By Corollary~\ref{cor:Toda},
we have a desired embedding $M_H(v)^* \to M_{(\beta,\omega)}(v)$.

(3)
If $d=0$, then the stability is preserved by
$E \mapsto \Phi_{X \to X'}^{{\bf E}^{\vee}}(E)^{\vee}$
by \cite[Thm. 2.3]{Y:twist1}.
\end{proof}

Let us explain the relation of this section
with \cite{Y:birational} and
\cite{YY}.
Let $w_1$ be a primitive isotropic Mukai vector.
For the Mukai vector $v$ in Proposition~\ref{prop:moduli-on-wall},
we set $w_2:=v-\frac{\langle v^2 \rangle}{2} w_1$.
Then $\langle w_2^2 \rangle=0$,
$\langle w_1,w_2 \rangle=1$ and $\deg_G(w_2)>0$.
We set
\begin{align*}
w_i:=r_i e^\beta+a_i \varrho_X+(d_i H+D_i+(d_i H+D_i,\beta)\varrho_X).
\end{align*}
Since $\langle w_i^2 \rangle =0$, we see that
$a_i=\frac{((d_i H+D_i)^2)}{2 r_i}$. 
Hence
\begin{equation}\label{eq:slope-behavior}
\begin{split}
(\omega^2)=&
2\frac{a_1/d_1-a_2/d_2}{r_1/d_1-r_2/d_2}\\
=& \frac{((H+D_1/d_1)^2)d_1/r_1-((H+D_2/d_2)^2)d_2/r_2}
{(r_1/d_1-r_2/d_2)}\\
=& -\frac{d_1 d_2}{r_1 r_2}(H^2)+
\frac{d_2 r_2 (D_1^2)-r_1 d_1 (D_2^2)}{r_1 r_2(r_1 d_2-r_2 d_1)}.
\end{split}
\end{equation}

Assume that $r_1<0$ and $r_2>0$. 
$(\omega_+^2)>(\omega^2)>(\omega_-^2)$ means that
$\phi_+(w_1)>\phi_+(w_2)$ and
$\phi_-(w_1)<\phi_-(w_2)$.
For $\omega_-$,
a general stable object $E$ of $v(E)=v$ 
fits in an exact sequence
$$
0 \to A[1] \to E \to B \to 0,
$$
where $A[1] \in M_{(\beta,\omega)}(w_1)$ and
$B \in M_{(\beta,\omega)}(w_2)$.
In particular, $E$ is an honest complex with
$H^{-1}(E)=A$ and $H^0(E)=B$.
On the other hand,
for $\omega_+$,
a general stable object $E'$ of $v(E')=v$ 
fits in an exact sequence
$$
0 \to B \to E' \to A \to 0,
$$
where $A[1] \in M_{(\beta,\omega)}(w_1)$ and
$B \in M_{(\beta,\omega)}(w_2)$.
Thus by crossing the wall $W_{w_1}$ from $\omega_-$ to
$\omega_+$, 
$\rk H^i(E)$, $i=-1,0$ become smaller.
This fact was a key in the proof of \cite[Thm. 1.1]{Y:birational}.

\begin{rem}
If $\NS(X)={\Bbb Z}H$, then \eqref{eq:slope-behavior}
implies that we always have $r_1 r_2<0$.
Thus if there is a stable object $E$ which is a sheaf
up to shift, then a general member is Gieseker stable.

In \cite{YY}, we studied the set of pairs
$\{(w_1,w_2) \}$ by using the theory of quadratic forms.
Thus a computation in \cite{YY} can be regarded as a computation of
the set of walls $\{W_{w_1} \}$. 
For a more detailed study of walls, see \cite{YY2}.  
\end{rem}


\subsection{Relations of Picard groups}
\label{subsect:abel:Picard}

Let $X$ be an abelian surface over ${\Bbb C}$.
For a primitive Mukai vector
$v=r+\xi+a\varrho_X$, $\xi \in \NS(X)$, assume that
$r>0$ or $r=0$ and $\xi$ is represented by an effective divisor.  
We assume that $H$ is general with respect to $v$.
Then we have a locally trivial morphism
${\frak a}:M_H(v) \to X \times \widehat{X}$, which is the Albanese map
of $M_H(v)$. 
Denote by $K_H(v)$ 
a fiber of ${\frak a}$.
Then $K_H(v)$ is an irreducible symplectic manifold of 
$\dim K_H(v)=\langle v^2 \rangle-2$ which is
deformation equivalent to a generalized Kummer variety constructed by
Beauville \cite{Beauville}.
For an irreducible symplectic manifold,
Beauville constructed an integral bilinear form
on the second cohomology group.
In our case, we have an isomorphism 
\begin{equation*}
\theta_v:v^{\perp} \to H^2(K_H(v),{\Bbb Z})
\end{equation*} 
which preserves the Hodge structure and the bilinear form
\cite{Y:7}.
Then we have an isomorphism
$$
 \theta_v: A^*_{\alg}(X) \cap v^{\perp} \to \Pic(K_H(v)).
$$
The same claims also hold for the moduli of stable twisted sheaves
(cf. \cite[Thm. 3.19]{Y:twisted}).

For a Fourier-Mukai transform $\Phi:{\bf D}(X) \to {\bf D}(X')$,
if there is an open subset $U \subset M_H(v)$ 
such that $\dim (M_H(v) \setminus U) \leq \dim M_H(v)-2$
and $\Phi(E) \in M_{\widehat{H}}(\Phi(v))$ for $E \in U$,
then we have an identification
$\Phi_*:H^2(M_H(v),{\Bbb Z}) \to H^2(M_{\widehat{H}}(\Phi(v)),{\Bbb Z})$
and a commutative diagram
\begin{equation*}
\xymatrix{
v^{\perp} \ar[r]^{\Phi} \ar[d]_{\theta_v} & 
\Phi(v)^{\perp} \ar[d]^{\theta_{\Phi(v)}}
\\
H^2(K_H(v),{\Bbb Z}) \ar[r]_{\Phi_{*}}  & 
 H^2(K_{\widehat{H}}(\Phi(v)),{\Bbb Z})
}
\end{equation*}
Thus Proposition~\ref{prop:pic-isom} 
enables us to study the Picard groups of $K_H(v)$ unless
there is no wall of type (b1) in Lemma~\ref{lem:codim}.

\subsubsection{A study on the wall of type (b1)}

For the wall of type (b1), the universal families are different 
along a divisor. We shall study the
relation of two families in order to compare the Picard groups.
For this purpose, we start with the analysis of a family
of torsion free sheaves of rank 2.

\begin{NB}
\begin{lem}
We set $v=2+\xi+a \varrho_X$.
Assume that $E$ is a simple torsion free sheaf with $v(E)=v$
such that $\dim E^{**}/E \leq 1$.
If $\langle v^2 \rangle \geq 12$, then
in a neighborhood of $E$ in the stack of torsion free sheaves
${\cal M}(v)$,
 $\dim({\cal M}(v) \setminus {\cal M}_H^\beta(v)^{\mu \text{-}ss}) 
\leq \langle v^2 \rangle
-1$.
\end{lem}

\begin{proof}
We first prove that 
$\dim({\cal M}(v) \setminus {\cal M}_H^\beta(v)^{ss}) 
\leq \langle v^2 \rangle
-1$
in a neighborhood of $E$. 
We may assume that $E$ is not $\beta$-twisted semi-stable.
Then we have an exact sequence
$$
0 \to I_1 \to E \to I_2 \to 0
$$ 
such that $I_1$, $I_2$ are torsion free sheaves of rank 1 with 
$\deg(I_1) \geq \deg(I_2)$ and 
$\Hom(I_1,I_2)=0$.
We set $v(I_i)=1+\xi_i+a_i \varrho_X$.
If $\Hom(I_1,I_2^{**}) \ne 0$, then
$\xi_1=\xi_2$.
Since $E$ is simple, we have
$\Hom(I_2,I_1)=0$.
Moreover $\Ext^1(I_2,I_1) \ne 0$.
Hence $\langle v(I_1),v(I_2) \rangle \geq 1$.
Since $I_1^{**}/I_1 \subset E^{**}/E$ and
$\dim _{\frak k} E_1^{**}/E_1 \leq 1$,
we see that
$\dim \Hom(I_2^{**},I_1^{**}) \leq 1$.
\begin{NB2}
By the exact sequence
$$
0 \to \Hom(I_2^{**},I_1)\to \Hom(I_2^{**},I_1^{**}) \to 
\Hom(I_2^{**},I_1^{**}/I_1)
$$
and $\Hom(I_2^{**}, I_1) \subset \Hom(I_2,I_1)=0$, the claim holds.
\end{NB2}
Hence $((\xi_1-\xi_2)^2 ) \leq 2$.
If $\langle v(I_1),v(I_2) \rangle \geq 3$, then
$\dim({\cal M}(v) \setminus {\cal M}_H^{ss}) \leq \langle v^2 \rangle
-1$.
Assume that 
$\langle v(I_1),v(I_2) \rangle =1,2$.
Then $\langle v^2 \rangle=((\xi_1-\xi_2)^2)+
4\langle v(I_1),v(I_2) \rangle \leq 2+8=10$.
Hence the claim holds.

We next prove that
$\dim {\cal M}_H^\beta(v)^{ss} \setminus {\cal M}_H(v)^{\mu \text{-}ss}
\leq \langle v^2 \rangle-1$.
Assume that there is an exact sequence
$$
0 \to I_1 \to E \to I_2 \to 0
$$
such that $I_1$, $I_2$ are torsion free sheaves of rank 1
with $\deg(I_1)=\deg(I_2)$.
\end{proof}
\end{NB}

We set $v=2+\xi+a \varrho_X$.
Let $\alpha$ be a representative 
of $[\alpha] \in H^2_{\text{\'{e}t}}(X,{\cal O}_X^{\times })$.
Let $M$ be an open subscheme of the moduli space of
simple torsion free $(\alpha^{-1})$-twisted sheaves
$E$ with $v(E)=v^{\vee}$ such that
$E^{**}$ is simple and $\dim_{\frak k}(E^{**}/E) \leq 1$,
where $E^*:={\cal H}om_{{\cal O}_X}(E,{\cal O}_X)$.
Let ${\cal E}$ be a universal family 
as a $(\alpha^{-1} \times \alpha')$-twisted sheaf on $X \times M$,
where $\alpha'$ is a representative of $[\alpha'] 
\in H^2_{\text{\'{e}t}}(X',{\cal O}_{X'}^{\times})$.
Then ${\cal E}^{\vee}:=
{\bf R}{\cal H}om_{{\cal O}_X}({\cal E},{\cal O}_{X \times M})$
is not a family of torsion free $\alpha$-twisted sheaves.
We shall modify this complex to a flat family of torsion free
$\alpha$-twisted sheaves.
Let $D$ be the divisor of $M$
parametrizing non-locally free sheaves.
We shall prove that
$D$ is smooth and 
${\cal E}xt^1_{{\cal O}_{X \times M}}({\cal E},{\cal O}_{X \times M})$
is an ${\cal O}_D$-module.
We take a locally free resolution of ${\cal E}$:
$$
0 \to U \to V \to {\cal E} \to 0.
$$ 
Then ${\cal E}^{\vee}$ is the complex
which is represented by $V^{\vee} \to U^{\vee}$.
The support of ${\cal E}xt^1_{{\cal O}_{X \times M}}({\cal E},
{\cal O}_{X \times M})$ is $X \times D$.
Assume that $E \in D$.
By the local-global spectral sequence,
we have an exact sequence
$$
\Ext^1(E,E) \overset{\varphi}{\to} H^0(X,{\cal E}xt^1_{{\cal O}_{X}}(E,E)) \to 
H^2(X,{\cal H}om_{{\cal O}_{X}}(E,E)) \overset{\psi}{\to} \Ext^2(E,E).
$$ 
Since $E^{**}$ is simple,
$\psi$ is isomorphic and $\varphi$ is surjective.
Since ${\cal E}xt^1_{{\cal O}_{X}}(E,E) \cong {\frak k}_x$ 
is the Zariski tangent space of 
local deformation of $E$ at a pinch point $x$ of $E$,
$D$ is smooth in a neighborhood of $E$.
Moreover 
$$
{\cal E}xt^1_{{\cal O}_{X \times M}}({\cal E},{\cal O}_{X \times M})
=i_*(F),
$$
 where 
$i:X \times D \to X \times M$
is the inclusion and $F$ is a flat family of skyscraper $\alpha$-twisted
sheaves
${\frak k}_x$, $x \in X$
parametrized by $D$.
We have an exact triangle
\begin{equation}\label{eq:univ_pm}
{\cal E}^* \to 
{\bf R}{\cal H}om_{{\cal O}_{X}}({\cal E},{\cal O}_{X \times M})
\to i_*(F)[-1] \to {\cal E}^*[1].
\end{equation}
We set ${\cal G}:=\ker(U^{\vee} \to i_*(F))$.
Since ${\cal T}or_2^{{\cal O}_M}(i_*(F),{\cal O}_{D})=0$, we get
${\cal T}or_1^{{\cal O}_M}({\cal G},{\cal O}_{D})=0$. Thus ${\cal G}$ is
flat over $M$ and ${\cal E}^*_{|X \times D} \to
V^{\vee}_{|X \times D}$ is injective. 
We have exact sequences
\begin{equation*}
\begin{split}
&0 \to {\cal E}^*_{|X \times D} \to V_{|X \times D}^{\vee}
\to {\cal G}_{|X \times D} \to 0,
\\
&0 \to F(-D) \to {\cal G}_{|X \times D} 
\to U_{|X \times D}^{\vee} \to F \to 0.
\end{split}
\end{equation*}
For $t \in D$, ${\cal E}_{|X \times \{t \}}^*$ and
${\cal E}^{\vee}_{|X \times \{t \}}$ fits in
exact sequences in 
${\frak A}_{(\beta',\omega')}$:
\begin{equation*}
\begin{split}
&0 \to F_{|X \times \{t \}}[1]
\to
{\cal E}_{|X \times \{t \}}^*
\to ({\cal E}_{|X \times \{t \}}^*)^{**} \to 0,
\\
&0 \to ({\cal E}_{|X \times \{t \}}^*)^{**}[1] \to
{\cal E}_{|X \times \{t \}}^{\vee}[1] \to F_{|X \times \{ t \}}
\to 0,
\end{split}
\end{equation*}
where the intersection number satisfies
$(\beta',\omega') \gg (\xi,\omega')$.
These are non-trivial extensions.

Let us study the case (b1) in Lemma~\ref{lem:codim}.
Let $W_{w_1}$ be a wall for $v$ such that
$\langle w_1^2 \rangle=0$ and $\langle v,w_1 \rangle=2$.
In the notation of \S\,\ref{subsect:wall=w_1},
$\Phi_{X \to X_1}^{{\bf E}^{\vee}[1]}(E)$ or its dual
is a torsion free sheaf of rank 2 on $X_1$.
We shall prove that each moduli space is represented by
an open subscheme $M$ of $M_L^\gamma(w)$
such that $\dim(M_L^\gamma(w) \setminus M) \geq 2$, 
where $L$ is an ample line bundle on an abelian surface $Y$ 
and $w \in A^*_{\alg}(Y)$ is a primitive Mukai vector.

We shall prove this claim inductively by crossing walls of type (b1). 
So we may assume that one of the 
$M_{(\beta,\omega_\pm)}(v)$ is represented by a scheme $M$ up to 
codimension 1, and there is a universal family
${\cal E}_{\pm}$ as a twisted sheaf.
Indeed we may choose $M_L^\gamma(w)$ 
as the moduli scheme of rank 1 torsion free
sheaves, $M_H^\beta(v)$ or $M_H^{-\beta}(-v^{\vee})$. 
Then $\Phi_{X \to X_1}^{{\bf E}^{\vee}[1]}({\cal E}_{\pm})$
or its dual is a family of torsion free sheaves of rank 2.
For simplicity, assume that 
${\cal E}:=\Phi_{X \to X_1}^{{\bf E}^{\vee}[1]}({\cal E}_{\pm})^{\vee}$
is a family of torsion free sheaves.
Applying the above observation to
this situation,
we get a family of
torsion free sheaves
${\cal E}^*=(\Phi_{X \to X_1}^{{\bf E}^{\vee}[1]}({\cal E}_{\pm})^{\vee})^*$.
Then 
${\cal E}_{\mp}:=\Phi_{X_1 \to X}^{{\bf E}[1]}
((\Phi_{X \to X_1}^{{\bf E}^{\vee}[1]}({\cal E}_{\pm})^{\vee})^*)
$ is a family of 
$\sigma_{(\beta,\omega_\mp)}$-stable objects.
Therefore we have an exact triangle
\begin{equation*}
{\cal E}_\mp \to {\cal E}_\pm \to {\cal F} \to
{\cal E}_\mp[1], 
\end{equation*}
where 
${\cal F}$ is a family of stable objects with the Mukai vector
$w_1$ parametrized by a divisor $D$.
Then we have an identification 
$M_{(\beta,\omega_\pm)}(v) \cong M_{(\beta,\omega_\mp)}(v)$
up to codimension 1.
Thus 
$M_{(\beta,\omega_\mp)}(v)$ is also represented by 
$M$.
We set $K:=M \cap K_L(w)$.
Since $\codim_{K_L(w)}(K_L(w) \setminus K) \geq 2$,
$H^2(K_L(w),{\Bbb Z}) \cong H^2(K,{\Bbb Z})$.
For $x \in v^{\perp}$, we set
\begin{equation*}
\theta_\pm(x):= -[{p_{K*}}(\ch({\cal E}_\pm)^{\vee} p^*(x))]_1 
\in H^2(K,{\Bbb Z}),
\end{equation*}
where $p_K:X \times K \to K$ and $p:X \times K \to X$
are the projections, and
$[z]_1$ means the $H^2(K,{\Bbb Z})$-component
of $z$. 
We shall compare two isometries $\theta_\pm$.
We write $D=\theta_\mp(d)$.
\begin{lem}\label{lem:D}
$$
d=v-\frac{\langle v^2 \rangle}{2}w_1.
$$
\end{lem}

\begin{proof} 
Let ${\cal E}$ be the family of torsion free sheaves as above.
We set $v':=v({\cal E}_{|X_1 \times \{s \}})$, 
$s \in K_L(w)$.
It is sufficient to prove
\begin{equation}\label{eq:c_1(D)}
2c_1(D)=-[{p_{K*}}(\ch({\cal E}^*)^{\vee} 
p^*(2v'-\langle v', v' \rangle \varrho_{X_1}))]_1,
\end{equation}
where $p_K$ and $p$ are the projections from $X_1 \times K$.
Let ${\cal L}$ be an 
$(\alpha^2 \times (\alpha')^2)$-twisted 
line bundle on $X_1 \times K_L(w)$
which is an extension of $\det {\cal E}^*$.
Indeed since
$[\alpha \times \alpha']$ is a 2-torsion element of the 
Brauer group,  
the category of 
$(\alpha^2 \times (\alpha')^2)$-twisted
sheaves is equivalent to the category of coherent sheaves.
Since $\Pic^0(K_L(w))$ is trivial,
we have a decomposition ${\cal L} \cong L_1 \boxtimes L_2$,  
where $L_1$ and $L_2$ are $(\alpha^2)$-twisted line bundle on $X_1$
and $(\alpha')^2$-twisted line bundle on $K_L(w)$ respectively.
We note that ${\cal E}^*$ is reflexive and 
${\cal E} \cong {\cal E}^* \otimes 
(\det {\cal E}^*)^{\vee}$.
Hence 
${\cal E} \cong {\cal E}^* \otimes (L_1 \boxtimes L_2)^{\vee}$.
Let $G$ be an $\alpha$-twisted vector bundle of rank 2 on $X_1$
such that $v(G)=v'$. Then $v(G^{\vee} \otimes L_1)=v(G)=v'$.
We have 
$\det p_{K!} [{\cal E}xt^1_{{\cal O}_{X_1 \times K}}
({\cal E}\otimes G , {\cal O}_{X_1 \times K})] 
\cong {\cal O}_K(2D)$,
where for an object $F$ of ${\bf D}(X)$ we denote by $[F]$ 
the class of $F$ in the Grothendieck group.
By using the relative duality,
we see that
\begin{align*}
p_{K!} [{\cal E}xt^1_{{\cal O}_{X_1 \times K}}
({\cal E} \otimes G,{\cal O}_{X_1 \times K})] 
&=-p_{K!} [{\cal E}^{\vee} \otimes G^{\vee}]
  +p_{K!} [{\cal E}^* \otimes G^{\vee}]
\\
&=-p_{K!} [{\bf R}{\cal H}om_{{\cal O}_{X_1 \times K}}(
            {\cal E}^* \otimes (L_1 \boxtimes L_2)^{\vee}, 
            G^{\vee} \boxtimes {\cal O}_K)]
  +p_{K!} [{\cal E}^* \otimes G^{\vee}]
\\
&=-p_{K!} [{\bf R}{\cal H}om_{{\cal O}_{X_1 \times K}}(
            {\cal E}^* \otimes L_1^{\vee} \otimes G, 
            {\cal O}_{X' \times K}) \otimes L_2]
  +p_{K!} [{\cal E}^* \otimes G^{\vee}].
\end{align*}
Hence 
\begin{align*}
{\cal O}_K(2D)=
 & \det p_{K!} [{\cal E}xt^1_{{\cal O}_{X_1 \times K}}
({\cal E} \otimes G,{\cal O}_{X_1 \times K})]
\\
=&
 \det p_{K!} [(({\cal E}^*)^{\vee} \otimes L_1 \otimes G^{\vee})^{\vee}]
 \otimes L_2^{\langle v',v(L_1 \otimes G) \rangle}
 \otimes \det p_{K!} [{\cal E}^* \otimes G^{\vee} ]
\\
=& 
 \det p_{K!}\bigl([({\cal E}^*)^{\vee}] \otimes ([L_1 \otimes G^{\vee}]
  -\langle v',v(L_1 \otimes G) \rangle [A]+[G])\bigr)^{\vee},
\end{align*}
where $A$ is an $\alpha$-twisted structure sheaf of a point $x_1 \in X_1$. 
Since
$$
v([L_1 \otimes G^{\vee}]-\langle v',v(L_1 \otimes G) \rangle [A]+[G])=
v'-\langle v',v' \rangle \varrho_{X'}+v',
$$
we get \eqref{eq:c_1(D)}.
\end{proof}

\begin{prop}
We have 
$
\theta_\pm(x)=
\theta_\mp(x-\langle w_1,x \rangle d), x \in v^{\perp},
$
and
$$
x-\langle w_1,x \rangle d=
x-2\frac{\langle d,x \rangle}{\langle d,d \rangle} d,\quad 
x \in v^{\perp}.
$$
Thus $\theta_\mp^{-1} \circ \theta_\pm$ is the reflection by $d$.
In particular, it is an isometry of
$v^{\perp}$.
\end{prop}

\begin{proof}
For $x \in v^{\perp}$, \eqref{eq:univ_pm} implies that
\begin{equation}\label{eq:theta_pm}
\theta_\pm(x)=\theta_\mp(x)-\langle w_1,x \rangle \theta_\mp(d).
\end{equation}
By Lemma~\ref{lem:D},
we have $\langle d^2 \rangle=-\langle v^2 \rangle$. 
Since 
$\frac{\langle v^2 \rangle}{2} \langle w_1,x \rangle
=-\langle d,x \rangle$,
we get the claim.
\end{proof}

\begin{rem}
If $\theta_+$ and $\theta_-$ are isometries,
then we can determine $d$ as follows.
Since $\theta^\pm$ are isometries, \eqref{eq:theta_pm} implies that
$$
\langle x^2 \rangle =
(\theta_\pm(x)^2)=(\theta_\mp(x-\langle w_1,x \rangle d)^2)
=\langle (x-\langle w_1,x \rangle d)^2 \rangle
=\langle x^2 \rangle-2\langle x,\langle w_1,x \rangle d \rangle
+\langle w_1,x \rangle^2 \langle d^2 \rangle
$$
for $x \in v^{\perp}$.
Hence we get
$$
-2\langle x,d \rangle
+\langle w_1,x \rangle \langle d^2 \rangle=0.
$$
Since $(v^{\perp})^{\perp}={\Bbb Z}v$, we have
$-2d+\langle d^2 \rangle w_1=nv$, $n \in {\Bbb Z}$.
Then $n \langle v^2 \rangle=\langle d^2 \rangle \langle w_1,v \rangle
=2\langle d^2 \rangle$, which implies that
$d=-(n/2)v+(n/4) \langle v^2 \rangle w_1$.
Then we have 
$$
 n \langle v^2 \rangle
 =2\bigl\langle 
    (-(n/2)v+(n/4)\langle v^2 \rangle w_1 )^2
   \bigr\rangle.
$$
Thus $n=0$ or $-2$.
Since $d \ne 0$, we should have $n=-2$.
Therefore 
$$
d=v-\frac{\langle v^2 \rangle}{2}w_1
$$
\end{rem}

\begin{rem}
If $\langle w_1,u \rangle \in 2{\Bbb Z}$ 
for any $u \in A^*_{\alg}(X)$, then
$M_{(\beta,\omega_\pm)}(v)$ are isomorphic to open subschemes
of the moduli of stable twisted sheaves on $X'$.
Thus we have a moduli space
$M_{(\beta,\omega_\pm)}(v)$ as a scheme. 
\end{rem}

We have the following refinement of \cite{Y:birational}.
\begin{cor}
\label{cor:lagrange}
Let $X$ be an abelian surface.
Assume that there is a primitive Mukai vector $u \in A^*_{\alg}(X)$
such that $\langle u,v \rangle =\langle u^2 \rangle=0$.
Then we can explicitly find a Fourier-Mukai transform
$\Phi$ which induces a birational map 
$M_H(v) \dashrightarrow M_{H'}^{\alpha}(v')$, where 
$M_{H'}^{\alpha}(v')$ is a moduli scheme of $\alpha$-twisted sheaves
of dimension 1 on an abelian surface.
Moreover this birational map induces an isomorphism
of the second cohomology groups provided $\langle v^2 \rangle>6$.
Thus we can explicitly describe the line bundle on 
$K_H(v)$ which induces a rational Lagrangian fibration
$K_H(v) \dashrightarrow {\Bbb P}^{\langle v^2 \rangle/2-1}$.   
\end{cor}


\section{Appendix}
\label{sect:appendix}


\subsection{Stability condition and base change}

Now we discuss on the behaviour of our stability condition 
$\sigma_{(\beta,\omega)}=({\frak A}_{(\beta,\omega)},Z_{(\beta,\omega)})$ 
under a field extension $L/ {\frak k}$.
A similar claim holds for any surfaces.
Let us write by $X_L$ the base change of $X$, and set standard morphisms as 
\begin{equation*}
\xymatrix{
X_L  \ar[r]^{p'} \ar[d]_{\pi'} & X \ar[d]^{\pi}
\\
\Spec L     \ar[r]_{p}         & \Spec {\frak k}}
\end{equation*}

For given $\beta,H,\omega$ on $X$, 
we denote by  $\beta',H',\omega'$ the pull-backs on $X_L$.
The stability function $Z_{(\beta',\omega')}: {\bf D}(X_L) \to {\Bbb C}$ is 
defined in the same way as $Z_{(\beta,\omega)}$:
\begin{equation*}
Z_{(\beta',\omega')}(E_L) 
:= \langle e^{\beta'+\sqrt{-1} \omega'},v(E_L) \rangle,\quad
E_L \in {\bf D}(X). 
\end{equation*}
The function $\Sigma_{(\beta',\omega')}(E'_L,E_L)$ 
is also defined in the same way:
\begin{equation*}
\Sigma_{(\beta',\omega')}(E',E):=\det
\begin{pmatrix}
\mathrm{Re} Z_{(\beta',\omega')}(E'_L) & 
\mathrm{Re} Z_{(\beta',\omega')}(E_L)  \\
\mathrm{Im} Z_{(\beta',\omega')}(E'_L) & 
\mathrm{Im} Z_{(\beta',\omega')}(E_L)
\end{pmatrix},\quad
E_L, E'_L \in {\bf D}(X_L).
\end{equation*}

\begin{lem}\label{lem:stability_basechange}
Assume that the extension $L/ {\frak k}$ is finite.
\begin{enumerate}
\item[(1)]
For an object $E_L \in {\bf D}(X_L)$ we have
\begin{equation*}
[L:{\frak k}] Z_{(\beta',\omega')}(E_L) =Z_{(\beta,\omega)}(p'_{*}E_L).
\end{equation*}
\item[(2)]
For objects $E_L, E'_L\in {\bf D}(X_L)$ we have
\begin{equation*}
\Sigma_{(\beta',\omega')}(E'_L,E_L) \ge 0 
\iff 
\Sigma_{(\beta,\omega)}(p'_{*} E'_L, p'_{*} E_L) \ge 0.
\end{equation*}
\end{enumerate}
\end{lem}

\begin{proof}
Note that for an object $E_L \in {\bf D}(X_L)$ we have
\begin{equation*}
Z_{(\beta',\omega')}(E_L) = 
\dim_{L} [\pi'_{*}(v(E_L) \cdot e^{\beta'+\sqrt{-1} \omega'})]_0,
\end{equation*}
where $\cdot$ in the right hand side means the intersection product 
on the Chow ring $A^{*}(X_L)$ and $[\ ]_0$ denotes 
the degree zero part of an element of $A^{*}(X_L)$.
Then we have
\begin{equation*}
\begin{split}
[L:{\frak k}] Z_{(\beta',\omega')}(E_L) 
&= \dim_{\frak k}
     p_{*} [\pi'_{*}( v(E_L) \cdot e^{\beta'+\sqrt{-1}\omega'})]_0
 = \dim_{\frak k} 
    [\pi_{*} p'_{*}(v(E_L) \cdot e^{\beta'+\sqrt{-1}\omega'})]_0
\\
&= \dim_{\frak k} [\pi_{*} (v(p'_{*}E_L) \cdot e^{\beta+\sqrt{-1}\omega})]_0
 = Z_{(\beta,\omega)}(p'_{*}E_L).
\end{split}
\end{equation*}
Here at the third equality we used the projection formula.
Thus we have (1).
Then (2) follows from the definition of $\Sigma_{(\beta,\omega)}$.
\end{proof}

We denote by 
${\frak C}_{L}$ and  
${\frak A}_{(\beta',\omega'),L}$ 
the categories over $L$, 
which are similarly defined as 
${\frak C}$ and ${\frak A}_{(\beta,\omega)}$.
Then the stability condition 
$\sigma_{(\beta',\omega'),L}
:=({\frak A}_{(\beta',\omega'),L},Z_{(\beta',\omega')})$
is well-defined.

\begin{lem}\label{lem:catA_basechange}
\begin{enumerate}
\item[(1)]
The derived pull-back $p'^{*}: {\bf D}(X)\to {\bf D}(X_L)$ induces 
an exact functor 
$p'^{*}: {\frak A}_{(\beta,\omega)} \to {\frak A}_{(\beta',\omega'),L}$. 
\item[(2)]
The derived push-forward $p'_{*}: {\bf D}(X_L)\to {\bf D}(X)$  
induces an exact functor
$p'_{*}: {\frak A}_{(\beta',\omega'),L} \to {\frak A}_{(\beta,\omega)}$.
\end{enumerate}
\end{lem}

\begin{proof}
What should be shown is that the image $p'^{*}({\frak A}_{(\beta,\omega)})$ 
(resp. $p'_{*}({\frak A}_{(\beta',\omega'),L})$) 
is indeed in ${\frak A}_{(\beta',\omega'),L}$ 
(resp. in ${\frak A}_{(\beta,\omega)}$).
The case (1) is clear, since the twisted semi-stability 
is preserved under pull-back
(which is a consequence of the uniqueness of Harder-Narasimhan
filtration).

For the case (2), 
let $E$ be a $(\beta')$-twisted semi-stable object of ${\frak C}_L$.
it is enough to prove that $p'_*(E)$ is also $\beta$-twisted semi-stable.

We may assume that $L$ is a normal extension of ${\frak k}$.
Indeed for a normal extension $L'$ of ${\frak k}$ containing
$L$, $q'_*(q'^*(E))=E^{\oplus [L':L]}$, where $q':X_{L'} \to X_L$
is the projection.

Let $F$ be a subobject of $p'_*(E)$.
Then $p'^*(F)$ is a subobject of $p'^* p'_*(E)$.
By Lemma~\ref{lem:pull-back} below,
$p'^* p'_*(E)$ is a $\beta'$-twisted semi-stable object
with $\chi(p'^* p'_*(E)(n))=\chi(p'_*(E)(n))$.
Hence 
$$
\frac{\chi(F(n))}{\rk F}=
\frac{\chi(p'^*(F)(n))}{\rk F} \leq \frac{\chi(p'_*(E)(n))}{\rk p'_*(E)}.
$$
Therefore $p'_*(E)$ is $\beta$-twisted semi-stable.
\end{proof}

\begin{lem}\label{lem:pull-back}
$p'^* p'_*(E)$ is a successive extension of $E$.
\end{lem}

\begin{proof}
We note that $p'^* p'_*(E)=E \otimes_L (L \otimes_{\frak k} L)$.
Since $L$ is a normal extension of ${\frak k}$,
$R:=L \otimes_{\frak k}L$
is a successive extension of
$R$-modules $R/{\frak m}_i$, where ${\frak m}_i$ are 
maximal ideals of $R$ and $R/{\frak m}_i \cong L$.
\begin{NB}
If $L$ is not normal, then $R/{\frak m}_i$ is bigger
than $L$ in general.
If $L={\frak k}[x]/(f(x))$ and
$f(x)=\prod_i f_i(x)$ in $L[x]$, then
$L \otimes_{\frak k} L=\prod_i L[x]/(f_i(x))$. 
\end{NB}
Since $E \otimes_L (R/{\frak m}_i)\cong E$,
 $p'^* p'_*(E)$ is a successive extension of
$E$.
\end{proof}

\begin{prop}\label{prop:basechange}
Assume that the extension $L/ {\frak k}$ is finite.
If $E$ is a $\sigma_{(\beta,\omega)}$-semi-stable object,
then $p'^{*}E$ is 
$\sigma_{(\beta',\omega'),L}$-semi-stable,
\end{prop}

\begin{proof}
Assume that $p'^*(E)$ is not 
$\sigma_{(\beta',\omega'),L}$-semi-stable. 
(Here we are implicitly using Lemma~\ref{lem:catA_basechange} (1).)
Take a distabilizing subobject $F$ of $p'^*(E)$, 
so that we have $\Sigma_{(\beta',\omega')}(F,p'^*(E)) < 0$.
Then by Lemma~\ref{lem:stability_basechange} (2) 
we have $\Sigma_{(\beta,\omega)}(p'_{*} F, p'_{*}p'^{*} E) < 0$.

On the other hand, 
since $p'_{*} p'^{*} E \simeq E^d$ with $d := {[L:{\frak k}]}$, 
$p'_{*} p'^{*} E$ is 
$\sigma_{(\beta,\omega)}$-semi-stable. 
Then since $p'_{*} F \in{\frak A}_{(\beta,\omega)}$ and 
it is a subobject of $p'_{*} p'^{*} E$ 
by Lemma~\ref{lem:catA_basechange} (2), 
we have $\Sigma_{(\beta,\omega)}(p'_{*}F, p'_{*}p'^{*} E) \ge 0$.
Therefore by contradiction the claim holds.
\end{proof}

\begin{rem}
This claim implies that our stability condition $\sigma_{(\beta',\omega'),L}$
is a member of $\Stab(X_L)_p$ in \cite[Definition 3.1]{Sosna}.
\end{rem}

\begin{cor}
Let $L=\overline{\frak k}$, the algebraic closure of ${\frak k}$.
If $E$ is a $\sigma_{(\beta,\omega)}$-semi-stable object,
then $p'^{*}E$ is $\sigma_{(\beta',\omega'),L}$-semi-stable.
\end{cor}

\begin{proof}
Assume that $p'^{*}E$ is not 
$\sigma_{(\beta',\omega'),L}$-semi-stable.
Let $F$ be the distabilizing subobject of $p'^{*}E$.
Since $F$ is a class of a complex consisting of coherent sheaves,
we may assume that $F$ is defined on a field $L'$ such that 
${\frak k} \subset L' \subset L=\overline{\frak k}$ 
and $[L':{\frak k}]<\infty$. 
Then by Proposition~\ref{prop:basechange} 
$p'^{*}_{L'/{\frak k}}\, E$ is semi-stable 
with respect to $\sigma_{(\beta'',\omega''),L'}$, 
where $p'_{L'/{\frak k}}:X_{L'} \to X$ and 
$\beta'',\omega''$ are pull-backs of $\beta,\omega$ via 
${\frak k} \to L'$.
But it contradicts the choice of $F$.
\end{proof}

\subsection{Non-emptyness of the moduli spaces}
Over a field of characteristic 0,
the moduli spaces of stable objects are non-empty if the expected dimension
is non-negative. We shall explain that a similar claim also holds
over a field of positive characteristic
under a technical condition.
Let ${\frak k}$ be a field of characteristic $p>0$ and
$X$ is an abelian or a $K3$ surface defined over ${\frak k}$.
Let $\alpha$ be a 2-cocycle of the \'{e}tale sheaf
${\cal O}_X^{\times}$.
Let $K^\alpha(X)$ be the Grothendieck group of
$\alpha$-twisted sheaves on $X$.

\begin{prop}\label{prop:mod-p}
Let $G$ be a locally free $\alpha$-twisted sheaf such that
$K^\alpha(X)={\Bbb Z}G+K^\alpha(X)_{\leq 1}$,
where $K^\alpha(X)_{\leq 1}$ is the subgroup generated by
torsion sheaves.
Assume that $(p,\rk G)=1$.
Let $v$ be the Mukai vector of an $\alpha$-twisted sheaf
$E$.
If $v$ is primitive and
$\langle v^2 \rangle \geq -2 \varepsilon$, then
$M_H^{G'}(v) \ne \emptyset$ for a general $(G',H)$. 
Moreover $M_H^{G'}(v)$ is deformation equivalent to
$M_H(1,0,-\ell)$, $\ell:=\langle v^2 \rangle/2$.
\end{prop}

We first treat the case where $\alpha$ is trivial.
By \cite[Cor. 3.2]{NO} and \cite[Cor. 1.8]{D},
for a polarized abelian or a K3 surface $(X,H)$, we have
a lifting $({\cal X},{\cal H})$ to characteristic 0.
\begin{rem}
Assume that ${\cal X} \to T$ is a family of abelin surfaces with polarizations.
By \cite[sect. 6.2]{MFK}, a polarization
is a $T$-homomorphism $\lambda:{\cal X} \to \widehat{\cal X}$
such that for any geometric point $t \in T$,
$\lambda_t=\phi_L$ for some ample line bundle $L$
on ${\cal X}_t$.
By \cite[Prop. 6.10]{MFK},
we have a relatively ample line bundle which induces
the polarization $2\lambda$.
We consider $\pi:\Pic_{{\cal X}/T}^{\xi} \to T$, 
where $\Pic_{{\cal X}/T}^{\xi}$ is the 
connected component containing $H$ on $X$. 
Since $\Pic_{{\cal X}/T}^{\xi} \to T$ is a projective morphism,
$\im \pi$ is a closed subscheme of $T$.
For the category of Artinian rings, 
polarizations come from ample line bundles
(\cite[Lem. 2.3.2]{Oort}).
Thus \cite[Cor. 3.2]{NO} implies that
the infinitesimal lifting of $H$ is unobstructed. 
Then $\pi$ is dominant. Then for a suitable finite covering 
$T' \to T$ of $T$, we have a family of line bundles
which is an extension of $H$. 
\end{rem}

We first prove the simplest case in order to
explain our starategy.
Let $v=(r,\xi,a) \in {\Bbb Z} \oplus \NS(X) \oplus {\Bbb Z}$
be a primitive Mukai vector with $v>0$ and consider the untwisted case.
We assume that $H$ is general with respect to $v$.
In this case, the twisted semi-stability
is independent of $G$.
Hence we can set $G={\cal O}_X$. 
For the proof of $M_H(v) \ne \emptyset$,
we may assume that
${\frak k}$ is algebraically closed.
Replacing $v$ by $v e^{nH}$, $n \gg 0$,
we may assume that $\xi$ is ample.
Since ${\Bbb Q}H$ is sufficiently close to ${\Bbb Q}\xi$,
we may assume that $H \in {\Bbb Q}\xi$.
Thus we may assume that $v=(r,dH,a)$.
In this case, we have a relative moduli scheme
$M_{({\cal X},{\cal H})/S}(v) \to S$ which is smooth and
projective.
Therefore the claim holds.
In order to cover general cases, we prepare the following two lemmas.

\begin{lem}\label{lem:Azumaya-deformation}
Assume that $(p,\rk G)=1$.
Let $(R,tR)$ be a discrete valuation ring with
$R/tR={\frak k}$ and
$({\cal X},{\cal H})$ a flat family of polarized surfaces
over $T=\Spec(R)$. 
There is an extension $(R',t')$ of $(R,t)$ and
a flat family of projective bundle
$P \to {\cal X}_{R'}$ such that the Brauer class of 
$P \otimes_{R'} R'/t'R$ is $[\alpha] \in 
H^2_{\text{\'{e}t}}(X, {\cal O}_X^{\times})$. 
\end{lem}
 
\begin{proof}
\begin{NB}
We may assume that $\rk G>1$.
We take a torsion free $\alpha$-twisted sheaf
such that $G_1$ fits an exact sequence
$$
0 \to G_1 \to G \to \oplus_i {\frak k}_{x_i} \to 0
$$  
and $\Hom(G_1,G_1(K_X))_0=0$, where
$\Hom(G_1,G_1(K_X))_0$ is the space of trace free homomorphisms.
Since $\rk G>1$ and $G_1$ is $\mu$-stable for any polarization, 
$G_1$ deforms to a locally free $\alpha$-twisted
sheaf. Replacing $G$ by $G_1$, we may assume that
$\Hom(G,G(K_X))_0=0$.
\end{NB}
We take a smooth curve $D \in |nH|$, $n \gg 0$. 
Let $G_1$ be a torsion free $\alpha$-twisted sheaf on
$X$ such that $\rk G_1>1$ and $\rk G_1 \equiv \rk G \mod p$. 
Then there is a torsion free $\alpha$-twisted sheaf
such that $G_2$ fits an exact sequence
$$
0 \to G_2 \to G_1 \to \oplus_i {\frak k}_{x_i} \to 0
$$  
and $\Hom(G_2,G_2(K_X+D))_0 =0$.
Let $\Def_0(E)$ is the deformation space
of a fixed determinant $\det G_1=\det G_2$.
Hence $\Def_0(G_2)$ is smooth and
we have a surjection
\begin{equation}
\Ext^1_{{\cal O}_X}(G_2,G_2)_0 \to 
\Ext^1_{{\cal O}_D}(G_{2|D},G_{2|D})_0.
\end{equation}
Thus $\Def_0(G_2) \to \Def_0(G_{2|D})$ is submersive.
Since $G_{2|D}$ deforms to a $\mu$-stable vector bundle,
$G_2$ deforms to a torsion free $\alpha$-twisted sheaf $G_3$
such that $G_{3|D}$ is $\mu$-stable.
Then $G_3$ is a $\mu$-stable $\alpha$-twisted sheaf.
Replacing $G_1$ by $G_3^{\vee \vee}$, we assume that $G_1$ is $\mu$-stable
with respect to $H$ and $H^2((G_1^{\vee} \otimes G_1)/{\cal O}_X)=0$.
Since $(p,\rk G_1)=1$,
we have a decomposition
$$
G_1^{\vee} \otimes G_1 \cong 
{\cal O}_X \oplus 
(G_1^{\vee} \otimes G_1)/{\cal O}_X.
$$
Hence $H^2(X,(G_1^{\vee} \otimes G_1)/{\cal O}_X)=0$.
Then there is a discrete valuation ring $(R',t')$ dominating
$(R,t)$ such that
${\Bbb P}(G_1) \to X$ is lifted to a projective bundle
$P \to {\cal X}_{R'}$.
\end{proof} 

For the morphism $g:P \to {\cal X}_{R'} \to \Spec(R')$,
$L:=\Ext^1_g(T_{P/{\cal X}},{\cal O}_{{\cal X}})$ is a line bundle
on $\Spec(R')$.
Since $\Hom_g(T_{P/{\cal X}},{\cal O}_{{\cal X}})=0$,
we have a non-trivial extension
$$
0 \to g^*(L^{\vee}) 
\to {\cal P} \to T_{P/{\cal X}} \to 0.
$$
We take an \'{e}tale trivialization 
$U \to {\cal X}$ of
$P \to {\cal X}$:
$P \times_{\cal X} U \cong {\Bbb P}^m \times U$, where
$m=\rk G_1-1$.
Then the relative tautological line bundle
on ${\Bbb P} \times U$ forms a 
twisted line bundle
${\cal O}_P(1)$ on ${\cal X}$.
${\cal P} \otimes {\cal O}_P(-1)$ gives a twisted sheaf ${\cal G}$ on
${\cal X}$ with ${\cal G} \otimes_{R'}R'/t' R' \cong
G_1$.
Assume that $E \in K^\alpha(X)$ with $v(E)=v$
satisfies $v(E \otimes G_1^{\vee}) \in 
{\Bbb Q} \oplus {\Bbb Q}H \oplus {\Bbb Q}\varrho_X$.
Then we have a relative moduli scheme
$M_{({\cal X},{\cal H})}^{{\cal G}}(v) \to \Spec(R')$, which is
smooth and projective.

\begin{lem}\label{lem:reduction:mu-stable}  
For an isotropic Mukai vector $u$,
assume that $M_H^u(u) \ne \emptyset$.
Then $M_H^u(v)$ is isomorphic to the moduli space of $\mu$-stable
(twisted) sheaves on $X$. 
\end{lem}

\begin{proof}
By the proof of \cite[Cor. 2.7.3]{PerverseII},
we get the claim.
\end{proof}

Thus Proposition \ref{prop:mod-p} follows from the claim for
the case where
$H$ is general with respect to $v$.

\begin{NB}
Assume that $H$ is ample and $(p,\rk G)=1$.
Then we see that $\overline{M}_H^G(u) \ne \emptyset$.
In particular, if $H$ is general with respect to $v$,
then $M_H^u(u) \ne \emptyset$.
If ${\frak k}={\Bbb F}_q$, then
${\cal M}_H^G(u)^{ss}({\Bbb F}_{q^n})$ is independent
of a general choice of $(G,H)$.
Hence $\overline{M}_H^u(u) \ne \emptyset$.
By the proof of \cite[Cor. \ref{II-cor:K3-non-empty}]{PerverseII},
we get $M_H^u(u) \ne \emptyset$.
\end{NB}

Let $G_1$ be a local projective generator of ${\frak C}$.
We have a local projective generator $G_2$ such that
$G_2=pNG_1+G$ in $K^\alpha(X)$ and 
$G_1$ and $G_2$ belong to the same chamber for $v$, 
where $N$ is sufficiently large
(\cite[Cor. 2.4.4]{PerverseI}).
$(\rk G_2) G_2+\frac{\langle v(G_2),v(G_2) \rangle}{2}{\frak k}_x
\in K^\alpha(X)$ is isotropic.
Hence there is a primitive and isotropic Mukai vector
$u$ such that 
$(\rk G_2) v(G_2)+\frac{\langle v(G_2),v(G_2) \rangle}{2}\varrho_X
\in {\Bbb Z}u$.

\begin{lem}\label{lem:non-empty:u}
There is a $u$-twisted stable object $E$ with
$v(E)=u$.
\end{lem}

\begin{proof}
We apply Lemma \ref{lem:Azumaya-deformation}
to $G_2$ on $X$.
Replacing $T$ by a finite extension $T' \to T$,
we have a family of $u$-twisted semi-stable objects
${\cal E}_\eta$ with Mukai vectors $u$ over the generic point $\eta$
of $T$.
By the valuative criterion of properness,
we see that ${\cal E}$ can be extended to a
family of $u$-twisted semi-stable objects
${\cal E}$ with Mukai vectors $u$ over $T$.
Therefore there is a $u$-twisted semi-stable object $E$ with
$v(E)=u$.
By the proof of \cite[Cor. 1.3.3]{PerverseII},
we get our claim.
\end{proof}

Proof of Proposition \ref{prop:mod-p}:
By Lemma \ref{lem:non-empty:u} and Lemma \ref{lem:reduction:mu-stable},
we may assume that $H$ is general with respect to $v$.
Replacing $v$ by $v e^{nH}$, $n \gg 0$,
$v/\rk v-u/\rk u=(0,\xi,a)$, where ${\Bbb Q}\xi$ is sufficiently
close to ${\Bbb Q}H$.
Replacing $H$ by $\xi$, we may assume that
$v \in {\Bbb Q}u+{\Bbb Q}H+{\Bbb Q}\varrho_X$.
Then we have a relative moduli scheme
$M_{({\cal X},{\cal H})}^{{\cal G}}(v) \to \Spec(R')$, which is
smooth and projective.
Therefore the claim holds.
\qed

\begin{NB}
In order to treat the perverse coherent sheaves,
we need to construct the moduli space of
stable perverse coherent sheaves over any field.
\end{NB}

\subsection{Moduli of perverse coherent sheaves}\label{subsect:Moduli}

In \cite[Prop. 1.4.3]{PerverseI}, one of the authors 
constructed the moduli of
semi-stable perverse coherent sheaves under
the characteristic 0 assumption.
The construction is reduced to the construction of semi-stable
${\cal A}$-modules, where ${\cal A}$ is a sheaf of ${\cal O}_Y$-algebras
on a projective scheme $Y$ via Morita equivalence.
Then by using Simpson's result \cite[Thm. 4.7]{S:1} 
on the moduli of semi-stable
${\cal A}$-modules, we get the moduli space.
In this subsection, we shall remark that Simpson's construction
also works for any characteristic case by Langer's results \cite{Langer:1}.
Let $S$ be a scheme of finite type over a universally Japanese ring.
Let $(Y,{\cal O}_Y(1)) \to S$ be a flat family of polarized schemes
over $S$.
Let ${\cal A}$ be a sheaf of ${\cal O}_Y$-algebras such that
${\cal A}$ is a coherent
${\cal O}_Y$-module, which is flat over $S$.
For an ${\cal A}_s$-module $E$ on $X_s$,
we write the Hilbert polynomial of $E$ as
$$
\chi(Y_s,E(m))=\sum_i a_i(E) \binom{n+i}{i},\; a_i(E) \in {\Bbb Z}.
$$ 
If $\chi(E(m))$ is a polynomial of degree $d$, then
$E$ is of dimension $d$.
By using the Hilbert polynomial of $E$, we have the notion of 
semi-stability and also the $\mu$-semi-stability.

\begin{lem}\label{lem:type}
We take a surjective morphism
${\cal O}_Y(-m)^{\oplus N} \to {\cal A}$.
Let $E$ be a $\mu$-semi-stable ${\cal A}$-module of dimension $d$.
Then 
$$
\frac{a_{d-1}(F)}{a_d(F)} \leq \frac{a_{d-1}(E)}{a_d(E)}+m
$$
for any subsheaf $F$ of $E$. 
\end{lem}

\begin{proof}
We may assume that $F$ is $\mu$-semi-stable.
For the multiplication morphism
$\phi:F \otimes {\cal A} \to E$,
$\im \phi$ is an ${\cal A}$-submodule of $E$.
Hence 
$$
\frac{a_{d-1}(\im \phi)}{a_d(\im \phi)} \leq \frac{a_{d-1}(E)}{a_d(E)}.
$$
By our assumption, we have a surjective morphism
$F(-m)^{\oplus N} \to  \im \phi$.
Hence 
$$
\frac{a_{d-1}(\im \phi)}{a_d(\im \phi)} \geq \frac{a_{d-1}(F(-m))}{a_d(F)}
= \frac{a_{d-1}(F))}{a_d(F)}-m.
$$
Therefore the claim holds.
\end{proof}

By this lemma, the set of $\mu$-semi-stable ${\cal A}_s$-modules
$E$ on $Y_s$ ($s \in S$) with the Hilbert polynomial $P$ is bounded
by Langer's boundedness theorem.
Hence we can parametrize semi-stable ${\cal A}_s$-modules
by a quot-scheme 
${\frak Q}:=\Quot_{{\cal A}(-n)^{\oplus N}/Y/S}^{{\cal A},P}$
whose points correspond to quotient ${\cal A}_s$-modules
${\cal A}_s(-n)^{\oplus N} \to E$ with $\chi(E(x))=P(x)$, 
where $N=P(n)$.
For a purely $d$-dimensional ${\cal A}_s$-module $E$ on $Y_s$, 
 $\mu_{\max}(E)$ is the maximum of 
$\frac{a_{d-1}(F)}{a_d(F)}$ where $F$ is an ${\cal A}_s$-submodule
of $E$. 
Combining this with Langer's important result
\cite[Cor. 3.4]{Langer:1}, we have the following 
\begin{lem}\label{lem:section}
For any purely $d$-dimensional 
${\cal A}$-module $E$ on $Y_s$ $(s \in S)$,
\begin{equation}
 \frac{h^0(Y_s,E)}{a_d(E)} \leq 
 \left[\frac{1}{d!}
 \left(\mu_{\max}(E)
 +c \right)^d \right]_+,
\end{equation}
where $c$ depends only on $(Y,{\cal O}_Y(1))$, ${\cal A}$, $d$ and $a_d(E)$.
\end{lem}

Let $(R,{\frak m})$ be a discrete valuation ring $R$ and the
maximal ideal ${\frak m}$.
Let $K$ be the fractional field and $k$ the residue field.
Let ${\cal E}$ be an $R$-flat family of ${\cal A}_R$-modules
such that ${\cal E} \otimes_R K$
is pure.
\begin{lem}\label{lem:valuative}
There is a $R$-flat family
of coherent ${\cal A}_R$-modules ${\cal F}$ and a homomorphism
$\psi:{\cal E} \to {\cal F}$ such that
${\cal F} \otimes_R k$ is pure,
$\psi_K$ is an isomorphism and
$\psi_k$ is an isomorphic at generic points of $\Supp({\cal F} \otimes_R k)$.
\end{lem}

\begin{proof}
By using \cite[Lem. 1.17]{S:1},
we first construct ${\cal F}$ as a usual family of sheaves.
Then by the very construction of it, ${\cal F}$ becomes an 
${\cal A}_R$-module.
\end{proof}

\begin{NB}
It also follows by using elementary transformations 
along $Y \otimes_R k$-sheaves on $Y \otimes_R k$. 
\end{NB}

Let ${\frak Q}^{ss}$ be the open subscheme of ${\frak Q}$
consisting of semi-stable ${\cal A}_s$-modules.
By using Lemma \ref{lem:section} and Lemma \ref{lem:valuative},
we can construct the moduli of semi-stable ${\cal A}_s$-modules
as a GIT-quotient of ${\frak Q}$ by the action of $\PGL(N)$, where
$n$ is sufficiently large. Thus 
we have a qood quotient ${\frak q}:{\frak Q}^{ss} \to {\frak Q}^{ss}/\PGL(N)$
and ${\frak Q}^{ss}/\PGL(N)$ is the coarse moduli space
of semi-stable ${\cal A}_s$-modules.
\begin{thm}
We have a coarse moduli scheme of semi-stable ${\cal A}$-modules
$\overline{M}_{(Y,{\cal O}_Y(1))/S}^{{\cal A},P} \to S$, 
which is projective over $S$.
\end{thm}
Let ${M}_{(Y,{\cal O}_Y(1))/S}^{{\cal A},P}$ be the open subscheme
of stable ${\cal A}_s$-modules. Then 
by this construction, 
${\frak q}:{\frak Q}^s \to {M}_{(Y,{\cal O}_Y(1))/S}^{{\cal A},P}$
 is a principal $\PGL(N)$-bundle, where
${\frak Q}^s$ is the open subscheme parametrizing
stable ${\cal A}_s$-modules. 

Let $X$ and $Y$ be flat families of projective varieties over a scheme
$S$ of finite type over a universally Japanese ring and 
assume that $X \to S$ is smooth.
Let $\pi:X \to Y$ be a family of $S$-morphisms and $G$ a locally free
sheaf on $X$ such that $G_s$ is a local projective generator
of a family of abelian categories ${\frak C}_s \subset {\bf D}(X_s)$
as in \cite[sect. 1.3]{PerverseI}. 
Since \cite[Cor. 1.3.10]{PerverseI} holds for any base $S$, 
we have the following.
\begin{cor}\label{cor:relative-moduli}
\begin{enumerate}
\item[(1)]
We have a relative moduli stack of $G_s$-twisted semi-stable objects
with the Hilbert polynomial $P$ as a quotient stack
$[Q^{ss}/\GL(N)]$, where
$Q^{ss}$ is an open subscheme
of $\Quot_{G^{\oplus N}/X/S}^{{\frak C},P}$ 
parametrizing $G_s$-twisted semi-stable objects.
\item[(2)] 
We have a relative moduli scheme 
$\overline{M}_{(X,{\cal O}_X(1))/S}^{G,P} \to S$ of
$G_s$-twisted semi-stable objects
with the Hilbert polynomial $P$
as a GIT-quotient:
$$
\overline{M}_{(X,{\cal O}_X(1))/S}^{G,P} \cong Q^{ss}/\PGL(N).
$$
\end{enumerate}
\end{cor}

Let ${\cal Q}$ be the universal quotient on
$X \times_S Q^{ss}$.
\begin{prop}\label{prop:universal-gfamily}
Assume that there is a bounded family of locally free (twisted) 
sheaves $V_{\bullet}$
on $X$ such that $\chi((V_{\bullet})_s,{\cal Q}_s)=1$, $s \in S$.
Then there is a universal family on 
$X \times_S M_{(X,{\cal O}_X(1))/S}^{G,P}$.
\end{prop}

\begin{proof}
We set ${\cal A}:=\pi_*(G^{\vee} \otimes_{{\cal O}_X} G)$.
Then $\pi_*(G^{\vee} \otimes {\cal Q})$ is a $\GL(N)$-equivariant
${\cal A} \otimes {\cal O}_{Q^{ss}}$-module.
Hence ${\cal Q} \cong
\pi^{-1}(\pi_*(G^{\vee} \otimes {\cal Q})) 
\overset{{\bf L}}{\otimes}_{\pi^{-1}({\cal A})} G$
is $\GL(N)$-equivariant.
Then
$F_{\bullet}:=
{\bf R}p_{Q^{ss}}(V_{\bullet}^{\vee} \otimes {\cal Q})$ is a bounded complex
of $\GL(N)$-equivariant 
locally free sheaves on $Q^{ss}$, where
$p_{Q^{ss}}:X \times_S Q^{ss} \to Q^{ss}$ is the projection.
Then ${\cal Q} \otimes p_{Q^{ss}}^*(\det(F_{\bullet})^{\vee})$ 
is a $\PGL(N)$-linearized
object on $X \times_S Q^{ss}$.
Let $Q^s$ be the open subscheme parametrizing stable objects. 
Since $Q^{s} \to Q^s/\PGL(N)$ is a principal bundle,
$({\cal Q} \otimes p_{Q^{ss}}^*(\det(F_{\bullet})^{\vee}))_{|X \times_S Q^s}$ 
is the pull-back of a 
family of perverse coherent sheaves ${\cal E}$ on 
$X \times_S M_{(X,{\cal O}_X(1))/S}^{G,P}$, which gives a universal family.
\begin{NB}
${\cal Q}$ is represented by a two term complem
${\cal I} \to (G(-n)\otimes {\cal O}_{Q^{ss}})^{\oplus N}$.
Since $\pi_*(G^{\vee} \otimes {\cal I}) \subset 
({\cal A}(-n)\otimes {\cal O}_{Q^{ss}})^{\oplus N}$,
$\pi_*(G^{\vee} \otimes {\cal I})$ is $\GL(N)$-linearized subsheaf of 
$({\cal A}(-n)\otimes {\cal O}_{Q^{ss}})^{\oplus N}$.
Hence ${\cal I} \cong \pi^{-1}(\pi_*(G^{\vee} \otimes {\cal I})) 
\otimes_{\pi^{-1}({\cal A})} G$ is $\GL(N)$-linearized and 
we have a commutative diagram
\begin{equation}
\begin{CD}
\pi^{-1}(\pi_*(G^{\vee} \otimes {\cal I})) 
\otimes_{\pi^{-1}({\cal A})} G @>>> 
\pi^{-1}(({\cal A}(-n) \otimes {\cal O}_{Q^{ss}})^{\oplus N})
\otimes_{\pi^{-1}({\cal A})} G \\
@VVV @VVV \\
{\cal I} @>>> ({\cal A}(-n)\otimes {\cal O}_{Q^{ss}})^{\oplus N}.
\end{CD}
\end{equation}  
Therefore ${\cal I} \to 
({\cal A}(-n)\otimes {\cal O}_{Q^{ss}})^{\oplus N}$ is $\GL(N)$-equivariant.
\begin{NB2}
Since
$\pi^{-1}(\pi_*(G^{\vee} \otimes {\cal Q})) 
\overset{{\bf L}}{\otimes}_{\pi^{-1}({\cal A})} G$
is a complex, we cannot apply the same argument to ${\cal Q}$ directly.
\end{NB2}
\end{NB}
\end{proof}

\begin{cor}\label{cor:universal-family}
Assume that $(X,H)$ is a polarized
$K3$ or an abelian surface over a scheme $S$.
For $v=(r,\xi,a) \in {\Bbb Z} \oplus \NS(X/S) \oplus {\Bbb Z}$,
if $\gcd(r,(\xi,D),a)=1$, $D \in \NS(X/S)$, then
there is a universal family on $X \times_S M_H^\beta(v)$.  
\end{cor}

\end{document}